\documentclass[11pt]{amsart}

\usepackage{amsmath, amsthm, amssymb}
\numberwithin{equation}{section}

\usepackage{amsthm}
\usepackage{hyperref}
\usepackage[all]{xy}
\usepackage[usenames,dvipsnames]{xcolor}
\usepackage{mleftright}

\usepackage{overpic}

\newtheorem{theo}{Theorem}[section]
\newtheorem{lemma}[theo]{Lemma}
\newtheorem{prop}[theo]{Proposition}

\newtheorem{coro}[theo]{Corollary}

\theoremstyle{definition}
\newtheorem{defn}[theo]{Definition}
\newtheorem{warn}[theo]{Warning}
\newtheorem{eg}[theo]{Example}

\newtheorem{rmk}[theo]{Remark}

\usepackage[left=2.5cm,right=2.5cm,top=3.5cm]{geometry}

\pdfoutput=1
\usepackage{amsmath}
\usepackage{amssymb}
\usepackage{wrapfig}
\usepackage{bbm}
\usepackage{arydshln}
\usepackage{multirow}
\usepackage{mathtools}
\usepackage{blkarray}
\usepackage{setspace}
\usepackage{dsfont}
\usepackage{float}
\usepackage{subcaption}
\usepackage{braket}
\usepackage[utf8]{inputenc}
\usepackage[english]{babel}
\usepackage[all]{xy}
\usepackage{tikz-cd}
\usepackage{mathtools}
\usepackage{tikz}
\usepackage{enumerate}

\usepackage{ytableau}

\makeatletter
\newcommand\mathcircled[1]{%
	\mathpalette\@mathcircled{#1}%
}
\newcommand\@mathcircled[2]{%
	\tikz[baseline=(math.base)] \node[draw,circle,inner sep=3pt] (math) {$\m@th#1#2$};%
}
\makeatother

\newcommand{\bb}{\mathbb}
\newcommand{\mc}{\mathcal}

\newcommand{\mf}{\mathfrak}
\newcommand{\Rees}{{\textup{Rees}}}
\newcommand{\Def}{{\textup{Def}}}
\newcommand{\fDef}{{\mathfrak{Def}}}

\newcommand{\spe}{{\textup{sp}}}

\newcommand{\DDD}{{{\bb D}}}

\newcommand{\Har}{{\textup{Har}}}

\newcommand{\AQ}{{\textup{AQ}}}

\newcommand{\cEnd}{{\textup{End}}}
\newcommand{\OEnd}{{\textup{End}}}
\newcommand{\iEnd}{{\mathcal{E}nd}}

\newcommand{\bL}{{\mathbb{L}}}

\newcommand{\Ph}{{\mathcal{P}_\hbar}}
\newcommand{\Qh}{{\mathcal{Q}_\hbar}}
\newcommand{\gPQ}{{\mathfrak{g}_{\mc P,\mc Q}}}
\newcommand{\gPQh}{{\mathfrak{g}_{\mc P_\hbar,\mc Q_\hbar}}}
\newcommand{\gPQvh}{{\mathfrak{g}^{\varphi_\hbar}_{\mc P_\hbar,\mc Q_\hbar}}}
\newcommand{\fph}{{[\![\hbar]\!]}}
\newcommand{\Kh}{{\K[\hbar]}}
\newcommand{\Kfph}{{\mathbb{K}[\![\hbar]\!]}}
\newcommand{\gd}{{\textup{gr}}}
\newcommand{\ob}{{\textup{ob}}}

\newcommand{\Phm}{{\mathcal{P}_{\hbar /\hbar^{m+1}}}}
\newcommand{\Qhm}{{\mathcal{Q}_{\hbar /\hbar^{m+1}}}}

\newcommand{\N}{{\mathbb{N}}}

\newcommand{\D}{{\mathcal{D}}}

\newcommand{\SH}{{\textup{Sh}}}

\DeclareMathOperator{\Hom}{Hom}
\DeclareMathOperator{\Vect}{Vect}

\DeclareMathOperator{\Ext}{Ext}
\newcommand{\Mod}{{\textup{-Mod}}}

\newcommand{\K}{{\bb K}}
\newcommand{\C}{{\bb C}}
\newcommand{\ik}{{\bb K}}

\newcommand{\im}{{\textup{im}}}

\newcommand{\End}{{\textup{End}}}
\newcommand{\colim}{{\textup{colim}}}
\newcommand{\coker}{{\textup{coker}}}
\newcommand{\into}{{\ \hookrightarrow\ }}
\newcommand{\onto}{{\ \twoheadrightarrow\ }}

\newcommand{\Obs}{{\textup{Obs}}}

\newcommand{\Diff}{{\textup{Diff}}}

\newcommand{\CE}{{\textup{CE}}}
\newcommand{\Der}{{\textup{Der}}}
\newcommand{\Spec}{{\textup{Spec\ }}}

\newcommand{\BDh}{{\textup{BD}_\hbar}}

\newcommand{\BDhat}{{\textup{BD}_{\hat{\hbar}}}}

\newcommand{\Quant}{{\textup{Quant}}}
\newcommand{\fQuant}{{\mathfrak{Quant}}}

\newcommand{\dR}{{\textup{dR}}}

\newcommand{\triv}{{\textup{triv}}}

\newcommand{\Maps}{{\textup{Maps}}}

\newcommand{\Ind}{{\textup{Ind}}}

\newcommand{\Sh}{{\textup{Sh}}}

\newcommand{\Alg}{{\textup{Alg}}}

\newcommand{\Pois}{{\textup{Pois}}}
\newcommand{\Comm}{{\textup{Comm}}}
\newcommand{\Lie}{{\textup{Lie}}}

\newcommand{\ch}{{\textup{ch}}}
\newcommand{\boxt}{{{\boxtimes }}}

\newcommand{\fset}{{\textup{fSet}}}

\newcommand{\un}{{\textup{un}}}

\newcommand{\pt}{{\textup{pt}}}

\newcommand{\q}{{\textup{q}}}

\newcommand{\col}{{{\textup{col }}}}

\newcommand{\Homi}{{\mc H\textup{om}}}
\newcommand{\coh}{{\textup{coh}}}

\newcommand{\ash}{{\textup{\mbox{!`}}}}

\newcommand{\Op}{{\textup{Op}}}

\newcommand{\Ass}{{\textup{Ass}}}

\newcommand{\Ss}{{{\mf{S}}}}

\newcommand{\MC}{{{\textup{MC}}}}
\newcommand{\fMC}{{{\mathfrak{MC}}}}

\newcommand{\ev}{{\textup{ev}}}

\newcommand{\z}{{\textup{z}}}

\newcommand{\CC}{{\textup{CC}}}

\newcommand{\Vir}{{\textup{Vir}}}

\newcommand{\PV}{{{\textup{PV}}}}

\newcommand{\DG}{{\textup{DG}}}

\newcommand{\Sch}{{\textup{Sch}}}

\newcommand{\cl}{{\text{cl}}}

\DeclareMathOperator{\QC}{QCoh}

\newcommand{\Dd}{{\textup{D}}}
\newcommand{\DD}{{D}}

\newcommand{\MM}{{\textbf{M}}}

\newcommand{\Sym}{\textup{Sym}}
\newcommand{\id}{{\textup{id}}}

\newcommand{\gr}{\textup{gr}}

\newcommand{\Z}{\bb Z}
\newcommand{\DGVect}{\textup{DGVect}}

\newcommand{\Homii}{{\underline{\textup{Hom}}}}

\newcommand{\g}{{\mf{g}}}

\newcommand{\del}{{{\partial}}}

\newcommand{\e}{{\varepsilon}}

\def\IA{\mathbb{A}}

\def\IN{\mathbb{N}}

\def\IZ{\mathbb{Z}}
\def\IG{\mathbb{G}}

\def\CC{{\mc C}}

\def\FF{{\mc F}}

\def\II{{\mc I}}
\def\JJ{{\mc J}}

\def\MM{{\mc M}}

\def\OO{{\mc O}}
\def\PP{{\mc P}}
\def\QQ{{\mc Q}}

\def\VV{{\mc V}}

\def\YY{{\mc Y}}

\def\dim{{\rm dim}}

\def\ch{{\rm ch}}

\def\gf{\mathfrak{g}}

\def\hf{\mathfrak{h}}

\def \mf{\mathfrak}

\def \id{\mathbbm{1}}

\def \Op{\mathrm{Op}}

\def \ie{{\it i.e.}}

\newcommand{\fibre}[1]{\underset{#1}{\times}}
\newcommand{\tensor}[1]{\underset{#1}{\otimes}}

\def \vph{\varphi_{\hat \hbar}}
\def \hhb{{\hat \hbar}}

\title{On the deformation theory of chiral quantizations}
\author{Dylan Butson and Sujay Nair}
\date{}

\begin{document} 

\begin{abstract}
	We give an operadic approach to deformation quantization of vertex Poisson algebras, a chiral analogue of the traditional problem of deformation quantization of Poisson algebras.
	Our main result is an order-by-order deformation-obstruction theory for such quantizations, controlled by the chiral analogue of Poisson cohomology. In the special case of chiral quantizations of affine symplectic varieties, quantizations of the vertex Poisson algebras of functions on their arc spaces, we prove that this deformation-obstruction theory is controlled by their de Rham cohomology. As another application, we prove that the boundary Virasoro minimal models are rigid under deformations.
\end{abstract}
\maketitle

\begingroup
\hypersetup{hidelinks}
\tableofcontents
\endgroup


\section{Introduction}

Let $A$ be a Poisson algebra, and recall that a \emph{quantization} of $A$ is an $\hbar$-adic associative algebra $A_\hbar$ over the ring of formal power series $\ik\fph$ such that the central fibre $A_0=A_\hbar\otimes_{\ik\fph}\ik_0$ is commutative, and thus canonically Poisson, with an isomorphism of as Poisson algebras $A_0\cong A$. Similarly, for $R$ a vertex Poisson algebra, a \emph{quantization} of $R$ is an $\hbar$-adic vertex algebra $R_\hbar$ such that the central fibre $R_0 = R_\hbar\otimes_{\ik\fph}\ik_0$ is a commutative vertex algebra, and thus canonically a vertex Poisson algebra, with an isomorphism of vertex Poisson algebras $R_0\cong R$.

In this work we view the quantization problem through the lens of deformation theory. In his ICM address \cite{Tamarkin03}, Tamarkin introduced the deformation-obstruction theory for chiral algebras, in the sense of Beilinson--Drinfeld \cite{BD1}, a generalization of vertex algebras defined over smooth curves $X$ which reduces to the traditional definition in the case $X=\bb A^1$. More recently, in \cite{BDSHK19}, Bakalov--de Sole--Heluani--Kac developed this theory in detail in the case of vertex algebras.

In fact, Tamarkin conjectures an analogue of the celebrated Kontsevich formality theorem \cite{Kontsevich1997}, and states that the importance of this conjecture can be seen from the corollary that deformations of chiral algebras are equivalent to those of coisson algebras, the analogous generalization of vertex Poisson algebras to other curves. The main theorem of this work achieves a somewhat similar goal, parameterizing the space of quantizations of a coisson algebra $R$, or in particular a vertex Poisson algebra, in terms of the cohomology $H^\bullet_{c}(R)$ of its operadic deformation-obstruction complex:

\begin{theo}\label{introthm:order by order}
Let $R$ be a vertex Poisson algebra and let $R_{\hbar/\hbar^m}$ be a quantization of $R$ defined modulo $\hbar^m$ for $m\geq 1$. Then there exists a quantization of $R$ defined modulo $\hbar^{m+1}$, agreeing with $R_{\hbar/\hbar^m}$ modulo $\hbar^m$, if and only if a certain obstruction class determined by $R_{\hbar/\hbar^m}$ vanishes, that is,
\[
[\Obs(R_{\hbar/\hbar^m})]=0 \ \in H^2_{c}(R)~.
\]
Further, if the obstruction class vanishes, the moduli space of such extensions, up to isomorphism, is a torsor for $H^1_{c}(R)$.
\end{theo}

The space of functions $\mc O(\mc J Y)$ on the arc space $\mc J Y$ of an affine Poisson variety $Y$ is naturally a vertex Poisson algebra, and this constitutes a rich source of examples of such. We call a quantization of the vertex Poisson algebra $\mc O(\mc J Y)$ a \emph{chiral quantization} of $Y$, following \cite{Ar}, though our definition is closer to that of strict chiral quantization in \textit{loc.\ cit.} Chiral quantization has become increasingly relevant in the physics literature; see for example \cite{Beem:2017ooy}.

In the case that $Y$ is symplectic, we show that the order-by-order deformation-obstruction theory for chiral quantizations of $Y$ is controlled by its de Rham cohomology:
\begin{theo}\label{introthm:symplectic order by order}
	Let $Y$ be an affine symplectic variety, and let $R_{\hbar/\hbar^m}$ be a quantization of $\mc O(\mc J Y)$ defined modulo $\hbar^m$ for $m\geq 1$.
Then there exists a quantization of $\mc O(\mc J Y)$ defined modulo $\hbar^{m+1}$, agreeing with $R_{\hbar/\hbar^m}$ modulo $\hbar^m$, if and only if the obstruction class vanishes, that is,
	\[
	[\Obs(R_{\hbar/\hbar^m}) ]=0\in H^4_{\dR}(Y)~.
	\]
	Further, if the obstruction class vanishes, the moduli space of such extensions, up to isomorphism, is a torsor for $H^3_{\dR}(Y)$.
\end{theo}
This result is evocative of the classical work \cite{LecomteWilde83, Fedosov1994, Deligne, BeK} on the deformation quantization of symplectic manifolds or varieties, which provide a full parametrization of the space of quantizations of $Y$ in terms of $H^2_\dR(Y)$. It also provides a partial generalization of the results of \cite{GMS1, GMS2} and \cite{BD1} on vertex algebras of chiral differential operators, which implicitly classify the chiral quantizations of $Y$ in the case $Y=T^*W$ is a global cotangent bundle.

Although our result is abstractly weaker, in the sense that it fails to provide a trivialization of the iterated torsor describing formal quantizations, it is often still strong enough to draw the necessary conclusions about full spaces of quantizations in many applications. 



Finally, we discuss the application of our results to boundary Virasoro minimal models: it has long been folklore that rational conformal field theories are isolated, and so in particular their vertex algebras are rigid under deformations. The boundary Virasoro minimal models are rational vertex algebras that are quantizations of a certain vertex Poisson algebra structure on the arc space of the thickened point $\Spec \ik[T]/T^n$. Using the technical machinery developed in this paper, we prove the following result on the rigidity of these vertex algebras under deformations:
\begin{theo}\label{introthm:minimal model}
Let $\MM_n$ be the boundary Virasoro minimal model for $(p,q) = (2,2n+1)$, the simple quotient of the Virasoro algebra $\Vir^c$ at $c = c_{p,q} = 1- \frac{6pq}{(p-q)^2}$. Then $\MM_n$ is rigid, that is, it does not admit non-trivial deformations.
\end{theo}

\subsection{Relation to previous work}

The problem of quantizing vertex Poisson algebras was first introduced by Li in \cite{Li04}, following the analogy with deformation quantization of Poisson algebras.

In recent years, chiral quantization has become increasingly relevant in the physics literature. In the SCFT/VOA correspondence of \cite{Beem4}, the resulting vertex algebras are widely expected to be chiral quantizations of the Higgs branch of the four-dimensional theory. Similarly, in the boundary vertex algebra construction of \cite{CosG} the resulting vertex algebras are chiral quantizations---albeit in a looser sense---of the Higgs and Coulomb branches.

Quantizations of vertex Poisson algebras also play an important role in the (micro)localization theory of vertex algebras; see for example \cite{Arakawa2015:local,Arakawa:2023cki}. Following \cite{KashiwaraRouquier,DoddKremnizer}, the microlocalization of a vertex algebra is defined by considering it as a (sheafy) quantization of the sheaf of functions on the arc space of its associated scheme. Localization has been used to great effect, recently, to produce free-field realizations of a whole host of vertex algebras.

The deformation theory of chiral algebras was first introduced by Tamarkin in his ICM address \cite{Tamarkin03}, as mentioned above. In a series of works \cite{BDSHK19, BDSHK20,BDHKV21} (subsets of) Bakalov, de Sole, Heluani, Kac, and Vignoli give a very concrete account of the deformation-obstruction complex for vertex algebras and vertex Poisson algebras.

A different cohomology theory for (graded) vertex algebras was developed by Huang in \cite{Huang:2010ud,Huang:2010yv}, with the property that its second cohomology group parameterizes first order deformations, up to isomorphism. This cohomology theory has been studied by \cite{Kovalchuk:2024rwd} in the case of freely generated vertex algebras to parameterize their first order deformations. Recently, \cite{Linshaw26} used Huang's cohomology theory to show that simple affine vertex algebras with positive integral levels were rigid. Like our results on the minimal models, this gives further evidence, in the affirmative, for the belief that strongly rational vertex algebras are rigid.

As mentioned, our Theorem \ref{introthm:symplectic order by order} is reminiscent of the classification of quantizations of symplectic varieties of \cite{BeK}, following the approach of Fedosov \cite{Fedosov1994} for symplectic manifolds.
Optimistically, one could hope for an analogous construction for chiral quantizations, with an analogue of the noncommutative period map of \cite{BeK} giving an explicit description of the moduli space of chiral quantizations. While this is not pursued in the current text, we believe that this could be a fruitful direction in future.

There have also been several approaches to the deformation theory of vertex algebras, and the quantization of vertex Poisson algebras, in terms of the BV formalism in mathematical physics. In \cite{Li2016}, Si Li proves that the quantum master equation for deformations of two-dimensional holomorphic field theories in the BV formalism is equivalent to the Maurer-Cartan equations for deformations of vertex algebras.

More recently, in \cite{Khan:2025rah}, the authors define a three-dimensional classical field theory associated to a vertex Poisson algebra, in analogy with the construction of the two dimensional Poisson sigma model associated to a Poisson algebra. Both these classical field theories can also be defined more tersely in terms of the universal bulk theory construction of \cite{ButsonYoo}. Optimistically, an analysis of the bulk-boundary Feynman diagrams in a quantization of this theory could give a quantization of the input vertex Poisson algebra structure, in analogy with the approach of Kontsevich \cite{Kontsevich1997}. 

Relatedly, and again optimistically, we hope that an analogue of the Kontsevich formality theorem could be proved in the chiral setting. In particular, we hope that an analogue of Tamarkin's proof of Kontsevich formality \cite{Tamarkin, Hinich} could also apply to the chiral setting.



\subsection{Outline}

We begin in Section \ref{sec:operad prelim} with some recollections on operads, which form the backbone of many of our arguments. We also recall the classical examples of the operads governing familiar algebraic structures, such as associative, commutative, Lie, and Poisson algebras.

In Section \ref{sec:dmodules}, we recall some preliminaries regarding D-modules on smooth algebraic varieties, with a particular focus on D-modules over an algebraic curve. In \ref{ssec:chiral algebras} we recall the various pseudo tensor structures on the category of D-modules on an algebraic curve $X$, and define chiral and coisson algebras.

In Section \ref{sec:vertex algebras}, we recall the, perhaps more familiar, notions of vertex algebras and vertex Poisson algebras. Following \cite{BD1,BDSHK19}, we also provide an operadic definition of these and relate them to the weakly translation invariant D-modules on $\IA^1$. To conclude this section we present an initial definition of the quantization of a vertex Poisson algebra, following \cite{Arakawa2015:local}.

In Section \ref{defsec} we introduce the DG Lie algebra $\gf_{\PP,\QQ}$ associated to a pair of $\PP$ and $\QQ$,  whose Maurer--Cartan elements parameterize the $\PP$-algebra structures internal to $\QQ$, and introduce the deformation-obstruction complex associated to a fixed $\PP$-algebra internal to $\QQ$.

We then review the deformation obstruction-complexes for familiar algebras and link them to well-known cohomology theories such as Hochschild and Chevalley--Eilenberg. We also relate the deformation-obstruction complex of a commutative algebra with its Andr\'e--Quillen cohomology, and exposit some properties of the natural bicomplex underlying the deformation-obstruction complex of a Poisson algebra.

In Section \ref{ssec:def chiral algebras}, we review the deformation-obstruction complexes for chiral and coisson algebras. As in the case of Poisson algebras, the deformation-obstruction complex for coisson algebras comes from a bicomplex, and shares many analogous properties.

In Section \ref{chsec}, we review the role of the deformation-obstruction complex in controlling the higher order deformation-obstruction theory of algebras. In particular, we prove our main theorem on the order-by-order deformation theory of algebras, which we later adapt to the setting of quantizations.

In Section \ref{quantsec}, we prove Theorem \ref{introthm:order by order}, which is Theorem \ref{DefObsQuantChThm} in the text. We first review the traditional case of deformation quantization of Poisson algebras and compare our order-by-order result, Theorem \ref{DefObsQuantThm}, to that of Fedosov quantization.

We then proceed to the quantization of coisson algebras. We introduce a filtration on the chiral operad whose corresponding Rees object is used to describe quantizations of a vertex Poisson algebras. After some technical digressions on the splitting of this filtration, we provide the proof of Theorem \ref{DefObsQuantChThm}.

In Section \ref{varsec}, we relate the variational cohomology of the arc space of a symplectic variety to its de Rham cohomology. This requires a digression on internal $\Hom$ objects in the category of D-modules, as well as introducing internal versions of the classical operad and the variational cohomology.

We conclude in Section \ref{sec:examples} with various the applications of our results. In particular, in Section \ref{sympsec} we prove Theorem \ref{introthm:order by order}, which is Theorem \ref{symptheo} in the text, and in Section \ref{ssec:minimal models}, we prove Theorem \ref{introthm:minimal model}, which is Theorem \ref{thm:rigidity of minimal model} in the text.

\subsection*{Acknowledgements}

The authors would like to thank Christopher Beem, whose advice during the early stages of this work was invaluable. We would also like to thank Tomoyuki Arakawa and Ivan Losev for useful discussions.



\section{Operadic preliminaries}\label{sec:operad prelim}

Let $\mc C$ be a symmetric monoidal category with monoidal structure $\otimes:\mc C\times \mc C\to \mc C$ and unit ${\bf 1}_\CC\in \CC$. In applications, we will essentially always have $\mc C=\K\Mod$ or $\mc C = (S\Mod, \otimes_S)$ for $S$ a commutative $\K$-algebra, with $\K$ a field of characteristic zero, or their differential graded (DG) variants.

\subsection{Plethystics}

Let $\frak S_n$ denote the symmetric group on $n$ letters. An $\frak S$-module (internal to $\CC$) also known as a collection, $M$, is a sequence $M=(M(0),M(1),\dots,M(n),\dots)$, where, for each $n\in\IN$, $M(n)$ is a representation of $\frak S_n$ internal to $\CC$. The $\frak S$-modules form a category, $\frak S\Mod$ with morphisms,  
\[
\Hom_{\Ss\Mod}(M,N) = \bigoplus_n \Hom_\CC(M(n),N(n))^{\Ss_n}~,
\]
for $M,N\in \Ss\Mod$. There are three notions of monoidal structure that one may be define on $\Ss\Mod$.

Given two $\frak S$-modules $M,N$, their \emph{tensor} product is the $\frak S$-module
\[
(M\tensor{\frak S} N)(n) \coloneq \bigoplus_{i+j=n}\Ind_{\frak S_i \times \frak S_j}^{\frak S_n} M(i)\otimes N(j)
\]
Then, $(\frak S\Mod, \tensor{\frak S})$ is symmetric monoidal with tensor unit $({\bf 1}_\CC,0\dots)$, the trivial $\frak S$-module.

Given two $\frak S$-modules $M,N$ their \emph{Hadamard} tensor product is the $\frak S$-module
\[
(M\otimes^H N)(n) \coloneqq M(n)\otimes N(n)~,
\]
with $\frak S_n$ acting diagonally. Then, $(\frak S\Mod, \otimes^H)$ is symmetric monoidal with tensor unit given by the collection $(\bf 1_\CC,1_\CC,\dots)$.

Given two $\frak S$-modules $M,N$ their \emph{composite} product is the $\frak S$-module
\[
M\circ N \coloneq \bigoplus_{k\in\IN} \bigg(M(k)\otimes N^{\tensor{\Ss} k}\bigg)_{\frak S_k}~,
\]
where $\frak S_k$ acts diagonally and acts on $N^k$ by permuting factors. Let $\id_\Ss = (0,{\bf 1}_\CC,0,\dots)\in \Ss\Mod$, then for any $\frak S$-module $M$,
\[
\id_{\frak S}\circ M \cong M\circ \id_\Ss\cong M~,
\]
and $(\frak S\Mod,\circ, \id_{\frak S})$ is a unital, monoidal (but not symmetric) category.

\subsection{Operads}\label{Opsec}

Let $\fset$ denote the category whose objects are finite sets with the obvious notion of morphisms between them. For a morphism $\pi:I\rightarrow J$, we write $I_j=\pi^{-1}(j)$ for the fibre of $\pi$ above $j\in J$.

\begin{defn}\label{opdefn} A (symmetric, coloured) \emph{operad} $\mc O$ in $\mc C$ is:
	\begin{enumerate}[(i)]
		\item A collection $\col\mc O$, elements of which are called \emph{colours} or \emph{objects} of $\mc O$,
		\item For each finite set $I\in\fset$ and indexed collection of objects $\{c_i\}_{i\in I}$ and $d$ of $\mc O$, an object $\mc O(\{c_i\}_{i\in I},d)\in \mc C$ called the \emph{multilinear operations} in $\mc O$.
		\item For each map of finite sets $I\to J$, and indexed collections of objects $\{c_i\}_{i\in I}$, $\{d_j\}_{j\in J}$ and $e$, a morphism
		$$ \bigotimes_{j\in J} \mc O(\{c_i\}_{i\in I_j}, d_j)\otimes \mc O(\{d_j\}_{j\in J},e) \to \mc O(\{c_I\}_{i\in I},e)$$
		of objects of $\mc C$ called the \emph{composition law} in $\mc O$.
		\item For each object $c\in \mc C$, a morphism $\id_c\in \mc O(c,c)$ called the \emph{identity map} on $c$, which is both a left and right unit for the composition law.
		\item For each sequence of maps $I\rightarrow J \rightarrow K$, and indexed collections of objects $\{c_i\}_{i\in I}$, $\{d_j\}_{j\in J}$, $\{e_k\}_{k\in K}$ and $f$, compatible commutativity data for the diagram
		\[
		\begin{tikzcd}
		\tensor{j\in J} \OO(\{c_i\}_{i\in J},d_j)\otimes \tensor{k\in K}\OO(\{d_j\}_{j\in J_k},e_k)	\otimes \OO(\{e_k\}_{k\in K},f) \ar[r] \ar[d] & \tensor{k\in K} \OO(\{c_i\}_{i\in I_k},e_k)\otimes \OO(\{e_k\}_{k\in K},f) \ar[d] \\
		\tensor{j\in J} \OO(\{c_i\}_{i\in I_j},d_j)\otimes \OO(\{d_j\},f) \ar[r] & \OO(\{c_i\}_{i\in I},f)
		\end{tikzcd}
		\]
		and similarly for higher arity compositions.
	\end{enumerate}
	A map $\varphi: \mc O\to \mc O'$ of operads in $\mc C$ is:
	\begin{enumerate}[(i)]
		\item A map $\col\mc O \to \col \mc O'$
		\item For each map of finite sets $I\to J$, and indexed collections of objects $\{c_i\}_{i\in I}$ and $d$ of $\mc O$, a morphism
		$$ \mc O(\{c_i\}_{i\in I_j}, d) \to \mc O'(\{\varphi(c_i)\}_{i\in I_j}, \varphi(d))$$ 
		such that $\id_c$ maps to $\id_{\varphi(c)}$ for each $c\in \col \mc O$.
		\item For each map of finite sets $I\to J$, and indexed collections of objects $\{c_i\}_{i\in I}$, $\{d_j\}_{j\in J}$ and $e$, the commutativity of the diagram
		\[
			\begin{tikzcd}
				\tensor{j\in J} \mc O(\{c_i\}_{i\in I_j}, d_j)\otimes \mc O(\{d_j\}_{j\in J},e) \ar[r] \ar[d] & \mc O(\{c_i\}_{i\in I},e) \ar[d] \\
				\tensor{j\in J} \mc O'(\{\varphi(c_i)\}_{i\in I_j}, \varphi(d_j))\otimes \mc O(\{\varphi(d_j)\}_{j\in J},\varphi(e)) \ar[r] & \mc O'(\{\varphi(c_i)\}_{i\in I},\varphi(e)) 
			\end{tikzcd}
		\]
	\end{enumerate}
	
	The collection of operads in $\mc C$ thus defines a category, denoted $\Op(\mc C)$. Let $\Alg_{\mc O}(\mc O')$ denote the space $\Hom_{\Op(\mc C)}(\mc O,\mc O')$ of maps of operads $\mc O\to \mc O'$.
	
\end{defn}

For an operad $\mc O$ with a single colour $\col \mc O=\{c\}$, we use the notation 
\[
\mc O(n)=\mc O(\{c\}_{i\in\{1,\dots,n\}},c)~.
\]
Note that $\mc O(n)$ is an $\Ss_n$-module internal to $\mc C$, and so the sequence, $\OO = (\OO(1),\OO(2),\dots)$, is an $\Ss$-module. 

\begin{rmk}\label{opsingcolrmk} 
Indeed, operads on a single colour admit a slicker definition: a single coloured operad is equivalent to a unital, monoidal object in the monoidal category $(\Ss\Mod,\circ)$. Namely, a single coloured operad is a tuple $(\OO, \mu, u)$, where $\OO \in \Ss\Mod$ and $\mu:\OO\circ \OO\rightarrow \OO$ is an associative product with unit $u:\id_\Ss\rightarrow \OO$.
\end{rmk}

A natural source of coloured operads are enriched symmetric monoidal categories.
\begin{eg}\label{symmonoidcatopeg} Let $\D$ be a symmetric monoidal category enriched over $\mc C$. Then $\D$ defines an operad $\mc O_{\D}\in \Op(\mc C)$ in $\mc C$, with objects given by objects of $\D$ and multilinear operations given by
	$$ \mc O_{\D} (\{C_i\}_{i\in I},D) =\Hom_{\D}(\otimes_{i\in I} C_i,D) \ .$$
	In this case, we abbreviate $\Alg_{\mc O}(\mc O_{\D})=\Alg_{\mc O}(\D)$. In particular, for any object $D \in \mc D$ we let $\cEnd_{\mc D}(D)$ denote the full suboperad on the single colour $D$, with space of $n$-ary operations given by
	\[ \cEnd_{\mc D}(D)(n) = \Hom_{\mc D} (D^{\otimes n}, D)\ .   \]
\end{eg}


\begin{eg}Following Example \ref{symmonoidcatopeg} above, if $\mc C$ is a closed monoidal category, then it defines an operad in itself and thus for any operad $\mc O$ in $\mc C$ we have a canonical category $\Alg_{\mc O}(\mc C):=\Alg_{\mc O}(\mc O_{\mc C})$ of algebras over $\mc O$ in $\mc C$.
	
	Furthermore, if $\mc O$ is an operad internal to $\mc C$ on a single colour, then an object $A\in \Alg_{\mc O}(\mc C)$ is equivalent to an object $A \in \mc C$ together with maps in $\mc C$
\begin{equation}\label{algopconceqn}
		 \mc O(n) \to \Homi_{\mc C}( A^{\otimes n} , A)  
\end{equation}
	such that the composition of operations in $\mc O$ is mapped to the internal composition of multi-linear maps in the closed symmetric monoidal category $\mc C$.
\end{eg}

\begin{rmk}
	We think of the multilinear maps from $A$ to itself induced by the image of the morphism in Equation \ref{algopconceqn} as equipping the object $A$ with an algebraic structure---hence the notation.
	These multilinear maps specify all possible $n$-ary operations on $A$ that are built up from some set of elementary operations, modulo any relations imposed on them.
\end{rmk}

%


\subsubsection{Generators and relations}
Since every single coloured operad has an underlying $\Ss$-module, we have a forgetful functor
\[
\Op(\CC)^\textup{s.\ col} \rightarrow \Ss\Mod~.
\]
The left adjoint to this functor is the free-operad construction:
\begin{defn}\label{freeopeg} 
The free operad on $E\in \Ss\Mod$, $\mc \FF(E)\in\Op(\CC)$, is a single coloured operad characterized by the property that for any single coloured operad, $\OO\in\Op(\CC)$,
\[
\Alg_{\mc F(E)}(\mc O) = \Hom_{\Ss\Mod}(E,\OO)~.
\]
The free operad admits an inductive construction, with $\FF(E) = \colim_{m\in \IN} \FF_m(E)$, where 
\[
\FF_0(E) = \id_\Ss~, \text{ and } \FF_m(E) = \id_\Ss\oplus (E\circ \FF_{m-1}(E))  \text{ for } m>0~.
\]
\end{defn}

We say a single coloured operad $\mc O$ is generated over a collection, $E$, if there exists some $\Ss$-submodule $R\hookrightarrow \FF(E)$, generating an operadic ideal $(R)\subset \FF(E)$, such that
\[
\OO = \FF(E)/(R)~.
\]
In particular, an operad is \emph{binary} if it is generated by a collection of the form, $E=(0,0,E(2),0\dots)$, \ie, it is generated by binary operations.

\subsubsection{Pseudo-tensor categories}

Let $\D$ be a symmetric monoidal category enriched over $\CC$. From Example \ref{symmonoidcatopeg} we know that $\D$ defines an operad, $\OO_\D$ in $\CC$. Now let $\D'\subset \D$ be a subcategory that is not necessarily closed under the monoidal structure of $\D$. While $\D'$ is not a monoidal subcategory, it still defines a suboperad $\OO_{\D'}\into \OO_\D$ with spaces of operations
\[
\mc O_{\D'}(\{c_i\}_{i\in I},d)=\Hom_{\D}(\otimes_{i\in I} c_i,d) \ .
\]
Again, we abbreviate $\Alg_{\mc O}(\mc O_{\D'})=\Alg_{\mc O}(\D')$.

One should therefore think of a general (coloured, symmetric) operad as a natural generalization of a symmetric monoidal category, \ie, the correct gadget internal to which we can define various kinds of algebras. For this reason coloured symmetric operads are also called pseudo-tensor categories, for example in \cite{BD1}.

Our primary objects of consideration will be algebraic structures of the form $A\in \Alg_{ \PP}(\QQ)$ where $\PP$ is an operad with a single colour encoding a familiar algebraic structure (such as Lie or Poisson algebras), and $\QQ$ is a coloured operad enhancing some familiar category of interest, such as $\Vect_\ik$ or D-modules on an algebraic curve. Thus, we make the following formal definition:

\begin{defn}\label{pseudotensordefn} Let $\D$ be a category enriched over $\CC$. A pseudo-tensor structure on $\D$ is a coloured operad $\mc O_\D\in \Op(\CC)$ such that
	\begin{enumerate}[(i)]
		\item the collection of colours $\col \mc O_{\mc D}$ is equal to the collection of objects of $\mc D$, 
		\item the spaces of 1-ary operations $\mc O_{\mc D}(M,N)$ and their compositions are given by the spaces of morphisms in $\D$ and their categorical composition, respectively.
	\end{enumerate}
	
In particular, given any object $M\in \mc D$, we let $\cEnd_{\mc D}(M)$ denote the full suboperad on the single colour $M$, whose space of $n$-ary operations is given by
\begin{equation}\label{pseudotensendopeqn}
	\cEnd_{\mc D}(M)(n) = \mc O_{\mc D}( \{ M\}_{i=1}^n, M) \ .  
\end{equation}
\end{defn}
Evidently, every symmetric monoidal category defines a pseudo-tensor category by Example \ref{symmonoidcatopeg}.

\subsection{Algebras of the classical kind}\label{ssec:operad examples}

We introduce some examples of binary operads and their respective algebras. These operads will be introduced over $\CC=\Vect_{\ik}$ with its usual tensor product and unit. Throughout, we let $\D$ denote a $\ik$-linear symmetric monoidal category.

\subsubsection{}

The commutative operad $\Comm\in \Op(\Vect_\ik)$ is generated by 
\[
E = (0,0,\ik_m,0,\dots)~,
\]
where $\ik_m$ is the trivial representation of $\Ss_2$, subject to the relation generated by
\[
\ik[\Ss_3]\cdot\Ass(m)\subset \FF(E)(3) \qquad \text{ defined by } \qquad \Ass(m) =  m\circ(m\otimes \id) - m\circ (\id\otimes m)\in \FF(E)(3)~.
\]
The category of $\Comm$ algebras $\Comm(\mc D):=\Alg_{\Comm }(\mc D)$ in $\mc D$ is the usual category of (non-unital) commutative algebra objects in $\mc D$.

As $\Ss$-modules
\[
\Comm = (\ik)_{n\in\IN}~,
\]
where $\ik$ is the trivial representation of each $\Ss_n$.

\subsubsection{}\label{Lieop}

The Lie operad $\Lie\in \Op(\Vect_\K)$ is generated by 
\[
E=(0,0,\K_b,0,\dots)~,
\]
where $\ik_b$ is the sign representation of $\Ss_2$, subject to the relation generated by
\[
\ik[\Ss_3]\cdot\text{Jac}(b) \subset \mc F(E)(3) \qquad\text{ defined by } \qquad  \text{Jac}(b)= b\circ( \id \otimes b) - b\circ( b\otimes \id) - b \circ (\id \otimes b)\circ \sigma_{12}\in \FF(E)(3) ~.
\]
The category $\Lie(\D) \coloneqq \Alg_\Lie(\D)$, is the usual category of Lie algebra objects in $\D$.

\subsubsection{}\label{Asseg} 

The associative operad $\Ass\in \Op(\Vect_\K)$ is generated by 
\[
E = (0,0,\ik[\Ss_2]\cdot m,0,\dots)~,
\]
subject to the relation generated by
\[
  \ik[\Ss_3]\cdot \Ass(m)\subset \FF(E)(3) \quad\quad \text{ defined by }\quad\quad \text{Ass}(m)=m\circ (m\otimes \id) - m \circ (\id \otimes m) \in \mc F(E)(3)~.
\]
The category $\Ass(\mc D):=\Alg_{\Ass }(\mc D)$ in $\mc D$ is the usual category of associative algebra objects in $\mc D$. 

As $\Ss$-modules,
\[
\Ass = \big(\ik[\Ss_n]\big)_{n\in \IN}~,
\]
\ie, the regular representation of $\Ss_n$ in arity $n$.

We have a sequence of operads over $\ik$,
\[
\Lie\hookrightarrow\Ass \twoheadrightarrow \Comm~.
\]
The surjection $\Ass \twoheadrightarrow \Comm$ induces a fully faithful embedding
\[
\Comm(\D) \hookrightarrow \Ass(\D)~,
\]
by forgetting the commutativity property.

Meanwhile, the inclusion $\Lie\hookrightarrow \Ass$ induces a forgetful functor
\[
\Ass(\D)\rightarrow \Lie(\D)~,
\]
sending an associative algebra $A$ to the Lie algebra $(A,[\cdot,\cdot])$, where $[\cdot,\cdot]$ is the commutator of the associative product on $A$. This forgetful functor admits a left adjoint given by the universal enveloping algebra construction
\begin{align*}
\Lie(\D) &\rightarrow \Ass(\D)~,\\
\hf &\mapsto U(\hf)~
\end{align*}

\subsubsection{}\label{Poisopdefn}

The Poisson operad $\Pois\in\Op(\Vect_\K)$, is generated by
\[
E = (0,0, \K_m\oplus \K_\pi,0\dots)
\]
where $\K_m$ is the trivial representation and $\K_\pi$ is the sign representation, subject to the relations generated by
\[
\ik[\Ss_3]\cdot\big(\text{Jac}(\pi)\oplus\Ass(m)\oplus\text{Dist}(m,\pi)\big)\subset \FF(E)(3)~,
\]
defined by
\begin{align*}
 \text{Jac}(\pi) & = \pi\circ( \id \otimes \pi) - \pi\circ( \pi\otimes \id) - \pi \circ (\id \otimes \pi)\circ \sigma_{12} \ , \\
  \text{Ass}(m) & =m\circ (m\otimes \id) - m \circ (\id \otimes m)~, \\
 \text{Dist}(m,\pi) & = \pi\circ(\id\otimes m) - m\circ(\pi\otimes \id)  - m \circ ( \id \otimes \pi) \ .
\end{align*}
The Poisson operad admits an auxiliary grading
\[
\Pois = \bigoplus_{n\in\IN}\bigoplus_{p=0}^n \Pois(n)^{\langle p\rangle}~,
\]
defined by assigning the grading $E = E^{\langle 0\rangle}\oplus E^{\langle -1\rangle}$ to the generators, with
\[
E^{\langle 0 \rangle} = (0,0,\ik_m,0,\dots)~,\qquad \text{and} \qquad E^{\langle -1 \rangle}=(0,0,\ik_\pi,0\dots)~.
\]

The category $\Pois(\D)\coloneqq \Alg_{\Pois}(\D)$ is the usual category of Poisson algebra objects in $\D$. In particular, the category $\Alg_{\Pois}(\Vect_\K)$ is evidently given by the usual category of Poisson algebras, that is, commutative algebras $A\in \Alg_\Comm(\Vect_\K)$ together with a Lie bracket $\pi: A^{\otimes 2}\to A$ such that $\pi$ is a derivation of the product $m$. 

Similarly, the category $\Alg^\gr_{\Pois}(\Vect_\ik^\gr)$ of graded Poisson algebras is the usual category of graded commutative algebras equipped with a Lie bracket $\pi: A^{\otimes 2}\to A$ of graded degree $-1$ which is a derivation of the commutative product. For later, we note that this is the natural graded degree of the Poisson bracket when the Poisson algebra is the associative graded of a filtered associative algebra.


The Poisson operad fits into the following sequence of operads over $\ik$,
\begin{align*}
	\Lie \hookrightarrow \Pois \twoheadrightarrow \Comm
\end{align*}
The surjection induces the fully faithful embedding
\[
\Comm(\D)\hookrightarrow \Pois(\D)~,
\]
by regarding a commutative algebra as a Poisson algebra with trivial bracket.

The inclusion induces a forgetful functor
\[
\Pois(\D) \rightarrow \Lie(\D)~,
\]
by forgetting the commutative multiplication. As with associative algebras, this functor admits a left adjoint 
\begin{align*}
	\Lie(\D)\rightarrow \Pois(\D)~,
	\hf\mapsto \Sym(\hf)
\end{align*}
and equipping $\Sym(\hf)$ with the Kirillov--Kostant--Souriau Poisson bracket.


\section{Algebras of the chiral kind}\label{sec:dmodules}

\subsection{The category of D-modules}\label{DModsec}

Let $X$ be a smooth variety of dimension $d_X$ over $\bb K=\C$ or a field of characteristic $0$. We write $\mc O_X$ for the sheaf of regular functions, $\mc D_X$ for the sheaf of differential operators, $\Theta_X$ for the tangent sheaf, $\Omega^1_X$ for the sheaf of K\"ahler differentials, $\Omega^{d_X}_X$ for the sheaf of sections of the canonical bundle, and $\omega_X=\Omega^{d_X}_X[d_X]$ for the dualizing sheaf on $X$. Let $\Dd(\mc O_X)$ be the derived category of complexes of $\mc O_X$-modules, and $\Sh_\z(X)$ the category of complexes of sheaves of $\bb K$-modules on $X$ in the Zariski topology. We will typically take $X$ to be a smooth algebraic curve, though we will also need to consider D-modules on more general varieties and thus we will recall the basic conventions to fix notation in general.

\begin{defn}\label{cohshvfun}
	Let $f:X\to Y$ a map of schemes. The inverse and direct image functors are
	\[ \bb L f^\bullet :\Dd(\mc O_Y) \to \Dd(\mc O_X) \qquad f^\bullet \mc F = f^{-1}F\otimes_{f^{-1}\mc O_Y} \mc O_X  \qquad\text{and}\qquad \bb R f_\bullet:\Dd(\mc O_X)\to \Dd(\mc O_Y) \qquad f_\bullet \mc F =  f_\bullet\mc F \ , \]
	where $f_\bullet:\Sh_\z(X)\to \Sh_\z(Y)$ and $f^{-1}:\Sh_\z(Y) \to \Sh_\z(X)$ are the usual direct and inverse image functors on sheaves of $\bb K$-modules.
\end{defn}

Let $\DD^\ell(X)$ (respectively, $\Dd^\ell(X)$) denote the (derived) category of (complexes of) left D-modules, that is, (complexes of) sheaves of left modules on $X$ over the sheaf of algebras $\mc D_X$, such that the underlying (cohomology) sheaves of $\mc O_X$-modules are quasi-coherent. Similarly, we let $\DD^r(X)$ (respectively, $\Dd^r(X)$) denote the (derived) category of (complexes of) right D-modules.

There is an equivalence of categories $\Dd^\ell(X)\xrightarrow{\cong} \Dd^r(X)$ sending a complex of left D-modules $M^\ell\in \Dd^\ell(X)$ to the complex $M^r=M^\ell\otimes_{\OO_X} \omega_X\in \Dd^r(X)$, and inducing an identification of the respective hearts up to a shift $\DD^\ell(X)[-d_X]\xrightarrow{\cong} \DD^r(X)$. We write $\Dd(X)$ for the common value of the derived category under this identification, though we will reference the specific models when defining functors between these categories.

We also use an alternative equivalence identifying the left and right hearts $\DD^\ell(X) \xrightarrow{\cong} \DD^r(X)$, sending a left $D$-module $M^\ell\in \DD^\ell(X)$ to $M^{r,\heartsuit}=M^\ell\otimes_{\mc O_X} \Omega^{d_X}_X=M^r[-d_X] \in \DD^r(X)$. We write $\DD(X)=\DD^r(X)$ for the abelian category given by the right heart, though we will also use the inverse identification
\[ (\cdot )^{\ell,\heartsuit}:\DD(X) \xrightarrow{\cong} \DD^\ell(X) \qquad\text{where}\qquad M^{\ell,\heartsuit} = M^\ell[d_X] \]
to define various structures on $\DD(X)$.
One should view this identification as less canonical than the derived analogue, but it will nonetheless be necessary to establish the theory of plain (non-DG) algebras of the chiral kind.

\begin{defn}\label{Dmodinvim}
	Let $f:X\to Y$ be a map of smooth varieties. The inverse image $f^!:\Dd(Y)\to \Dd(X)$ is defined by
	\[ f^!=\bb L f^{!,\heartsuit}:\Dd^\ell(Y)\to \Dd^\ell(X)   \qquad\text{where}\qquad f^{!,\heartsuit}(M)=f^\bullet(M) \quad\text{equipped with the pullback flat connection.} \]
\end{defn}

Evidently this functor is right exact with respect to the left t-structure, and we recall that for $f:X\to Y$ smooth it is exact.

\begin{defn}
	The exterior product is defined by
	\[ \boxtimes:\Dd(X)\times \Dd(Y)\to \Dd(X\times Y) \qquad\text{where}\qquad M\boxtimes N = \pi_X^!M \otimes \pi_Y^!N \ , \]
	for $\pi_X:X\times Y\to X, \pi_Y:X\times Y\to Y$. Note that for $X$ and $Y$ smooth, the maps $\pi_X$ and $\pi_Y$ are smooth and thus the relevant pullback functors are exact, and by abuse of notation we also define
		\[ \boxtimes:\DD(X)\times \DD(Y)\to \DD(X\times Y) \qquad\text{where}\qquad M\boxtimes N = \pi_X^!M \otimes \pi_Y^!N \ . \]
\end{defn}

The derived category $\Dd(X)$ is symmetric monoidal with respect to the derived tensor product
$$M\otimes^{!, \bb L} N= \Delta^!(M\boxtimes N) $$
for $M,N\in \Dd(X)$ and $\Delta:X\to X\times X$ the diagonal embedding.

The abelian category $\DD(X)$ is also symmetric monoidal with respect to the non-derived tensor product
$$M\otimes^{!} N= \Delta^{!,\heartsuit}(M^{\ell,\heartsuit}\boxtimes N^{\ell,\heartsuit})  \ .$$
Concretely, for $L,M\in \DD^\ell(X)$, we have $L\otimes^!M = L\tensor{\OO_X}M$
as $\OO_X$-modules and with $\tau\in \Theta_X\subset \D_X$ acting via the usual Leibniz rule, $\tau\otimes1+1\otimes\tau$.

Note that for $M\in \DD(X)$ locally projective we have $M\otimes^! N\cong M\otimes^{\bb L,!} N[d_X]$.

\begin{defn}\label{Dmoddirim}
	Let $f:X\to Y$ be a map of smooth varieties. The direct image $f_*:\Dd(X)\to \Dd(Y)$ is defined by
	\[ f_*:\Dd^r(X)\to \Dd^r(Y)  \quad\quad  f_*(M)= \bb R f_\bullet( M\otimes^{\bb L}_{\mc D_X} \mc D_{X\to Y}) \qquad\text{for}\qquad \mc D_{X\to Y}:= f^{!,\heartsuit} \mc D_Y \in (\mc D_X,f^{-1} \mc D_Y)\text{-Mod}\, \]
	where $\mc D_{X\to Y}= f^{!,\heartsuit} \mc D_Y \in \DD(X)$ is defined in terms of $\mc D_Y\in \DD^\ell(Y)$ as a left module, so that the $(\mc D_Y, \mc D_Y)$-bimodule structure on $\mc D_Y$ equips $\mc D_{X\to Y}$ with the structure of a $(\mc D_X,f^{-1} \mc D_Y)$-bimodule.
\end{defn}

Recall that for $f$ quasi-finite the functor $f_*$ is left exact, and for $f$ affine it is right exact, with respect to the right t-structure. Further, recall that for a closed embedding the functor $f_*$ is exact with respect to the right t-structure, and defines an equivalence $f_*:\DD(X)\xrightarrow{\cong} \DD(Y)_X$ of abelian categories between $\DD(X)$ and the full subcategory $\DD(Y)_X\subset \DD(Y)$ of right D-modules on $Y$ supported on $X$, with inverse equivalence $f^!_X=f^!|_{\DD(Y)_X}:\DD(Y)_X\xrightarrow{\cong} \DD(X)$.

\begin{defn}
For $\pi:X\to \pt$ and $M\in \Dd(X)$ we define the de Rham cochains of $M$ by
\[C^\bullet_\dR(X,M)=\pi_*M \in \Dd(\pt)=\Dd(\K\Mod)  \ .\]
\end{defn}

Concretely, for $M\in \Dd^r(X)$ we have
\[ \pi_*M =\bb R\pi_\bullet( M^r\otimes^{\bb L}_{\mc D_X} \mc O_X) \cong \bb R\pi_\bullet( M^\ell \otimes_{\mc O_X} \Omega^\bullet_X )[2d_X] \]
where $ M^\ell \otimes_{\mc O_X} \Omega^\bullet_X\in \Dd(\Sh_\z(X))$ denotes the algebraic de Rham complex with coefficients in the complex $M^\ell$, considered as a complex of Zariski sheaves on $X$, and the latter isomorphism follows from using the Spencer--Koszul resolution,
\[ \mc K^\bullet=\left[ \mc D_X \otimes_{\mc O_X} \wedge^{d_X}_{\mc O_X} \theta_X[d_X] \to \cdots \to  \mc D_X \otimes_{\mc O_X} \wedge^{2}_{\mc O_X} \theta_X[2]\to  \mc D_X \otimes_{\mc O_X} \theta_X[1]\to \mc D_X \right]\xrightarrow{\cong} \mc O_X \ ,\]
to compute the derived tensor product as
\[  M\otimes^{\bb L}_{\mc D_X} \mc O_X \cong M\otimes_{\mc D_X}\mc K^\bullet \cong M^\ell\otimes_{\mc O_X} \omega_X\otimes_{\mc D_X}\mc K^\bullet \cong  M^\ell \otimes_{\mc O_X} \Omega^\bullet_X[2d_X] \ . \]

We also define the middle de Rham cohomology as the functor of abelian categories
\[ h:\DD(X)\to \DD(\pt)=\K\Mod \qquad \text{by}\qquad h(M)=\pi_\bullet(M^r\otimes_{\mc D_X} \mc O_X) \cong \pi_\bullet( M^r/(M^r\cdot \theta_X)) \]
for $M=M^r\in \DD(X)\cong \DD^r(X)$, noting $\mc O_X \cong \mc D_X / (\mc D_X \cdot \theta_X)$. This functor computes the true de Rham cohomology in the following setting:

\begin{prop}\label{locprojaffprop} Let $X$ be a smooth, affine variety and $M\in \DD(X)$ be locally projective. Then there is a natural quasi-isomorphism
\[ \pi_*M \xrightarrow{\cong} h(M) \ . \]
\end{prop}

\subsection{Chiral, \texorpdfstring{$\Comm^!$}{Comm!}, and \texorpdfstring{$\Lie^*$}{Lie*} algebras}\label{ssec:chiral algebras}
\subsubsection{\texorpdfstring{$\text{Comm}^!$}{Comm!} algebras}
Let $\DD(X)^!$ denote the pseudo-tensor structure on $\DD(X)$ induced by the symmetric monoidal structure $\otimes^!$. Concretely, the spaces of multilinear maps are given by
\[
P^!_I(\{L_i\}, M)\equiv\Hom_{\DD(X)^!}(\{L_i\}, M) = \Hom_{\DD(X)}(\otimes^!_{i\in I} L_i , M) \ . 
\]

\begin{defn}
	A $\Comm^!$ algebra on $X$ is a commutative algebra object in the symmetric monoidal category $(\DD(X),\otimes^!)$. We write $\Comm^!(X)=\Alg_{\Comm}(\DD(X)^!)$ for the category of $\Comm^!$ algebras.
\end{defn}

The endomorphism operad of an object $M\in \DD(X)^!$ is given concretely by
\[ \OEnd_{\DD(X)^!}(M)(n)=P^!_n(\{M\}_{i=1}^n, M) = \Hom_{\DD(X)}(M^{\otimes^! n} , M) \ .  \]

\subsubsection{Lie* algebras}

Unlike the $!$-tensor structure, the $*$-tensor structure does not arise from a monoidal structure on $\DD(X)$.

\begin{defn} Let $I\in \fset$ and write $\Delta^{(I)}:X\rightarrow X^I$ for the diagonal embedding. For any $I$-family $L_i\in\DD(X)$, $i\in I$, and $M\in\DD(X)$, the multilinear operations in $\DD(X)^*$ are given by
\[
 P^*(\{L_i\}, M)\equiv\Hom_{\DD(X)^*}(\{L_i\}, M) = \Hom_{\DD^r(X^I)}(\boxt_{i\in I} L_i^r , \Delta^I_* M^r) \ , 
\]
with composition defined as in 2.2.3 of \cite{BD1}, noting that $\Delta^I_*$ is exact with respect to the right $t$-structures on $\Dd(X)$ and $\Dd(X^I)$.
\end{defn}

\begin{defn}
	A $\Lie^*$ algebra is a Lie algebra, $L\in \Alg_{\Lie}(\DD(X)^*)$ internal to the operad $\DD(X)^*$.
\end{defn}

Concretely, for $b\in \Hom_{\DD(X)^*}(\{L,L\},L)$, the composition $b\circ (b\otimes \id) \in \Hom_{\DD(X)^*}(\{L,L,L\},L)$ is defined by
\[
L\boxt L\boxt L  \xrightarrow{ b\,\boxt \id} \Delta_*L\boxt L \cong \Delta^{12,3}_*(L\boxt L) \xrightarrow{ \Delta^{12,3}_*(b)} \Delta^{12,3}_X\Delta^{(2)}_*L=\Delta^{(3)}_* L~,
\]
where $\Delta^{12,3}:X^2\to X^3$ is defined by $(x,y)\mapsto (x,x,y)$.

The endomorphism operad of an object $M\in \DD(X)^*$ is given concretely by
\[ \OEnd_{\DD(X)^*}(M)(n) =P_n^*(\{ M\}_{i=1}^n,M)= \Hom_{\DD(X^n)}(M^{\boxtimes n} , \Delta^{(n)}_* M) \ .  \]

\subsubsection{Chiral algebras}

For any finite set $I\in \fset$, let $\Delta^{(I)}:X\to X^I$ once again denote the diagonal embedding and let
\[ 
j^{(I)}:U^{(I)} \to X^I \qquad \text{for}\qquad U^{(I)}=\{ (x_i) \in X^I \ | \ x_i \neq x_j  \text{ for all } i\neq j \} 
\]
denote the inclusion of the complement of the union of all diagonals:

\begin{defn} The \emph{chiral} pseudo-tensor structure on $\DD(X)$, denoted as $\DD(X)^\ch$, is defined as having multilinear operations
	\[ 
	P^\ch_I(\{L_i\}, M) =\Hom_{\DD(X)^\ch}(\{L_i\}, M) = \Hom_{\DD^r(X^I)}(j^{(I)}_*j^{(I),*} (\boxt_{i\in I} L_i^r) , \Delta^{(I)}_* M^r) ~,
	\]
	for $L_i,M\in \DD(X)$, noting that $\Delta^{(I)}_*$, $j^{(I),*}$ and $j^{(I)}_*$ are all exact with respect to the right $t$-structures on $\Dd(X)$, $\Dd(X^I)$ and $\Dd(U^{(I)})$.
\end{defn}

There are natural composition maps on the spaces of chiral operations defined in 1.3.2 of \cite{BD1}.
For example, for $A\in\DD(X)$ and $\mu\in \Hom_{\DD(X)^\ch}(\{A,A\},A)$, the composition $\mu \circ (\mu \otimes \id) \in \Hom_{\DD(X)^\ch}(\{A,A,A\},A)$ is defined by
\begin{align*}
	j^{(3)}_*(j^{(3)})^*( A\boxt A \boxt A)  & =  j^{12,3}_*j^{12,3 *}( j_*j^*(A\boxt A) \boxt A) \xrightarrow{j^{12,3}_*j^{12,3 *}( \mu\boxt \id)} j^{12,3}_*j^{12,3 *}(\Delta_*(A)\boxt A) \\
	& \quad\quad = \Delta^{12,3}_*(j_*j^*(A\boxt A)) \xrightarrow{ \Delta^{12,3}_*(\mu)} \Delta^{12,3}_* \Delta^{(2)}_*A=\Delta^{(3)}_*A 
\end{align*}
where we let $\Delta^{12,3}:X^2\to X^3$ defined by $(x,y)\mapsto (x,x,y)$, $U^{12,3}=\{(x,y,z)| x,y\neq  z\}$ and $j^{12,3}:U^{12,3}\into X^3$.

\begin{defn}
	 A chiral algebra is a Lie algebra $A\in \Alg_{\Lie}( \DD(X)^\ch)$ internal to the operad $\DD(X)^\ch$.
	\end{defn}
	
	Evidently a non-unital chiral algebra structure is equivalent to specifying a map
	$$ \mu \in P_2^\ch( \{ A,A\},A)=\Hom_{\DD(X^2) }( j^*j_*(A\boxt A) , \Delta_*A)$$
	which is skew-symmetric and satisfies the appropriately defined Jacobi identity.
	
	\begin{defn}
	 A unital chiral algebra is a chiral algebra $A$ together with a map $u:\omega_X\to A$ such that the diagram
	 \begin{equation}\label{unchdiagram} \xymatrix{ j_*j^*(A\boxtimes \omega_X) \ar[r]\ar[d] & \Delta_* A \ar@{=}[d] \\ j_*j^*(A\boxtimes A) \ar[r] & \Delta_*A } 
	 \end{equation}
	 and its higher arity analogues commute, where the left vertical map is that induced by the map $u$, and the bottom horizontal map is given by the chiral multiplication, and the top horizontal map is given by the boundary map in the exact triangle for a complimentary open and closed embedding:
	 \[j_*j^*(A\boxtimes \omega_X)  \to \Delta_* \Delta^!(A\boxtimes \omega_X) =\Delta_* (A \otimes^! \omega_X) = \Delta_* A \ .  \]

\end{defn}

%
%
%
%
%
%
%
\begin{defn}\label{commchalg} A chiral algebra $A \in \Alg^\ch(X)$ is called commutative if the composition
	$$ A \boxt A \to j_*j^*(A \boxt A) \xrightarrow{\mu} \Delta_*A $$
	vanishes, where the first map is the unit of the $(j^*,j_*)$-adjunction and the second is the chiral algebra structure map $\mu$. A unital chiral algebra $A \in \Alg_\un^\ch(X)$ is called commutative if it is commutative as a non-unital chiral algebra.
\end{defn}

\begin{prop} There is a natural equivalence of categories between the full subcategory of commutative chiral algebras and $\Comm^!$ algebras on $X$,
\[\Alg^\ch(X)^\Comm \xrightarrow{\cong} \Alg_\Comm(\DD(X)^!) \]
identifying the respective forgetful functors to $\DD(X)$.
\end{prop}



\begin{rmk}\label{chiraltostar} There is a natural map of operads $\DD(X)^\ch\to \DD(X)^*$ which is the identity on objects and arity 1 morphisms, and defined on higher arity morphisms by the maps
	$$  \Hom_{\DD(X^I)}(j^{(I)}_*j^{(I),*} (\boxt_{i\in I} M_i) , \Delta^{(I)}_* L) \to \Hom_{\DD(X^I)}(\boxt_{i\in I} M_i , \Delta^{(I)}_* L)$$
	induced by
	$$\quad\quad \boxtimes_{i\in I} M_i \to  j^{(I)}_*j^{(I),*} (\boxt_{i\in I} M_i) \ , $$
	where the latter are given by the unit of the $(j^{(I),*}, j^{(I)}_*)$ adjunction. This in turn gives a forgetful functor $\Alg_{\Lie}(\DD(X)^\ch)\rightarrow \Alg_{\Lie}(\DD(X)^*)$
\end{rmk}

\begin{prop}[\cite{BD1}]\label{prop:chiral envelope}
	The forgetful functor $\Alg_{\Lie}(\DD(X)^!)\rightarrow \Alg_{\Lie}(\DD(X)^*)$ has a left adjoint
	\[
	\Alg_{\Lie}(\DD(X)^*)\xrightarrow{\VV} \Alg_{\Lie}(\DD(X)^\ch)~.
	\]
	called the chiral envelope.
\end{prop}
One should compare this to the free-forgetful adjunction $\Lie(\Vect_\ik)\leftrightarrows \Ass(\Vect_\ik)$ recalled in \ref{Asseg}.

\subsection{Quasicoherent sheaves on $\mc D_X$-schemes}

\begin{defn}
A $\mc D_X$-scheme is a scheme $\mc Y$ over $X$ together with a flat Ehresmann connection on $\mc Y$ over $X$.
\end{defn}

The forgetful functor $\DD(X)\to \QC(X)$ is symmetric monoidal, so every $\Comm^!$ algebra $R$ on $X$ determines a commutative algebra object $R$ in $\QC(X)$, or equivalently a scheme $\mc Y = \Spec R$ that is affine over $X$. The additional data of the $\DD$-module structure on $R$ determines a flat connection on $\mc Y$ over $X$. Thus, every $\Comm^!$ algebra on $X$ determines a $\mc D_X$-scheme, and such $\mc D_X$-schemes are called affine.

For $\underline{Y} \to X$ a scheme over $X$, we define the $\mc D_X$-scheme of jets $\mc J(\underline{Y})$ of sections of $\underline{Y}$ by the property that for any $\mc D_X$-scheme $\mc Z$, we have
\[ \Hom_{\mc D_X-\Sch}(\mc Z, \mc J(\underline{Y})) = \Hom_{\Sch_{/X}}(\mc Z, \underline{Y}) \ .\]
In particular, taking $\mc Z=X$ we obtain the familiar property $\Gamma_\nabla(X,\mc J(\underline{Y}))=\Gamma(X,\underline{Y})$. In the case $\underline{Y}=Y\times X$ for $Y$ a $\bb K$-scheme, we write simply $\mc J(Y)$ in place of $\mc J(\underline{Y})$.

In particular, if $\underline{Y}=\Spec B$ for $B$ a commutative algebra object in $\QC(X)$, so that $\underline{Y}$ is affine over $X$, then we can equivalently define the jet algebra $\mc J(B)\in \Comm^!(X)$ by the property
\begin{equation}\label{affinejetadjeqn}
	 \Hom_{\Comm^!(X)}( \mc J(B),R) = \Hom_{\Comm(\QC(X))}(B,R) \ .  
\end{equation}
 for any $R\in \Comm^!(X)$ and we have that $\mc J(\underline{Y})=\Spec \mc J(B)$.

\subsubsection{}
Let $R$ be a commutative algebra object in $\DD(X)^!$ and let $\DD(R)=R\Mod(\DD(X)^!)$ (respectively, $\Dd(R)=\Dd(R\Mod(\DD(X)^!))$) denote the (derived) category of (complexes of) $R$-module objects internal to $\DD(X)$. Equivalently, we let $R[\mc D_X]= R^{l,\heartsuit}\otimes_{\mc O_X} \mc D_X$ denote the sheaf of algebras given by differential operators with coefficients in the $\Comm^!$ algebra $R$. We recall that $\DD(R)$ (respectively, $\Dd(R)$) is equivalent to the (derived) category of (complexes of) sheaves of modules over $R[\mc D_X]$ which (have cohomology sheaves that) are quasicoherent as $\mc O_X$ modules.

The forgetful functor $\DD(X)\to \QC(X)$ is symmetric monoidal, so it induces a natural forgetful functor
\[ \DD(R)=R\Mod(\DD(X)^!) \to R\Mod(\QC(X)) \ , \]
where in the latter case we view $R$ as an algebra object in $\QC(X)$, and we let $\mc Y=\Spec_{\mc D_X}R$ denote the affine D-scheme corresponding to $R$, noting $R\Mod(\QC(X))=\QC(\mc Y)$ the category of quasicoherent sheaves on the underlying scheme of $\mc Y$. Thus, geometrically we can view $\DD(R)$ as the category of quasicoherent sheaves on $\mc Y$ equipped with a compatible flat connection along $X$, and $\Dd(R)$ its derived category.

\subsubsection{} Let $\alpha:P\to R$ be a map of $\Comm^!$ algebras on $X$, $\mc X=\Spec_{\mc D_X} R$, $\mc Y=\Spec_{\mc D_X} P$, and denote the corresponding map of D-schemes by $f:\mc X\to \mc Y$. We obtain functors
\[ \bb L f^\bullet:\Dd(P) \to \Dd(R) \qquad f^\bullet M= R\otimes^!_P M \qquad \text{and}\qquad \bb R f_\bullet:\Dd(R) \to \Dd(P) \qquad f_\bullet(M)=\ob_R^P(M)\]
where $\ob_R^P(M)\in \Dd(P)$ denotes the functor induced by precomposition with the map $\alpha$, preserving the forgetful functors to $\DD(X)$.

Let $P,R$ be $\Comm^!$ algebras on $X$, and note that $P\boxtimes R$ is naturally a $\Comm^!$ algebra on $X^2$. Further, given $M\in \DD(P)$ and $N\in \DD(R)$ (or $M\in \Dd(P)$ and $N\in\Dd(R)$), we obtain objects $ M\boxtimes N \in \DD(P\boxtimes R)$ (or  $ M\boxtimes N \in \Dd(P\boxtimes R)$). Similarly, $P\otimes^! R$ defines a $\Comm^!$ algebra on $X$, and the functors $\Delta^{!,\heartsuit}:\DD(X^2)\to \DD(X)$ and $\Delta^!:\Dd(X^2)\to \DD(X)$ determine canonical functors
\[\Delta^{!,\heartsuit} :\DD(P\boxtimes R) \to \DD(P\otimes^! R) \qquad \text{and} \qquad \Delta^!:\Dd(P\boxtimes R) \to \Dd(P\otimes^! R) \ . \]

\subsubsection{}\label{Rshotimessec} The multiplication structure map $m:R\otimes^!R\to R$ is a map of $\Comm^!$ algebras, and we denote by $\Delta_R:\mc Y \to \mc Y^{\times 2}$ the corresponding diagonal embedding for $\mc X=\Spec_{\mc D_X}R$.

The derived category $\Dd(R)$ is symmetric monoidal with respect to the derived tensor product
\[ M\otimes_R^{\bb L,!} N= \bb L\Delta_R^\bullet  \Delta^!(M\boxtimes N)  =  \bb L\Delta_R^\bullet (M\otimes^{\bb L,!}N) \ . \]
Geometrically, this is simply the usual derived tensor product of quasicoherent sheaves on $\mc Y$, with its natural tensor product D-module structure.

The abelian category $\DD(R)$ is also symmetric monoidal with respect to the non-derived tensor product
\[ M\otimes^!_R N =  \Delta_R^\bullet \Delta^{!,\heartsuit}(M\boxtimes N) = \Delta_R^\bullet( M\otimes^! N)  \ . \]
Note that for $M\in \DD(R)$ projective we have $M\otimes^!_R N\cong M\otimes_R^{\bb L,!} N[d_X]$.

\subsubsection{}

The symmetric monoidal structure $\otimes_R^{!}$ on $\DD(R)$ defined in \ref{Rshotimessec} above determines a coloured operad structure on $\DD(R)$ as usual:
\[ P^!_{I,R}(\{L_i\},M) = \Hom_{\DD(R)}( \otimes_{R,i\in I}^{!} L_i, M)  \ . \]

\subsubsection{} Next, we define the simplest analogue of the $*$-operations on $\DD(R)$:

\begin{defn}\label{Rlindefn}  The spaces of $R$-linear $*$-operations are defined by
\[ P_{I,R}^*(\{L_i\},M)=\Hom_{\DD(R^{\boxtimes I})}(\boxtimes_{i\in I} L_i, \Delta_*M)   \subset  P_{I}^*(\{L_i\},M)=\Hom_{\DD(X^I)}( \boxtimes_{i\in I} L_i,\Delta_* M)  \ , \]
the subspace of $I$-ary $*$-operations in $\DD(X)$ which are $R$-linear in the argument $L_i$ for each $i\in I$, where we note that the object $\Delta_*M \in \DD(X^I)$ inherits an $R^{\boxtimes I}$ module structure via\[ R^{\boxtimes I} \otimes^! \Delta_* M \cong \Delta_* \Delta^!(R^{\boxtimes I} \otimes^! \Delta_* M) \cong R^{\otimes^! I} \otimes^! M \xrightarrow{m} M \ . \]

\end{defn}

Concretely, a map $\varphi: \boxtimes_{i\in I} L_i \to \Delta_* M$ is $R$-linear in the $i_0^{th}$ argument if the rightmost square of the diagram
\[  \begin{tikzcd} \pi_{i_0}^! R \otimes (\boxtimes_{i\in I} L_i ) \ar[r,"\cong"] \ar[d,swap,"\id_{\pi_{i_0}^!R} \otimes \varphi"] & (R\otimes L_{i_0})\boxtimes  (\boxtimes_{i\neq i_0} L_i) \ar[r,"m_{L_0}\boxtimes  \id "] \ar[d,swap,"\id_R \otimes_{i_0} \varphi"] & \boxtimes_{i\in I} L_i \ar[d,"\varphi"] \\ \pi_{i_0}^!R\otimes  \Delta_*(M) \ar[r,"\cong"] & \Delta_*(R\otimes M)  \ar[r,"\Delta_* m_M"] & \Delta_*M
\end{tikzcd} \ \text{commutes,} \]  
where we define $\id_R \otimes_{i_0} \varphi: (R\otimes L_{i_0})\boxtimes  (\boxtimes_{i\neq i_0} L_i) \to \Delta_*(R\otimes M)$ so that the left square commutes, and $m_M:R\otimes M \to M$ and $m_{L_0}:R\otimes L_0 \to L_0$ are the $R$-module structure maps.

\begin{prop}\label{Rlinstarprop} The spaces of $R$-linear $*$-operations define a coloured operad structure on $\DD(R)$.
\end{prop}

It will sometimes be necessary to consider the following more general notion of $R$-polydifferential $*$-operations on $\DD(R)$:

\begin{defn}\label{Rpolydiffopsdefn} We define the pseudo-tensor structure of $R$-polydifferential $*$-operations on $\DD(R)$,
\[ P_{I,R-\Diff}^*(\{L_i\},M)  \subset  P_{I}^*(\{L_i\},M) \ ,\]
as in Section 1.4.8 of \cite{BD1}.
\end{defn}

Note that we have
\[ P_{I,R}^*(\{L_i\},M) \subset P_{I,R-\Diff}^*(\{L_i\},M)\]
as the subspace of all $R$-polydifferential $*$-operations that are supported strictly on the diagonal scheme-theoretically.

Let $M\in \DD(R)$ with structure map $m_M:R\otimes M \to M$.
\begin{defn}\label{Derchdefn} The space of derivations $\Der(R,M)\subset \Hom_{\DD(X)}(R,M)$ is the subspace of maps of D-modules $f:R\to M$ such that
\[  \begin{tikzcd} R\otimes R \ar[r,"m"] \ar[d,swap,"\id_R \otimes f+\sigma(f\otimes \id_R)"] & R \ar[d,"f"] \\ R\otimes M \ar[r,"m_M"] & M
	\end{tikzcd} \quad\quad \text{commutes.} \]
\end{defn}

In fact, the functor $\Der(R,\cdot):\DD(R)\to \Vect$ is representable by an object $\Omega_R\in \DD(R)$ equipped with a universal derivation $d_R:R\to \Omega_R$, called the cotangent sheaf and de Rham differential for $R$, respectively; see 1.4.16 in \cite{BD1}.

We define the tangent sheaf of $R$ (when it exists, for example if $R$ is smooth in the sense of 2.3.15 in \cite{BD1}) to be the dual object
\[ \Theta_R = \Omega_R^{\circ_R} =  \Homi_{\DD(R)}(\Omega_R,R) \ \in \DD(R) \ , \]
in the sense of Definition \ref{circRdefn}, which we postpone until Section \ref{varsec} where these results are primarily used.

\begin{prop} Suppose $R$ is a very smooth $\Comm^!$ algebra and $X$ is affine. Then
	\[ h( \Theta_R) =\Hom_{\DD(R)}(\Omega_R,R)=\Der(R,R)  \ .\]
\end{prop}

\subsection{$\Lie^*$ algebroids on $\mc D_X$-schemes}

Following 1.4.9 in \cite{BD1}. Let $L$ be a $\Lie^*$ algebra:

\begin{defn} \label{defn:Lie action by derivations}
An action of $L$ on $R$ by derivations is an $L$-module structure $\tau:L\boxtimes R \to \Delta_* R$ on the underlying $D$-module of $R$, such that the structure maps $m:R\otimes R \to R$ and $u:\omega_X \to R$ are maps of $L$-modules, that is
\[\begin{tikzcd} L \boxtimes (R\otimes R) \ar[r,"\id_L\boxtimes m "] \ar[d,swap,"\id_R \otimes \tau+\sigma(\tau \otimes \id_R)"] & L \boxtimes R \ar[d,"\tau"] \\ \Delta_*(R\otimes R)  \ar[r,"\Delta_* m"] & \Delta_*R
\end{tikzcd} \quad\quad \text{and}\qquad \begin{tikzcd} L\boxtimes \omega_X \ar[r,"\id_L\boxtimes u "] \ar[d,swap,"\tau_\triv"] & L\boxtimes R \ar[d,"\tau"] \\ \Delta_*(\omega_X)  \ar[r,"\Delta_* u "] & \Delta_* R
\end{tikzcd} \quad\quad \text{commute.}   \]
\end{defn}

\begin{eg}(\cite{BD1} 1.4.16) There is a canonical $\Lie^*$ structure on $\Theta_R$ determined by the condition that $\Theta_R$ admits an action on $R$ by derivations $\tau_\textup{univ}: \Theta_R \boxtimes R \to \Delta_* R$ defined by the element
	\[ \id_{\Omega_R}\in \End_R(\Omega_R) = \Der(R, \Omega_R) \subset \Hom_{\DD(X)}(R,\Omega_R) \cong P_2^{\circ !}(R,\{\Omega_R,R\}) \cong P_2^*(\{\Theta_R,R\},R)  \ , \]
or equivalently the universal derivation $d\in \Der(R, \Omega_R)$, under the above identifications.
\end{eg}

\begin{prop}\label{DeractLieprop} An action of $L$ on $R$ by derivations is equivalent to a $\Lie^*$ map $L\to \Theta_R$.
\end{prop}
\begin{proof} Evidently a $\Lie^*$ map $L\to \Theta_R$ determines an action of $L$ on $R$ by derivations, by composition with the universal action $\tau_\textup{univ}$. Conversely, given an action of $L$ on $R$ by derivations, each $l\in \Gamma(X,L)$ induces an endomorphism $\tau(l,\cdot) \in \Der(R,R) \subset \Hom_{\DD(X)}(R,R)$ and this lifts to a map of $\Lie^*$ algebras $L \to \Theta_R$.
\end{proof}

\begin{prop}\label{deractpolydiffprop} An action of $\mc L$ on $R$ by derivations is $R$-differential in the second argument, in the sense of Definition \ref{Rpolydiffopsdefn}.
\end{prop}
\begin{proof}
By definition a derivation is supported scheme-theoretically on the first-order infinitesimal thickening of the diagonal in $\mc Y \times \mc Y$, and thus supported set-theoretically on the diagonal.
\end{proof}

Fix an action $\tau$ of $L$ on $R$ by derivations and let $M \in \DD(R)$.

\begin{defn}\label{LRactiondefn} An action of $L$ on $M$ compatible with $\tau$ is an $L$-module structure $\rho:L\boxtimes M \to \Delta_* M$ such that the structure map $R\otimes M \to M$ is a map of $L$-modules, where $R$ is equipped with the $L$-module structure determined by $\tau$, that is
	\[ \begin{tikzcd} L \boxtimes (R\otimes M) \ar[r,"\id_L\boxtimes m "] \ar[d,swap,"\id_R \otimes \rho +\tau \otimes \id_M"] & L \boxtimes M \ar[d,"\rho"] \\ \Delta_*(R\otimes M)  \ar[r,"\Delta_* m"] & \Delta_*M
	\end{tikzcd}  \quad \text{commutes.}\]
\end{defn}

\begin{defn}(\cite{BD1} 1.4.11)
	A $\Lie^*$ algebroid on $\Spec R$ is an $R$-module $\mc L \in \DD(R)$, with
\begin{itemize}
	\item a $\Lie^*$ bracket $b: \mc L\boxtimes \mc  L \to \Delta_* \mc L$, and
	\item an action $\tau:\mc L\boxtimes R \to \Delta_* R$ of $\mc L$ on $R$ by derivations,
\end{itemize}
such that 
\begin{itemize}
	\item  $\tau$ is $R$-linear in the first argument $\mc L$, in the sense of Definition \ref{Rlindefn}, and
	\item the adjoint action $b:\mc L\boxtimes \mc L \to \mc L$ is $\tau$-compatible, in the sense of Definition \ref{LRactiondefn}.
\end{itemize}
\end{defn}

For example, the tangent sheaf $\Theta_R$ is naturally a $\Lie^*$-algebroid on $R$, and by Proposition \ref{DeractLieprop}, every

We note that the structure maps $b$ and $\tau$ are not $R$-linear in general, but they are necessarily $R$-polydifferential:

\begin{prop} Let $\mc L$ be a $\Lie^*$ algebroid on $R$. Then the structure maps $b: \mc L\boxtimes \mc  L \to \Delta_* \mc L$ and $\tau:\mc L \boxtimes R \to \Delta_* R$ are $R$-polydifferential in the sense of Definition \ref{Rpolydiffopsdefn}, that is, 
	\[ b \in P_{2,R-\Diff}^*(\{ \mc L, \mc L\}, \mc L) \quad\quad\text{and}\quad\quad  \tau \in P_{2,R-\Diff}^*(\{ \mc L, R\}, \mc R)\ . \]
\end{prop}
\begin{proof} The proof is the same as that for Lie algebroids. The map $\tau$ is by definition $R$-linear and thus $R$-differential in the first argument, and is $R$-differential in the second argument by Proposition \ref{deractpolydiffprop}. The map $b$ is $R$-differential in the second argument, and equivalently the first argument by skew-commutativity, by the formula provided by the $\tau$-compatibility of the adjoint action: the latter is equivalent to commutativity of the diagram
\[ \begin{tikzcd} \mc L \boxtimes (R\otimes \mc L) \ar[r,"\id_{\mc L}\boxtimes m "] \ar[d,swap,"\id_R \otimes b +\tau \otimes \id_{\mc L}"] & \mc L \boxtimes \mc L \ar[d,"b"] \\ \Delta_*(R\otimes M)  \ar[r,"\Delta_* m"] & \Delta_*M \end{tikzcd} \ . \]

\end{proof}

%
%
%
%
%
%
%

\subsection{Coisson algebras}\label{ssec:coisson intro}

\begin{defn}
	A \emph{coisson} algebra is a pair $(A,\{\cdot,\cdot\})$, where $A\in \Comm^!(\DD(X))$ and 
	\[
	\{\cdot,\cdot\}\in P_{2}^*(\{A,A\},A)~,
	\]
	is a $\Lie^*$ bracket such that the adjoint action of $A$ on $A$ is an action of a $\Lie^*$-algebra on a $\Comm^!$ algebra by derivations, as in Definition \ref{defn:Lie action by derivations}. Said more prosaically, $\{\cdot,\cdot\}$ must satisfy the Leibniz rule in each slot.
\end{defn}
One should note the similarity to the definition of a Poisson algebra---hence the nomenclature.

Coisson algebras also admit a description as algebras internal to the so-called classical operad. Let $Q(I)$ denote the set of equivalence relations on $I$ and identify each element, $\sim$, with the corresponding pair $(\pi_S,S)$ for $\pi_S:I\rightarrow S= I/\sim $. We use $I_s$ to denote the fibre, $\pi_S^{-1}(s)\subset I$ over $s\in S$. 

We recall the following the definitions from 1.4.27 of \cite{BD1}:

\begin{defn} The \emph{classical} pseudo-tensor structure on $\DD(X)$, denoted as $\DD(X)^c$, is defined by
	\[ 
	P_I^c(\{ L_i\},M) = \bigoplus_{(\pi_S,S)\in Q(I)}  P_I^c(\{L_i\},M)_S ~,
	\]
	where for each $\pi_S:I\to S$ we define
	\[
	P_I^c(\{L_i\},M)_S= P_S^*\big(\{ \otimes!_{i\in I_s} L_i\}_{s\in S},M\big)\otimes \big( \otimes_{s\in S} \Lie(|I_s|)\big)~,
	\]
	for any $L_i,M\in \DD(X)$. The composition maps are defined in $1.4.27$ of \cite{BD1}.
\end{defn}

There is a natural, auxiliary grading on the space of classical operations
\[
	P_I^c(\{ L_i\},M) = \bigoplus_{p=0}^{|I|-1}  P_I^c(\{ L_i\},M)^{\langle p \rangle} ~,
\]
where
\begin{equation}\label{auxgreqn}
	P_I^c(\{ L_i\},M)^{\langle p \rangle}=\bigoplus_{\substack{(\pi_S,S)\in Q(I)\\ |S|=p+1}} P_S^*(\{ \otimes^!_{i\in I_s} L_i\},M)\otimes ( \otimes_{s\in S} \Lie(|I_s|))   \ ,  
\end{equation}
with the direct sum running over quotients, $S$, of $I$ whose cardinality is $p+1$.

%
In particular, we have
\[ P_I^c(\{ L_i\},M)^{\langle 0\rangle} = P^!_I(\{L_i\},M)\otimes \Lie(|I|) \qquad \text{and}\qquad  P_I^c(\{ L_i\},M)^{\langle |I|-1 \rangle}= P^*_I(\{L_i\},M) \ , \]
and so one should, heuristically, think of the auxiliary grading as counting the number of $*$-operations appearing at a given arity. This gives rise to the following maps of coloured operads:
\begin{equation}\label{alphabetaeqn}
	\DD(X)^!\otimes^H \Lie \xrightarrow{\alpha} \DD(X)^c \xrightarrow{\beta} \DD(X)^*  
\end{equation}
extending the identity functor on $\DD(X)$, where $\otimes^H$ denotes the Hadamard tensor product of $\K$-linear operads. Similarly, the following are proved in \cite{BD1}:
\begin{lemma} For any coloured operad $\mc P$, there is a natural isomorphism
	\[ \Alg_{\Lie}(\mc P\otimes^H \Lie) = \Hom_{\Op_\K}( \Lie, \mc P\otimes^H \Lie) \xrightarrow{\cong}    \Hom_{\Op_\K}( \Comm, \mc P)=\Alg_{\Comm}(\mc P)\]
\end{lemma}
\begin{proof} This is proved in section 1.1.10 in \cite{BD1}.
\end{proof}

\begin{prop}\label{CoisLieprop} The category $\Alg_{\Lie}(\DD(X)^c)$ is equivalent to that of (non-unital) coisson algebras, such that the induced functors
\[ \Alg_{\Lie}(\DD(X)^! \otimes \Lie) \to  \Alg_{\Lie}(\DD(X)^c) \to  \Alg_{\Lie}(\DD(X)^*) \]
correspond to the inclusion of $\Comm^!$ algebras as coisson algebras with trivial $\Lie^*$ bracket, and the forgetful functor on coisson algebras remembering only the underlying $\Lie^*$ bracket, respectively.
\end{prop}
\begin{proof}
This is precisely Lemma 1.4.28 of \cite{BD1}, with some aspects of the proof explicitly stated. 
\end{proof}
Concretely, at the level of binary operations, the preceding Proposition \ref{CoisLieprop} follows from the fact that
\begin{equation}\label{cl2arydecompeqn}
	 P_2^c(\{L,L\},L) =\left( P_2^!(\{L,L\},L)\otimes \Lie(2)\right)  \oplus P_2^*(\{L,L\},L) ~,
\end{equation}
where the two summands correspond to the two equivalence relations on the set $\{1,2\}$, which have $p=0$ and $p=1$ in the above grading, respectively, and thus the 2-ary Lie bracket structure map internal to $P_2^c(\{L,L\},L)$ consists of precisely a commutative 2-ary operation internal to $\DD(X)^!$ and an alternating 2-ary operation internal to $\DD(X)^*$.

From the definition, we have a forgetful functor
\[
\Alg_{\Lie}(\DD(X)^c) \rightarrow \Alg_{\Lie}(\DD(X)^*)
\]
from coisson algebras to $\Lie^*$ algebras. If $L$ is a $\Lie^*$ algebra, then the $\Comm^!$ algebra $\Sym(L)$ has a unique coisson bracket such that the inclusion $L\hookrightarrow \Sym(L)$ is a morphism of $\Lie^*$ algebra---this is the Kirillov--Kostant--Souriau bracket. Indeed,
\[
\Sym: \Alg_{\Lie}(\DD(X)^*)\rightarrow \Alg_{\Lie}(\DD(X)^c)~,
\]
is left adjoint to the forgetful functor.

The endomorphism operad of an object $M\in \DD(X)^c$ can be described concretely by the following:

\begin{prop}\label{Oendcprop}

The spaces of operations of $\OEnd_{\DD(X)^c}(M)\in \Op_\K$ are given by
\[ 
\OEnd_{\DD(X)^c}(M)(n) = \bigoplus_{p=0}^{n-1}  \left( \bigoplus_{n_1+...+n_{p+1}=n} \Ind_{\Ss_{n_1}\times ... \times \Ss_{n_{p+1}}}^{\Ss_n}\left(P_{p+1}^*(\{ M^{\otimes^! n_i}\}_{i=1}^{p+1}  ,M)\otimes \bigotimes_{i=1}^{p+1} \Lie(n_i)\right) \right)_{\Ss_{p+1}} \ .  
\]
\end{prop}
\begin{proof}
Recall the auxiliary grading on $\DD(X)^c$,
\[ 
\OEnd_{\DD(X)^c}(M)(n) =P_n^c(\{ M\}_{i=1}^n,M)= \bigoplus_{p=0}^{n-1}P^c_n(\{M\}_{i=1}^n,M)^{\langle p\rangle} \ ,
\]
where $P^c_n(\{M\}^n_{i=1},M)^{\langle p\rangle}$ denotes the sum over quotients of cardinality $p+1$. Fix an ordering so that each quotient of cardinality $p+1$ may be identified with $\{1,\dots,p+1\}$. Summing over all quotients of cardinality $p+1$ is then almost the same as summing over all surjections $\{1,\dots,n\}\onto \{1,\dots,p+1\}$, except we should take $\Ss_{p+1}$-coinvariants to account for the choice of ordering.

The group $\Ss_n$ acts on the space of such surjections with equivalence classes labelled by the set of decompositions $n_1+\dots+n_p=n$. The stabilizer of the standard representative of such an equivalence class is $\Ss_{n_1}\times \dots \times \Ss_{n_{p+1}}$. Summing over all surjections is equivalent to summing over all decompositions but accounting for the size of the orbit and so we obtain
\[
P^c_n(\{M\},M)^{\langle p\rangle} = \left( \bigoplus_{n_1+...+n_{p+1}=n} \Ind_{\Ss_{n_1}\times ... \times \Ss_{n_{p+1}}}^{\Ss_n} \bigg( P_{p+1}^*(\{ M^{\otimes^! n_i}\}_{i=1}^{p+1}  ,M)\otimes \bigotimes_{i=1}^{p+1} \Lie(n_i)\bigg)  \right)_{\Ss_{p+1}} \ . 
\]

\end{proof}


\section{Algebras of the vertex kind}\label{sec:vertex algebras}

Fix $X=\IA^1$, as our algebraic curve. It is known, see Section $0.15$ of \cite{BD1}, that weakly translation invariant chiral, coisson or $\Lie^*$ algebras on $\IA^1$ recover the more traditional notions of vertex algebra, vertex Poisson algebra, and vertex Lie algebra, respectively. In this section, we review the definitions of the three  aforementioned algebras and spell out this equivalence, following the exposition in Appendix A of \cite{BDSHK19}.

The additive group $\IG_a$ acts on $\IA^1$ acts on $\IA^1$ by translations and we have an equivalence of categories,
\[
\DD(\IA^1)^{\IG_a,w} \simeq \ik[\partial]\textup{-Mod}
\]
where $\ik[\partial]$ is the ring of polynomials in the indeterminate $\del$. The functor $\ik[\partial]\textup{-Mod}\rightarrow D(\IA^1)^{\IG_a,w}$ is given by
\[
V \mapsto V\otimes \OO_{\IA^1}~,
\]
with D-module structure defined by letting the section $z\in \OO_{\IA^1}$ act by multiplication on $\OO_{\IA^1}$ and letting $\partial_z\in \Theta_{\IA^1}$ act as $\partial - \partial_z$. The inverse functor is given by taking the fibre at zero.

\subsection{Preliminaries and conventions}


\subsubsection{Commutative algebras with derivation}

The category $\ik[\partial]$-Mod is symmetric monoidal with the tensor product $\otimes = \tensor{\ik}$ and the action of $\partial$ given by the coproduct $\partial\mapsto \partial\otimes1 +1\otimes \partial$. We write $\ik[\partial]\Mod^!$ for the pseudo-tensor structure induced by this tensor product.
\begin{prop}
	A commutative algebra internal to $\ik[\partial]\Mod^!$ is a commutative algebra over $\ik$ equipped with an action of $\partial$ by derivations.
\end{prop}

The forgetful functor $\Comm(\ik[\partial]\Mod^!)\rightarrow \Comm(\Vect_\ik)$ admits a left adjoint which sends a commutative $\ik$-algebra $A$ to the jet algebra $\JJ A\in \Comm(\ik[\partial]\Mod^!)$.
Indeed, the corresponding schemes $\Spec \JJ A$ are an interesting class of weakly translation invariant $\D_{\IA^1}$-schemes.

\subsubsection{Vertex Lie algebras}

We recall the definition of a vertex Lie algebra, also known as a Lie conformal algebra or (rarely) simply a conformal algebra.

\begin{defn}\label{defn:vertex Lie algebra}
	A vertex Lie algebra is a pair $(L,[\cdot_\lambda\cdot])$ consisting of an $L\in \ik[\partial]$-Mod and a $\lambda$-bracket
	\[
	[\cdot_\lambda \cdot]: L\otimes L \rightarrow L[\lambda]~,
	\]
	satisfying
	\begin{enumerate}[(i)]
		\item  $[{a}_\lambda b] = - [{a}_{-\lambda-\del} b]$ \emph{(skewsymmetry)}
		\item $[\partial {a}_\lambda b] = -\lambda [{a}_\lambda b]$ and $[{a}_\lambda \partial b] = (\lambda+\partial)[{a}_\lambda b]$ \emph{(sesquilinearity)}
		\item $[a_\lambda [b_\mu c]] - [c_\mu[a_\lambda b]]= [[a_\lambda b]_{\lambda+\mu} c]$ \emph{(Jacobi identity)}
	\end{enumerate}
	for $a,b,c\in L$.
\end{defn}
\subsubsection{Vertex algebras}

We adopt the following, somewhat unconventional, definition of a vertex algebra, first appearing in the work of  \cite{KacDeSole06}.
\begin{defn}
A (non-unital) vertex algebra is a pair $(V,I_\lambda)$ consisting of an $L\in \ik[\partial]$-Mod and a so-called integral $\lambda$-bracket
	\[
	I_\lambda: V\otimes V\rightarrow V[\lambda]~,
	\]
	satisfying
	\begin{enumerate}[(i)]
		\item $\partial I_\lambda(a,b) = I_\lambda(\partial a,b)+I_\lambda(a,\partial b)$ and $\tfrac{d}{d\lambda}I_\lambda(\partial a,b) = \lambda \tfrac{d}{d\lambda}I_\lambda(a,b)$ \emph{(sesquilinearity)}
		\item $I_\lambda(a,b) = I_{-\lambda-\partial}(b,a)$ \emph{(skewsymmetry)}
		\item $I_\lambda(a, I_\mu(b,c)) - I_\mu(b,I_\lambda(a,c)) - I_{\lambda+\mu}(I_\lambda(a,b),c) = 0$ \emph{(Jacobi identity)}
	\end{enumerate}
A \emph{unital} vertex algebra is a vertex algebra $(V,I_\lambda)$ with a distinguished vector $\ket{0}\in \ker\partial\subset V$ such that, for any $a\in V$
\[
I_\lambda(a,\ket{0})= a~.
\]
\end{defn}

\begin{rmk}
Often, one splits the integral $\lambda$-bracket into two operations
\[
:\,\,: \,\,\, = I_\lambda(\cdot,\cdot)|_{\lambda=0} ~, \text{ and } [\cdot_\lambda \cdot] = \frac{d}{d\lambda}I_\lambda(\cdot,\cdot)
\]
where $[\cdot_\lambda\cdot]$ is an honest $\lambda$-bracket on $V$, making it a vertex Lie algebra, and $:\,\,:$, the normally ordered product, is a noncommutative, nonassociative binary operation.
	
\end{rmk}

\begin{rmk}
	The usual definition of a vertex algebra structure on $L\in\ik[\partial]$-Mod consists of an integer family of products $\{_{(n)},n\in\IZ\}$
	\[
	_{(n)} : L\otimes L \rightarrow L ~,
	\]
	satisfying some conditions. One can recover the products from the integral $\lambda$-bracket as 
	\[
	{a}_{(n)} b = 
	\begin{cases}
	\tfrac{1}{n!}\tfrac{d^{n+1}}{d\lambda^{n+1}} I_\lambda(a,b)|_{\lambda=0}~, \text{ for }n\ge0\\
	\partial^{1+n} I_\lambda(a,b)|_{\lambda=0} ~, \text{ for } n <0
	\end{cases} 
	~.
	\]
	The usual conditions on the products are equivalent to the conditions $(i)$--$(iii)$ for $I_\lambda$, see \cite{KacDeSole06}.
\end{rmk}

\begin{rmk}
We say that a subspace $W\subset V$ \emph{strongly generates} $V$ if any vector in $V$ can be written as a linear combination of normally ordered monomials
\[
:\partial^{i_1}w_1 : \partial^{i_2}w_2:_{\dots\dots}:\partial^{i_{k-1}}w_{k-1}\partial{i_k} w_k:_{\dots\dots} :
\]
for $w_i\in W$. In the literature, one usually says that a \emph{set} strongly generates $V$ if the $w_i$ above can be chosen amongst this set---clearly any basis of $W$ gives a strongly generating set of $V$.
\end{rmk}


\subsubsection{Vertex Poisson algebras}

\begin{defn}
	Let $V\in\Comm(\ik[\partial]\Mod^!)$, \ie $V$ is a commutative $\ik$-algebra with derivation $\del$. A vertex Poisson algebra structure on $V$ is a $\lambda$-bracket 
	\[
	[\cdot_\lambda\cdot]:V\otimes V\rightarrow V[\lambda]~,
	\]
	such that 
	\begin{enumerate}[(i)]
		\item $[\cdot_\lambda\cdot]$ defines a vertex Lie algebra structure on the underlying $\ik[\partial]$-module of $V$,
		\item $[\cdot_\lambda \cdot]$ distributes over the commutative product on $V$ via
		\begin{itemize}
			\item $[{a}_\lambda b c] = [{a}_\lambda,b]c + [{a}_\lambda c]b$ \emph{(right Leibniz rule)}
			\item $[a{b}_\lambda c]= [{a}_{\lambda+\partial} c]_{\rightarrow}b + [{b}_{\lambda+\partial} c]_{\rightarrow} a$ \emph{(left Leibniz rule)}
		\end{itemize}
		where the arrow subscript means that $\del$ acts to the right of the bracket.
		\end{enumerate}
\end{defn}

A rich source of vertex Poisson algebras comes from ordinary Poisson algebras and vertex Lie algebras as explained in the propositions below.

\begin{prop}
	Let $Y$ be a Poisson variety over $\ik$, then $\ik[\JJ Y]$ has a unique vertex Poisson algebra structure such that 
	\[
	[a_\lambda b] = \{a,b\}~,
	\]
	where $a,b\in \ik[Y]\subset \ik[\JJ Y]$ and $\{\cdot,\cdot\}$ is the Poisson bracket on $\ik[Y]$.
\end{prop}

\begin{prop}
	Let $(L,[\cdot_\lambda\cdot])$ be a vertex Lie algebra then $\Sym(L)$ has a unique vertex Poisson algebra structure given by the bracket $\{\cdot_\lambda\cdot\}$ such that
	\[
	\{a_\lambda b\} = [a_\lambda b]~,
	\]
	for $a,b\in L= \Sym^1(L)\subset \Sym(L)$.
\end{prop}

\subsection{Chiral algebras on the line and vertex algebras}\label{ssec:weakly trans inv chiral alg}
Note that the equivalence $\DD(\IA^1)^{\IG_a,w} \simeq \ik[\partial]\textup{-Mod}$ lifts to an equivalence of symmetric monoidal categories
\[
\DD(\IA^1)^!\simeq \ik[\partial]\textup{-Mod}^!
\]
The following proposition follows immediately from unravelling the definitions.
\begin{prop}
A weakly translation invariant $\Comm^!$ algebra on $\IA^1$, \ie, a commutative algebra internal to $\DD(\IA^1)^{\IG_a,w,!}$ is equivalent to a commutative algebra in $\ik$-Mod, equipped with a derivation.
\end{prop}

For the rest of this subsection, we fix a finite set $I\in \fset$ and an $I$-family $L_i\in \ik[\partial]\Mod$, for $i\in I$, with corresponding $I$-family of weakly translation invariant D-modules $\widetilde{L_i}\in\DD(X)^{\IG_a,w}$. Fix also a $M\in \ik[\del]\Mod$ with corresponding D-module $\widetilde{M}\in \DD(X)^{\IG_a,w}$.

Let $\Delta^{(I)}:\IA^1\rightarrow \IA^I$ denote the closed embedding of the smallest diagonal and let $j^{(I)}:U^{(I)} \rightarrow \IA^I$ denote the open embedding of the complement of the union of all diagonals.

\subsubsection{The star pseudotensor structure} 
We start with defining the star pseudo-tensor structure. The global sections of $\Delta^{(I)}_*\widetilde{M}$ can be computed in terms of $M$ as
\[
\Gamma(\IA^I, \Delta^{(I)}_* \widetilde{M}) \cong M\tensor{\ik[\partial]} \ik[z,\lambda_i|i\in I]\in \mc{D}(\IA^I)\textup{-Mod}~,
\]
where $\partial_{z_i}$ acts as $-\lambda_i$ and $z_i$ acts as $z - \partial_{\lambda_i}$. For $L_i\in \ik[\partial]$-Mod with associated D-modules $\widetilde{L}_i\in \DD(X)^{\IG_a,w}$, we have the isomorphisms
\begin{align*}
	\Hom_{\D(\IA^I)}\big(\Gamma(\IA^I, \boxt_{i\in I}\widetilde{L}_i),\Gamma(\IA^I,\Delta_* \widetilde{M}))\big) 
	&\cong \Hom_{\D(\IA^I)}\big(\tensor{i\in I} L_i[z_i],\, M\tensor{\ik[\partial]} \ik[z,\lambda_i|i\in I]\big)~,\\
	&\cong \Hom_{\ik[\partial_i|i\in I]}\big(\tensor{i\in I} L_i, M\tensor{\ik[\partial]} \ik[z,\lambda_i|i\in I]\big)~.
\end{align*}
where $\partial_i$ acts on $L_i$ and $\partial_i$ acts as $\partial_{z_i}$ on the codomain. An element $t\in\IG_a$ acts on the codomain by $z\mapsto z+t$. Thus, the translation-invariant morphisms are precisely those whose images land in the subspace $M\otimes_{\ik[\partial]}\ik[\lambda_i|i\in I]\subset M\otimes_{\ik[\del]}\ik[z,\lambda_i |i\in I]$.

\begin{defn}
	The star pseudo-tensor structure on $\ik[\partial]$-Mod, denoted as $\ik[\partial]\textup{-Mod}^*$, is defined by having multilinear operations
	\[
	P^{*,\IG_a}(\{L_i\}_{i\in I}, M) = \Hom_{\ik[\partial_i|i\in I]}(\tensor{i\in I} L_i, M\tensor{\ik[\partial]}\ik[\lambda_i| i\in I])~.
	\]
	We forgo defining the composition of operations but note that in the case of a single colour the compositions are defined in \cite{BDSHK19}.
\end{defn}

\begin{prop}[\cite{BD1,BDSHK19}]\label{prop:Lie star is vertex Lie}
	A vertex Lie algebra is a Lie algebra internal to $\ik[\partial]\textup{-Mod}^*$, or equivalently a weakly translation invariant $\Lie^*$ algebra on $\IA^1$.
\end{prop}
\begin{proof}
 The subspace of binary operations on a single colour $L$ is
 \[
 P_2^*(\{L,L\},L) = \Hom_{\ik[\partial]\otimes\ik[\partial]}\big(L\otimes L, L[\lambda]\big)~,
 \]
 where we identify $L\tensor{\ik[\partial]}\ik[\lambda_1,\lambda_2]$ with $L[\lambda]$ by $\lambda_1=\lambda$ and $\lambda_2= -\partial-\lambda$. A Lie algebra structure on $L$ is fixed by choosing some $b\in P_2^*(\{L,L\},L)$ satisfying the Jacobi relation.

 One recognises $P_2^*(\{L,L\},L)$ as the space of $\lambda$-brackets satisfying conditions $(i)$ and $(ii)$ of Definition \ref{defn:vertex Lie algebra}, \ie, skewsymmetry and sesquilinearity. Condition $(iii)$ being equivalent to $b\in P_2^*(\{L,L\},L)$ satisfying the Jacobi relation is the content of Proposition 5.1 of \cite{BDSHK19}.
\end{proof}

\subsubsection{The chiral pseudotensor structure} 
Let $\OO_I^\star= \OO(U^{(I)})$ be the co-ordinate ring of the complement of the union of all diagonals in $\IA^I$, then $\OO_I^\star\in \DD(\IA^I)$.

Then,
\[
\Gamma(\IA^I,j^{(I)}_*j^{(I),*}(\boxtimes_{i\in I} \widetilde{L}_i)) \cong\tensor{i\in I} L_i\otimes \OO_I^{\star}\in \mc{D}(\IA^I)\textup{-Mod}~,
\]
with $z_i$ acting on the right on $\OO_n^\star$ by multiplication and $\partial_{z_i}$ acting as $\partial^{(i)} + \partial_{z_i}$, where $\partial^{(i)}$ acts on $L_i$. Therefore, we have an isomorphism
\[
 \Hom_{D(\IA^I)}(j^{(I)}_*j^{(I)*}(\boxtimes_{i\in I} L),\Delta^{(I)}_*M) \cong \Hom_{{\mc D}(\IA^n)}\big(\tensor{i\in I} L_i\otimes \OO_I^\star, M\tensor{\ik[\partial]}\ik[z,\lambda_i|i\in I]\big)~.
\]

\begin{prop}[\cite{BDSHK19}]
The subspace of translation invariant morphisms can be identified with
\[
 \Hom_{{\mc D}(\IA^I)^{\IG_a}}\big(\tensor{i\in I} L_i\otimes \OO_I^{\star,\IG_a}, M\tensor{\ik[\partial]}\ik[\lambda_i|i\in I]\big)~,
\]
\end{prop}
\begin{proof}
	This is Lemma A.1 of \cite{BDSHK19}.
\end{proof}

\begin{defn}
	The chiral pseudo-tensor structure on $\ik[\partial]$-Mod, denoted by $\ik[\partial]\textup{-Mod}^\ch$, is defined by having space of multilinear operations
	\[
	P^{ch}_I(\{L_i\}_{i\in I}, M) = \Hom_{{\mc D}(\IA^I)^{\IG_a}}\big(\tensor{i\in I} L_i\otimes \OO_I^{\star,\IG_a}, M\tensor{\ik[\partial]}\ik[\lambda_i|i\in I]\big)~.
	\]
\end{defn}

\begin{prop}[\cite{BD1,BDSHK19}]
	A vertex algebra is a Lie algebra internal to the operad $\ik[\partial]\textup{-Mod}^{\rm ch}$.
\end{prop}
\begin{proof}
	The subspace of binary operations on a single colour $L\in\ik[\del]\Mod$ is 
	\[
	P^{\ch}_2(\{L,L\},L) = \Hom_{\mc{D}(\IA^2)^{\IG_a}}(L\otimes L \otimes \ik[(z_1-z_2)^{\pm1}],L[\lambda])~,
	\]
	where we identify $\lambda_1=\lambda$ and $\lambda_2 = - \partial-\lambda$. Any such morphism is fully determined by the image of $L\otimes L \otimes (z_1-z_2)^{-1}$. Let $\mu\in P_2^{\ch}(\{L,L\},L)$; for $l_1,l_2\in L$, define a putative integral $\lambda$-bracket
	\[
	I_\lambda(l_1,l_2) = \mu(l_1,l_2,(z_1-z_2)^{-1})~.
	\]
	That $I_\lambda$ satisfies the axioms of an integral $\lambda$-bracket if and only if $\mu$ satisfies the Jacobi relation is, essentially, the content of Theorem $6.12$ of \cite{BDSHK19}.
\end{proof}

\subsubsection{The classical pseudotensor structure}
Recall that a vertex Poisson algebra is a vertex Lie algebra structure on a commutative algebra with derivation---satisfying a distributivity law. In light of the preceding discussion, we can recast this as a weakly translation invariant $\Lie^*$ structure on a weakly translation invariant $\Comm^!$ algebra on $\IA^1$.

\begin{prop}
	A vertex Poisson algebra is a weakly translation invariant coisson algebra on $\IA^1$.
\end{prop}

As before let $Q(I)$ denote the set of equivalence relations on $I$ and we enumerate its elements by $(\pi_S,S)$, where $S= I/\sim$ and $\pi_S: I\rightarrow S$ is the projection.
\begin{defn}
	The classical pseudotensor structure on $\ik[\partial]\Mod$, denoted as $\ik[\partial]\Mod^c$ has spaces of multilinear operations
	\[
	P^c_I(\{L_i\}_{i\in I},M) = \bigoplus_{(\pi_S,S)\in Q(I)} P^c_I(\{L_i\}_{i\in I},M)_S
	\]
	where 
	\[
	P^c_I(\{L_i\}_{i\in I},M) = P_S^*(\{\otimes_{i\in I_s}^! L_i\}_{s\in S},M)\otimes \big(\otimes_{s\in S}\Lie(|I_s|)\big)~.
	\]
	with 
	\[
	P_S^*(\{\otimes_{i\in I_s}^! L_i \}_{s\in S},M) = \Hom_{\otimes_{s\in S}\ik[\partial]}\big(\otimes_{s\in S} (\otimes_{i\in I_s}^!L_i), M\tensor{\ik[\partial]}\ik[\lambda_s|s\in S]\big)~.
	\]
\end{defn}
\begin{prop}[\cite{BDSHK19}]
	A vertex Poisson algebra is a Lie algebra internal to the operad $\ik[\partial]\Mod^c$.
\end{prop}
\begin{proof}
The subspace of binary operations on a single colour is 
\[
P^c_2(\{L,L\},L) = \Hom_{\ik[\partial]\otimes \ik[\partial]}(L\otimes L , L[\lambda]) \oplus\Hom_{\ik[\partial]}(L\otimes^!L,L)\otimes\Lie(2)
\]
and so the binary Lie bracket structure map internal to $P^c_2(\{L,L\},L)$ decomposes as a direct sum of a commutative multiplication in $\ik[\partial]\Mod^!$ and a Lie bracket internal to $\ik[\partial]\Mod^*$. Checking that the Jacobi relations imply the remaining conditions is the content of Theorem 10.7 of \cite{BDSHK19}. 
\end{proof}

\subsection{Chiral quantization}\label{ssec:chiral quantization intro}

Before formalizing the notion of chiral quantization in the language of deformation theory, we recall a more hands on definition of chiral quantization and review a few examples. First, however, we need to introduce $\hbar$-adic vertex algebras.


\begin{defn}
	A $\hbar$-adic (non-unital) vertex algebra is a pair $(V_\hhb,I_{\hhb,\lambda})$ where 
	\begin{enumerate}[(i)]
		\item $V_\hhb\in\ik[\partial]\Mod(\ik\fph\Mod)$, is such that the underlying $\ik[\![\hbar]\!]$-module is complete in the $\hbar$-adic topology and trivializable, \ie, there exists an isomorphism $V_\hhb \xrightarrow{\sim} V_0 \widehat{\otimes}\ik\fph$, where $V_0 = V\otimes_{\ik\fph}\ik_0$
		\item $I_{_\hhb,\lambda}$ is an $\hbar$-adic integral $\lambda$-bracket $V_\hhb\otimes_{\ik\fph}V_\hhb\rightarrow V_\hhb[\lambda]$,
	\end{enumerate}
	such that for each $m\in \IN$, the truncations $(V/\hbar^m , I_{\lambda}/\hbar^m)$ are vertex algebras over $\ik[\hbar]/\hbar^m$. 

	A \emph{unital} $\hbar$-adic vertex algebra is a $\hbar$-adic vertex algebra with a distinguished element $\ket{0}\in V_\hhb$ such that each truncation, modulo $\hbar^m$, is a unital vertex algebra over $\ik[\hbar]/\hbar^m$. 
\end{defn}

Consider the action of the multiplicative group on $\ik\fph$ which scales $\hbar$ with unit weight---we denote this distinguished multiplicative group by $\IG_\hbar$. 

\begin{lemma}\label{lem:filtered is a VOA}
	Let $V_\hhb$ be a $\hbar$-adic vertex algebra with an action of $\IG_\hbar$ by automorphisms. If the induced $\IG_\hbar$-grading on $V_\hhb$ is bounded below, the filtered $\ik[\partial]$-module 
	\[
	{\rm Fil}(V_\hhb) \coloneqq \big( V_\hhb\tensor{\ik[\![\hbar]\!]} \ik(\!(\hbar)\!)\big)^{\IG_\hbar}~,
	\]
	is a vertex algebra over $\ik$. 
\end{lemma}
\begin{proof}
Since the grading on $V_\hhb$ is bounded below, ${\rm Fil}(V_\hhb)$ is a filtered $\ik$-vector space, \ie, it contains no infinite linear combinations. Thus, to show that ${\rm Fil}(V_\hhb)$ is a vertex algebra all that remains is to show that the restriction of $I_{\hhb,\lambda}$ to ${\rm Fil}(V_\hhb)$ preserves it. 

For any homogeneous $v,w\in V_\hhb$, the coefficients of $\lambda$ in $I_{\lambda}(v,w)$ are polynomial in $\hbar$ since the grading is bounded below. Thus, the restriction of $I_{\hhb,\lambda}$ to ${\rm Fil}(V)$ is an integral $\lambda$-bracket.
\end{proof}

\begin{prop}\label{prop:Rees for hadic VOA}
	The filtered construction gives an equivalence of categories between graded $\hbar$-adic vertex algebras with bounded below grading and filtered vertex algebras over $\ik$:
	\begin{align*}
	{\rm Vert}_{\ik\fph}^{\gr,+} &\xrightarrow{\sim} {\rm Vert}_{\ik}^{\rm fil}~,\\
	V_\hhb &\mapsto {\rm Fil}(V)
	\end{align*}
\end{prop}
\begin{proof}
 This is essentially the Rees correspondence for filtered vertex algebras.
\end{proof}

A $\hbar$-adic unital vertex algebra $V_\hhb, I_{\hhb,\lambda}$ is said to be \emph{almost commutative} if the image of
\[
I_{0,\lambda} : V_0 \otimes V_0 \rightarrow V_0[\lambda]
\]
is contained in $V_0$. Then, $I_{0,\lambda}$ actually defines a commutative product on $V/\hbar V$ making it a unital commutative algebra with derivation.
\begin{prop}\label{prop:almost commutative hbar central fibre}
	Suppose $(V_\hhb,I_{\hhb,\lambda})$ is an almost commutative, unital $\hbar$-adic vertex algebra. Then $V_0$ is naturally a vertex Poisson algebra with commutative multiplication $I_{0,\lambda}$ and $\lambda$-bracket
	\[
	[\cdot_\lambda \cdot] = \lim_{\hbar\rightarrow 0} \bigg(\frac{1}{\hbar}\frac{d}{d\lambda}I_{\hhb,\lambda}(\cdot,\cdot) \bigg)~.
	\]
\end{prop}

\begin{defn}\label{defn:chiral quantization}
A \emph{quantization} of a vertex Poisson algebra, $V$, is an almost commutative $\hbar$-adic vertex algebra $(V_\hhb,I_{\hhb,\lambda}$ such that $V_0\cong V$ as vertex Poisson algebras, where $V_0$ possesses the vertex Poisson algebra structure from Proposition \ref{prop:almost commutative hbar central fibre}.

A \emph{chiral quantization} of a Poisson variety $Y$ is a quantization of the vertex Poisson algebra $\ik[\JJ Y]$.
\end{defn}

\subsubsection{Current algebras}
Let $\gf$ be a simple Lie algebra and let $\gf^*$ be its linear dual, considered as Poisson variety with the Kirillov--Kostant--Souriau Poisson bracket. While $\gf^*$ is not symplectic, it is smooth and so falls within the purview of our results. The functions on the arc space $\OO(\JJ\gf^*)$, have the structure of a coisson algebra, with $\lambda$-bracket
\[
[x_\lambda y]  = [x,y]~,
\]
for $x,y\in \gf\subset \OO(\JJ \gf^*)$.

For $k\in \ik$, we write $V^k(\gf)$ for the universal affine vertex algebra of $\gf$. This is a vertex algebra strongly generated by $\gf\subset V^k_\hbar(\gf)$ with integral $\lambda$-bracket
\[
I_\lambda(x,y) =\,  :xy: +\lambda  [x,y] + \frac{\lambda^2}{4 h^\vee} \langle x,y\rangle~,
\]
corresponding to the OPE
\[
x(z)y(w) \sim \frac{k}{2h^\vee}\frac{ \langle x, y\rangle}{(z-w)^2} + \frac{[x,y](w)}{z-w}~,
\]
for $x,y\in \gf$, $\langle \cdot,\cdot\rangle$ the Killing form, and $h^\vee$ the dual Coxeter number. Similarly, let $V^k_\hbar(\gf)$ denote the homogenized universal affine vertex algebra of $\gf$ at level $k$. This is a $\hbar$-adic vertex algebra strongly generated by the subspace $\gf\subset V^k_\hbar(\gf)$ with integral $\lambda$-bracket
\[
I_\lambda(x,y) =\,  :xy: +\lambda \hbar [x,y] + \frac{\lambda^2\hbar^2}{4 h^\vee} \langle x,y\rangle~,
\]

\begin{prop}[\cite{Arakawa2015:local}]
	We have an isomorphism of vertex algebras over $\ik$,
	\[
	{\rm Fil}(V^k_{\hbar}(\gf)) = \big(V^k_\hbar(\gf)\otimes_{\ik\fph}\ik(\!(\hbar)\!)\big)^{\IG_\hbar} \cong V^k(\gf)~.
	\]
\end{prop}
\begin{proof}
	This is a special case of Theorem 3.2.6.1 of \cite{Arakawa2015:local}.
\end{proof}

\begin{prop}
 The $\hbar$-adic vertex algebras $V_\hbar^k(\gf)$ give a one-parameter family of graded chiral quantizations of $\gf^*$.
\end{prop}

\begin{rmk}
	Indeed, this is a special case of a much larger class of quantizations: given a vertex Lie algebra $L$, the homogenized chiral envelope $\VV_\hbar(L)$ is a quantization of the coisson algebra $\Sym(L)$, see \cite{Li04,Castellan:2023jnu} for more details.
\end{rmk}

\subsubsection{Chiral differential operators}\label{sssec:cdos intro}

Let $Y$ be a smooth affine variety with vanishing second Chern class%
\footnote{Chiral differential operators can actually be defined on more general smooth schemes with vanishing second Chern class but we shall only consider the affine case}
. Let $\alpha\in \Omega^{3,cl}_{\rm dR}(M)$ be a closed three-form viewed as a map
\[
\alpha:\Theta(Y)\otimes \Theta(Y)\rightarrow \Omega^1(Y)\cong \partial\OO(Y)~,
\]
and extended to $\Theta(\JJ Y)$ by $\del$-linearity. Write $L_Y^\alpha$ for the extension%
\footnote{In the language of \cite{BD1}, $L^\alpha_Y$ is a chiral $\Lie^*$ algebroid and is a chiral extension of the tangent $\Lie^*$ algebroid of $\JJ Y$. We shall sweep this detail under the rug.} 
\[
	0\rightarrow \OO(\JJ Y)\hookrightarrow L^\alpha_Y \twoheadrightarrow \Theta(\JJ Y) \rightarrow 0
\]
of the tangent vertex Lie algebroid, $\Theta(\JJ Y)$, of $\JJ Y$. The closed three form defines a $\lambda$-bracket on $L_Y^\alpha$ via
\begin{equation}
\begin{split}
 [x_\lambda y]    &= [x,y]- \alpha(x,y) ~,\\
[x_\lambda a]  &= x(a)~,
\end{split}
\end{equation}
for $x,y\in \Theta(Y)\subset \Theta(\JJ Y)$ and $a\in \OO(Y)\subset \OO(\JJ Y)$. 

\begin{defn}
	A vertex algebra of chiral differential operators (cdos), $\D^{\ch}_{\alpha}(Y)$, is defined as the vertex envelope of $L_Y^\alpha$ for some $\alpha\in \Omega^{c,cl}_{\rm dR}(Y)$.

	Concretely $\D^{\ch}_{\alpha}(Y)$ is strongly generated by the subspace $L_Y^\alpha\subset \D^{\ch}_{\alpha}(Y)$ with non-vanishing $\lambda$-brackets
	\[
	I_{\lambda}(x,y) =  [x,y] - \alpha(x,y) ~, \text{ and } I_\lambda(x,a) = x(a) ~. 
	\]
\end{defn}



One can define the $\hbar$-adic analogue, $L^\alpha_{\hbar,Y}$, by homogenizing the $\lambda$-brackets 
\begin{equation}
\begin{split}
 [x_\lambda y]    &= \hbar [x,y]- \hbar^2 c(x,y) ~,\\
[x_\lambda a]  &=\hbar x(a)~.
\end{split}
\end{equation}
Note that $L_{\hbar,Y}^\alpha$ is graded by $\IG_\hbar$ with $\Theta(\JJ Y)$ having weight one and $\OO(\JJ Y)$ having weight zero. Analogously, we define the $\hbar$-adic chiral differential operators $\D^{\ch}_{\hbar,\alpha}$ as the vertex envelope of $L_{\hbar,Y}^\alpha$.

\begin{prop}\label{prop:hadic cdos are cdos}
	We have an isomorphism of vertex algebras over $\ik$
	\[
	{\rm Fil}(\D^{\ch}_{\hbar,\alpha}(Y)) \cong \D^{\ch}_{\alpha}(Y)
	\]
\end{prop}
\begin{proof}
	Proven in Appendix A of \cite{BuN2}.
\end{proof}

\begin{prop}
	For every $\alpha\in \Omega^{3,cl}(Y)$, $\D^{\ch}_{\hbar,\alpha}$ is a graded chiral quantization of $T^*Y$.
\end{prop}
\begin{proof}
  As a commutative algebra (with derivation) the classical limit of the vertex envelope is just $\Sym(\Theta(\JJ Y))\cong \ik[\JJ T^*Y]$. Furthermore, the $\hbar \rightarrow 0$ limit of the $\lambda$-brackets of $L^\alpha_{\hbar,Y}$ recovers the usual $\lambda$-bracket on $\ik[\JJ T^*Y]$ induced from the symplectic structure on $T^*Y$ and thus $\D^{\ch}_{\hbar,\alpha}$ is a chiral quantization of $T^*Y$.

	The grading on $L_{\hbar,Y}^\alpha$ induces a $\IG_\hbar$-grading on $\D^{\ch}_{\hbar,\alpha}$ making it a graded chiral quantization.
\end{proof}

\subsubsection{The Virasoro algebra and its minimal models}\label{ssec:dfn virasoro}

Let $\Vir^c$ be the Virasoro vertex algebra at central charge $c\in \ik$, \ie, the algebra strongly generated by the field $T$ with integral $\lambda$-bracket
\[
I_\lambda(T,T)= \frac{1}{8}\lambda^4 c + \lambda^2 T + \lambda \del T~,
\]
corresponding to the OPE
\[
T(z)T(w) \sim \frac{c/2}{(z-w)^4} + \frac{2 T(w)}{(z-w)^2} + \frac{\partial T(w)}{z-w}~.
\]

It is well known that for certain values of the central charge, $c = c_{p,q}$, with $p,q$ coprime and 
\[
c_{p,q} = 1 - 6\frac{(p-q)^2}{pq}~,
\]
the Virasoro algebra develops a vertex ideal such that the quotient is a rational vertex algebra known as the $(p,q)$ minimal model. In particular, for $(p,q)=(2,2n+1)$, with $n\in \IN$, the vertex ideal $\II_n\subset \Vir^c$ is generated by a state of the form
\[ 
T^n + \text{ states with conformal weight } 2n~.
\]
We write, 
\[
\MM_n \coloneqq \Vir^{c_{2,2n+1}} /\II_n~,
\]
for the corresponding minimal model. 

We define the homogenized $\hbar$-adic vertex algebra $\Vir^c_{\hbar}$ as being generated by $T$ with integral $\lambda$-bracket
\[
I_\lambda(T,T)= \frac{1}{8}\hbar ^2\lambda^4 c + \lambda^2 \hbar T + \lambda \hbar \del T~.
\]
Clearly, the integral $\lambda$-bracket is invariant with respect to the $\IG_\hbar$-action which assigns unit weight to $T$. Indeed, this grading is related (via the Rees correspondence) to the, so-called, $R$-filtration of \cite{ArabiArdehali:2025fad}.

For $c= c_{2,2n+1}$, $\Vir^c_{\hbar}$ develops a vertex ideal, $\II_{\hbar,n}$, generated by a state of the form $T^n + O(\hbar)$. We denote the quotient as
\[
\MM_{\hbar,n} \coloneqq \Vir^{c_{2.2n+1}}/ \II_{\hbar,n}~.
\]

\begin{lemma}\label{lem:filtered version of Virasoro}
	We have isomorphisms of vertex algebras
	\begin{equation}
		\begin{split}
		\big(\Vir^c_{\hbar}\tensor{\ik[\![\hbar]\!]}\ik(\!(\hbar)\!)\big)^{\IG_\hbar} &\cong \Vir^c~,\\
		\big(\MM_{\hbar,n}\tensor{\ik[\![\hbar]\!]}\ik(\!(\hbar)\!)\big)^{\IG_\hbar} &\cong \MM_n~.
		\end{split}
	\end{equation}
\end{lemma}
\begin{proof}
	Since $\Vir^c_\hbar$ and $\MM_{\hbar,n}$ are positively graded by $\IG_\hbar$, by Lemma \ref{lem:filtered is a VOA} the invariants are a vertex algebra. We have a $\IG_\hbar$-equivariant embedding of $\hbar$-adic vertex algebras
	\begin{equation}
	\begin{split}
		\Vir^c_\hbar &\rightarrow \Vir^c[\![\hbar]\!] \\
		T &\mapsto \hbar T
	\end{split}
	\end{equation}
	which is an isomorphism after inverting $\hbar$. Since $\big(\Vir^c(\!(\hbar)\!)\big)^{\IG_\hbar}\cong \Vir^c$ we have the first desired isomorphism. The result for the minimal model follows since the embedding maps $\II_{\hbar,n}$ to $\hbar^n\II_n[\![\hbar]\!]$.
\end{proof}

We identify the classical limit of $\Vir^c_{\hbar}$ with $V=\JJ\ik[T]$ with vertex Poisson algebra structure
\[
[T_\lambda T] = 2\lambda T +\partial T~.
\]
By construction, $\Vir^c_{\hbar}$ is a graded quantization of $V$ for any $c\in \ik$.

Let $I_n\subset V$ be the differential ideal generated by $T^n$ and note that $\II_{\hbar,n}|_{\hbar=0}\cong I_n$. We write $M_n \coloneqq (\JJ \ik[T])/I_n \cong \JJ(\ik[T]/T^n)$. 

\begin{lemma}
	The vertex algebra $\MM_{\hbar,n}$ is a graded quantization of the vertex Poisson algebra $M_n$.
\end{lemma}

\begin{rmk}
	Unlike the chiral differential operators and current algebras of the previous example the Virasoro algebra is not a chiral quantization of a Poisson algebra, though it is a quantization of a vertex Poisson algebra.

	Moreover, the minimal models furnish a nice example of quantizations of vertex Poisson algebras, whose underlying commutative algebras are not smooth. This will lead to some challenges when we address their deformation theory in Section \ref{ssec:minimal models}.
\end{rmk}

\begin{rmk}
Every $(p,q)$-minimal model has a Zhu's $C_2$ algebra of the form $\ik[T]/T^m$ for some $m$. However, the boundary minimal models of form $(2,2n+1)$ are special in that they are classically free, \ie, there is an isomorphism between the associated graded of the Li filtration and the arcs to the Zhu's $C_2$ algebra.
\end{rmk}


\section{The deformation-obstruction complex}\label{defsec}

For the remainder of the text, unless stated otherwise, all operads will be operads on a single colour, internal to the symmetric monoidal categories $\Vect_\K$ or $\DGVect_\K$, which we will call plain or DG operads, respectively. We also identify plain operads with DG operads concentrated in cohomological degree zero, as usual.

\subsection{Koszul preliminaries}

First, let us recall the notion of a single coloured co-operad, dual to the definition of an operad recalled in Remark \ref{opsingcolrmk}.
\begin{defn}
	A single coloured co-operad is a counital, comonoid in the category $(\Ss\Mod,\circ)$, \ie, a tuple $(\CC, \Delta,\eta)$ where $\CC\in \Ss\Mod$, $\Delta: \CC\rightarrow \CC\circ \CC$ is a coassociative comultiplication---called the decomposition---with counit $\eta: \CC\rightarrow \id_\Ss$.


\end{defn}
Let $E\in \Ss\Mod$, similar to the free operad $\FF(E)$, the free cooperad on $E$, $\FF^c(E)$, is a cooperad structure on the same underlying $\Ss$-module as $\FF(E)$. We say that a cooperad $\CC$ is generated by $E$ if there exist some subspace of corelations $R\subset \FF^c(E)$ defining a cooperadic coideal, $(R)$, such that
\[
\CC \hookrightarrow \FF^c(E)\twoheadrightarrow \FF^c(E)/(R)~.
\]

The free operad admits a grading, $\FF(E) = \oplus_{n\in \IN} \FF(E)^{(n)}$, called the \emph{weight grading}, characterized by setting
\[
\FF(E)^{(0)} = (0, \ik,0,\dots)~, \qquad \FF(E)^{(1)} = E\subset \FF(E)~,
\]
and letting the weight be additive under composition. For a binary, free operad, the weight grading corresponds to the arity, up to a shift by one: $\FF(E)^{(n)}=\FF(E)(n+1)$. Similarly, $\FF^c(E)=\oplus_{n\in\IN} \FF^{c,(n)}(E)$ has the same weight grading.

\begin{defn}\label{quadopdefn} 
A quadratic operad, $\PP=\PP(E,R)$ is an operad that is presented as a quotient of the free operad, $\FF(E)$, on $E$ by a relation $R\into \FF(E)$ that factors as $R\into \FF(E)^{(2)}$.

Similarly, a quadratic cooperad $\CC = \CC(E,R)$ is a cooperad presented as a suboperad of the free cooperad $\FF^c (E)$ by a corelation $R\into \FF^c(E)$ that factors as $R\into\FF^{c,(2)}(E)\subset \FF^c(E)$.
\end{defn}
Since the (co)relations of a quadratic (co)operad are homogeneous of weight two, the weight grading descends to the quadratic (co)operad.

%
%
%

The operads $\Comm$, $\Lie$, $\Ass$, and $\Pois$ that we recalled in Section \ref{ssec:operad examples} are all examples of quadratic operads---in fact they are all binary quadratic operads.

Quadratic operads are nice in that they enjoy a kind of Koszul duality:

\begin{defn} 

The Koszul dual cooperad $\mc P^\ash$ of a quadratic operad $\mc P=\mc P(E,R)$ is the DG cooperad
\[ 
\mc P^\ash= \mc C(E[1],R[2]) \qquad \text{with} \qquad R[2] \to \mc F^c(E)^{(2)}[2] \cong \mc F^c(E[1])^{(2)}  ~.
\]

\end{defn}

We finish by recalling the cobar construction. Let $\CC=\CC(E,R)$ be a quadratic cooperad, then the cobar construction is the quasi-free DG operad
\[
\Omega\CC = \FF(\overline{\CC}[-1],d_{\Omega \CC})
\]
on the coideal $\overline{\CC} = \CC/\CC^{(0)}$ with differential $d_{\Omega\mc C}:\mc F( \overline{\mc C}[-1])\to \mc F( \overline{\mc C}[-1])[1]$ defined in terms of the decomposition $\Delta: \CC\rightarrow \CC\circ \CC$ as 

\[
d_{\Omega\mc C}|_{\overline{\mc C}[-1]} : \overline{\mc C}[-1] \xrightarrow{\Delta_{\mc C}[-1]}  \overline{\mc C}[-1] \circ \overline{\mc C}[-1] [1] \subset \mc F( \overline{\mc C}[-1])[1]  \ . 
\]
For the special case where $\CC = \PP^\ash$ for a quadratic operad $\PP=\PP(E,R)$, the cobar construction and Koszul duality are related by
\begin{prop} 
There is a natural projection of operads $\Omega\PP^\ash\rightarrow (\PP^\ash)^\ash\cong\PP$ inducing an isomorphism
\[ 
 H^0(\Omega \mc P^\ash ) \xrightarrow{ \cong} \mc P ~,
\]
on cohomology.
\end{prop}
\begin{proof}
See the discussion preceding Proposition 7.3.2 in \cite{LV}.
\end{proof}

One might hope that $\Omega\PP^\ash$ provides a quasi-free resolution of $\PP$. Though this is not always the case, it is true for a special class of operads:
\begin{defn} 
An operad $\mc P$ is called \emph{Koszul} if it is quadratic and the natural projection $\Omega \mc P^\ash \to \mc P$ is a quasi-isomorphism.
\end{defn}

\begin{rmk}
	The familiar operads $\Comm$, $\Lie$, $\Ass$, and $\Pois$ are all Koszul.
\end{rmk}


\subsection{Maurer--Cartan elements}

Let $\CC$ be a cooperad and $\QQ$ be an operad, then
\[
\Hom_\Ss(\CC,\QQ)= \bigoplus_{n\in \IN} \Hom_{\Ss_n}(\CC(n),\QQ(n))~,
\]
has a natural convolution product $\star: \Hom_\Ss(\CC,\QQ)\otimes \Hom(\CC,\QQ) \rightarrow \Hom_\Ss(\CC,\QQ)$. In fact, this convolution product is naturally pre-Lie, \ie, its commutator defines a Lie bracket
\begin{equation}\label{defliebreqn} [ \cdot , \cdot ] : \Hom(\mc C, \mc Q)^{\otimes 2} \to \Hom(\mc C, \mc Q)  \qquad [f,g] = f\star g - g\star f  \ . 
\end{equation}

In summary:

\begin{prop} 
Let $\mc C$ be a DG cooperad and $\mc Q$ a DG operad. Then 
\[
\g_{\mc C,\mc Q}=\Hom_{\Ss}(\mc C, \mc Q) ~,
\]
is naturally a DG Lie algebra with respect to the convolution Lie bracket defined above, and the natural cohomological grading and differential induced by those on $\CC$ and $\QQ$.
\end{prop}
\begin{proof} 
See Proposition 6.4.5 in \cite{LV}.
\end{proof}

In particular, if $\PP$ is a quadratic operad we write 
\[
\gf_{\PP,\QQ} \coloneqq \gf_{\PP^\ash,\QQ}=\Hom_\Ss(\PP^\ash,\QQ)~,
\]
with cohomological grading,
\[
\gf_{\PP,\QQ}=\bigoplus_{n=1}^\infty \gf_{\PP,\QQ}(n)[1-n] \qquad\text{where}\qquad \gf_{\PP,\QQ}(n) = \Hom_{\Ss_n}(\PP^\ash(n),\QQ(n))~.
\]

\subsubsection{Algebra structures and Maurer--Cartan elements} 



Recall the weight grading on the quadratic operad $\PP$. The DG Lie algebra $\g_{\mc P,\mc Q}$ thus inherits an auxiliary grading, also called the \emph{weight grading}, given by
\[ 
\g_{\mc P,\mc Q}=\bigoplus_{\ell\in \IN}\g_{\mc P,\mc Q}^{(\ell)} \qquad\text{where}\qquad  \g_{\mc P,\mc Q}^{(\ell)}=\Hom_{\Ss}({\mc P^{\ash}}^{ (\ell)} ,\mc Q) \ .
\]
\begin{defn}
	A Maurer--Cartan element, or twisting homomorphism, $\alpha\in \gf^1_{\CC,\QQ}$ is an element satisfying
	\[
	d\alpha + \tfrac12 [\alpha,\alpha]=0~.
	\]
	We let $\MC_\infty(\g_{\mc C,\mc Q})\subset \g_{\mc C,\mc Q}^1$ denote the subvariety of Maurer--Cartan elements. In particular, for the space of Maurer--Cartan elements of pure weight one, we drop the subscript, \ie,
	\[
	\MC(\gf_{\PP,\QQ})\coloneqq \MC_\infty(\gf_{\PP,\QQ})\cap \gf_{\PP,\QQ}^{(1)}~.
	\]
\end{defn}


\begin{prop}\label{MCplainprop}

Let $\PP$ be a quadratic, Koszul, DG operad	and $\QQ$ some arbitrary single coloured DG operad, then
\[
\Alg_\PP(\QQ) \cong \MC\big(\gf_{\PP,\QQ}\big)\subset \gf_{\PP,\QQ}^{1,(1)}~.
\]
\end{prop}
\begin{proof}
Recall the cobar construction for $\PP^\ash$, $\Omega \PP^\ash$---since it is quasi-free 
\[
\Alg_{\Omega\PP^\ash}(\QQ)=\Hom_{\Op_\DG}(\Omega\PP^\ash,\QQ) \cong \Hom_{\Ss\Mod_\DG}(\overline{\PP}^\ash,\QQ)[1] \cong \MC(\gf_{\PP,\QQ})~,
\]
where the first isomorphism is from the free-forgetful adjunction and the second isomorphism is precisely the content of Theorem $6.5.10$ of \cite{LV}.

From Proposition 10.1.7 of \emph{loc.\ cit.}, for a Koszul operad $\Alg_{\PP}(\QQ)\hookrightarrow \Alg_{\Omega\PP^\ash}(\QQ)$ and in terms of Maurer-Cartan elements this corresponds to the subvariety of Maurer--Cartan elements that are of pure weight one, \ie,
\[
\Alg_{\PP}(\QQ) \cong \MC(\gf_{\PP,\QQ})\subset \gf_{\PP,\QQ}^{1,(1)}~.
\]

\end{proof}

\begin{rmk}
	The subvariety of all Maurer--Cartan elements $\MC(\gf_{\PP,\QQ})$ corresponds to the space of $\Omega \PP^\ash$ structures, which are homotopy algebras.
\end{rmk}

\subsubsection{Graded Maurer--Cartan elements}

Given two graded operads $\mc P$ and $\mc Q$, an algebra $\varphi \in \Alg_{\mc P}(\mc Q)=\Hom_{\Op_\K}(\mc P, \mc Q)$ is called graded if it corresponds to a map of graded operads, that is, the underlying maps of $\K[\mf S_n]$-modules respect the gradings at each arity. We let $\Alg_{\mc P}^\gd(\mc Q)=\Hom_{\Op_\K^\gd}(\mc P, \mc Q)$ denote the space of graded $\mc P$-algebras internal to $\mc Q$.

Under the preceding hypotheses, the DG Lie algebra $\g_{\mc P, \mc Q}$ inherits an auxiliary grading coming from those on $\mc P$ and $\mc Q$, which we denote by
\[ \g_{\mc P, \mc Q}= \bigoplus_{j\in \Z} \g_{\mc P, \mc Q}^{\langle j \rangle } \ , \]
and with which the DG Lie algebra structure maps are naturally compatible, that is, the Lie bracket and differential have degree $0$. In this setting, we obtain the following identification:

\begin{prop} \label{prop:graded MC elements}
An element $ \Alg_{\mc P}(\mc Q)=\MC(\gPQ) \subset \g_{\mc P,\mc Q}^{1,(1)}$ defines a graded algebra structure, that is, $\varphi \in \Alg_{\mc P}^\gd(\mc Q)$, if and only if the corresponding Maurer--Cartan element is of degree $0$. Equivalently, the canonical identification of Proposition \ref{MCplainprop} defines an identification
	\[ \Alg_{\mc P}^\gd(\mc Q) =\Hom_{\Op_\DG^\gd}(\mc P, \mc Q)  \xrightarrow{\cong } \MC_\gr(\g_{\mc P, \mc Q}) \subset \g_{\mc P, \mc Q}^{1 ,(1),\langle 0 \rangle}  \ .\]
\end{prop}
\begin{proof} 
This follows from tracking the grading in the proof of Proposition \ref{MCplainprop}.

\end{proof}

\subsubsection{The twisted DG Lie algebra} 

Given $\varphi \in \Alg_{\mc P}(\mc Q)$ viewed as a Maurer--Cartan element, the deformations of $\varphi$ are controlled by an analogous DG Lie algebra, defined as follows:

\begin{defn}\label{gvarphidef} 
The twisted DG Lie algebra $\g^\varphi_{\mc P,\mc Q}$ is defined by
	\[  
	\g^\varphi_{\mc P,\mc Q} = \left( \Hom_{\Ss}(\mc P^\ash ,\mc Q) \ , \  [ \cdot , \cdot ]   \ , \  d_\varphi  \right) 
	\]
where $[ \cdot , \cdot ]:\Hom(\mc P^\ash, \mc Q)^{\otimes 2} \to \Hom(\mc P^\ash, \mc Q) $ is the bracket defined in Equation \ref{defliebreqn} and
\begin{equation}\label{dphieqn}
	d_\varphi:\g^\varphi_{\mc P,\mc Q} \to \g^\varphi_{\mc P,\mc Q} [1] \qquad \text{is defined by} \qquad d_\varphi(\cdot) = d(\cdot) + [\varphi , \cdot] \ .
\end{equation}
\end{defn}

The twisted DG Lie algebra controls deformations of $\varphi$ in the following sense.
\begin{prop} 
For any $\mc P$-algebra structure $\varphi\in\Alg_{\mc P}(\mc Q)\cong \MC(\g_{\mc P, \mc Q})^{(1)}$, we have
\[ \alpha \in \MC(\g_{\mc P,\mc Q}^\varphi) \qquad \text{if and only if}\qquad \varphi+\alpha \in \MC(\g_{\mc P,\mc Q})  \ , \]
and the resulting deformed DG Lie algebras are naturally equal, that is, we have $(\g_{\mc P,\mc Q}^\varphi)^\alpha = \g_{\mc P,\mc Q}^{\varphi+\alpha}$.
\end{prop}
\begin{proof} These follow from straightforward calculations; the former is Proposition 12.2.4 of \cite{LV}.
\end{proof}
Iterating the preceding Proposition, we deduce:
\begin{coro} 
Let $\varphi\in\Alg_{\mc P}(\mc Q)\cong \MC(\g_{\mc P, \mc Q})$ and $\alpha \in \MC(\g_{\mc P,\mc Q}^\varphi)$. Then we have
\[
\beta \in \MC(\g_{\mc P,\mc Q}^{\varphi+\alpha}) \qquad \text{if and only if}\qquad \alpha+\beta \in \MC(\g_{\mc P,\mc Q}^\varphi) \ .  
\]
\end{coro}
In other words, further deformations of the deformed algebra $\varphi+\alpha$ are equivalent to deformations of the original algebra. Equivalently, given deformations $\alpha_1$ and $\alpha_2$ of $\varphi$, we obtain that $\beta = \alpha_2-\alpha_1$ a deformation of $\varphi+\alpha_1$ to $\varphi+\alpha_2$, and conversely.

We finish by defining the operadic deformation-obstruction complex of $\varphi \in \Alg_{\mc P}(\mc Q)$, a $\mc P$-algebra internal to $\mc Q$:

We now introduce the operadic deformation-obstruction complex of $\varphi\in \Alg_{\PP}(\QQ)$, a $\PP$-algebra internal to $\QQ$. As we shall see in \ref{mapdefsec}, this complex controls the deformation theory of $\varphi$.

\begin{defn}\label{defobscxdefn} The operadic deformation-obstruction complex of $\varphi \in \Alg_{\mc P}(\mc Q)$ is the cochain complex
\[ C^\bullet_{\mc P,\mc Q}(\varphi)=(\g^\varphi_{\mc P,\mc Q} \ , \ d_\varphi ) \qquad\text{where}\qquad  d_\varphi:\g^\varphi_{\mc P,\mc Q} \to \g^\varphi_{\mc P,\mc Q} [1]   \]
is the differential defined in Equation \ref{dphieqn}, and denote by $Z^\bullet_{\mc P,\mc Q}(\varphi)$ and $H^\bullet_{\mc P, \mc Q}(\varphi)$ the subcomplex of cocycles and the cohomology of $C^\bullet_{\mc P,\mc Q}(\varphi)$, respectively.
\end{defn}

Given $\varphi \in \Alg_{\mc P}^\gd(\mc Q)$, the auxiliary grading is also compatible with the DG Lie algebra $\g_{\mc P,\mc Q}^\varphi$, since the modified differential $d_\varphi$ is still of degree $0$ with respect to the auxiliary grading, as the Maurer--Cartan element corresponding to $\varphi$ is itself of degree $0$. In particular, the deformation complex inherits a grading, which is preserved by the differential and thus induces a grading on cohomology,
\[ 
C^\bullet_{\mc P,\mc Q}(\varphi) = \bigoplus_{j\in \Z} C^{\bullet,\langle j\rangle}_{\mc P,\mc Q}(\varphi)
\qquad \text{and} \qquad 
H^\bullet_{\mc P,\mc Q}(\varphi) = \bigoplus_{j\in \Z} H^{\bullet,\langle j\rangle}_{\mc P,\mc Q}(\varphi)~. 
\]

\subsection{Deformation-obstruction complex for algebras of the classical kind}
Throughout this section we fix a vector space $A\in \Vect_\ik$ and set $\QQ= \cEnd_{\Vect_\ik}(A)$ to be the endomorphism operad of $A$. 

\subsubsection{Associative algebras} 

\begin{prop}\label{Hoccxprop}
The deformation-obstruction complex of Definition \ref{defobscxdefn} for $\mc P=\Ass_\K$ and $\mc Q=\End_{ \Vect_\ik}(A)$ is given by the usual complex of (reduced) Hochschild cochains, that is,
\[ C^\bullet_{\mc P,\mc Q}(\varphi) = C^\bullet_{\textup{Hoc}}(A,m_\varphi) \qquad \text{where}\qquad C^\bullet_{\textup{Hoc}}(A,m_\varphi) = \left( \bigoplus_{n=1}^\infty \Hom(A^{\otimes n}, A)[1-n] \ , \ d_\textup{Hoch} \right)  \ , \]
for $m_\varphi:A^{\otimes 2} \to A$ the associative algebra structure map and $d_\text{Hoch}:C^\bullet_{\textup{Hoc}}(A,m_\varphi)\to C^\bullet_{\textup{Hoc}}(A,m_\varphi)[1]$ the usual Hochschild differential determined by $m_\varphi$.
\end{prop}
\begin{proof}
This follows from the fact that at each arity $n$, we have
\[ \g_{\mc P,\mc Q}(n)= \Hom_{\K[\Ss_n]\Mod}(\Ass^\ash(n), \End(V)(n)) \cong \K[\Ss_n] \otimes_{\K[\Ss_n]} \Hom(V^{\otimes n},V) \cong  \Hom(V^{\otimes n},V) \]
and that the differential $d_\varphi=[\varphi,\cdot]$ given by commutator against the Maurer-Cartan element corresponding to the associative algebra structure map is equivalent to the usual Hochschild formula of cyclic pre- and post-composition, under the infinitesimal composition maps, with the multiplication structure map $m_\varphi$, which is precisely the usual Hochschild differential.
\end{proof}

\subsubsection{Commutative algebras} 
Recall from \ref{Asseg} that there is a canonical map $\Ass \to \Comm$, the pullback along which corresponds to the natural inclusion of commutative into associative algebras $\Alg_\Comm(\mc Q)\to \Alg_\Ass(\mc Q)$. This induces a map on Koszul dual cooperads 
\[
\Ass^\ash \to \Comm^\ash~.
\]
In fact, this implies the following:

\begin{prop}\label{Harcxprop}
	The deformation-obstruction complex of Definition \ref{defobscxdefn} for $\mc P=\Comm_\K$ and $\mc Q=\End_{\mc C}(A)$ is given by the usual complex of (reduced) Harrison cochains, that is,
	\[ 
	C^\bullet_{\mc P,\mc Q}(\varphi) = C^\bullet_{\textup{Har}}(A,m_\varphi) \qquad \text{where}\qquad C^\bullet_{\textup{Har}}(A,m_\varphi) = \left( \bigoplus_{n=1}^\infty \Hom(A^{\otimes n}, A)^{\SH_n}[1-n] \ , \ d_\textup{Har} \right)  \ , 
	\]
	where $\Hom(A^{\otimes n}, A)^{\SH_n}\subset\Hom(A^{\otimes n}, A)$ denotes the subspace of $n$-ary multilinear maps which vanish on sums over $(p,q)$-shuffles for some fixed $p+q=n$ of the correspondingly permuted tensor products 
	, that is,
	\[
	\Hom(A^{\otimes n},A) ^{\SH_n} = \{ m\in \Hom(A^{\otimes n},A)\ | \sum_{\sigma \in \SH_{p,q}} m(v_{\sigma^{-1}(1)}\otimes\dots\otimes v_{\sigma^{-1}(p+q)}) =0 \text{ for all } v\in A^{\otimes p},\  p+q=n \}~.
	\]
 Also $m_\varphi:A^{\otimes 2} \to A$ denotes the generating structure map of the commutative algebra, and $d_\text{Har}$ is the usual Harrison differential.

\end{prop}
\begin{proof} 
This is essentially Proposition 13.1.4 of \cite{LV}.
\end{proof}

%

We now describe an equivalent model for the deformation theory of commutative algebras, called Andr\'e--Quillen cohomology, and a generalization thereof which will be relevant for describing the deformation theory of Poisson algebras below.

Viewing $A$ as a commutative $\ik$-algebra with commutative multiplication $m_\varphi$, we write $\Omega^1_A\in A\Mod$ for the module of K\"ahler differentials. Given a cofibrant resolution $\tilde A \to A$, the cotangent complex $\bL_A\in\Dd(A\Mod)$ is defined as the image of the object $\Omega^1_{\tilde A}\in \Dd(\tilde A\Mod)$ under the equivalence $\Dd(\tilde A\Mod)\xrightarrow{\sim}\Dd(A\Mod)$. Similarly, we define $\bL^p_A\in \Dd(A \Mod)$ as the image of the object $\Omega^p_{\tilde A} = \bigwedge_{\tilde A}^p \Omega^1_{\tilde A}\in \Dd(\tilde A\Mod)$ under the same equivalence.

\begin{defn}\label{AQDefn} The Andr\'e--Quillen cochains on $A$ are defined by
\[ C^\bullet_\AQ(A)= \Hom_{\Dd(A\Mod)}(\bL_A,A) \ .  \]
Similarly, the degree-$p$ higher Andr\'e--Quillen cochains on $A$ are defined by
\[ C^{\bullet,\langle p \rangle}_\AQ(A)= \Hom_{\Dd(A\Mod)}(\bL^{\wedge_A p+1}_A,A) \ .  \]
\end{defn}
Note that our conventions are shifted by one from the more standard conventions of \cite{Lod}, so that for $p=0$ we have $ C^{\bullet,\langle p \rangle}_\AQ(A)= C^\bullet_\AQ(A)$. Moreover, in this case the relevant cohomology groups are equivalent to the Harrison cohomology defined above:
\begin{theo} There is a natural isomorphism
\[ H^\bullet_\AQ(A)\xrightarrow{\cong} H^\bullet_{\textup{Har}}(A,m_\varphi) \ . \]
\end{theo}
\begin{proof}
	This is a special case of Theorem 12.4.3 of \cite{LV}.
\end{proof}

Furthermore, we have the following standard vanishing results for smooth algebras:

\begin{theo} \label{theo:smooth then kahler isproj}
Suppose $A$ is smooth over $\K$. Then $\bL_A\xrightarrow{\cong} \Omega_A^1\in A\Mod=\Dd(A\Mod)^\heartsuit$, and the object $\Omega_A^1$ is projective.
\end{theo}
\begin{proof} This is a standard result in algebraic geometry; see \cite{Stks} tags 08R4, 08RB, and 00T1.
\end{proof}

\begin{coro}\label{smthcoro} Suppose $A$ is smooth over $\K$. Then $C^\bullet_\AQ(A)$ is acyclic away from degree zero and
	\[ H^0_\AQ(A)\cong \Hom_{A\Mod}(\Omega^1_A,A) = \Der(A,A) = :\Theta_A \]
	the module of derivations $ \Der(A,A) $ of $A$, or tangent module $\Theta_A$ of $A$.
	
	More generally, for each $p\geq 1$, $C^{\bullet,\langle p \rangle }_\AQ(A)$ is acyclic in non-zero degree, and
	\[ 
	H^{0,\langle p \rangle}_\AQ(A)\cong \Hom_{A\Mod}(\Omega^{p+1}_A,A) \cong \wedge^{p+1}_A \Theta_A \ , 
	\]
	the module of degree $p+1$ polyvector fields $\wedge^{p+1}_A \Theta_A$.
\end{coro}

\subsubsection{Lie algebras} 

\begin{prop}\label{CEcxprop}
	The deformation-obstruction complex of Definition \ref{defobscxdefn} for $\mc P=\Lie_\K$ and $\mc Q=\End_{\mc C}(A)$ is given by the usual complex of (reduced) Chevalley--Eilenberg cochains, that is,
	\[ C^\bullet_{\mc P,\mc Q}(\varphi) = C^\bullet_{\textup{CE}}(A,b_\varphi) \qquad \text{where}\qquad C^\bullet_{\textup{CE}}(A,b_\varphi) = \left( \bigoplus_{n=1}^\infty \Hom(A^{\otimes n}, A)_{\Ss_n}[1-n] \ , \ d_\textup{CE} \right)  \ , \]
	where $\Hom(A^{\otimes n}, A)_{\Ss_n}=\K \otimes_{\K[\Ss_n]}\Hom(A^{\otimes n}, A) $ denotes the $S_n$-coinvariants, $b_\varphi:A^{\otimes 2} \to A$ the Lie algebra structure map and $d_\text{CE}:C^\bullet_{\textup{CE}}(A,m_\varphi)\to C^\bullet_{\textup{CE}}(A,m_\varphi)[1]$ the usual Chevalley--Eilenberg differential determined by $b_\varphi$.
\end{prop}
\begin{proof}
This is Proposition 13.2.2 of \cite{LV}.
	
\end{proof}

\subsubsection{Poisson algebras}

\begin{defn} Let $\varphi \in \Alg_\Pois(\End(A))$ a Poisson algebra structure on $A$ with corresponding commutative product $m_\varphi:A^{\otimes 2} \to A$ and Lie bracket $b_\varphi:A^{\otimes 2} \to A$ . We define the operadic Poisson cochains on $(A,m_\varphi,b_\varphi)$ by
\[ C_\Pois^\bullet(A,m_\varphi,b_\varphi)=C_{\Pois,\End(A)}^\bullet(\varphi) \ . \]
\end{defn}

%

\subsubsection{A bicomplex structure on Poisson cochains}

Recall that the Poisson operad has an auxiliary grading $\Pois = \bigoplus_{m\in\IN}{\Pois}^{\langle m\rangle}$, with the bracket operation having weight one. This induces a grading on the Koszul dual cooperad
\[
\Pois^\ash = \bigoplus (\Pois^\ash)^{-p} \qquad \text{where} \qquad (\Pois^{\ash})^{-p} = \Lie^\ash(p+1)\circ \Comm^\ash ~,
\]
in turn inducing an auxiliary grading on the Poisson cochains $\gf_{\PP, \QQ} = \bigoplus_{n=1}^\infty\bigoplus_{p=0}^{n-1} \gf_{\PP,\QQ}(n)^{\langle p\rangle}[1-n]$ where
\[
\gf_{\PP,\QQ}(n)^{\langle p \rangle} = \Hom_{\ik[\Ss_n]\Mod}\big((\Pois^\ash)(n)^{\langle-p\rangle}, \End_{\Vect_{\ik}}(A)(n)\big)~.
\]
With respect to this auxiliary grading $d_{m_{\varphi}}$ has degree zero but $d_{b_{\varphi}}$ has degree one, and in fact we may identify this auxiliary grading with the column number of a bicomplex:

\begin{prop}\label{Poisbicxprop} There is a natural bicomplex $C^{\bullet,\bullet}_\Pois(A,m_\varphi,b_\varphi)=\oplus_{p,q=0}^\infty C^{p,q}_\Pois$ defined by
\[
C^{p,q}_\Pois=C^{p,q}_\Pois(A,m_\varphi,b_\varphi)= \g_{\mc P, \mc Q}(p+q+1)^{\langle p \rangle } \qquad  
\]
with horizontal and vertical differentials given, respectively, by
\[
d_h=d_{b_\varphi}:C^{p,q}_\Pois \to C^{p+1,q}_\Pois \qquad \text{ and } \qquad d_v=d_{m_\varphi}:C^{p,q}_\Pois \to C^{p,q+1}_\Pois~,
\]
such that the total complex is naturally isomorphic to the operadic Poisson cochains, that is, \[\textup{Tot } C^{\bullet,\bullet}_\Pois(A,m_\varphi,b_\varphi) \xrightarrow{\cong} C_\Pois^\bullet(A,m_\varphi,b_\varphi) \ .\]
\end{prop}

We can depict the bicomplex explicitly as follows, such that the graded degree corresponds to column number:
\begin{equation}\label{Poisdecompeqn}
\begin{tikzcd} \vdots & \vdots & \vdots &  \\
	\g_{\mc P,\mc Q}(3)^{\langle 0 \rangle} \ar[r,"d_{b_\varphi} "] \ar[u,"d_{m_\varphi}"] & \g_{\mc P,\mc Q}(4)^{\langle 1 \rangle} \ar[r,"d_{b_\varphi} "] \ar[u,"d_{m_\varphi}"] & \g_{\mc P,\mc Q}(5)^{\langle 2 \rangle} \ar[r,"d_{b_\varphi} "] \ar[u,"d_{m_\varphi}"]  & \hdots  \\
	\g_{\mc P,\mc Q}(2)^{\langle 0 \rangle}\ar[r,"d_{b_\varphi} "] \ar[u,"d_{m_\varphi}"] &  	\g_{\mc P,\mc Q}(3)^{\langle 1 \rangle} \ar[r,"d_{b_\varphi} "] \ar[u,"d_{m_\varphi}"] & 	\g_{\mc P,\mc Q}(4)^{\langle 2 \rangle}\ar[r,"d_{b_\varphi} "] \ar[u,"d_{m_\varphi}"] & \hdots\\
	\g_{\mc P,\mc Q}(1)^{\langle 0 \rangle} \ar[r,"d_{b_\varphi} "] \ar[u,"d_{m_\varphi}"]& 	\g_{\mc P,\mc Q}(2)^{\langle 1 \rangle}\ar[r,"d_{b_\varphi} "] \ar[u,"d_{m_\varphi}"] & 	\g_{\mc P,\mc Q}(3)^{\langle 2 \rangle}\ar[r,"d_{b_\varphi} "] \ar[u,"d_{m_\varphi}"] & \hdots
\end{tikzcd}  \ . 
\end{equation}
The zeroth row and column define quotient complexes
\[ C^\bullet_{\Pois,m}(A,m_\varphi)=(\bigoplus_{n=1}^\infty \g_{\mc P,\mc Q}(n)^{\langle 0 \rangle}[1-n] \ , \ d_{m_\varphi} ) \qquad \text{and}\qquad C^\bullet_{\Pois,b}(A,b_\varphi)=(\bigoplus_{n=1}^\infty \g_{\mc P,\mc Q}(n)^{\langle n-1 \rangle}[1-n] \ , \ d_{b_\varphi} )  \]
which control the induced deformation of commutative algebra and Lie algebra, respectively, coming from a deformation of Poisson algebra, as the following Proposition confirms:

\begin{prop}\label{rowcol1prop}There are natural identifications of complexes
\[C^\bullet_{\Pois,m}(A,m_\varphi) \xrightarrow{\cong} C^\bullet_{\textup{Har}}(A,m_\varphi) \qquad \text{and}\qquad  C^\bullet_{\Pois,b}(A,b_\varphi) \xrightarrow{\cong} C^\bullet_{\textup{CE}}(A,b_\varphi)  \ . \]
\end{prop}
\begin{proof} This follows straightforwardly from the preceding discussion by unwinding the definitions.

\end{proof}

More generally, for each $p\geq 0$ the $p^{th}$ row of the above bicomplex defines a complex
\[ C^{\bullet,\langle p \rangle}_{\Pois,m}(V,m_\varphi)=(\bigoplus_{n=p+1}^\infty \g_{\mc P,\mc Q}(n)^{\langle p \rangle}[1-n] \ , \ d_{m_\varphi} ) \]
and we have the following generalization of the preceding result:

\begin{prop}\label{AQPoisequivprop} For each $p\geq 0$, there are natural identifications of complexes
\[C^{\bullet,\langle p \rangle}_{\Pois,m}(A,m_\varphi)[p] \xrightarrow{\cong} C^{\bullet,\langle p \rangle}_\AQ(A) \ , \]
for $A=(A,m_\varphi)$. In particular, we have $H^{p,\langle p \rangle}_{\Pois,m}(A,m_\varphi) \cong \Hom_{A\Mod}(\Omega^{p+1}_A,A)$.
\end{prop}
\begin{proof} This follows from Proposition 4.5.13 of \cite{Lod}, together with Proposition \ref{finclfiltprop} below.
\end{proof}

\subsubsection{} The bicomplex decomposition of Equation \ref{Poisdecompeqn} is a manifestation of the fact that a deformation of a Poisson algebra is equivalent to a deformation of the underlying commutative algebra and Lie algebra in a way that is compatible with the Poisson algebra relations. A common variant that is often considered in the geometric context is the deformations of Poisson structure on a fixed underlying variety, or equivalently in the affine setting the deformations of the underlying Lie bracket which are compatible with the fixed commutative algebra structure. Such deformations are controlled by the following subcomplex:

\begin{defn}\label{gPCdefn} Let $\varphi \in \Alg_\Pois(\End(A))$ a Poisson algebra stucture on $A$ with corresponding commutative product $m_\varphi:A^{\otimes 2} \to A$ and Lie bracket $b_\varphi:A^{\otimes 2} \to A$ . The (reduced) geometric Poisson cochains on $(A,m_\varphi,b_\varphi)$ is the subcomplex of  $C_\Pois^\bullet(A,m_\varphi,b_\varphi)$ given by
	\[ C^\bullet_\PV(A,m_\varphi,b_\varphi)= \left( \bigoplus_{p=0}^{\infty} H^{p,\langle p \rangle}_{\Pois,m}(A,m_\varphi)[-p] \ , \ d_{b_\varphi}  \ \right)  \ .\]
\end{defn}

The following geometric description follows from Proposition \ref{AQPoisequivprop}:

\begin{coro}\label{gPCcoro} The complex of geometric Poisson cochains admits an isomorphism of complexes
	\[ C^\bullet_\PV(A,m_\varphi,b_\varphi) \cong \left( \bigoplus_{p=0}^\infty\Hom_{A\Mod}(\Omega^{p+1}_A,A)[-p] \ , \ \tilde d_{b_\varphi} \ \right) \ , \]
	where $\tilde d_{b}$ is determined uniquely by the fact that the following diagram commutes
	\begin{equation}\label{tildedbeqn}
		\begin{tikzcd}
		\Hom_{A\Mod}(\Omega^{p}_A,A)\arrow[r,"\tilde d_{b_\varphi}"] \arrow[d," (\cdot)\circ (d_A^{\dR})^{\wedge p}"] & \Hom_{A\Mod}(\Omega^{p+1}_A,A)  \arrow[d," (\cdot)\circ (d_A^{\dR})^{\wedge p+1}"]\\
			\Hom(\wedge^p A,A) \arrow[r,"d_{b_\varphi}"]&  	\Hom(\wedge^{p+1}A,A)	\end{tikzcd}  \ .
	\end{equation}
\end{coro}

Finally, by Corollary \ref{smthcoro}, we obtain the following:

\begin{coro}\label{GOequivcoro} Suppose the underlying commutative algebra $(A,m_\varphi)$ is smooth over $\K$. Then the inclusion of the subcomplex of geometric Poisson cochains is a quasi-isomorphism,
\[C^\bullet_\PV(A,m_\varphi,b_\varphi) \xrightarrow{\cong} C_\Pois^\bullet(A,m_\varphi,b_\varphi) \ .\]
\end{coro}


\subsection{Deformation-obstruction complex for algebras of the chiral kind}\label{ssec:def chiral algebras}
\subsubsection{Chiral algebras} In this section, we take $\mc P=\Lie_\K$ the Lie operad over $\K$ and $\mc Q=\End_{\DD(X)^\ch}(R)$ the endomorphism of operad of a D-module $R\in \mc C = \DD(X)^\ch$ with respect to the chiral pseudo tensor structure, so that the space $\Alg_{\mc P}(\mc Q)=\Hom_{\Op_\K}(\mc P,\mc Q)$ of $\mc P$-algebras internal to $\mc Q$ is equal to the space of chiral algebra structures on the D-module $M$.

\begin{prop}\label{chHoccxprop}
	The deformation-obstruction complex of Definition \ref{defobscxdefn} for $\mc P=\Lie_\K$ and $\mc Q=\End_{\DD(X)^\ch}(R)$ is given by the complex of (reduced) chiral cochains, that is,
	\[ \hspace*{-1cm} C^\bullet_{\mc P,\mc Q}(\varphi) = C^\bullet_{\ch}(R,\mu_\varphi) \qquad \text{where}\qquad C^\bullet_{\ch}(R,\mu_\varphi) = \left( \bigoplus_{n=1}^\infty  P_n^\ch(\{R\}_{i=1}^n ,R)_{\Ss_n}[1-n] \ , \ d_\textup{CE} \right)  \ , \]
	for $\mu_\varphi:j_*j^{*}(R^{\boxtimes 2}) \to R$ the chiral Lie bracket structure map and $d_\text{CE}:C^\bullet_{\ch}(R,\mu_\varphi)\to C^\bullet_{\ch}(R,\mu_\varphi)[1]$ the Chevalley--Eilenberg differential determined by $\mu_\varphi$, internal to the pseudo-tensor category $\DD(X)^\ch$.
\end{prop}
\begin{proof} The is a specialization of Proposition \ref{CEcxprop}, internal to the pseudo-tensor category $\DD(X)^\ch$.
\end{proof}

\subsubsection{Commutative chiral algebras}

In this section, we take $\mc P=\Comm_\K$ the commutative operad over $\K$ and $\mc Q=\End_{\DD(X)^!}(R)$ the endomorphism of operad of a D-module $R\in \mc C = \DD(X)^\ch$ with respect to the $!$ monoidal structure, so that the space $\Alg_{\mc P}(\mc Q)=\Hom_{\Op_\K}(\mc P,\mc Q)$ of $\mc P$-algebras internal to $\mc Q$ is equal to the space of $\Comm^!$ algebra structures on the D-module $R$.

\begin{prop}\label{chHoccxprop}
	The deformation-obstruction complex of Definition \ref{defobscxdefn} for $\mc P=\Comm_\K$ and $\mc Q=\End_{\DD(X)^!}(R)$ is given by the complex of (reduced) chiral Harrison cochains, that is,
	\[ \hspace*{-1cm} C^\bullet_{\mc P,\mc Q}(\varphi) = C^\bullet_{\Har}(R,m_\varphi) \qquad \text{where}\qquad C^\bullet_{\Har}(R,m_\varphi) = \left( \bigoplus_{n=1}^\infty  P_n^!(\{R\}_{i=1}^n , R)^{\Sh_n}[1-n] \ , \ d_\textup{Har} \right)  \ , \]
	for $m_\varphi:R^{\otimes 2} \to R$ is the $\Comm^!$ structure map and $d_\text{Har}:C^\bullet_{\Har}(R,m_\varphi) \to C^\bullet_{\Har}(R,m_\varphi)[1]$ the Harrison differential determined by $m_\varphi$, internal to the monoidal category $\DD(X)^!$.
\end{prop}
\begin{proof} The proof is precisely the same as that of Proposition \ref{Harcxprop}, internal to the monoidal category $\DD(X)^!$, \emph{mutatis mutandis}.
\end{proof}

We now describe the analogue of (higher) Andr\'e--Quillen cohomology in the chiral setting. Consider 
$R$ as $\Comm^!$-algebra with the multiplication $m_\varphi$ and let $\Omega^1_R\in \DD(R)$ be the module of K\"ahler differentials, as defined in the discussion following Definition \ref{Derchdefn}. Given a cofibrant resolution $\tilde R \to R$ of commutative algebras internal to $\Dd(X)^!$, the cotangent complex $\bL_R\in \Dd(R)$ is defined as the image of the object $\Omega^1_{\tilde R} \in \Dd(\tilde R)$ under the equivalence $\Dd(\tilde R) \xrightarrow{\cong} \Dd(R)$. Similarly, we define $\bL^{p}_R\in \Dd(R)$ as the image of the object $\Omega^p_{\tilde R} = (\Omega^1_{\tilde R})^{\wedge_{\tilde R} p} \in \Dd(\tilde R)$ under the same equivalence.

\begin{defn}\label{AQchDefn} 
The chiral Andr\'e--Quillen cochains on $R$ are defined by
	\[ C^\bullet_\AQ(R)= \Hom_{\Dd(R)}(\bL_R,R) \ .  \]
	Similarly, the degree-$p$ higher Andr\'e--Quillen cochains on $R$ are defined by
	\[ C^{\bullet,\langle p \rangle}_\AQ(R)= \Hom_{\Dd(R)^*}(\{\bL_R\}_{k=1}^{p+1};R)_{\Ss_{p+1}} \ ,  \]
	where $\Dd(R)^*$ denotes the derived category of $R$-modules internal to $\Dd(X)^!$, equipped with the $*$ pseudo-tensor structure, recalled in Definition \ref{Rlindefn} and Proposition \ref{Rlinstarprop}.
\end{defn}
Note that for $p=0$ we have $ C^{\bullet,\langle p \rangle}_\AQ(R)= C^\bullet_\AQ(R)$. Moreover, in this case the relevant cohomology groups are equivalent to the chiral Harrison cohomology defined above:
\begin{theo} There is a natural isomorphism
	\[ 
	H^\bullet_\AQ(R)\xrightarrow{\cong} H^\bullet_{\textup{Har}}(R,m_\varphi) \ . 
	\]
\end{theo}
\begin{proof}
	This is again a special case of Theorem 12.4.3 of \cite{LV}, internal to the $\K$-linear monoidal category $\DD(X)^!$.
\end{proof}

Finally, we formulate the analogue of the vanishing of higher Andre-Quillen cohomology for smooth algebras in the chiral setting: we will use the notion of very smooth $\Comm^!$ algebra, as in Definition 2.3.15 of \cite{BD1}, which guarantees the following analogue of Theorem \ref{theo:smooth then kahler isproj} in the chiral setting:

\begin{prop}\label{verysmoothprop} Let $R$ be a very smooth $\Comm^!$ algebra. Then $\bL_R\xrightarrow{\cong} \Omega_R^1\in \DD(R)=\Dd(R)^\heartsuit$, and the object $\Omega_R^1$ is projective.
\end{prop}
\begin{proof} This follows essentially by definition.
\end{proof}

The primary source of examples of very smooth $\Comm^!$ algebras used in the present work comes from the following:

\begin{prop} Let $B$ be a smooth algebra over $\K$ and $R=\mathcal{J}(B)$ be the $\Comm^!$ algebra of jets to $B$, as defined preceding Equation \ref{affinejetadjeqn}. Then $R$ is very smooth.
\end{prop}
\begin{proof}
	This is explained in 2.3.17 of \cite{BD1}.
\end{proof}

\begin{coro}\label{smthchcoro} Let $R$ be a very smooth $\Comm^!$ algebra. Then $C^\bullet_\AQ(R)$ is acyclic in non-zero degree, and
	\[ H^0_\AQ(R)\cong \Hom_{\DD(R)}(\Omega^1_R,R) = \Der(R,R) = h(\Theta_R) \]
	the module of derivations $ \Der(R,R) $ of $R$, as defined in Definition \ref{Derchdefn}, or equivalently the middle de Rham cohomology of the tangent module $\Theta_R  \in \DD(R)$.
	
	More generally, for each $p\geq 0$, $C^{\bullet,\langle p \rangle }_\AQ(R)$ is acyclic except in cohomological degree $p$, and
	\[ H^{p,\langle p \rangle}_\AQ(R)\cong P^*_{R,p+1}(\{\Omega^1_R\}_{k=1}^{p+1},R)_{\Ss_{p+1}} \cong h(\wedge^{p+1}_R\Theta_R) \ , \]
	the middle de Rham cohomology of the module of degree ${p+1}$ polyvector fields $\wedge^{p+1}_R \Theta_R \in \DD(R)$.
\end{coro}

\subsubsection{\texorpdfstring{$\Lie^*$}{Lie*} algebras}

Throughout this section, we take $\mc Q=\Lie_\K$ and $\mc P=\End_{\DD(X)^*}(R)$ the endomorphism of operad of a D-module $R\in \mc C = \DD(X)^*$ with respect to the * pseudo-tensor structure, so that the space $\Alg_{\mc P}(\mc Q)=\Hom_{\Op_\K}(\mc P,\mc Q)$ of $\mc P$-algebras internal to $\mc Q$ is equal to the space of $\Lie^*$ algebra structures on the D-module $R$.

\begin{prop}\label{chPoiscxprop}
	The deformation-obstruction complex of Definition \ref{defobscxdefn} for $\mc P=\Lie_\K$ and $\mc Q=\End_{\DD(X)^c}(R)$ is given by the complex of (reduced) Chevalley--Eilenberg cochains in $\DD(X)^*$, that is,
	\[ 
	C^\bullet_{\mc P,\mc Q}(\varphi) = C^\bullet_{\CE}(R,b_\varphi) \qquad \text{where}\qquad C^\bullet_{CE}(R,b_\varphi) = \left( \bigoplus_{n=1}^\infty   P_n^*(\{ R\}_{i=1}^n,R)_{\Ss_n}[1-n] \ , \ d_\text{CE} \right)  \ , 
	\]
for $d_\CE:C^\bullet_{\CE}(R,b_\varphi)\to C^\bullet_{\CE}(R,b_\varphi)[1]$ the usual Chevalley--Eilenberg differential determined by $b_\varphi$, internal to the pseudo-tensor category $\DD(X)^*$.
\end{prop}
\begin{proof}
This is a specialization of Proposition \ref{CEcxprop}, internal to the pseudo-tensor category $\DD(X)^*$.
\end{proof}

\subsubsection{Coisson algebras}\label{Coisdefsec}

In this section, we take $\mc Q=\Lie_\K$ and $\mc P=\End_{\DD(X)^c}(R)$ the endomorphism of operad of a D-module $R\in \mc C = \DD(X)^c$ with respect to the classical pseudo-tensor structure, so that the space $\Alg_{\mc P}(\mc Q)=\Hom_{\Op_\K}(\mc P,\mc Q)$ of $\mc P$-algebras internal to $\mc Q$ is equal to the space of coisson algebra structures on the D-module $R$.

\begin{prop}\label{chPoiscxprop}
	The deformation-obstruction complex of Definition \ref{defobscxdefn} for $\mc P=\Lie_\K$ and $\mc Q=\End_{\DD(X)^c}(R)$ is given by the complex of (reduced) Chevalley--Eilenberg cochains in $\DD(X)^c$, that is,
	\[ \hspace*{-1cm} C^\bullet_{\mc P,\mc Q}(\varphi) = C^\bullet_{c}(R,\mu_\varphi^c) \qquad \text{where}\qquad C^\bullet_{c}(R,\mu_\varphi^c) = \left( \bigoplus_{n=1}^\infty   P_n^c(\{ R\}_{i=1}^n,R)_{\Ss_n}[1-n] \ , \ d_\text{CE} \right)  \ , \]
	for $\mu_\varphi^c=(m_\varphi,b_\varphi)\in P_2^c(\{R,R\},R)_{\Ss_2}$ the classical Lie bracket structure map corresponding to a Coisson structure under Proposition \ref{CoisLieprop} above, and $d_\text{CE}:C^\bullet_{c}(R,\mu_\varphi)\to C^\bullet_{c}(R,\mu_\varphi)[1]$ the usual Chevalley--Eilenberg differential determined by $\mu_\varphi$, internal to the pseudo-tensor category $\DD(X)^c$.
\end{prop}
\begin{proof} 
This is a specialization of \ref{CEcxprop}, internal to the pseudo-tensor category $\DD(X)^c$.
\end{proof}

As we will explain, there is a bicomplex structure underlying the complex of classical cochains, analogous to that of Proposition \ref{Poisbicxprop} in the Poisson case. 

%
%

\begin{prop}\label{Coisbicxprop} 
There is a natural bicomplex $C^{\bullet,\bullet}_c(R,m_\varphi,b_\varphi)=\oplus_{p,q=0}^\infty C^{p,q}_c$ defined by
	\[
	C^{p,q}_c\equiv C^{p,q}_c(R,m_\varphi,b_\varphi)= P_{p+q+1}^c(\{ R\}_{i=1}^{p+q+1},R)_{S_{p+q+1}}^{\langle p \rangle }
	\]
	with horizontal and vertical differentials given, respectively, by
	\[
  d_h=d_{b_\varphi}:C^{p,q}_c \to C^{p+1,q}_c \qquad d_v=d_{m_\varphi}:C^{p,q}_c \to C^{p,q+1}_c 
	\]
	such that the total complex is naturally isomorphic to the classical cochains, that is, \[\textup{Tot } C^{\bullet,\bullet}_c(R,m_\varphi,b_\varphi) \xrightarrow{\cong} C_c^\bullet(R,\mu_\varphi^c) \ .\]
\end{prop}
\begin{proof}
	The Lie operad is generated by a single binary operation, which thus uniquely determines the Chevalley--Eilenberg differential. The decomposition of the space of binary operations in $P^c$ given in Equation \ref{cl2arydecompeqn} induces the decomposition of the differential as $d_{\CE}=d_{\mu_\varphi}+d_{b_\varphi}$ and these components respect the bigrading essentially by definition, and---by the coisson algebra relations---they commute appropriately.
\end{proof}

The zeroth row and column define quotient complexes
\[ 
C^\bullet_{c,m}(R,m_\varphi)=(\bigoplus_{n=1}^\infty P_{n}^c(\{ R\}_{i=1}^{n},R)_{\Ss_{n}}^{\langle 0 \rangle }[1-n] \ , \ d_{m_\varphi} )  ~,
\]
and
\[
C^\bullet_{c,b}(R,b_\varphi)=(\bigoplus_{n=1}^\infty P_{n}^c(\{ R\}_{i=1}^{n},R)_{\Ss_{n}}^{\langle n-1 \rangle }[1-n] \ , \ d_{b_\varphi} )  ~,
\]
which control the induced deformation of $\Comm^!$ and $\Lie^*$ algebra, respectively, coming from a deformation of Coisson algebra, as the following Proposition confirms:

\begin{prop}\label{firstrcprop}There are natural identifications of complexes
	\[C^\bullet_{c,m}(R,m_\varphi) \xrightarrow{\cong} C^\bullet_{\textup{Har}}(R,m_\varphi) \qquad \text{and}\qquad  C^\bullet_{c,b}(R,b_\varphi) \xrightarrow{\cong} C^\bullet_{\textup{CE}}(R,b_\varphi)  \ . \]
\end{prop}
\begin{proof} This follows straightforwardly from the preceding discussion by unwinding the definitions.
\end{proof}

More generally, for each $p\geq 0$ the $p^{th}$ column of the above bicomplex defines a complex
\[ C^{\bullet,\langle p \rangle}_{c,m}(R,m_\varphi)=(\bigoplus_{n=p+1}^\infty P_{n}^c(\{ R\}_{i=1}^{n},R)_{S_{n}}^{\langle p \rangle }[1-n] \ , \ d_{m_\varphi} ) \]
and we have the following analogue of Proposition \ref{AQPoisequivprop}:

\begin{prop}\label{AQPoischequivprop} For each $p\geq 0$, there are natural identifications of complexes
	\[
	C^{\bullet,\langle p \rangle}_{c,m}(R,m_\varphi)[p] \xrightarrow{\cong} C^{\bullet,\langle p \rangle}_\AQ(R) \ .
	\]
	In particular, we have $H^{p,\langle p \rangle}_{c,m}(R,m_\varphi) \cong P_{R,p+1}^*(\{\Omega^1_R\}_{k=1}^{p+1},R)_{\Ss_{p+1}}$.
\end{prop}
\begin{proof} The proof is given in detail for the corresponding statement about internal Hom objects in Proposition \ref{AQPoischequivintprop}, and the same proof applies here with fewer complications, \emph{mutatis mutandis}.
\end{proof}

\begin{rmk}
This is proved in Theorem 7.2 of \cite{BDHKV21} for $X=\bb A^1$ in the translation invariant setting
\end{rmk}

\subsubsection{} The analogue of the geometric Poisson cochains of Definition \ref{gPCdefn}, controlling deformations of a Poisson structure on a fixed underlying variety, is given by the following subcomplex:

\begin{defn}\label{gPCchdefn} 
Let $\varphi \in \Alg_\Lie(\End_{\DD(X)^c}(R))$ be a Coisson algebra 
structure on $R$ with corresponding commutative product $m_\varphi\in P_2^!(\{R,R\};R)$ and Lie bracket $b_\varphi \in P_2^*(\{R,R\},R)$. The (reduced) variational Poisson cochains on $(R,m_\varphi,b_\varphi)$ is the subcomplex of $C_c^\bullet(R,m_\varphi,b_\varphi)$ given by
	\[ C^\bullet_\PV(R,m_\varphi,b_\varphi)= \left( \bigoplus_{p=0}^\infty H^{p,\langle p \rangle}_{c,m}(R,m_\varphi)[-p] \ , \ d_{b_\varphi}  \ \right)  \ .\]
\end{defn}

The following concrete description, analogous to Corollary \ref{gPCcoro}, follows from Proposition \ref{AQPoischequivprop}:

\begin{coro}\label{gPCchcoro} The complex of variational Poisson cochains admits an isomorphism of complexes
	\[ 
	C^\bullet_\PV(R,m_\varphi,b_\varphi) \cong \left( \bigoplus_{p=0}^\infty P^*_{R,p+1}(\{\Omega^1_R\}_{k=1}^{p+1},R)_{\Ss_{p+1}}[-p] \ , \ \tilde d_{b_\varphi} \ \right) \ ,  
	\]
		where $\tilde d_{b}$ is determined uniquely by the fact that the following diagram commutes
	\begin{equation}\label{tildedbeqn}
		\begin{tikzcd}
			 P^*_{R,p}(\{\Omega^1_R \}_{k=1}^{p};R)_{\Ss_{p}} \arrow[r,"\tilde d_{b_\varphi}"] \arrow[d," (\cdot)\circ (d_R^{\dR})^{\boxtimes p}"] & \mc P^*_{R,p+1}(\{\Omega^1_R \}_{k=1}^{p+1};R)_{\Ss_{p+1}}  \arrow[d," (\cdot)\circ (d_R^{\dR})^{\boxtimes p+1}"]\\
			\mc P^*_p(\{R\}_{k=1}^p,R)_{\Ss_{p}} \arrow[r,"d_{b_\varphi}"]&  	\mc P^*_{p+1}(\{R\}_{k=1}^{p+1},R)_{\Ss_{p+1}} 	\end{tikzcd}  \ .
	\end{equation}
\end{coro}

Finally, by Corollary \ref{smthchcoro}, we obtain the following:

\begin{coro}\label{GOequivchcoro} Suppose the underlying $\Comm^!$ algebra $(R,m_\varphi)$ is very smooth. Then the inclusion of the subcomplex of variational Poisson cochains is a quasi-isomorphism,
	\[C^\bullet_\PV(R,m_\varphi,b_\varphi) \xrightarrow{\cong} C_c^\bullet(R,m_\varphi,b_\varphi) \ .\]
\end{coro}

We remark that this was proved for $R$ an algebra of differential polynomials in \cite{BDHKV21}, that is, the algebra of jets to an affine space $\bb A^{d}$. The preceding corollary implies that this generalizes to the algebra of jets on a smooth $\K$-algebra.


%
%
%
%
%
%


\section{Deformation theory}\label{chsec}
Throughout, let $\IA^1_\hbar = \Spec \ik[\hbar]$ and let $\IG_m$ act on $\IA^1_\hbar$ by scaling $\hbar $ with unit weight. We write $\IG_\hbar$ for this distinguished multiplicative group action.

\subsection{Deformation theory for algebras}\label{mapdefsec}

Let $\PP$ be a binary, Koszul operad over $\ik$ and let $\QQ$ be a single coloured operad over $\ik$. We now formalize the moduli problem controlling formal deformations of a $\mc P$-algebra structure internal to $\mc Q$. 

Let $(S,\mf m_S)$ be a local Artinian $\ik$-algebra with a closed embedding of $\ik$-schemes $\Spec S \hookrightarrow \IA^1_{\hbar}$ such that the closed point in $\Spec S$ is mapped to the origin of $\IA^1$. We write $\PP_S = \PP\otimes S$ and $\QQ_S = \QQ\otimes S$ for the constant families of $\PP$ and $\QQ$ over $\Spec S$.

\begin{prop} There is an isomorphism of DG Lie algebras
\[ \g_{\mc P_S,\mc Q_S} \cong \g_{\mc P, \mc Q} \otimes_{\K} S \]
where the latter is equipped with the trivial $S$-linear extension of the differential $d:\g_{\mc P,\mc Q} \to \g_{\mc P,\mc Q}[1]$ and Lie bracket on $ \g_{\mc P,\mc Q}$.
\end{prop}
\begin{proof}
This is Lemma 12.2.5 in \cite{LV}.
\end{proof}

There is a natural map of DG Lie algebras given by the evaluation at $\mf m_S$
\[ (\cdot)/\mf m_S : \g_{\mc P_S,\mc Q_S} \to \g_{\mc P_S,\mc Q_S} \otimes_{S} S/\mf m_S \cong \g_{\mc P, \mc Q} \  \]
and thus a map on Maurer--Cartan solutions
\[
\MC(\gf_{\PP_S,\QQ_S}) \rightarrow \MC({\gf_{\PP,\QQ}})~,
\]
taking an $S$-family of $\PP$ algebras, internal to $\QQ$, to its restriction at the closed point $\mf m_S$.


\begin{defn}\label{Defdefn}
 Let $\varphi\in \Alg_{\mc P}(\mc Q)\cong \MC(\g_{\mc P,\mc Q})$ be a $\mc P$-algebra internal to $\mc Q$ considered as a Maurer--Cartan element. The space of deformations of $\varphi$ over $\Spec S$ is
 \[ 
 \Def_\varphi(S) = \MC(\g_{\mc P_S,\mc Q_S}) \fibre{\MC(\g_{\mc P, \mc Q})} \{\varphi\}  \ .
 \]
\end{defn}
Given some $\varphi\in \Alg_\PP(\QQ)=\Hom_{\Op_\ik}(\PP,\QQ)$, we can view this as a map between the fibres of $\PP_S$ and $\QQ_S$ at the closed point $\mf m_S$. A deformation $\varphi_S\in\Def_\varphi(S)$, therefore, can be understood to be an extension of this map to the whole $S$-family, such that the restriction to the central fibres agrees with $\varphi$.

The variety of deformations has a concrete realisation in terms of Maurer--Cartan elements:
\begin{prop}\label{DefMCprop} 
We have an isomorphism of varieties, 
\[
\Def_\varphi(S) \cong \MC( \g^\varphi_{\mc P, \mc Q}\otimes_{\bb K} \mf m_S) \ ,
\] 
where $\g^\varphi_{\PP,\QQ}\otimes \mf m_S$ is the DG Lie algebra with the trivial $\mf m_S$-linear extension of the differential and Lie bracket on $\g^\varphi_{\PP,\QQ}$.
\end{prop}
\begin{proof}
This is Proposition 12.2.6 in \cite{LV}.
\end{proof}

\subsubsection{The stack of formal deformations}

The weight $0$ component of the Koszul dual co-operad $\PP^\ash$ is simply a copy of the base field $\ik$ in arity $1$, \ie,
\[
\PP^{\ash,(0)} = (0,\ik,0,0\dots)  \qquad \text{and so} \qquad \gf_{\PP,\QQ}^{(0)} = \QQ(1)
\]

For any single coloured (DG) operad $\QQ$, the space of unary operations $\QQ(1)$ is a unital (DG) associative $\ik$-algebra. Moreover, in the DG setting, the cohomological degree zero component, $\QQ(1)^0$ is an associative subalgebra of $\QQ(1)$. For example if $\QQ $ is the endomorphism operad on the single colour $V\in \Vect_\ik$, then $\QQ(1) = \End(V)$ is just the endomorphism algebra of $V$. 

Two deformations $\varphi_S, \widetilde{\varphi}_S\in \Def_\varphi(S)$ of $\varphi\in \Alg_{\PP}(\QQ)$ are called equivalent if there is an invertible element $g\in \QQ_S(1)^0)^\times$ such that $\varphi_S = g\widetilde{\varphi}g^{-1}$ and $g/\mf m_S =\id \in \QQ(1)^0$.


\begin{defn}\label{fDefdefn}
We define the stacks of formal deformations $\fDef_\varphi$ by
	\begin{equation}\label{fDefeqn}
		\fDef_\varphi(S) = \Def_{\varphi}(S) / \sim  
	\end{equation}
	where the right hand side denotes the quotient groupoid of $\Def_{\varphi}(S)$ by the equivalence relation defined above.
\end{defn}

\subsubsection{The Maurer--Cartan stack}

The weight zero subspace $\g_{\mc P,\mc Q}^{(0)}$ is naturally a Lie subalgebra and thus acts, via the adjoint action, on $\g_{\mc P,\mc Q}$ preserving the weight grading. Similarly, the weight and cohomological degree zero subspace $\g_{\mc P,\mc Q}^{0,(0)}\subset \gf_{\PP, \QQ}^{(0)}$ is a further Lie subalgebra which acts on $\g_{\mc P,\mc Q}$ preserving the weight and cohomological grading.

For a local Artinian ring $(S,\mf m_S)$ as above, one can define the analogous subalgebras 
\[ 
\g_{\mc P,\mc Q,\mf m_S}^{0,(0)} = \g_{\mc P,\mc Q}^{0,(0)} \otimes_{ \K}\mf m_S\  \subset \ \g_{\mc P,\mc Q,\mf m_S}^0 = \g_{\mc P,\mc Q}^0\otimes_{ \K}\mf m_S \ \subset \  \g_{\mc P,\mc Q,\mf m_S} = \g_{\mc P,\mc Q}\otimes_{ \K}\mf m_S   ~.
\]
Since these subalgebras are all nilpotent, we have formal exponential algebraic groups
\[
\mf G_{\PP,\QQ,\mf m_S}^0 = \exp(\gf_{\PP,\QQ,\mf m_S}^0) \qquad \text{and}\qquad \mf G_{\PP,\QQ,\mf m_S}^{0,(0)} = \exp(\gf_{\PP,\QQ,\mf m_S}^{0,(0)})~.
\]

These groups act on the space of solutions to the Maurer--Cartan equations, so that we have the following definition:
\begin{defn}\label{fMCdefn}
The Maurer--Cartan stack $\fMC_\varphi$ is defined as  	
\begin{equation}\label{MCstackeqn}
		\fMC_\varphi(S) =\left[ \MC( \g^\varphi_{\mc P, \mc Q}\otimes_{\bb K} \mf m_S)^{(1)}/\mf G_{\mc P,\mc Q,\mf m_S}^{0,(0)}\right]   \ .
	\end{equation}
\end{defn} 

This action encodes the automorphisms of the corresponding $\mc P$-algebra internal to $\mc Q$ in the sense of the following proposition.

\begin{prop}\label{fDefMCprop} 
There are isomorphisms of algebraic stacks
\[ 
\fMC_\varphi \xrightarrow{\cong} \fDef_\varphi \ . 
\]
\end{prop}
\begin{proof} This is Theorem 12.2.10 in \cite{LV}.
\end{proof}

\begin{rmk}\label{MCdefgensubsec}
 We also record for later reference the following formally equivalent description of the Maurer--Cartan stack, which will be more amenable to generalization later, in particular in Definition \ref{Defhbardefn}:
\begin{align} 
\notag	\fMC_\varphi(S) & =\left[ \MC( \g^\varphi_{\mc P, \mc Q}\otimes_{\bb K} \mf m_S)/\mf G_{\mc P,\mc Q,\mf m_S}^{0,(0)}\right]   \\
	\notag & \xrightarrow{\cong} \left[  \left(\MC(\g_{\mc P_S,\mc Q_S}) \times_{\MC(\g_{\mc P, \mc Q})} \{\varphi\}\right) /\mf G_{\mc P,\mc Q,\mf m_S}^{0,(0)}\right]  \\
	& \xrightarrow{\cong} \left[\MC(\g_{\mc P_S,\mc Q_S}) /\mf G_{\mc P,\mc Q,\mf m_S}^{0,(0)}\right] \times_{\MC(\g_{\mc P, \mc Q})} \{\varphi\} \label{MCconceqn}
\end{align}
where the first isomorphism follows from Proposition \ref{DefMCprop}, and the second since the group acts trivially on $\MC(\g_{\mc P, \mc Q})$ and $\{\varphi\}$, so that in particular the restriction map 
\[\left[\MC(\g_{\mc P_S,\mc Q_S}) /\mf G_{\mc P,\mc Q,\mf m_S}^{0,(0)}\right] \to \MC(\g_{\mc P, \mc Q})\]
 is well-defined.
\end{rmk}

\subsubsection{}

Recall the deformation-obstruction complex of \ref{defobscxdefn}, this complex controls the deformation theory of $\varphi\in \Alg_{\PP}(\QQ)$ in the sense of the following Proposition:

\begin{prop}\label{tgtOpMapprop} 
We have an isomorphism of varieties and stacks, respectively,
\[
\Def_\varphi(\ik[\hbar]/\hbar^2) \cong Z^1_{\PP,\QQ}(\varphi)~, \qquad \text{ and } \qquad \fDef_\varphi(\ik[\hbar]/\hbar^2) \cong H^1_{\PP,\QQ}(\varphi)~.
\]
\end{prop}
\begin{proof} 
This is precisely Theorem 12.2.7 in \cite{LV}. 
\end{proof}

At higher order in $\hbar$, the existence of a deformation is, in general, obstructed. Nevertheless, the space of such order-by-order deformations remains a torsor for $H^1_{\PP,\QQ}(\varphi)$. What follows is a strengthening of Theorems 12.2.8 and 12.2.9 of \cite{LV}, following similar proofs to those of \emph{loc. cit.}

Anticipating the following proposition, we define the obstruction cochain associated with $\varphi_{\hbar/\hbar^{m}}\in \Def^{\mc P,\mc Q}_\varphi(\K[\hbar]/\hbar^{m}) $ by
\[ \Obs_m(\varphi_{\hbar/\hbar^{m}})=-\frac{1}{2} \sum_{k=1}^{m-1}[ \varphi_k,\varphi_{m-k} ] \ \in (\g_{\mc P, \mc Q}^{\varphi})^{2,(2)} \ . \]
Then we have:

\begin{prop}\label{DefObsprop} 
Let $\varphi \in \Alg_{\mc P}(\mc Q)$ and $\varphi_{\hbar/\hbar^{m}}\in \Def_\varphi( \K[\hbar]/\hbar^{m})$. Then the obstruction associated to $\varphi_{\hbar/\hbar^m}$ is closed, \ie,
\[
d_\varphi \Obs_m(\varphi_{\hbar/\hbar^m}) =0 
\]
Moreover, for some putative extension to $m$th order in $\hbar$, $\varphi_{\hbar/\hbar^{m+1}} = \varphi_{\hbar/\hbar^m}+ \hbar^m \varphi_m$, we have 
\[ 
\varphi_{ / \hbar^{m+1}} \in \Def_\varphi(\K[\hbar]/\hbar^{m+1}) \times_{\Def_\varphi(\K[\hbar]/\hbar^{m}) } \{\varphi_{\hbar/\hbar^{m}} \} \qquad\text{if and only if} \qquad d_\varphi \varphi_m = \Obs_m(\varphi_{\hbar/\hbar^{m}})  \ .
\]

Thus, there exists an extension
\[ \varphi_{ \hbar/ \hbar^{m+1}} \in \Def_\varphi(\K[\hbar]/\hbar^{m+1}) \times_{\Def_\varphi(\K[\hbar]/\hbar^{m}) } \{\varphi_{\hbar/\hbar^{m}} \} \]
if and only if the cohomology class of the obstruction vanishes, that is,
\[ [ \Obs_m(\varphi_{\hbar/\hbar^{m}})]= 0 \ \in H^2_{\mc P,\mc Q}(\varphi)\ .\]
Further, when the obstruction class vanishes, the space of such extensions up to isomorphism,
\[  \fDef_\varphi(\K[\hbar]/\hbar^{m+1}) \times_{\fDef_\varphi(\K[\hbar]/\hbar^{m}) } \{\varphi_{\hbar/\hbar^{m}} \} \ , \]
is a torsor for $H^1_{\mc P,\mc Q}(\varphi)$.
\end{prop}
\begin{proof} 
First we show that the obstruction,	$\Obs_m(\varphi_{\hbar/\hbar^{m}})$, is closed. The Bianchi identity for the DG Lie algebra $\gf^\varphi_{\PP,\QQ}\otimes \ik[\hbar]/\hbar^{m+1}$ implies that 
\[
\big[\varphi_{\hbar/\hbar^{m}}, d_\varphi \varphi_{\hbar/\hbar^{m}} + \frac12[\varphi_{\hbar/\hbar^{m}},\varphi_{\hbar/\hbar^{m}}]\big]=0~.
\]
Taking the coefficient of $\hbar^m$ in the above tells us that $\Obs_m(\varphi_{\hbar/\hbar^m})$ must be closed.

Now, consider our putative lift
\[ 
\varphi_{\hbar / \hbar^{m+1}}= \varphi + \sum_{k=1}^m \hbar^k \varphi_k \ \in (\g_{\mc P, \mc Q}^{\varphi})^{1,(1)}\otimes_{\K} \K[\hbar]/\hbar^{m+1} ~,
\]
whose restriction to $\ik[\hbar]/\hbar^m$,
\[
\varphi_{\hbar/\hbar^m} := \varphi_{\hbar/\hbar^{m+1}} \mod \hbar^m  =\varphi+  \sum_{k=1}^{m-1} \hbar^k \varphi_k \ \in  (\g_{\mc P, \mc Q}^{\varphi})^{1,(1)}\otimes_{\K} \K[\hbar]/\hbar^{m} ~,
\] 
satisfies the Maurer--Cartan equations. Then $\varphi_{ \hbar / \hbar^{m+1}}$ satisfies the Maurer--Cartan equations if and only if they hold at order $\hbar^m$, which is equivalent to the condition that
\begin{equation}\label{defMCeqn}
	d_\varphi \varphi_m +[\varphi,\varphi_m] + \frac{1}{2} \sum_{k=1}^{m-1}[ \varphi_k,\varphi_{m-k} ] = 0 \ .
\end{equation}
Thus, by the definition of the obstruction cocyle, the data $\varphi_m$ defining such an extension is equivalent to an element $\varphi_m\in (\g_{\mc P, \mc Q}^{\varphi}) ^{1,(1)}$ satisfying $d_\varphi \varphi_m = \Obs_m(\varphi_{\hbar/\hbar^{m}})$, or equivalently a trivialization of the cohomology class $[ \Obs_m(\varphi_{\hbar/\hbar^{m}})]\in H^2_{\mc P,\mc Q}(\varphi)$.
	
Given two such extensions of $\varphi_{\hbar/\hbar^{m}}$ to order $\hbar^{m}$, denoted
\[\varphi_{ / \hbar^{m+1}} \  , \ \tilde \varphi_{ / \hbar^{m+1}} \in \Def_\varphi(\K[\hbar]/\hbar^{m+1}) \times_{\Def_\varphi(\K[\hbar]/\hbar^{m}) } \{\varphi_{\hbar/\hbar^{m}} \} \ , \]
their corresponding elements $\varphi_m, \tilde\varphi_m \in (\g_{\mc P, \mc Q}^{\varphi})^{1,(1)} $ evidently satisfy
\[ d_\varphi(\varphi_m- \tilde\varphi_m )=\Obs_m(\varphi_{\hbar/\hbar^{m}})-\Obs_m(\varphi_{\hbar/\hbar^{m}})=0~,\]
so that they define a cocycle $\varphi_m- \tilde\varphi_m \in Z^1_{\mc P,\mc Q}(\varphi)$, and conversely the same argument implies any such cocycle $\zeta \in Z^1_{\mc P,\mc Q}(\varphi)$ added to a given $\varphi_m$ yields an alternative extension $\tilde \varphi_m=\varphi_m + \zeta$ which still satisfies $d_\varphi(\tilde \varphi_m)=\Obs_m(\varphi_{\hbar/\hbar^{m}})$.

Furthermore, two such deformations are identified by an automorphism
\[
e^{\hbar^m a} = \id + \hbar^m a \in {\frak G}^{0,(0)}_{\PP,\QQ,\hbar^{m+1}} ~,
\]
with $a\in (\g_{\mc P,\mc Q}^\varphi)^{0,(0)}$if and only if $e^{\hbar^m a}\circ \varphi_{\hbar/\hbar^{m+1}}\circ e^{-\hbar^m a} = \widetilde{\varphi}_{\hbar/\hbar^{m+1}}$, which in turn implies that
\[ 
\varphi_m - \tilde\varphi_m = d_\varphi a \ ,
\]
noting the only non-trivial constraint is at order $\hbar^m$.

\end{proof}

\subsection{Deformation theory for families of algebras}\label{deffamsec} 

In this section, as in the previous, we study the order-by-order deformation-obstruction theory of maps between families of operads but unlike the preceding section we allow the operads to be non-constant.

Let $\Ph$ and $\QQ_\hbar$ be operads internal to $\K[\hbar]\Mod$, which we view as families of operads over $\IA^1_\hbar$. The space of $\mc P_\hbar$-algebra structures internal to $\mc Q_\hbar$ is controlled by the weight graded, DG Lie algebra internal to $\ik[\hbar]\Mod$, defined by
\[  
\g_{\mc P_\hbar,\mc Q_\hbar}=\bigoplus_{\ell\in \Z_{\geq 0}}\g_{\mc P_\hbar,\mc Q_\hbar}^{(\ell)} \qquad\text{where}\qquad  \g_{\mc P_\hbar,\mc Q_\hbar}^{(\ell)}=\Hom_{\Ss}(\overline{\mc P^{\ash}_\hbar}^{ (\ell)} ,\mc Q_\hbar) =\bigoplus_{n\in \N} \Hom_{\K[\hbar][\Ss_n]\Mod}(\overline{\mc P^{\ash}_\hbar}^{ (\ell)}(n) ,\mc Q_\hbar(n))  \ .
\]
In terms of this grading, we have the following application of Proposition \ref{MCplainprop}:
\begin{coro}\label{MCplainhbarprop} 
A $\Ph$-algebra internal to $\Qh$ is equivalent to a solution to the Maurer-Cartan equations in $\gPQh$ of pure weight $1$, that is,
	\[ 
	\Alg_{\Ph}(\Qh) =\Hom_{\Op_{\K[\hbar]}}(\Ph, \Qh)  \xrightarrow{\cong } \MC(\g_{\Ph, \Qh}) \subset \g_{\Ph, \Qh}^{1 ,(1)} ~.
	\]
\end{coro}

Given $\varphi_\hbar \in  \Alg_{\Ph}(\Qh)$, a $\Ph$-algebra internal to $\Qh$, we obtain a new DG Lie algebra over $\K[\hbar]$
\[ \gPQvh =  \left( \Hom_{\Ss}(\mc P^\ash_\hbar ,\mc Q_\hbar) \ , \  [ \cdot , \cdot ]_{\gPQh}   \ , \  d_{\varphi_\hbar}  \right) \]
where $d_{\varphi_\hbar} =d_{\gPQh}+[\varphi_\hbar,\cdot]_{\gPQh}:\gPQh\to \gPQh[1]$ as in Definition \ref{gvarphidef}. In particular, we obtain the operadic deformation-obstruction complex
\[ C^\bullet_{\mc P_\hbar,\Qh} (\varphi_\hbar) = (\gPQvh \ , \ d_{\varphi_\hbar}) \ . \]
More generally, for any fixed order $m\geq 0$ there is a natural DG Lie algebra over $\K[\hbar]/\hbar^{m+1}$
\[ \g_{\Ph,\Qh /\hbar^{m+1}} = \g_{\Phm,\Qhm} \qquad \text{and similarly}\qquad \g_{\Ph,\Qh /\hbar^{m+1}}^{\varphi_{\hbar/\hbar^{m+1}}} = \g_{\Phm,\Qhm}^{\varphi_{\hbar/\hbar^{m+1}}}  \]
a DG Lie algebra defined for each $\varphi_{\hbar/\hbar^{m+1}}\in \Alg_{\Phm}(\Qhm)$
which manifestly controls the deformation obstruction theory of $\mc P_{ / \hbar^{m+1}}$-algebras internal to $\mc Q_{ / \hbar^{m+1}}$, where we define
\[ \Phm = \Ph\otimes_{\K[\hbar]} \K[\hbar]/\hbar^{m+1}\qquad \text{and} \qquad \Qhm = \Qh\otimes_{\K[\hbar]} \K[\hbar]/\hbar^{m+1}   \ . \]

Taking $m=0$ in this construction recovers the specializations of the DG Lie algebras $\gPQh$ and $\gPQvh$ at $\hbar=0$, which agree with the usual DG Lie algebras associated to the pair of operads $(\mc P_0,\mc Q_0)$ and the algebra $\varphi=\varphi_{\hbar/\hbar} \in \Alg_{\mc P_0}(\mc Q_0)$, that is, we have
\[ \gPQh \otimes_{\K[\hbar]} \K_0 \cong  \g_{\mc P_0,\mc Q_0} \qquad \text{and}\qquad \gPQvh \otimes_{\K[\hbar]} \K_0 \cong  \g_{\mc P_0,\mc Q_0}^\varphi \ . \]

We now introduce the analogue of Definition \ref{Defdefn} of the space of deformations, and its generalization to the stack of deformations in Equation \ref{fDefeqn}, to the present setting. Let $(S,\mf m_S)$ be a local Artinian ring with a closed embedding $\Spec S \into \IA^1_\hbar$ sending the closed point of $\Spec S$ to the origin of $\IA^1$.


The following definition describes the space of deformations of $\varphi\in\Alg_{\PP_0}(\QQ_0)$ over $\Spec S$ to a $\mc P_\hbar$-algebra internal to $\mc Q_\hbar$:

\begin{defn}\label{Defhbardefn} 
Let $\varphi\in \Alg_{\mc P_0}(\mc Q_0)$ be an $\mc P$-algebra internal to $\mc Q$. The space of deformations of $\varphi$ over $\Spec S$, as an $\mc P_\hbar$-algebra internal to $\mc Q_\hbar$, is defined as the fibre product
\[ 
\Def_\varphi^{\Ph,\Qh}(S) = \MC(\gPQh\otimes_{\K[\hbar]} S) \fibre{\MC(\g_{\mc P_0,\mc Q_0})} \{\varphi\} \ .  
\]
Similarly, the stack of such deformations is defined as the fibre product of Maurer-Cartan stacks
\[
\fDef_{\varphi}^{\Ph,\Qh}(S) = \fMC(\gPQh\otimes_{\K[\hbar]} S) \fibre{\fMC(\g_{\mc P_0,\mc Q_0})} \{\varphi_0\} \ , 
\]
where the Maurer-Cartan stack is defined as the quotient stack
\[ 
\fMC(\g_{\Ph,\Qh }\otimes_{\K[\hbar]} S)= \left[ \MC(\g_{\Ph,\Qh }\otimes_{\K[\hbar]} S)\  \big/\  \mf G^{0,(0)}_{\Ph,\Qh,\mf m_S} \right]  
\]
and the group $\mf G^{0,(0)}_{\Ph,\Qh,\mf m_S}=\exp(\g^{0,(0)}_{\Ph,\Qh}\otimes_{\K[\hbar]} \mf m_S)$, the formal exponential of the nilpotent Lie algebra $\g^{0,(0)}_{\Ph,\Qh}\otimes_{\K[\hbar]} \mf m_S$---analogous to the definition of the Maurer-Cartan stack in Equation \ref{MCstackeqn}.
\end{defn}

In particular, the scheme and stack of deformations over $\K[\hbar]/\hbar^{m+1}$ are given by
\[
	 \Def_\varphi^{\Ph,\Qh}(\K[\hbar]/\hbar^{m+1})  =\MC(\g_{\PP_{\hbar/\hbar^{m+1}},\QQ_{\hbar/\hbar^{m+1}}})\fibre{\MC(\g_{\mc P_0,\mc Q_0})} \{ \varphi \} ~,
\]
and
\[
	 \fDef_\varphi^{\Ph,\Qh}(\K[\hbar]/\hbar^{m+1})  =\fMC(\g_{\PP_{\hbar/\hbar^{m+1}},\QQ_{\hbar/\hbar^{m+1}}})\fibre{\fMC(\g_{\mc P_0,\mc Q_0})} \{ \varphi \} ~,
\]
respectively. Note that
\[ 
\MC(\g_{\mc P_0,\mc Q_0})\xrightarrow{\cong }\fMC(\g_{\mc P_0,\mc Q_0}) \ , 
\]
since the maximal ideal of the local ring is trivial, and thus the group $\mf G^{0,(0)}_{\Ph,\Qh,\mf m_S}$ is as well. Thus, the fibre product condition in the definition of  $\fDef_\varphi^{\Ph,\Qh}(R) $ requires strict agreement of the restriction to the central fibre with the initial algebra $\varphi$, rather than up to some isomorphism.

Compared to the simpler descriptions of Propositions \ref{DefMCprop} and \ref{fDefMCprop}, the Maurer--Cartan stacks does not make reference to the nilpotent DG Lie algebra $\gf_{\PP_0,\QQ_0}^{\varphi}\otimes \mf m_S$. Nevertheless, we will be able to establish the analogue of Proposition \ref{DefObsprop} on the order by order deformation theory.

\subsubsection{Order by order deformation theory}

Before we proceed, a warning for the reader:

\begin{warn}\label{flatwarn}
	For the remainder of this section, we will assume that $\mc P_\hbar$ and $\mc Q_\hbar$ are such that at each arity $n$, the  $\K[\hbar]$-modules $\mc P_\hbar^\ash(n)$ and $\mc Q_\hbar(n)$ are trivializable, and in the DG case we assume that the analogue is true for the underlying  $\Kh$ module in each cohomological degree. 

\end{warn}

We will refer to this data formally in terms of the following definition:

\begin{defn}\label{Ssmodtrivdefn} 
By a trivialization of the underlying $\Ss$-module of an operad $\QQ_\hbar$ or a cooperad $\CC_\hbar$ over $\ik[\hbar]$ we mean a choice of trivialization at each arity $n$ of the $\ik[\hbar]$-modules $\QQ_\hbar(n)$ and $\CC_\hbar(n)$, \ie am isomorphism of $\ik[\hbar]$-modules
\[
\QQ_{\hbar}^\ash(n) \cong \QQ_0^\ash(n) \otimes \ik[\hbar] \qquad \text{ and } \qquad \CC_\hbar(n) \cong \CC_0(n) \otimes \ik[\hbar]~.
\]
Note that this is much weaker than a trivialization of the operadic or cooperadic structure.
\end{defn}

We fix such trivializations of the underlying $\Ss$-modules of $\mc P^\ash_\hbar$ and $\mc Q_\hbar$, thus giving an identification
\begin{equation}\label{gtriveqn}
	 \gPQh \cong \g_{\mc P_0,\mc Q_0}\otimes_{\K} \K[\hbar] \qquad\text{with}\qquad d_{\gPQh}= \sum_{j=0}^\infty \hbar^j d_j \quad \text{and} \quad [\cdot ,\cdot]_\gPQh = \sum_{j=0}^\infty \hbar^j[\cdot,\cdot]_j \ ,
\end{equation}
for maps $d_j:\g_{\mc P_0,\mc Q_0} \to \g_{\mc P_0,\mc Q_0}[-1]$ and $[\cdot,\cdot] :\g_{\mc P_0,\mc Q_0}^{\otimes 2}\to \g_{\mc P_0,\mc Q_0}$. The zeroth order maps $d_0:g_{\mc P_0,\mc Q_0} \to \g_{\mc P_0,\mc Q_0}[-1]$ and $[\cdot,\cdot]_0 :\g_{\mc P_0,\mc Q_0}^{\otimes 2}\to \g_{\mc P_0,\mc Q_0}$ are the usual DG Lie algebra structure on $\g_{\mc P_0,\mc Q_0}$, while the higher order in $\hbar$ terms describe the deformation of this DG Lie algebra as the operads $\mc P_0$ and $\mc Q_0$ deform over $\K[\hbar]$.


Recall from the preceding section that when $\PP_\hbar$ and $\QQ_\hbar$ are constant, there is no obstruction to deforming $\varphi\in\Alg_{\PP_0}(\QQ_0)$ over $\ik[\hbar]/\hbar^2$. This is because one always has a trivial deformation of $\varphi$ by extending it as a constant map of operads over $\ik[\hbar]/\hbar^2$. It is only after choosing a non-trivial deformation of $\varphi$ to first order that one might encounter a potential obstruction.

However, in the current setting, $\PP_\hbar$ and $\QQ_\hbar$ are non-constant and there is no canonical extension of $\varphi$ as a constant. Thus, one should expect even the first order deformation of $\varphi$ to be potentially obstructed. 


We proceed to repeat the analysis preceding Proposition \ref{DefObsprop} in the present setting: For $m\geq 1$, we define the generalized obstruction cochain associated with $\varphi_{\hbar/\hbar^{m}}\in \Def^{\Ph,\Qh}_\varphi(\K[\hbar]/\hbar^{m}) $ by
\begin{equation}\label{Obshbareqn}
	 \Obs_m^{\Ph,\Qh}(\varphi_{\hbar/\hbar^{m}}):=-\frac{1}{2} \sum_{k=1}^{m-1}[ \varphi_k,\varphi_{m-k} ] -  \sum_{j=1}^m d_j \varphi_{m-j} - \sum_{j=1}^m \sum_{k=0}^{m-j} \frac{1}{2}[\varphi_k, \varphi_{m-j-k}]_j  \ \in \g^{2,(2)}_{\mc P,\mc Q} \ .
\end{equation}
We now state the common generalization of Propositions \ref{tgtOpMapprop} and \ref{DefObsprop} in this setting:

\begin{theo}\label{DefObsThm} Let $\varphi\in \Alg_{\mc P_0}(\mc Q_0)$ be a $\mc P_0$-algebra internal to $\mc Q_0$ and $\varphi_{\hbar/\hbar^{m}}\in \Def^{\Ph,\Qh}_\varphi(\K[\hbar]/\hbar^{m}) $. Then $\Obs_m(\varphi_{\hbar/\hbar^{m}})$ satisfies $d_\varphi \Obs_m(\varphi_{\hbar/\hbar^{m}})=0$, and for  $\varphi_{ \hbar/ \hbar^{m+1}} =  \varphi_{\hbar/\hbar^{m}} + \hbar^m \varphi_m $ we have
		\[ \varphi_{ \hbar/ \hbar^{m+1}} \in \Def^{\Ph,\Qh}_\varphi(\K[\hbar]/\hbar^{m+1}) \times_{\Def^{\Ph,\Qh}_\varphi(\K[\hbar]/\hbar^{m}) } \{\varphi_{\hbar/\hbar^{m}} \} \qquad \text{if and only if} \qquad d_{\varphi} \varphi_m  = \Obs_m(\varphi_{\hbar/\hbar^{m}}) \ . \]
		Thus, there exists an extension
\[  \varphi_{ \hbar/ \hbar^{m+1}} \in \Def^{\Ph,\Qh}_\varphi(\K[\hbar]/\hbar^{m+1}) \times_{\Def^{\Ph,\Qh}_\varphi(\K[\hbar]/\hbar^{m}) } \{\varphi_{\hbar/\hbar^{m}} \} \]
if and only if the cohomology class of the obstruction vanishes, that is,
\[ [ \Obs^{\Ph,\Qh}_m(\varphi_{\hbar/\hbar^{m}})]= 0 \ \in H^{2}_{\mc P_0,\mc Q_0}(\varphi) \ .\]
Further, when the obstruction class vanishes, the space of such extensions up to isomorphism,
\[  \fDef^{\Ph,\Qh}_\varphi(\K[\hbar]/\hbar^{m+1}) \times_{\fDef^{\Ph,\Qh}_\varphi(\K[\hbar]/\hbar^{m}) } \{\varphi_{\hbar/\hbar^{m}} \} \ , \]
is a torsor for $H^{1}_{\mc P_0,\mc Q_0}(\varphi)$.
\end{theo}

\begin{proof}
We repeat the argument of Proposition \ref{DefObsprop}, keeping track of the additional $\hbar$-dependence coming from the families $\Ph$ and $\Qh$, and the fixed trivializations of their underlying $\Ss$-modules in the sense of Definition \ref{Ssmodtrivdefn}. In particular, via these trivializations one has an identification of $\Kh$-modules
\[
\gPQh \;\cong\; \g_{\mc P_0,\mc Q_0}\otimes_\K \Kh
\]
under which the differential and bracket on $\gPQh$ admit expansions
\[
d_{\gPQh}=\sum_{j\ge 0}\hbar^j d_j
\qquad\text{and}\qquad
[\cdot,\cdot]_{\gPQh}=\sum_{j\ge 0}\hbar^j[\cdot,\cdot]_j
\]
for $d_j$ and $[\cdot,\cdot]_j$ defined on the underlying vector space $\g_{\mc P_0,\mc Q_0}$, as in the description following Equation \ref{gtriveqn}.

Fix $m\ge 1$ and suppose that
\[
\varphi_{\hbar/\hbar^m}=\varphi+\sum_{k=1}^{m-1}\hbar^k\varphi_k
\;\in\;(\gPQh)^{1,(1)}\otimes_{\Kh}\Kh/(\hbar^m)
\]
is a Maurer-Cartan element, that is
\[
d_{\gPQh}(\varphi_{\hbar/\hbar^m})
+\frac12[\varphi_{\hbar/\hbar^m},\varphi_{\hbar/\hbar^m}]_{\gPQh}= 0 \ \in (\gPQh)^{2,(2)}\otimes_{\Kh}\Kh/(\hbar^{m}).
\]
Consider an arbitrary lift
\[
\varphi_{\hbar/\hbar^{m+1}}
=\varphi+\sum_{k=1}^{m}\hbar^k\varphi_k
\ \in (\gPQh)^{1,(1)}\otimes_{\Kh}\Kh/(\hbar^{m+1}),
\]
and expand the Maurer-Cartan equation modulo $\hbar^{m+1}$. The Maurer-Cartan equations for $\varphi_{\hbar/\hbar^{m+1}}$ modulo $\hbar^m$ are precisely the same as those for $\varphi_{\hbar/\hbar^{m}}$, which hold by hypothesis. So it remains to check the vanishing of the coefficient of $\hbar^m$. Using the expansions of $d_{\gPQh}$ and $[\cdot,\cdot]_{\gPQh}$ fixed above and collecting the terms of order $\hbar^m$ gives
\[
d_0\varphi_m+[\varphi,\varphi_m]_0
+\frac12\sum_{k=1}^{m-1}[\varphi_k,\varphi_{m-k}]_0
+\sum_{j=1}^m d_j\varphi_{m-j}
+\sum_{j=1}^m\sum_{k=0}^{m-j}\frac12[\varphi_k,\varphi_{m-j-k}]_j
=0.
\]
By the Definition of $\Obs_m^{\Ph,\Qh}(\varphi_{\hbar/\hbar^m})$ in Equation \ref{Obshbareqn}, this is equivalent to
\begin{equation}\label{defMChbareqn}
	d_\varphi(\varphi_m)=\Obs_m^{\Ph,\Qh}(\varphi_{\hbar/\hbar^m}),
\end{equation}
where $d_\varphi:=d_0+[\varphi,-]_0$ is the differential on the operadic deformation complex
$C^\bullet_{\mc P_0,\mc Q_0}(\varphi)$. Thus, a lift $\varphi_{\hbar/\hbar^{m+1}}=\varphi_{\hbar/\hbar^m}+\hbar^m\varphi_m$ is Maurer-Cartan
modulo $\hbar^{m+1}$ if and only if $d_\varphi\varphi_m=\Obs_m^{\Ph,\Qh}(\varphi_{\hbar/\hbar^m})$, as claimed. In particular, an extension exists if and only if the cochain
$\Obs_m^{\Ph,\Qh}(\varphi_{\hbar/\hbar^m})$ is exact, equivalently
\[
[\Obs_m^{\Ph,\Qh}(\varphi_{\hbar/\hbar^m})]=0\in H^2_{\mc P_0,\mc Q_0}(\varphi).
\]

As in Proposition \ref{DefObsprop}, we use the Bianchi identity to show that $\Obs_m^{\Ph,\Qh}(\varphi_{\hbar/\hbar^m})$ is in general a $d_\varphi$-cocycle. Since $\varphi_{\hbar/\hbar^m}$ satisfies the Maurer-Cartan equation up to order $\hbar^m$, we have the expansion 
\begin{equation}\label{MChbarmpoeqn}
	d_\hbar\varphi_{\hbar/\hbar^{m+1}} \;+\; \frac{1}{2}[\varphi_{\hbar/\hbar^{m+1}},\varphi_{\hbar/\hbar^{m+1}}] \;=\; \hbar^{\,m}\,\Obs_m^{\Ph,\Qh}(\varphi_{\hbar/\hbar^m}) \;+\; \textup{terms of order } \hbar^{\,m+1}\,.
\end{equation}
The Bianchi identity modulo $\hbar^{m+1}$ is given by
\[ 
d_{\varphi_{ \hbar / \hbar^{m+1}}}\!\Big(d_\hbar\varphi_{\hbar/\hbar^{m+1}} \;+\; \frac{1}{2}[\varphi_{\hbar/\hbar^{m+1}},\varphi_{\hbar/\hbar^{m+1}}] \Big) \;=\; 0\,.
\] 
so that applying the differential $d_{\varphi_{ \hbar / \hbar^{m+1}}}$ to both sides of Equation \ref{MChbarmpoeqn}, the vanishing at order $\hbar^m$ is equivalent to
\[ 
d_\varphi\!\Big(\Obs_m^{\Ph,\Qh}(\varphi_{\hbar/\hbar^m})\Big) \;=\; 0\,,
\] 
and thus $\Obs_m^{\Ph,\Qh}(\varphi_{\hbar/\hbar^m})$ is indeed closed with respect to the differential $d_\varphi$.

Finally, suppose $[\Obs_m^{\Ph,\Qh}(\varphi_{\hbar/\hbar^m})]=0$ so that extensions exist. If $\varphi_{\hbar/\hbar^{m+1}}=\varphi_{\hbar/\hbar^m}+\hbar^m\varphi_m$ and $\tilde\varphi_{\hbar/\hbar^{m+1}}=\varphi_{\hbar/\hbar^m}+\hbar^m\tilde \varphi_m$ are two such extensions, then we have
\[
d_\varphi(\tilde\varphi_m-\varphi_m)=0,
\]
so $\delta:=\tilde\varphi_m-\varphi_m$ is a $1$-cocycle. Moreover, $\delta$ is a coboundary, $\delta=d_\varphi(\eta)$, if and only if the two lifts are related by the action of
$\mf G^{0,(0)}_{\Ph,\Qh,\mf m}$ at order $\hbar^m$, where the group $\mf G^{0,(0)}_{\Ph,\Qh,\mf m_R}$ is as in Definition \ref{Defhbardefn}. Thus, isomorphism classes of such extensions again form a torsor under $H^1_{\mc P_0,\mc Q_0}(\varphi)$, as claimed.
\end{proof}

\subsection{Deformation theory for graded algebras}\label{ssec:graded def}

Recall the distinguished multiplicative group $\IG_\hbar$ that acts on $\IA^1_\hbar$ by scaling $\hbar$ with unit weight. Let $\PP_\hbar$ and $\QQ_\hbar$ be graded operads over $\ik[\hbar]$, which we think of as families of operads over the stack $[\IA^1_\hbar/\IG_\hbar]$.

We write $\PP_0\coloneqq \PP_\hbar \otimes_{\ik[\hbar]}\ik_0$ for the central fibre at $0\in [\IA^1/\IG_\hbar]$ and $\PP_1\coloneqq \PP_\hbar\otimes_{\ik[\hbar]}\ik_1$ for the fibre over the closed point $1\in [\IA^1/\IG_\hbar]$.

The fibre operad $\PP_1$ is naturally equipped with an (ascending) filtration $F_\bullet \PP_1$, induced by the grading on $\PP_\hbar$, while the fibre operad $\PP_0$ is a graded operad over $\ik$. These two fibres are related by
\[
\gr^\bullet_{F}\PP_1 \cong \PP_0~.
\]
We have, more or less, just recalled the Rees construction for operads. Namely, any filtered $\ik$-linear operad, $\PP$ defines a graded $\ik[\hbar]$-operad $\PP_\hbar$ via
\[
\PP_\hbar = \Rees_\hbar\, \PP_1 = \bigoplus_{m\in\IZ} F_m \PP_1\hbar^m~,
\]
such that 
\[
\PP_1 \cong \PP \qquad \text{ and } \qquad \PP_0 \cong \gr^\bullet_F \PP~.
\]

In paticular, if $\PP$ is trivially filtered, $F_m \PP = \PP$ for any $m\ge 0$ and zero otherwise, then $\gr_F \PP \cong \PP$ and the associated Rees operad is the constant extension
\[
\PP_\hbar = \PP\otimes \ik[\hbar]~.
\]
Thus, taking $\PP_\hbar$ and $\QQ_\hbar$ to be the Rees operads of trivially filtered operads $\PP$ and $\QQ$, respectively, reduces to the results of section \ref{mapdefsec}.

Once again, let $(S,\mf m_S)$ be a graded local Artinian ring with maximal ideal $\mf m_S$ with a $\IG_\hbar$-equivariant, closed embedding $\Spec S \into \IA^1_\hbar$ sending the closed point of $\Spec S$ to the origin.

\begin{defn}\label{Defhbargrdefn} 
Let $\varphi\in \Alg^\gr_{\mc P_0}(\mc Q_0)$ be a graded $\mc P_0$-algebra internal to $\mc Q_0$. The space of deformations of $\varphi$ over $\Spec S$ is defined the fibre product
\[
\Def_\varphi^{\Ph,\Qh}(S) = \MC_\gr(\gPQh\otimes_{\ik[\hbar]} S) \fibre{\MC_\gr(\gf_{\PP_0,\QQ_0})}\{\varphi\}~.
\]
Similarly the stack of such deformations is defined as the fibre product of Maurer--Cartan stacks
\[
\fDef_\varphi^{\Ph,\Qh}(S) = \fMC_\gr(\gPQh\otimes_{\ik[\hbar]} S) \fibre{\MC_\gr(\gf_{\PP_0,\QQ_0})}\{\varphi\}~,
\]
where the Maurer-Cartan stack is defined as the quotient stack
\[ 
\fMC^\gr(\g_{\Ph,\Qh }\otimes_{\K[\hbar]} S)= \left[ \MC_\gr(\g_{\Ph,\Qh }\otimes_{\K[\hbar]} S)\  \big/\  \mf G^{0,(0),\langle 0\rangle}_{\Ph,\Qh,\mf m_S} \right]  
\]
and the group $\mf G^{0,(0),\langle 0 \rangle}_{\Ph,\Qh,\mf m_S}=\exp(\g^{0,(0),\langle 0\rangle}_{\Ph,\Qh}\otimes_{\K[\hbar]} \mf m_S)$, the formal exponential of the nilpotent Lie algebra $\g^{0,(0),\langle 0 \rangle}_{\Ph,\Qh}\otimes_{\K[\hbar]} \mf m_S$---analogous to the definition of the Maurer-Cartan stack in Equation \ref{MCstackeqn}.
\end{defn}

In particular, the variety and stack of deformations over $\K[\hbar]/\hbar^{m+1}$ are given by
\[
	\Def_{\varphi,\gd}^{\Ph,\Qh}(\K[\hbar]/\hbar^{m+1})  =\MC_\gr(\g_{\Ph,\Qh /\hbar^{m+1}})\fibre{\MC_\gr(\g_{\mc P_0,\mc Q_0})} \{ \varphi \}~,
\]
and
\[
	\fDef_{\varphi,\gd}^{\Ph,\Qh}(\K[\hbar]/\hbar^{m+1})  =\fMC_\gr(\g_{\Ph,\Qh /\hbar^{m+1}})\fibre{\fMC_\gr(\g_{\mc P_0,\mc Q_0})} \{ \varphi \}   ~,
\]
respectively.

\subsubsection{Order by order deformation theory}

Once again we return to the setting of Warning \ref{flatwarn} and assume that the underlying $\Ss$-modules $\PP_\hbar^\ash$ and $\QQ_\hbar$ admit trivializations. Moreover, we assume that these trivializations are $\IG_\hbar$-equivariant.

We also have the graded analogue of the preceding Theorem \ref{DefObsThm}.

\begin{theo}\label{DefObsgrThm} Let $\varphi\in \Alg_{\mc P_0}^\gd(\mc Q_0)$ be a graded $\mc P_0$-algebra internal to $\mc Q_0$ and $\varphi_{\hbar/\hbar^{m}}\in \Def^{\Ph,\Qh}_{\varphi,\gd}(\K[\hbar]/\hbar^{m}) $. Then there exists an extension
	\[ \varphi_{ / \hbar^{m+1}} \  \in \Def^{\Ph,\Qh}_{\varphi,\gd}(\K[\hbar]/\hbar^{m+1}) \times_{\Def^{\Ph,\Qh}_{\varphi,\gd}(\K[\hbar]/\hbar^{m}) } \{\varphi_{\hbar/\hbar^{m}} \}\]
	if and only if the cohomology class of the obstruction vanishes, that is,
	\[ [ \Obs^{\Ph,\Qh}_m(\varphi_{\hbar/\hbar^{m}})]= 0 \ \in H^{2,\langle -m \rangle}_{\mc P_0,\mc Q_0}(\varphi) \ .\]
	Further, when the obstruction class vanishes, the space of such extensions up to isomorphism,
	\[  \fDef^{\Ph,\Qh}_{\varphi,\gd}(\K[\hbar]/\hbar^{m+1}) \times_{\fDef^{\Ph,\Qh}_{\varphi,\gd}(\K[\hbar]/\hbar^{m}) } \{\varphi_{\hbar/\hbar^{m}} \} \ , \]
	is a torsor for $H^{1,\langle -m \rangle}_{\mc P_0,\mc Q_0}(\varphi)$, the component of $H^1_{\mc P_0,\mc Q_0}(\varphi)$ that has degree $-m$.
\end{theo}
\begin{proof}
	This is the graded analogue of Theorem \ref{DefObsThm}. We again work under fixed trivializations of the underlying $\Ss$-modules, in the sense of Definition \ref{Ssmodtrivdefn}, so that we may expand the differential and bracket on $\g_{\Ph,\Qh}$ as $d_\hbar=\sum_{j\ge 0}\hbar^j d_j$ and $[\cdot,\cdot]_\hbar = \sum_{j\ge 0}\hbar^j[\cdot,\cdot]_j$, where the components $d_j$ and $[\cdot,\cdot]_j$ have degree $-j$.
	
	Let $m\ge 1$ and let $\varphi_{\hbar/\hbar^m}\in \Def^{\Ph,\Qh}_{\varphi,\gd}(\Kh/(\hbar^m))$ be a graded $m$th order deformation in the sense of Definition \ref{Defhbargrdefn}. By construction,
	$\varphi_{\hbar/\hbar^m}$ lies in the degree $0$ subspace $(\g_{\Ph,\Qh}\otimes_{\Kh}\Kh/(\hbar^m))^{1,(1),\langle 0\rangle}$,
	and any prospective lift
	\[
	\varphi_{\hbar/\hbar^{m+1}}=\varphi_{\hbar/\hbar^m}+\hbar^m\varphi_m \ \in (\g_{\Ph,\Qh}\otimes_{\Kh}\Kh/(\hbar^{m+1}))^{1,(1),\langle 0\rangle}
	\]
 must again lie in degree $0$. Equivalently, $\varphi_m$ must have degree $-m$, since the factor $\hbar^m$ contributes degree $m$ in our Rees conventions.
	
	Expanding the Maurer-Cartan equation exactly as in the proof of Theorem \ref{DefObsThm} shows that
	$\varphi_{\hbar/\hbar^{m+1}}$ is Maurer-Cartan modulo $\hbar^{m+1}$ if and only if the coefficient
	of $\hbar^m$ vanishes, or equivalently
	\[
	d_\varphi(\varphi_m)=\Obs^{\Ph,\Qh}_m(\varphi_{\hbar/\hbar^m}),
	\]
	with both sides lying in the homogeneous degree $-m$ subspace of $C^2_{\mc P_0,\mc Q_0}(\varphi)$. Thus, an extension exists if and only if
	\[
	[\Obs^{\Ph,\Qh}_m(\varphi_{\hbar/\hbar^m})]=0\in H^{2,\langle -m \rangle}_{\mc P_0,\mc Q_0}(\varphi).
	\]
	
	When the obstruction class vanishes, the same subtraction argument as in Theorem \ref{DefObsThm} shows that two choices of $\varphi_m$ differ by a $1$-cocycle of degree $-m$, and that changing $\varphi_m$ by a degree $-m$ coboundary corresponds precisely to the adjoint action of $\mf G^{0,(0),\langle 0\rangle}_{\Ph,\Qh,(\hbar)}$, as in Definition \ref{Defhbargrdefn}. Hence, the isomorphism classes of graded extensions form a torsor for
	$H^{1,\langle -m\rangle}_{\mc P_0,\mc Q_0}(\varphi)$.
\end{proof}

\subsection{Formal deformations of algebras}\label{ssec:formal def}

The deformation theory of the preceeding sections also makes sense for the complete local ring $\ik\fph$. Write $\PP_{\hat \hbar} = \PP_{\hbar}\otimes_{\ik[\hbar]}\ik\fph$ for the restriction of the family $\PP_\hbar$ to the formal neighborhood of the origin of $\IA^1_\hbar$, and analogously for $\QQ_{\hat \hbar}$.

\begin{prop}
	We have an isomorphism of DG Lie algebras over $\ik\fph$
	\[
	\gf_{\PP_{\hat \hbar} , \QQ_{\hat \hbar}} \cong \gf_{\PP_\hbar,\QQ_\hbar} \widehat{\otimes}_{\ik[\hbar]} \ik\fph~.
	\]
\end{prop}
\begin{proof}
	By construction.
\end{proof}

\begin{defn}\label{Defformaldef}
	The schemes $\Def_{\varphi}(\K\fph)$ and $\Def_{\varphi,\gr}(\ik\fph)$ of deformations and graded deformations, respectively, over $\K\fph$ are defined as
	\[ 
	\Def_{\varphi}(\K\fph) = \MC(\g_{\mc P_{\hat{\hbar}},\mc Q_{\hat{\hbar}}}) \fibre{\MC(\g_{\mc P_0, \mc Q_0})} \{\varphi\} ~, 
	\quad \text{ and } \quad
	\Def_{\varphi,\gr}(\ik\fph) = \MC_\gr(\gf_{\PP_{\hat \hbar},\QQ_{\hat \hbar}})\fibre{\MC_\gr(\gf_{\PP_0,\QQ_0})}\{\varphi\}~,
	\]
	generalizing Definition \ref{Defhbardefn}. Similarly the stacks of deformations, $\fDef_{\varphi}(\K\fph)$, and graded deformations, $\fDef_{\varphi,\gr}(\ik\fph)$, over $\K\fph$ are defined as
	\[ 
	\fDef_{\varphi}(\K\fph) = \fMC(\g_{\mc P_{\hat{\hbar}},\mc Q_{\hat{\hbar}}}) \fibre{\MC(\g_{\mc P_0, \mc Q_0})} \{\varphi\}   \ , 
	\quad \text{ and } \quad 
	\fDef_{\varphi,\gr}(\ik\fph) = \fMC_\gr(\gf_{\PP_{\hat \hbar},\QQ_{\hat \hbar}})\fibre{\fMC_\gr(\gf_{\PP_0,\QQ_0})}\{\varphi\}~.
	\]
	The Maurer--Cartan stacks are defined as the quotients
	\[
	\fMC(\g_{\mc P_{\hat{\hbar}},\mc Q_{\hat{\hbar}}})= \left[\MC(\g_{\mc P_{\hat{\hbar}},\mc Q_{\hat{\hbar}}}) /\mf G_{\mc P,\mc Q,\mf m_{\hat{\hbar}}}^{0,(0)}\right] 
	\]
	and
	\[ 
	\fMC_\gr(\g_{\PP_{\hat \hbar},\QQ_{\hat \hbar} })= \left[ \MC_\gr(\g_{\PP_{\hat \hbar},\QQ_{\hat \hbar} })\  \big/\  \mf G^{0,(0),\langle 0\rangle}_{\Ph,\Qh,\mf m_{\hat \hbar}} \right]  ~,
	\]
	for $\mf m_{\hat{\hbar}}=(\hbar) \subset \K\fph$ and the groups $\mf G^{0,(0)}_{\PP_{\hat \hbar},\QQ_{\hat \hbar},\mf m_{\hat \hbar}}$ and $\mf G^{0,(0),\langle 0 \rangle}_{\PP_{\hat \hbar},\QQ_{\hat \hbar},\mf m_{\hat \hbar}}$ from Definitions \ref{Defhbardefn} and \ref{Defhbargrdefn}.
\end{defn}

The results of Theorems \ref{DefObsThm} and \ref{DefObsgrThm} most naturally provide information about spaces of deformations over $\K\fph$, in the following sense:
\begin{prop}\label{Deflimgrprop} 
There are natural isomorphisms of varieties and stacks
	\[  \Def_{\varphi}^{\Ph,\Qh}(\K\fph)\xrightarrow{\cong }  \lim_{m} \Def_\varphi^{\Ph,\Qh}(\K[\hbar]/\hbar^{m}) \qquad \text{and}\qquad \fDef_{\varphi}^{\Ph,\Qh}(\K\fph) \xrightarrow{\cong } \lim_{m} \fDef_\varphi^{\Ph,\Qh}(\K[\hbar]/\hbar^{m})  \ ,  \]
	as well as their graded variants
	\[  \Def_{\varphi,\gd}^{\Ph,\Qh}(\K\fph) \xrightarrow{\cong } \lim_{m} \Def_{\varphi,\gd}^{\Ph,\Qh}(\K[\hbar]/\hbar^{m}) \qquad \text{and}\qquad \fDef_{\varphi,\gd}^{\Ph,\Qh}(\K\fph) \xrightarrow{\cong } \lim_{m} \fDef_{\varphi,,\gd}^{\Ph,\Qh}(\K[\hbar]/\hbar^{m})  \ .  \]
\end{prop}
\begin{proof}
Given a formal deformation $\varphi_{\hat{\hbar}} \in \Def^{\Ph,\Qh}_\varphi(\K\fph)$, for each $m\ge 1$ we can consider its truncation $\varphi_{\hbar/\hbar^m}:=\varphi_{\hat{\hbar}} \otimes_{\K\fph} \K[\hbar]/\hbar^m$. Each $\varphi_{\hbar/\hbar^m}$ lies in $\Def^{\Ph,\Qh}_\varphi(\K[\hbar]/\hbar^m)$, and these truncations are clearly compatible as $m$ varies. This defines a natural map 
\[ \Def^{\Ph,\Qh}_\varphi(\K\fph)\to \lim_{m}\Def^{\Ph,\Qh}_\varphi(\K[\hbar]/\hbar^m) ,\] 
sending $\varphi_{\hat{\hbar}}$ to the projective system $\{\varphi_{\hbar/\hbar^m}\}_{m\ge1}$. This map is injective (if two formal deformations agree to all finite orders, they are equal as formal power series) and it remains to show it is surjective.

For surjectivity, let $\{\varphi_{\hbar/\hbar^m}\}_{m\ge1}$ be a compatible family with $\varphi_{\hbar/\hbar^{m+1}}\in \Def^{\Ph,\Qh}_\varphi(\K[\hbar]/\hbar^m)$ and $\varphi_{\hbar/\hbar^{m+1}}\equiv \varphi_{\hbar/\hbar^{m}} \pmod{\hbar^{m}}$ for each $m$. By construction, in terms of the fixed trivializations of the underlying $\Ss$-modules of $\Ph$ and $\Qh$, we can write 
\[ \varphi_{\hbar/\hbar^m} \;=\; \varphi + \hbar\,\varphi_1 + \hbar^2\,\varphi_2 + \cdots + \hbar^{m-1}\,\varphi_{m-1} \ ,\] 
for some elements $\varphi_i\in \g_{\mc P_0,\mc Q_0}$ (in particular $\varphi_0=\varphi$ is the original $\mc P_0$-algebra structure) and similarly
\[ \varphi_{\hbar/\hbar^{m+1}} \;=\; \varphi + \hbar\,\varphi_1 + \hbar^2\,\varphi_2 + \cdots + \hbar^{m-1}\,\varphi_{m-1}+\hbar^m\varphi_m \ ,\] 
for some $\varphi_m\in \g_{\mc P_0,\mc Q_0}$ satisfying
\[ d_\varphi \varphi_m \;=\; \Obs_m^{\Ph,\Qh}\!\big(\varphi_{\hbar/\hbar^m}\big) \ . \] 
In particular, we obtain an infinite sequence $\{\varphi_i\}_{i\geq 0}$ with $\varphi_0=\varphi$ and $d_\varphi \varphi_m = \Obs_m(\varphi_{\hbar/\hbar^m})$ for each $m\geq 1$. The resulting formal power series 
\[ \varphi_{\hat{\hbar}} \;:=\; \varphi + \hbar\,\varphi_1 + \hbar^2\,\varphi_2 + \hbar^3\,\varphi_3 + \cdots  \] 
defines an element of $\g_{\Ph,\Qh}\otimes_{\K[\hbar]}\K\fph$. By construction, this $\varphi_{\hat{\hbar}}$ is a solution to the Maurer-Cartan equation \emph{to arbitrary order} in $\hbar$, hence in fact 
\[ d_{\g_{\Ph,\Qh}}(\varphi_{\hat{\hbar}}) + \frac{1}{2}\big[\varphi_{\hat{\hbar}},\varphi_{\hat{\hbar}}\big]_{\g_{\Ph,\Qh}} = 0 \] 
as a formal power series. In other words, $\varphi_{\hat{\hbar}}\in \Def^{\Ph,\Qh}_\varphi(\K\fph)$ is a formal deformation whose truncation at each order $m$ is precisely $\varphi_{\hbar/\hbar^m}$. Thus, the map $\lim_m \Def^{\Ph,\Qh}_\varphi(\K\fph) \to \Def^{\Ph,\Qh}_\varphi(\K[\hbar]/\hbar^m)$ is also surjective, and thus an isomorphism.

Finally, the isomorphism of deformation stacks follows by passing to the quotient by gauge equivalences. Indeed, the pro-nilpotence of the automorphism DG Lie algebra $\g_{\mc P_0,\mc Q_0}^{0,(0)}$ ensures that the construction of the gauge group $\mf G^{0,(0)}_{\Ph,\Qh,\mf m_R}=\exp(\g^{0,(0)}_{\Ph,\Qh}\otimes \mf m_R)$ commutes with filtered inverse limits, so that we have 
\[ 
\lim_{m}\mf G^{0,(0)}_{\Ph,\Qh,\mf m_{(m)}} \;\cong\; \mf G^{0,(0)}_{\Ph,\Qh,\mf m_{\hat \hbar}} 
\] 
for the maximal ideals $\mf m_{(m)}$ of $\Kfph/\hbar^m$ and $\mf m_{\hat \hbar}=(\hbar)$ of $\Kfph$. Thus, the above bijection on Maurer--Cartan solutions descends to an isomorphism of the relevant quotient stacks. This yields the stated isomorphism $\lim_{m}\fDef_\varphi^{\Ph,\Qh}(\K[\hbar]/\hbar^m) \cong \fDef_\varphi^{\Ph,\Qh}(\K\fph)$.

The graded case follows from the same argument, \emph{mutatis mutandis}.
\end{proof}

In particular, we obtain the following key corollaries about spaces of formal (graded) deformations:

\begin{coro}\label{forexcoro} Let $\varphi\in \Alg_{\mc P_0}(\mc Q_0)$ a $\mc P_0$-algebra internal to $\mc Q_0$ such that $H^{2}_{\mc P_0,\mc Q_0}(\varphi)=\{0\}$. Then there exists a formal deformation $\varphi_{\hat{\hbar}}\in \Def_{\varphi}^{\Ph,\Qh}(\K\fph)$.
	
\end{coro}
\begin{proof}
For each $m\ge 1$, consider the inductive extension problem
\[
\Def^{\Ph,\Qh}_\varphi(\Kh/(\hbar^{m+1}))\times_{\Def^{\Ph,\Qh}_\varphi(\Kh/(\hbar^{m}))}\{\varphi_{\hbar/\hbar^m}\} \ .
\]
By Theorem \ref{DefObsThm}, the obstruction to extending from order $m-1$ to order $m$ is the class of $\Obs_m^{\Ph,\Qh}(\varphi_{\hbar/\hbar^m})$ in $H^2_{\mc P_0,\mc Q_0}(\varphi)$.
Since $H^2_{\mc P_0,\mc Q_0}(\varphi)=0$, every such obstruction class vanishes, so extensions exist at every order. Starting from $\varphi_{\hbar/\hbar}=\varphi$, we obtain a compatible system of deformations $\{\varphi_{\hbar/\hbar^m}\}_{m\ge 1}$. Passing to the limit and using Proposition \ref{Deflimgrprop}, we obtain an element
\[
\varphi_{\hat{\hbar}}\in \Def_{\varphi}^{\Ph,\Qh}(\K\fph) \ ,
\]
as claimed.
\end{proof}

\begin{coro}\label{foruncoro} 
Let $\varphi\in \Alg_{\mc P_0}(\mc Q_0)$ be a $\mc P_0$-algebra internal to $\mc Q_0$ such that $H^{1}_{\mc P_0,\mc Q_0}(\varphi)=\{0\}$. Then if there exists a formal deformation $\varphi_{\hat{\hbar}}\in \Def_{\varphi}^{\Ph,\Qh}(\K\fph)$, it is unique up to isomorphism.
\end{coro}
\begin{proof}
Fix $m\ge 1$ and an $m^{th}$-order deformation $\varphi_{\hbar/\hbar^m}$. Given $\varphi_{\hbar/\hbar^{m+1}}=\varphi_{\hbar/\hbar^m}+\hbar^m\varphi_m$ and $\varphi'_{\hbar/\hbar^{m+1}}=\varphi_{\hbar/\hbar^m}+\hbar^m\varphi'_m$ are two $(m+1)^{th}$-order extensions, by the torsor statement in Theorem \ref{DefObsThm} their isomorphism classes differ by an element of $H^1_{\mc P_0,\mc Q_0}(\varphi)$. Under the assumption $H^1_{\mc P_0,\mc Q_0}(\varphi)=0$, this torsor is trivial, so there is at most one isomorphism class of extension at each order. By induction on $m$, any two compatible inverse systems of truncations $\{\varphi_{\hbar/\hbar^m}\}$ are isomorphic at every finite stage, hence define the same point of the limit stack. By Proposition \ref{Deflimgrprop}, the resulting formal deformation over $\K\fph$ is therefore unique up to isomorphism.
\end{proof}

\begin{coro}\label{forexgrcoro} 
Let $\varphi\in \Alg_{\mc P_0}^\gd(\mc Q_0)$ a graded $\mc P_0$-algebra internal to $\mc Q_0$ such that $H^{2,\langle -m \rangle}_{\mc P_0,\mc Q_0}(\varphi)=\{0\}$ for all $m$. 

Then, there exists a formal, graded deformation $\varphi_{\hat{\hbar}}\in \Def_{\varphi,\gd}^{\Ph,\Qh}(\K\fph)$.
\end{coro}
\begin{proof}
This is proved as in Corollary \ref{forexcoro}, using the graded extension criterion of Theorem \ref{DefObsgrThm}. Namely, at step $m$ the obstruction lies in $H^{2,\langle -m\rangle}_{\mc P_0,\mc Q_0}(\varphi)$, which vanishes by hypothesis, hence graded extensions exist for all
$m$. Taking the inverse limit and applying the graded limit statement of Proposition \ref{Deflimgrprop} gives a formal graded deformation over $\K\fph$.
\end{proof}

\begin{coro}\label{forungrcoro} 
Let $\varphi\in \Alg_{\mc P_0}^\gd(\mc Q_0)$ be a graded $\mc P_0$-algebra internal to $\mc Q_0$ such that $H^{1,\langle -m\rangle}_{\mc P_0,\mc Q_0}(\varphi)=\{0\}$ for any $m\ge1$.

Then if there exists a formal, graded deformation $\varphi_{\hat{\hbar}}\in \Def_{\varphi,\gd}^{\Ph,\Qh}(\K\fph)$, it is unique up to isomorphism.
	
\end{coro}
\begin{proof}
Fix $m\ge 1$. By the torsor statement in Theorem \ref{DefObsgrThm}, isomorphism classes of graded extensions from order $m$ to $m+1$ form a torsor for $H^{1,\langle m\rangle}_{\mc P_0,\mc Q_0}(\varphi)$, which is trivial by hypothesis. Thus graded extensions are unique up to isomorphism at each stage. Passing to
the inverse limit and using Proposition \ref{Deflimgrprop} (graded form) yields uniqueness of the resulting formal graded deformation over $\K\fph$.
\end{proof}

\begin{coro}\label{grtorcoro} 
Let $\varphi\in \Alg_{\mc P_0}^\gd(\mc Q_0)$ a graded $\mc P_0$-algebra internal to $\mc Q_0$ such that \mbox{$H^{1,\langle -m \rangle}_{\mc P_0,\mc Q_0}(\varphi)=\{0\}$} for all $m\neq m_0$, with $m_0 \in \IN$. 

Then if there exists a formal, graded deformation $\varphi_{\hat{\hbar}}\in \Def_{\varphi,\gd}^{\Ph,\Qh}(\K\fph)$, the stack of formal, graded deformations $\fDef_{\varphi,\gd}^{\Ph,\Qh}(\K\fph)  $ is a torsor for $H^{1,\langle -m_0 \rangle}_{\mc P_0,\mc Q_0}(\varphi)$.
\end{coro}
\begin{proof}
	Assume a formal graded deformation exists. By Theorem \ref{DefObsgrThm}, at the $m$th extension step the set of isomorphism classes of graded extensions is a torsor for $H^{1,\langle -m\rangle}(\varphi)$. By hypothesis, this torsor is trivial for $m\neq m_0$, and potentially nontrivial only at $m=m_0$. Consequently the only degree of freedom in the inverse system of truncations occurs in degree $-m_0$, and it is governed precisely by the action of $H^{1,\langle -m_0\rangle}(\varphi)$. 

	Passing to the limit via Proposition \ref{Deflimgrprop}, we conclude that $\fDef_{\varphi,\gd}^{\Ph,\Qh}(\K\fph)$ is a torsor for $H^{1,\langle - m_0\rangle}(\varphi)$.
\end{proof}

\section{Quantizations via deformation theory}\label{quantsec}

We will now apply the results of the previous section to the problem of deformation quantization. First, we revisit the traditional notion of deformation quantization of Poisson algebras, and outline the application of operadic deformation theory in this setting to recover some well-known results about their spaces of quantizations. In the following sections, we apply these results to the deformation theory of coisson algebras to give analogous descriptions of their spaces of quantizations.

\subsection{Quantization of Poisson algebras}\label{Poisquantsec}

%
%
%

\subsubsection{}\label{clfiltsec} 
There is a natural increasing filtration $F_\bullet^\cl\Ass$ on the associative operad, defined by
\[ 
F_{-2}^\cl \Ass(2) = \{0\} \quad\subset\quad F_{-1}^\cl\Ass(2) =\Lie(2)  \quad\subset\quad F_0^\cl\Ass(2) = \Ass(2)~,
\]
\ie, the standard inclusion of the commutator into the space of bilinear maps on 2-ary operations. Using the fact that the filtration on 2-ary operations is naturally split by taking anti-symmetric and symmetric components (given by commutator and anti-commutator) and that these generate $\Ass(n)$ under composition, we define the filtration $F^\cl_\bullet\Ass(n)$ by
\[  
F^\cl_{-k}\Ass(n) = \{ \mu \in \Ass(n) \ | \ \text{every binary tree presentation of $\mu$ has at least $k$ commutators}\} \ . 
\]
The following is a standard result, following essentially from the presentation of the Rees operad below.
\begin{prop}\label{finclfiltprop} The associated graded operad $\gr_\bullet^{F^\cl}(\Ass)$ is isomorphic to the Poisson operad,
	\[  
	\Pois \xrightarrow{\cong } \gr_\bullet^{F^\cl}(\Ass) \ \in\Op(\Vect_\K^\gr)~,
	\]
	as graded operads, where the Poisson operad is equipped with the auxiliary grading from \ref{Poisopdefn}.
\end{prop}

We define
\begin{equation}\label{eq:BD dfn}
\BDh = \Rees_\hbar^{F^\cl_\bullet}(\Ass) = \bigoplus_{m\in \Z} F_m^\cl \Ass \cdot \hbar^m ~,
\end{equation}
The Rees operad, $\BDh$, is known as the Beilinson--Drinfeld operad, see \cite{Melani2018} and \cite{CosGw}.

We write 
\[
\BDhat \coloneqq \BDh\widehat{\otimes}_{\ik[\hbar]}\ik\fph~,
\]
for its completion, which is a $\ik\fph$-operad.

Let $A_{\hat \hbar}\in \ik\fph\Mod$, then we write $\OEnd_{\hat \hbar}(A_{\hat \hbar})$ for the endomorphism operad of $A_{\hat \hbar}$ internal to the symmetric monoidal category $\ik\fph\Mod$
\begin{rmk}\label{rmk:BD-algebras concretely}
	The data of $\varphi_\hbar \in \Alg_{\BDhat}(\OEnd_{\hat \hbar}(A_{\hat \hbar})$ (or its graded counterpart) corresponds to a (graded) associative multiplication
	\[
	m_{\hbar}:A_{\hat \hbar} \otimes_{\ik\fph}A_{\hat \hbar} \rightarrow A_{\hat \hbar}~,
	\]
	and a (graded) Lie bracket 
	\[
	[\cdot,\cdot]_{\hbar}:A_{\hat \hbar}\otimes_{\ik\fph}A_{\hat \hbar} \rightarrow A_{\hat \hbar}~,
	\]
	that is a biderivation of $m$ and satisfies $\hbar[a,b]_{\hbar} = m_\hbar(a,b)-m_\hbar(b,a)$ for any $a,b\in A$. In the limit $\hbar=0$, $m_0:A_0\otimes A_0\rightarrow A_0$ is commutative and so $(m_0,[\cdot,\cdot]_{0})$ defines a (graded) Poisson algebra structure on the central fibre $A_0$.
\end{rmk}

\subsubsection{}

In light of Remark \ref{rmk:BD-algebras concretely} we make the following definition.
\begin{defn}
	Given a Poisson algebra $(A,\varphi)$, a \emph{quantization}, $(A_{\hat \hbar} \varphi_{\hat \hbar})$, is a
	\begin{enumerate}[(i)]
		\item flat $\ik\fph$-module $A_{\hat \hbar}$, complete in the $\hbar$-adic topology, with an isomorphism of the central fibre $A_0=A\otimes_{\ik\fph}\ik_0\xrightarrow{\sim} A$
		\item $\varphi_{\hat \hbar}\in \Alg_{\BDhat}(\OEnd_{\hat \hbar}(V_{\hat \hbar}))$
	\end{enumerate}
	such that the central fibre Poisson algebra $(A_0,\varphi_0)$  is identified with the Poisson algebra $(A,\varphi)$.

	If $(A,\varphi)$ is a graded Poisson algebra, we say that a quantization $(A_{\hat \hbar},\varphi_{\hat \hbar})$ is a \emph{graded quantization} if $\varphi_{\hat \hbar} \in \Alg^\gd_{\BDhat}(\OEnd_{\hat \hbar}(A_{\hat \hbar})$ and the isomorphism $(A_0,\varphi_0)\cong (A,\varphi)$ is an isomorphism of graded Poisson algebras.
\end{defn}

Now, take $A_{\hat \hbar}$ to be the trivial family $A\widehat{\otimes}\ik\fph$. Given a $\varphi\in \Alg_{\Pois}(\End(A))$, consider a deformation 
\[
\vph\in \Def_{\varphi}^{\BDhat,\OEnd_\hhb(A_\hhb)}(\varphi)~.
\]
By construction, $A_\hbar$ is a flat, $\hbar$-adic $\ik\fph$-module with central fibre equal to $A$ and $\vph\in \Alg_{\BDhat}(\OEnd_\hhb(A_\hhb))$ is such that $\varphi_0 = \varphi\in \Alg_{\Pois}(\OEnd(A))$. Thus, $(A_\hhb,\vph)$ is a quantization of $(A,\varphi)$. Indeed, the scheme of deformations parameterises all quantizations of $(A,\varphi)$ whose underlying $\ik\fph$-module is the trivial family.

Let $S$ be the Artinian local ring, $ \ik[\hbar]/\hbar^m$, or the complete local ring $\ik\fph$.
Therefore, we make the following definitions:
\begin{equation}
\Quant_\varphi(S)  =\Def_\varphi^{\BDhat,\OEnd_{\hat \hbar}(A_{\hat\hbar})}(S)~,\qquad \text{and} \qquad    \Quant_{\varphi,\gd}(S)  =\Def_{\varphi,\gd}^{\BDhat,\OEnd_{\hat \hbar}(A_{\hat\hbar})}(S)~,
\end{equation}
for the schemes of quantizations and graded quantizations, with the stipulation that the underlying $\ik\fph$-module is $A\widehat{\otimes}\ik\fph$.

We define the stacks of quantizations and graded quantizations as 
\begin{equation}
	\fQuant_\varphi(S)  =\fDef_\varphi^{\BDhat,\OEnd_{\hat \hbar}(A_{\hat\hbar})}(S)    ~,\qquad \text{ and } \qquad \fQuant_{\varphi,\gd}(S)  =\fDef_{\varphi,\gd}^{\BDhat,\OEnd_{\hat \hbar}(A_{\hat\hbar})}(S)   .
\end{equation}

\begin{rmk}\label{rmk:stack of finite quant is good}
The stacks $\fQuant_\varphi(S)$ (and its graded counterpart) actually parameterise the isomorphism classes of all (graded) quantizations of $(A,\varphi)$. This is because the underlying $\ik\fph$-module of every quantization is isomorphic to $A\widehat{\otimes} \ik\fph$ and so every quantization of $(A,\varphi)$ has a representative whose underlying $\ik\fph$-module is the trivial family.

One may think of our scheme of quantizations as a kind of partial gauge fixing of an \textit{a priori} larger, scheme of quantizations, which in addition parameterizes the data of the central fibre isomorphism.
\end{rmk}

\begin{rmk}
	Unlike the other source operads we have considered, the $\ik\fph$-operad $\BDh$ is not quadratic and, therefore, not Koszul. To apply the results of Section \ref{chsec}, we need to resolve $\BDh$ by a DG, Koszul operad $\BDh^{\DG}$. Since quantization of Poisson algebras is not the focus of our work we will not address this complication presently, but point the interested reader to \cite{Melani2018,CosGw} for more details.
\end{rmk}

\subsubsection{}
Applying our order by order deformation theory results, we have:
\begin{theo}\label{DefObsQuantThm} 
Let $\varphi\in \Alg_{\Pois}(\OEnd(A))$ a Poisson algebra and $\varphi_{\hbar/\hbar^{m}}\in\Quant_\varphi(\K[\hbar]/\hbar^{m}) $ a quantization defined modulo $\hbar^m$. Then there exists an extension
	\[ \varphi_{ \hbar / \hbar^{m+1}} \  \in \Quant_\varphi(\K[\hbar]/\hbar^{m+1}) \times_{\Quant_\varphi(\K[\hbar]/\hbar^{m}) } \{\varphi_{\hbar/\hbar^{m}} \}\]
	if and only if the cohomology class of the obstruction vanishes, that is,
	\[ [ \Obs^{\BDh,\OEnd(A_\hbar)}_m(\varphi_{\hbar/\hbar^{m}})]= 0 \ \in H^{2}_{\Pois}(A,\varphi) \ .\]
	Further, when the obstruction class vanishes, the space of such extensions up to isomorphism,
	\[  \fQuant_\varphi(\K[\hbar]/\hbar^{m+1}) \times_{\fQuant_\varphi(\K[\hbar]/\hbar^{m}) } \{\varphi_{\hbar/\hbar^{m}} \} \ , \]
	is a torsor for $H^{1}_{\Pois}(A,\varphi)$.
\end{theo}

\begin{coro} Let $\varphi\in \Alg_{\Pois}(\OEnd(A))$ a Poisson algebra such that $H^{2}_{\Pois}(A,\varphi)=\{0\}$. Then there exists a quantization $\varphi_{\hat{\hbar}}\in \Quant_\varphi(\K\fph)$.
\end{coro}
\begin{proof}
	This follows from applying Corollary \ref{forexcoro} to the present setting.
\end{proof}

We will say that a Poisson algebra $A$ is symplectic if the corresponding affine variety $\Spec A$ is symplectic, and in particular smooth. We have the following important corollary:

\begin{theo}\label{QuantdRthm}
Let $\varphi\in \Alg_{\Pois}(\End(A))$ be symplectic. Let $\varphi_{\hbar/\hbar^m}\in \Quant_\varphi(\ik[\hbar]/\hbar^m)$ be a quantization modulo $\hbar^m$. If the obstruction class vanishes, then the space of extensions up to isomorphism,
\[  
	\fQuant_\varphi(\K[\hbar]/\hbar^{m+1}) \times_{\fQuant_\varphi(\K[\hbar]/\hbar^{m}) } \{\varphi_{\hbar/\hbar^{m}} \} \ , 
\]
is a torsor for $H^2_{\dR}(\Spec A)$, the second de Rham cohomology of $\Spec A$.
\end{theo}
\begin{proof}
 The additional content of this Theorem is the identification of Poisson cohomology and de Rham cohomology, which can be understood structurally as follows: a Poisson structure on $A$ determines a Lie algebroid structure on $\Omega^1_A$ such that the anchor map $\alpha:\Omega^1_A\to \Theta_A$ is given by contraction with the Poisson tensor. Moreover, the (reduced) Chevalley--Eilenberg cochains on this Lie algebroid, with coefficients in the module $A$, recovers the (reduced) geometric Poisson cochains on $A$ in the sense of Definition \ref{gPCdefn}, by Corollary \ref{gPCcoro}.

 In the symplectic case, the (reduced) geometric Poisson cochains on $A$ is quasi-isomorphic to the operadic Poisson cochains which feature in Theorem \ref{DefObsQuantThm}, by Corollary \ref{GOequivcoro}, since $A$ is smooth. On the other hand, the symplectic hypothesis implies that the anchor map $\alpha$ defines an isomorphism of Lie algebroids to $\Theta_A$, and the (reduced) Chevalley--Eilenberg cochains on $\Theta_A$ with coefficients in $A$ is canonically identified with the (truncated) de Rham complex.

\end{proof}

This result is, of course, evocative of the classical results on quantization of symplectic manifolds, see \cite{LecomteWilde83, Fedosov1994, Kontsevich1997} as well as \cite{BeK} for the algebraic setting. In our context, their main result implies that
\begin{equation}
	\fQuant_\varphi(\ik\fph) \cong \hbar H^2_{\dR}(\Spec{A})\fph
\end{equation}
with an explicit isomorphism given by the noncommutative period map of \cite{BeK}.

\begin{rmk}
From the perspective of this work, the stronger result of \cite{BeK}, implies that there exist compatible trivializations of the family of torsors constructed at each order by Theorem \ref{QuantdRthm}.
\end{rmk}


\subsection{The filtration on the chiral operad}


\subsubsection{}\label{specialsec} 

We now describe the analogue in the chiral setting of the classical filtration on the associative operad introduced in \ref{clfiltsec}, following section 3.2 of \cite{BD1}. To do so, it will be convenient to return to the convention of labelling the arity of operations in the operad by $I\in\fset$.

Given a $T\in \fset$ such that $I\onto T$, one can define a diagonal embedding
\[
\Delta^{(I/T)}:X^T\rightarrow X^I~,
\]
where the embedding is specified by $I\onto T$.

Recall from 3.1.6 of \emph{loc. cit.} that a D-module $M\in \DD(X^I)$ is called \emph{special} if it admits a finite, increasing filtration $F_\bullet^\spe M$ such that each of the subquotients $\gr^\ell_{F^\spe} M$ is isomorphic to a finite direct sum of objects of the form $\Delta^{(I/T)}_* \omega_{X^T}$ for $I\onto T$ a finite quotient of $I$ of cardinality $-\ell$. In particular, the graded degree $\ell$ component of the associated graded has support of dimension $-\ell$, for each $-|I|\leq \ell \leq 0$, with $F_{-|I|-1}^\spe M=0$ and $F_\ell^\spe M=M$ for $\ell >0$.

In particular, we have
\begin{prop} The object $j^{(I)}_*j^{(I)*}\omega_{X}^{\boxtimes I} \in \DD(X^I)$ is special.
\end{prop}
\begin{proof}  This is Lemma 3.1.7 in \cite{BD1}.
\end{proof}

We let $F_\bullet^\spe j^{(I)}_*j^{(I)*}\omega_{X}^{\boxtimes I} $ denote the associated special filtration; in particular, we have
\begin{equation*}
F_{-1}^\spe j^{(I)}_*j^{(I)*}\omega_{X}^{\boxtimes I} =j^{(I)}_*j^{(I)*}\omega_{X}^{\boxtimes I}  \qquad\text{and}\qquad 	F_{-|I|}^\spe j^{(I)}_*j^{(I)*}\omega_{X}^{\boxtimes I}  = \omega_X^{\boxtimes I}   \ ,
\end{equation*}
so that the corresponding top and bottom degrees of the associated graded are given by
\begin{equation}\label{Fspegreqn}
	 \gr^{-1}_{F^\spe}j^{(I)}_*j^{(I)*}\omega_{X}^{\boxtimes I} = \Delta_*^{(I)} \omega_X \otimes \Lie(|I|)   \qquad\text{and}\qquad \gr^{-|I|}_{F^\spe}j^{(I)}_*j^{(I)*}\omega_{X}^{\boxtimes I}= M^{\boxtimes I}  \ .
\end{equation}
Note that this induces a canonical bounded, increasing filtration $F_\bullet^\spe P^\ch(M)$ on the chiral endomorphism operad $P^\ch(M)=\End_{\DD(X)^\ch}(M)$, defined by
\begin{align*}
	 F_k^\spe P^\ch(M)(I)  & \equiv  F_k^\spe \Hom_{\DD(X^I)}(j^{(I)}_*j^{(I) *}M^{\boxtimes I}, \Delta^{(I)}_*M )  \\ & := \Hom_{\DD(X^I)}((j^{(I)}_*j^{(I)*}\omega_{X}^{\boxtimes I}\otimes^! M^{\boxtimes I} )\big/ ((F_{-k-2}^\spe j^{(I)}_*j^{(I)*}\omega_{X}^{\boxtimes I}) \otimes^! M^{\boxtimes I}), \Delta^{(I)}_*M )  & ,
\end{align*}
where we have used the natural identification
\[ j^{(I)}_*j^{(I) *}M^{\boxtimes I} \cong (j^{(I)}_*j^{(I)*}\omega_{X}^{\boxtimes I}) \otimes^! M^{\boxtimes I}  \ . \]
Note that we have shifted the degrees of the filtration relative to the definition of Equation 3.2.4.2 of \cite{BD1}, so that the induced grading on the associated grading is the auxiliary grading introduced in Equation \ref{auxgreqn}, rather than that of \emph{loc. cit.}; this ensures that the binary commutative product is in degree 0 while the binary Poisson bracket operation is in degree $-1$, in analogy with the finite type case.

In particular, in our conventions we have the following Proposition, which is stated in \cite{BD1} in the alternative grading in the paragraph following Equation 3.2.4.2:
\begin{prop} Suppose $M\in \DD(X)$ is projective. There are natural isomorphisms
	\[  \gr_0^{F^\spe_\bullet} P^\ch(M)\xrightarrow{\cong} P^!(M)\otimes\Lie   \qquad \text{and} \qquad \gr_{|I|-1}^{F^\spe_\bullet} P^\ch(M) \xrightarrow{\cong}   P^*(M)\]
such that the induced inclusion and projection
\[   P^!(M)\otimes\Lie \cong \gr_0^{F^\spe_\bullet} P^\ch(M) \to P^\ch(M) \qquad \text{and} \qquad  P^\ch(M) \to \gr_{|I|-1}^{F^\spe_\bullet} P^\ch(M)\cong P^*(M)  \]
identify with the maps of colored operads $\alpha$ and $\beta$ of Equation \ref{alphabetaeqn}, respectively, restricted to the full, single-colored suboperad on the object $M$.
\end{prop}
\begin{proof}We supply a proof to normalize the grading conventions: First, note that for $k=-1$ we have $-k-2=-1$ and $F_{-1}^\spe j^{(I)}_*j^{(I)*}\omega_{X}^{\boxtimes I} =j^{(I)}_*j^{(I)*}\omega_{X}^{\boxtimes I}$ so that $F_{-1}^\spe P^\ch(M)(I) =0$ and thus
\[	\gr_0^{F_\bullet^\spe}P^\ch(M)(I) = F_0^\spe P^\ch(M)(I) =\Hom_{\DD(X^I)}((j^{(I)}_*j^{(I)*}\omega_{X}^{\boxtimes I}\otimes^! M^{\boxtimes I} )\big/ ((F_{-2}^\spe j^{(I)}_*j^{(I)*}\omega_{X}^{\boxtimes I}) \otimes^! M^{\boxtimes I}), \Delta^{(I)}_*M ) \ . \]
Now, note that again since $F_{-1}^\spe j^{(I)}_*j^{(I)*}\omega_{X}^{\boxtimes I} =j^{(I)}_*j^{(I)*}\omega_{X}^{\boxtimes I}$ we have
\begin{align*}
	 (j^{(I)}_*j^{(I)*}\omega_{X}^{\boxtimes I}\otimes^! M^{\boxtimes I} )\big/ ((F_{-2}^\spe j^{(I)}_*j^{(I)*}\omega_{X}^{\boxtimes I}) \otimes^! M^{\boxtimes I})  & = ((F_{-1}^\spe j^{(I)}_*j^{(I)*}\omega_{X}^{\boxtimes I}) \otimes^! M^{\boxtimes I})/(F_{-2}^\spe j^{(I)}_*j^{(I)*}\omega_{X}^{\boxtimes I}) \otimes^! M^{\boxtimes I})  \\
	 & = (\gr_{-1}^{F^\spe}j^{(I)}_*j^{(I)*}\omega_{X}^{\boxtimes I}) \otimes^! M^{\boxtimes I}) \\
	 &  \cong (\Delta_*^{(I)} \omega_X)\otimes^! M^{\boxtimes I}  \otimes_{ \K} \Lie_\K(I) \\
	 & \cong \Delta_*^{(I)}  (M^{\otimes^! I})\otimes_{ \K} \Lie_\K(I)
\end{align*} where the first isomorphism is that of Equation \ref{Fspegreqn} and the second follows by Kashiwara's lemma. Thus, we obtain that $\gr_0^{F_\bullet^\spe}P^\ch(M)(I) = F_0^\spe P^\ch(M)(I)$ is given by
\begin{align*}
	 F^0P^\ch(M)(I) & =\Hom_{\DD(X^I)}((j^{(I)}_*j^{(I)*}\omega_{X}^{\boxtimes I}\otimes^! M^{\boxtimes I} )\big/ ((F_{-2}^\spe j^{(I)}_*j^{(I)*}\omega_{X}^{\boxtimes I}) \otimes^! M^{\boxtimes I}), \Delta^{(I)}_*M )  \\
	 & \cong \Hom_{\DD(X^I)}( \Delta_*^{(I)}  (M^{\otimes^! I}), \Delta^{(I)}_*M )\otimes_{ \K} \Lie_\K(I) \\
	 & \cong \Hom_{\DD(X)}(M^{\otimes^! I},M)\otimes_{ \K} \Lie_\K(I) \\
	& = (P^!(M)\otimes \Lie)(I) & ,
\end{align*}
as desired, where the first isomorphism follows from the preceding identification of the source objects of the Hom space, and the latter follows again by Kashiwara's lemma. It follows from Theorem 3.1.5 of \cite{BD1} that the operadic composition in this suboperad of $P^\ch(M)$ agrees with that of $P^!(M)\otimes \Lie$ under the above identifications, and that this agrees with the definition of the map $\alpha$.

For $k=|I|-1$ we have $-k-2=-|I|-1$ and $F_{-|I|-1}^\spe=0$ so that $F_{|I|-1}^\spe P^\ch(M) = P^\ch(M)$, while for $k=|I|-2$ we have $-k-2=-|I|$ and $F^\spe_{-|I|}  j^{(I)}_*j^{(I)*}\omega_X^{\boxtimes I} = \omega_{X}^{\boxtimes I}$ so that
\begin{align*}
	 F_{|I|-2}^\spe P^\ch(M) & = \Hom_{\DD(X^I)}((j^{(I)}_*j^{(I)*}\omega_{X}^{\boxtimes I}\otimes^! M^{\boxtimes I} )\big/ ((F_{-|I|}^\spe j^{(I)}_*j^{(I)*}\omega_{X}^{\boxtimes I}) \otimes^! M^{\boxtimes I}), \Delta^{(I)}_*M ) \\
	 &  = \Hom_{\DD(X^I)}(j^{(I)}_*j^{(I)*}M^{\boxtimes I} )\big/ M^{\boxtimes I}, \Delta^{(I)}_*M ) 
\end{align*}
and thus the short exact sequence
\[ M^{\boxtimes I}  \to j^{(I)}_*j^{(I)*}M^{\boxtimes I} \to (j^{(I)}_*j^{(I)*}M^{\boxtimes I} )\big/ M^{\boxtimes I}  \]
induces, under application of the functor $\Hom_{\DD(X^I)}(\cdot, \Delta_*^{(I)}M)$, the left exact sequence
\[\Hom_{\DD(X^I)}((j^{(I)}_*j^{(I)*}M^{\boxtimes I} )\big/ M^{\boxtimes I} , \Delta_*^{(I)}M) \to \Hom_{\DD(X^I)}(j^{(I)}_*j^{(I)*}M^{\boxtimes I} , \Delta_*^{(I)}M) \to \Hom_{\DD(X^I)}(M^{\boxtimes I} , \Delta_*^{(I)}M) \]
which is exact as long as $M$ and thus $M^{\boxtimes I}$ is a projective D-module, yielding the desired result:
\[ \gr_{|I|-1}^{F_\bullet^\spe} P^\ch(M)(I) = F_{|I|-1}^\spe P^\ch(M)(I) \big/ F_{|I|-2}^\spe P^\ch(M)(I) \cong P^*(M)(I) \ . \]
It is clear that this quotient identifies with the standard map of operads $\beta$, and in particular is compatible the operadic composition.
\end{proof}

In fact, the preceding Proposition follows as a special case of the following more general Theorem, stated in 3.2.5 of \cite{BD1} under the same projectivity hypothesis. For our applications, we will need a strengthening of this result, and in particular of the preceding Proposition, to allow for the following more general class of D-modules:

\begin{defn} A D-module $M\in \DD(X)$ is called unital-projective if it admits a decomposition
	\[ M^r= \Omega^{d_X}_X \oplus P_M\]
	where $P_M\in \DD^r(X)$ is a countable direct sum of locally projective D-modules of finite rank.
\end{defn}

\begin{theo}\label{BDgrclthm} Let $X$ be affine and let $\DD(X)_\textup{up}$ denote the full subcategory of $\DD(X)$ on unital-projective objects. There is an isomorphism of graded pseudo-tensor categories
	\[ \gr_\bullet^{F_\bullet^\spe} \DD(X)_\textup{up}^\ch  \xrightarrow{\cong} \DD(X)_\textup{up}^c \ , \]
	extending the identity functor on $\DD(X)_\textup{up}$.
\end{theo}
\begin{proof} 
	
	Let $I$ be a finite set, let $\{ L_i^\textup{Tot}\}_{i\in I},\ M^\textup{Tot}\in \DD(X)_{\textup{up}}$, and let $\{L_i\}_{i\in I},\ M\in\DD(X)$ be arbitrary choices of summand in their unital-projective presentation, so that each is either locally projective of finite rank, or simply equal to $\Omega_X^{d_X}$. Then each of $\{L_i\}$ are in particular $\mc O_X$-flat and thus by the arguments of 3.2.5 in \cite{BD1}, we have for each $k=0,...,|I|-1$ an injection
	\[ \gr_{ k}^{F^\spe_\bullet} P^\ch(\{L_i\},M) \into \Hom_{\DD(X^I)}( \gr_{-k-1}^{F^\spe_\bullet}(j_*j^*\boxtimes_{i\in I} L_i), \Delta_*M)= P^c(\{L_i\},M)^{\langle  k \rangle} \]
	and it remains to check surjectivity.
	
	We recall from \emph{loc. cit.} that we have
	\[\gr_{-k-1}^{F^\spe_\bullet}(j_*j^*\boxtimes_{i\in I} L_i)=\bigoplus_{\substack{(\pi_S,S)\in Q(I)\\ |S|=k+1}} \Delta^{(I/S)}_*( \boxtimes_{s\in S} (\otimes^!_{i\in I_s} L_i) )\otimes \bigotimes_{s\in S} \Lie(|I_s|)  \ , \]
	which explains the equality with $P^c(\{L_i\},M)^{\langle  k \rangle}$ above. Moreover, from the short exact sequence
	\[\gr_{-k-1}^{F^\spe_\bullet}(j_*j^*\boxtimes_{i\in I} L_i ) \to (j_*j^*\boxtimes_{i\in I} L_i)/ F^\spe_{-k-2}(j_*j^*\boxtimes_{i\in I} L_i) \to (j_*j^*\boxtimes_{i\in I} L_i)/ F^\spe_{-k-1}(j_*j^*\boxtimes_{i\in I} L_i)  \ ,\]
	we apply the functor $\Hom_{\DD(X^I)}(\cdot , \Delta_* M)$ to obtain the long exact sequence
	\begin{align*}
		0 & \to F_{k-1}^\spe P^\ch(\{L_i\},M) \to F_{k}^\spe P^\ch(\{L_i\},M) \to \Hom_{\DD(X^I)}(\gr_{-k-1}^{F^\spe_\bullet}(j_*j^*\boxtimes_{i\in I} L_i ), \Delta_*M) \\
		& \to \Ext^1_{\DD(X^I)}((j_*j^*\boxtimes_{i\in I} L_i)/ F^\spe_{-k-1}(j_*j^*\boxtimes_{i\in I} L_i) , \Delta^*M ) \to \cdots & .
	\end{align*}
	Thus, it suffices to show the vanishing of $\Ext^1((j_*j^*\boxtimes_{i\in I} L_i)/ F^\spe_{-k-1}(j_*j^*\boxtimes_{i\in I} L_i) , \Delta_*M )$. Moreover, the source object $(j_*j^*\boxtimes_{i\in I} L_i)/ F^\spe_{-k-1}(j_*j^*\boxtimes_{i\in I} L_i)$ is evidently filtered with associated graded components given by the associated graded components $\gr_{\ell}^{F^\spe_\bullet}(j_*j^*\boxtimes_{i\in I} L_i )$, so that by inductive application of the long exact sequence, it suffices to show
	\[ \Ext^1(\gr_{\ell}^{F^\spe_\bullet}(j_*j^*\boxtimes_{i\in I} L_i ),\Delta_*M) = 0  \ .\]
	Evidently the combinatorial factors of the Lie operations do not contribute to the $\Ext$, so that it remains to show the vanishing of
	\[ \Ext^1_{\DD(X^I)}(\Delta^{(I/S)}_*( \boxtimes_{s\in S} (\otimes^!_{i\in I_s} L_i) ),\Delta^I_*M) \cong \Ext^1_{\DD(X^S)}( \boxtimes_{s\in S} (\otimes^!_{i\in I_s} L_i), \Delta^S_*M) \ ,\]
	where the isomorphism follows by Kashiwara's lemma.
	
	Next, we note that whenever the fibre $I_s$ contains at least one $i$ for which $L_i$ locally projective of finite rank, then the exterior factor $L_s:=\otimes^!_{i\in I_s} L_i$ is so as well, and otherwise $L_s=\Omega_X^{d_X}$. In this notation, it thus remains to show
	\[  \Ext^1_{\DD(X^S)}( \boxtimes_{s\in S} L_s , \Delta^S_*M)=0 \]
	where each $L_s$ is either locally projective of finite rank, or equal to $\Omega_X^{d_X}$.
	
	Note that each such $L_s$ is in particular dualizable, as are the $\otimes^!$ products of their duals, so that the functor $\Delta^*$ is well defined on the source and we can apply the adjunction for the derived Hom
	\[ \Hom_{\Dd(X^S)}( \boxtimes_{s\in S} L_s , \Delta^S_*M) \cong \Hom_{\Dd(X)}( \Delta^{S*}( \boxtimes_{s\in S} L_s),M) \ .\]
	We can compute the source object as
	\[ \Delta^{S*}( \boxtimes_{s\in S} L_s) = \DDD_{X^S} \Delta^! \DDD_{X^S} ( \boxtimes_{s\in S} L_s) = \DDD_{X^S} ( \otimes^{!,\bb L}_{s\in S} \DDD_X L_s ) \ .\]
	We recall that the dual of a locally projective, finite rank D-module is again such, and similarly $\Omega_X^{d_X}$ is self-dual, up to cohomological degree shifts which cancel with those relating $\otimes^{!,\bb L}_{s\in S}$ to the plain tensor product $\otimes^!_{s\in S}$.
	
	We now prove this $\Ext^1$ vanishes by analyzing the various cases, corresponding to whether each of the $\{L_i\}_{i\in I}$ or $M$ is either locally projective of finite rank, or equal to $\Omega_X^{d_X}$. If at least one of the $L_s$ is locally projective of finite rank, then, by the preceding paragraph, so is $\DDD_{X^S} ( \otimes^{!,\bb L}_{s\in S} \DDD_X L_s )$, and thus
	\[ \Ext^1_{\DD(X^S)}( \boxtimes_{s\in S} L_s , \Delta^S_*M) \cong \Ext^1_{\DD(X)}( \Delta^{S*}( \boxtimes_{s\in S} L_s),M) = 0\]
	since the source is locally projective and $X$ is affine.
	
	It remains to treat the case where we have $L_s=\Omega_X^{d_X}$ for every $s\in S$. If $M$ itself is equal to $\Omega_X^{d_X}$, then by the proof of Theorem 3.1.5 in \cite{BD1}, the filtration is trivial and we obtain a canonical identification $P^\ch(\Omega_X^{d_X}) \cong \Lie \cong P^c(\Omega_X^{d_X})$. Otherwise, if $M$ is locally projective of finite rank, then $P^c(\{\Omega_X^{d_X}\},M)=0$ and our desired surjectivity is trivial, independent of the $\Ext^1$ vanishing.
	
\end{proof}

\begin{coro}
	Let $M\in \DD(X)$ be unital-projective. Then 
	\[
	\gr_{F^\spe} P^\ch(M) \cong P^c(M)~.
	\]
	In particular, the preceding holds for $M = \JJ A$ with $A$ a flat $\mc O_X$-algebra.
\end{coro}

\subsubsection{}

Let $(M,F^M_\bullet M)$ be a filtered, unital-projective $D$-module. We assume there exists, and fix once and for all, a splitting of the filtration, \ie, a grading compatible with the $D$-module structure,
\begin{equation}\label{Mgdeqn}
	 M=  \bigoplus_{k\in \Z} M^{\langle k \rangle}\ \in \DD(X) \qquad \text{such that} \qquad F_kM =\bigoplus_{l\leq k} M^{\langle l \rangle}\langle -l \rangle   \ . 
\end{equation}

The corresponding Rees object is
\[
M_\hbar = \Rees_\hbar^{F^M_\bullet}(M) = \bigoplus_{m\in\IZ} F^M_m M\cdot \hbar^m\in D(X)\boxtimes_\ik \ik[\hbar]\Mod_\gr,
\]
and the splitting of \eqref{Mgdeqn} gives a trivialization, \ie, a $\IG_\hbar$-equivariant isomorphism
\begin{equation}\label{Vhtrivgreqn}
M_\hbar \xrightarrow{\sim} \bigoplus_{k\in \IZ} M^{\langle k\rangle}\otimes \ik[\hbar]~.
\end{equation}

\subsubsection{}

The filtration on $M$ induces one on the operad $P^\ch(M)$, given by
\begin{equation}\label{Mfilteqn}
	F_k^M P^\ch(M)(I) = \sum_{p-q=k}  \{ \varphi \in P^\ch(M)(I)  \ | \  \varphi(j^{(I)}_*j^{(I)*}F_p^{M^{\boxtimes I}}M^{\boxtimes I})\subset \Delta^{(I)}_*F_q^M  \ \}~,
\end{equation}
where the induced filtration on $M^{\boxtimes I}$ is defined by
\begin{equation}\label{FMboxeqn}
	F_p^{M^{\boxtimes I}}M^{\boxtimes I} = \sum_{p_1+\dots+p_{|I|}\le p}F^M_{p_1}M\boxtimes \dots \boxtimes F_{p_{|I|}}^MM~.
\end{equation}

The splitting of the filtration on $M$ induces a splitting of the filtration on $P^{\ch}_M$ given by the bigrading
\[
	 P^\ch(M) = \bigoplus_{p,q \in \Z} P^\ch(M)^{\langle p,q \rangle }~,
\]
where the components are given by
\begin{equation}\label{PchMpqeqn}
	 P^\ch(M)^{\langle p,q \rangle } = \Hom_{\DD(X^I)}(j^{(I)}_*j^{(I)*}(M^{\boxtimes I})^{\langle p \rangle}, \Delta^{(I)}_*M^{\langle q \rangle}) \ ,
\end{equation}
with the filtration given by
\[ 
F_k^MP^\ch(M) = \bigoplus_{q-p\leq k}  P^\ch(M)^{\langle p,q \rangle } \ , 
\]
and where we have used the induced splitting of the filtration $F_\bullet^{M^{\boxtimes I}}$ on $M^{\boxtimes I}$ given by
\[
M^{\boxtimes I} = \bigoplus_{p\in \Z} (M^{\boxtimes I})^{\langle p \rangle} \qquad \text{where}\qquad (M^{\boxtimes I})^{\langle p \rangle}= \bigoplus_{p_1+...+p_{|I|} = p}  M^{\langle p_1 \rangle} \boxtimes \hdots \boxtimes M^{\langle p_{|I|} \rangle} \ . 
\]
Note that for each $p,q\in \Z$, the projection to the $(p,q)$-graded component
\begin{equation}\label{pipqeqn}
	 \pi_{p,q}^{P^\ch(M)}:P^\ch(M) \to P^\ch(M)^{\langle p,q \rangle } \qquad \text{is given by}\qquad \mu \mapsto \pi_q^M \circ \mu \circ \iota_p^{M^{\boxtimes I}}  \ ,
\end{equation}
where $\iota_p^{M^{\boxtimes I}}:(M^{\boxtimes I})^{\langle p \rangle}  \to M^{\boxtimes I}$ denotes the inclusion of the $p$th graded component of $M^{\boxtimes I}$ and $\pi_q^M:M \to M^{\langle q \rangle} $ denotes the projection to the $q$th graded component of $M$.

\subsubsection{} 

There is a natural refinement of the special filtration of Section \ref{specialsec} to account for the underlying filtration on the object $M$, as fixed above. We define the filtration $F_\bullet^{M,\spe} P^\ch(M)$ by
\begin{equation}
F_k^{M,\spe} P^\ch(M)(I)= \sum_{\ell \in \Z} (F_{\ell}^{M} \cap F^\spe_{k-\ell} )P^\ch(M)(I) \ ,
\end{equation}
where we have the following explicit description of the intersections occuring as summands
\[
{\hspace{-1cm}}
(F_{\ell}^{M}\cap F_j^{\spe}) P^\ch(M)(I) = \sum_{p-q=\ell} \{ \varphi \in  P^\ch(M)(I) \ | \   \varphi( F_{p}^{M^{\boxtimes I}} M^{\boxtimes I}) \subseteq \Delta_* F_{q}M  \text{ , } \ \varphi((F_{-j-2}^\spe j^{(I)}_*j^{(I)*}\omega_{X}^{\boxtimes I}) \otimes^! M^{\boxtimes I})=\{0\}  \ \}
\]
and the filtration $F_\bullet^{M^{\boxtimes I}}M^{\boxtimes I}$ on $M^{\boxtimes I}$ is as defined in Equation \ref{FMboxeqn}.

\subsubsection{} 





\begin{prop}\label{prop:joint split}
	The joint filtration $F_\bullet^{M,\spe} P^\ch(M)$ on $P^\ch(M)$ can be split, \ie, there exists a grading
	\[
	P^{\ch}(M) = \bigoplus_{k\in \IZ} P^{\ch}(M)^{\langle k \rangle^\spe} \qquad \text{such that}\qquad F_k^{M,\spe} P^\ch(M) = \bigoplus_{i\le k} P^\ch(M)^{\langle i\rangle^\spe}~.
	\]
\end{prop}

%

The proof will consist primarily of bookkeeping; one chooses splittings of the special filtration restricted to each summand of the graded decomposition of Equation \ref{PchMpqeqn}, giving a splitting of the special filtration which is compatible with the graded decomposition in such a way that it can be naturally modified to define a splitting of the joint filtration.

We define the induced special filtration $ F_\bullet^\spe P^\ch(M)^{\langle p,q \rangle}$ on the $(p,q)$-graded component $P^\ch(M)^{\langle p,q \rangle}$ as the intersection
\[ 
F_k^\spe P^\ch(M)^{\langle p,q \rangle} = P^\ch(M)^{\langle p,q \rangle} \cap F_k^\spe P^\ch(M) \ . 
\] 
Then the essential content of the proof of Proposition \ref{prop:joint split} is isolated in the following elementary lemma:

\begin{lemma}\label{splitlemma} The restriction of the projection $\pi_{p,q}^{P^\ch(M)}$ of Equation \ref{pipqeqn} to $F_k^\spe P^\ch(M)$ factors as
	\[ 
	\begin{tikzcd}[column sep=large] F_k^\spe P^\ch(M) \arrow[r,"\pi_{p,q}^{F_k^\spe P^\ch(M) }"]\arrow[d] \arrow[r] & F_k^\spe P^\ch(M)^{\langle p,q \rangle }\arrow[d]  \\
		P^\ch(M) \arrow[r,"\pi_{p,q}^{P^\ch(M)}"]  & P^\ch(M)^{\langle p,q \rangle } 
		 \end{tikzcd} \ . 
		 \]
\end{lemma}
\begin{proof}
	Note that the intersection defining $F_k^\spe P^\ch(M)^{\langle p,q \rangle }$ admits the following concrete description
\[	
{\hspace{-1.5cm}}
P^\ch(M)^{\langle p,q \rangle}(I) \cap F_k^\spe P^\ch(M)(I) = \Hom_{\DD(X^n)}\bigg((j^{(I)}_*j^{(I)*}\omega_{X}^{\boxtimes I}\otimes^! (M^{\boxtimes I})^{\langle p \rangle} )\big/ ((F_{-k-2}^\spe j^{(I)}_*j^{(I)*}\omega_{X}^{\boxtimes I}) \otimes^! (M^{\boxtimes I})^{\langle p \rangle}), \Delta^{(I)}_*M^{\langle q \rangle} \bigg)   .
\]
Evidently the image of $\pi_{p,q}^{F_k^\spe P^\ch(M)}=\pi_{p,q}^{P^\ch(M)}|_{F_k^\spe P^\ch(M)}:F_k^\spe P^\ch(M) \to P^\ch(M)^{\langle p,q \rangle}$ is precisely this subset, so the map factors as claimed.
\end{proof}

The preceding lemma on the compatibility of the grading and filtration provides the following essential corollary:

\begin{coro}\label{Fspedseqn} 
The grading on $P^\ch(M)$ of Equation \ref{PchMpqeqn} induces a grading
	\[ 
	F_k^\spe P^\ch(M) = \bigoplus_{p,q \in \Z} F_k^\spe P^\ch(M)^{\langle p,q\rangle}   \ ,
	\]
	for each $k\in \Z$.
\end{coro}
\begin{proof}
	Taking the direct sum over $p,q\in \Z$ of the horizontal maps in the commutative diagram of Lemma \ref{splitlemma}, we obtain the commutative diagram
		\[ 
		\begin{tikzcd}[column sep=huge] 
		F_k^\spe P^\ch(M) \arrow[r,"\bigoplus_{p,q \in \Z} \pi_{p,q}^{F_k^\spe P^\ch(M) }"{xshift=5pt}]\arrow[d] \arrow[r] &\bigoplus_{p,q \in \Z}  F_k^\spe P^\ch(M)^{\langle p,q \rangle }\arrow[d]  \\
		P^\ch(M) \arrow[r,"\bigoplus_{p,q \in \Z} \pi_{p,q}^{P^\ch(M)}"]  & \bigoplus_{p,q \in \Z}  P^\ch(M)^{\langle p,q \rangle } 
	\end{tikzcd} \ . 
	\]
	The bottom arrow identifies with the identity map under the direct sum decomposition of Equation \ref{PchMpqeqn}, so that the lower composition is injective. Commutativity of the diagram implies the same must be true of the upper composition, and thus in particular the upper horizontal arrow is injective. As remarked in the proof of Lemma \ref{splitlemma}, each of the maps $\pi_{p,q}^{F_k^\spe P^\ch(M)}:F_k^\spe P^\ch(M) \to P^\ch(M)^{\langle p,q \rangle}$ is surjective, and thus so is their direct sum, establishing the desired isomorphism.
\end{proof}

We can now complete the proof of Proposition \ref{prop:joint split}:
\begin{proof}[Proof of Proposition \ref{prop:joint split}] 

For each $p,q \in \Z$, the induced special filtration $F_\bullet^\spe P^\ch(M)^{\langle p,q \rangle}$ is bounded and thus admits a splitting, that is, a grading
	\[
	P^\ch(M)^{\langle p,q \rangle} = \bigoplus_{i \in \Z} P^\ch(M)^{\langle p,q,i \rangle} \qquad \text{such that} \qquad F_k^\spe P^\ch(M)^{\langle p,q \rangle} = \bigoplus_{i \leq k} P^\ch(M)^{\langle p,q,i \rangle} \ .  
	\]
	Then by the preceding Corollary \ref{Fspedseqn}, we obtain
	\[
	F_k^\spe P^\ch(M) = \bigoplus_{p,q \in \Z} F_k^\spe P^\ch(M)^{\langle p,q\rangle}  = \bigoplus_{p,q \in \Z}\bigoplus_{i \leq k} P^\ch(M)^{\langle p,q,i \rangle} = \bigoplus_{i \leq k}\bigoplus_{p,q \in \Z} P^\ch(M)^{\langle p,q,i \rangle} \ . 
	\]
	Thus, we obtain that
\begin{equation}\label{trigreqn}
		 P^\ch(M) = \bigoplus_{i \in \Z}P^\ch(M)^{\langle i \rangle }    \qquad \text{for}\qquad P^\ch(M)^{\langle i \rangle} =\bigoplus_{p,q \in \Z} P^\ch(M)^{\langle p,q,i \rangle} \ ,
\end{equation}
	which thus provides a splitting of the special filtration $F_\bullet^\spe P^\ch(M)$.
	
Note that the joint filtration can also be described as:
\begin{align*}
		 F_k^{M,\spe} P^\ch(M) & = \sum_{\ell \in \Z} (F_{\ell}^{M}\cap F_{k-\ell}^\spe) P^\ch(M) \\
		 & = \sum_{\ell \in \Z} \left(\bigoplus_{p,q \in \Z}\bigoplus_{i \leq \ell} P^\ch(M)^{\langle p,q,i \rangle}   \right) \cap \left( \bigoplus_{p-q\leq k-\ell}\bigoplus_{i\in \Z} P^\ch(M)^{\langle p,q,i \rangle} \right) \\
		 & = \bigoplus_{p,q,i \in R(k)} P^\ch(M)^{\langle p,q,i \rangle} \ ,
\end{align*}
where $R(k)$ is the set of triples $p,q,i \in \Z$ such that there exists $\ell \in \Z$ satisfying
\[ 
i \leq \ell \qquad \text{and} \qquad  q-p\leq k-\ell  \ .
\]
For a fixed $i$, one has the most choices for $p$ and $q$ if $\ell=i$--- since this subtracts the least possible amount from the upper bound on $q-p$. Thus, $R(k)$ is simply the set of triples of integers satisfying
\[ q-p \leq k- i \qquad \text{or equivalently} \qquad q-p+i\leq k \ , \]
so that we obtain the desired result
\[ 
F_k^{M,\spe} P^\ch(M) = \bigoplus_{q-p+i\leq k} P^\ch(M)^{\langle p,q,i \rangle}(n) \ .   
\]
The decomposition of Equation \ref{trigreqn} thus induces the desired splitting
\begin{equation}\label{eq:splitting joint filtration}
P^\ch(M) = \bigoplus_{k\in \Z}P^\ch(M)^{\langle k \rangle^\spe}  \qquad \text{where} \qquad P^\ch(M)^{\langle k \rangle^\spe}=\bigoplus_{\substack{i,p,q \\ q-p+i=k}} P^\ch(M)^{\langle p,q,i \rangle} \ ,
\end{equation}
that is, we have
\[  
F_k^{M,\spe} P^\ch(M) = \bigoplus_{i \leq k}P^\ch(M)^{\langle i \rangle^\spe} 
\ . 
\]

\end{proof}

\subsubsection{}

Recall that the classical operad has an auxiliary grading defined as in Section \ref{ssec:coisson intro}. When the D-module $M$ is graded, one can refine the grading on $P^c(M)$ to a trigrading
\[
P^c(M) =\bigoplus_{p,q,i} P^c(M)^{\langle p,q,i\rangle}~,
\]
where
\begin{equation}
	P^c(M)^{\langle{p,q,i\rangle}}(I) =\bigoplus_{\substack{p_1,p_2,\dots,p_{|I|}\\ p_1+p_2+\dots+ p_{|I|} =p}}P^c\big(\{M^{\langle p_1\rangle},M^{\langle p_2\rangle },\dots,M^{\langle p_{|I|} \rangle}\},M^{\langle q\rangle}\big)^{\langle i\rangle}~. 
\end{equation}
We can coarsen this to a single \emph{joint} grading, analogous to the joint filtration on $P^\ch(M)$:
\begin{equation}\label{eq:joint grading on classical}
	P^c(M) = \bigoplus_{k\in \IZ} P^c(M)^{\langle k\rangle^\spe} ~, \qquad \text{ where } \qquad P^c(M)^{\langle k \rangle ^\spe} = \bigoplus_{\substack{p,q,i\\q-p+i=k}}P^c(M)^{\langle p,q,i\rangle}~.
\end{equation}

We now prove the generalisation of Theorem \ref{BDgrclthm} to the case of the joint filtration.

\begin{theo}\label{thm:gr of joint chiral is joint classical}
Let $(M,F^M)$ be a filtered, unital-projective D-module which is split in the sense of \eqref{Mgdeqn}, then we have an isomorphism of graded operads
	\[
	\gr_{F^{M,\spe}}P^\ch(M) \xrightarrow{\sim} P^c( \gr_{F^M} M)~,
	\]
	where $P^c(\gr_{F^M} M)$ is equipped with the grading from \eqref{eq:joint grading on classical}.
\end{theo}
\begin{proof}
Recall from the proof of Proposition \ref{prop:joint split} that the grading induced by $M$ is compatible with the splitting of the special filtration of $M$---in particular, \eqref{trigreqn} implies that the associated graded with respect to the special filtration inherits a trigrading
\[
\gr_{F^\spe} P^\ch(M) = \bigoplus_{p,q,i\in \IZ} (\gr^i_{F^\spe} P^\ch(M))^{\langle p,q \rangle}~,
\]
where 
\[
(\gr^i_{F^\spe}P^\ch(M))^{\langle p ,q\rangle}(I)  = \bigoplus_{\substack{p_1,p_2,\dots,p_{|I|}\\ p_1+p_2+\dots+ p_{|I|} =p}} \gr^i_{F^\spe}P^\ch\big(\{M^{\langle p_1\rangle},M^{\langle p_2\rangle },\dots,M^{\langle p_{|I|} \rangle}\},M^{\langle q\rangle}\big)~.
\]
Moreover, \eqref{eq:splitting joint filtration} implies that we have the identification
\[
\gr^k_{F^{M,\spe}}P^\ch(M) = \bigoplus_{q-p+i=k} (\gr^i_{F^\spe} P^\ch(M))^{\langle p,q\rangle}~.
\]
But from \ref{BDgrclthm}, $\gr^i_{F^\spe}P^\ch(M) = P^c(M)^{\langle i\rangle}$, and so---at least at the level of trigraded $\Ss$-modules---one has an isomorphism
\[
\gr_{F^{M,\spe}} P^\ch(M) \cong P^c(M)~.
\]
That this is an isomorphism of operads follows from the fact that the identification of Theorem \ref{BDgrclthm} respects the compositions.
\end{proof}

\subsubsection{}

The joint filtration on $P^\ch(M)$ induces a filtration on $\gf_{\Lie, P^\ch(M)}$ with 
\[
F^{M,\spe}_k \gf_{\Lie,P^\ch(M)} = \Hom_{\Ss\Mod}(\Lie^\ash,F^{M,\spe}_k P^\ch(M))~.
\]

\begin{prop}\label{prop:filt on Lie alg}
	Let $(M, F^M)$ be a split, filtered D-module, then 
	\[
	\gr_{F^{M,\spe}}(\gf_{\Lie,P^\ch(M)}) = \gf_{\Lie, P^c(\gr_{F^M} M)}~.
	\]

	Moreover, if $\varphi\in \Alg_\Lie(P^\ch(M))$ is a chiral algebra structure on $M$ such that it induces the coisson algebra structure $\overline \varphi \in \Alg_{\Lie}(P^c(\gr M))$ on $\gr M$. Then we have an isomorphism of DG Lie algebras
	\[
	\gr_{F^{M,\spe}}\ \gf^\varphi_{\Lie,P^\ch(M)} \cong \gf^{\overline{\varphi}}_{\Lie,P^c(\gr_{F^M} M)}~.
	\]
\end{prop}
\begin{proof}
Since $\Ss\Mod$ is semisimple, $\gr_{F^M,\spe}\gf_{\Lie,P^\ch(M)} \cong \gf_{\Lie, \gr_{F^M, \spe} P^\ch(M)}$, whereupon the result follows from Theorem \ref{thm:gr of joint chiral is joint classical}, noting that the second result follows since the isomorphism of \textit{loc.\ cit.} is an isomorphism of operads.
\end{proof}

\begin{prop}\label{prop:bound on cohomology}
Let $(M, F^M)$ be a split, filtered D-module, and suppose $\varphi\in \Alg_\Lie(P^\ch(M))$ is a chiral algebra structure on $M$ such that it induces the coisson algebra structure $\overline \varphi \in \Alg_{\Lie}(P^c(\gr M))$ on $\gr M$.

Then we have the following inequality
\[
\dim\ H^i_\ch(M,\varphi) \le \dim\ H^i_c(\gr M, \overline \varphi)~.
\]
\end{prop}
\begin{proof}
Follows from Proposition \ref{prop:filt on Lie alg}, using the usual spectral sequence on a filtered complex.
\end{proof}

\subsubsection{}

Write $D_\hbar(X)\Mod\cong D(X)\boxtimes_\ik \ik[\hbar]\Mod$ for the category of families of D-modules over $\IA^1_\hbar$. Likewise, we write $D_\hbar(X)^\gd$ for the category of families of D-modules over the stack $[\IA^1_\hbar/\IG_\hbar]$.

We say that $M_\hbar\in D_\hbar(X)$ is \emph{trivializable} if there exists a trivialization:
\[
M_\hbar \xrightarrow{\sim} M_0 \otimes \ik[\hbar]~,
\]
where $M_0 \cong M\otimes_{\ik[\hbar]}\ik_0\in\D(X)$. Similarly, $M_\hbar\in D_\hbar(X)$ is said to be trivializable if there exits a graded trivialization $M_\hbar \xrightarrow{\sim} M_0 \otimes \ik[\hbar]$.

The trivializable $D_\hbar(X)$-modules are equivalent to ordinary D-modules by taking the central fibre. Moreover, the graded, trivializable $D_\hbar(X)$ modules are equivalent to the split, filtered D-modules by 
\[
M_\hbar \mapsto (M_0,F^{M_0})
\]
where $M_0= M\otimes_{\ik[\hbar]}\ik_0$ the central fibre and $F^{M_0}$ is the split filtration induced by the grading on $M_0$. The inverse is
\[
M_0\mapsto \Rees_\hbar^{F^{M_0}}(M_0)~.
\]
Indeed, viewing $M_\hbar\in D_\hbar(X)$ as a trivially graded $D_\hbar(X)$-module, the corresponding D-module $M_0$ is trivially filtered.

Thus, for any $M_\hbar \in D_\hbar(X))$ or $M_\hbar\in D_\hbar(X)^\gd$ we may define the $\ik[\hbar]$-operad
\[
P_\hbar^\ch(M_\hbar) \coloneqq \Rees^{F^{M_0,\spe}}_\hbar(P^\ch(M_0))~,
\]
where the filtration on $M_0$ is taken to be trivial if $M_\hbar$ is not graded. Since the filtration on $M_0$ is split in both cases, by Proposition \ref{prop:joint split}, the filtration on $P^\ch(M_0)$ is split too and there exists a trivialization of $\Ss$-modules,
\[
P^\ch_\hbar(M_\hbar) \xrightarrow{\sim} (P_\hbar^\ch(M_\hbar))_0 \otimes \ik[\hbar] \cong P^c(M_0)\otimes \ik[\hbar]~,
\]
where we used Theorem \ref{thm:gr of joint chiral is joint classical} to identify the central fibre with the classical operad.

Like in Section \ref{ssec:formal def} we can consider the completed version of the above construction. Let $D_{\hat \hbar}(X) = D(X)\widehat\boxtimes_\ik\ik\fph$ be the category of families of D-modules over the formal neighbourhood of the origin of $\IA^1_\hbar$, and write $D_{\hat \hbar}^\gd(X)$ for it graded analogue. We can define the notion of a trivializable $D_{\hat \hbar}(X)$-module and $D_{\hat \hbar}(X)$-module analogously.

Every trivializable $M_{\hat \hbar}\in D_{\hat \hbar}(X)^\gd$ defines a split, filtered D-module $(M_0,F^{M_0})$, with filtration induced by the grading on $M_0$ and 
\[
M_{\hat \hbar} = \widehat{\Rees}_{\hbar}^{F^{M_0}}(M_0) = \Rees_{\hbar}^{F^{M_0}}(M_0)\widehat{\otimes}_{\ik[\hbar]}\ik\fph~.
\]
Once again, for $M_{\hat \hbar}\in D_{\hat \hbar}(X)$ the corresponding $M_0$ is trivially filtered.

Given a trivializable $M_{\hat \hbar}\in D_{\hat \hbar}(X))$ or $D_{\hat \hbar}(X)^\gd$ we define the $\ik\fph$-operad
\[
P_{\hat \hbar}(M_\hbar) = \widehat{\Rees}_\hbar^{F^{M_0,\spe}}(P^\ch(M_0))~.
\]
As before, the splitting of the joint filtration induces a trivialization
\[
P_{\hat \hbar}(M_{\hat \hbar} \xrightarrow{\sim} (P^\ch_\hhb(M_\hhb))_0\widehat{\otimes}\ik\fph \cong P^c(M_0)\widehat{\otimes} \ik\fph~.
\]

\begin{rmk}\label{rmk:algebras in chiral Rees}
	The data of $\varphi_{\hat \hbar}\in \Alg_{\Lie_{\hat \hbar}}(P^\ch_{\hat \hbar}(M_{\hat \hbar}))$ (or its graded counterpart $\varphi_{\hat \hbar} \in \Alg^\gd_{\Lie_{\hat \hbar}}(P^\ch_{\hat \hbar}(M_{\hat \hbar}))$ corresponds to a (graded) chiral algebra structure on $M_{\hat \hbar}$ such that its restriction to the central fibre, $\varphi_0\in \Alg_{\Lie}(P^c(M_0))$, defines a (graded) coisson algebra structure on $M_0$.

\end{rmk}

\begin{rmk}\label{rmk:concrete algebraic sturcture in chiral Rees}
Let $\varphi_{\hat \hbar} \in \Alg_{\Lie_{\hat \hbar}}(P^\ch_{\hat \hbar}(M_{\hat \hbar}))$, by way of the trivialization of \eqref{eq:splitting joint filtration} we can think of $\varphi_{\hat \hbar}$ as a pair 
\[
m_{\hat \hbar}\oplus b_{\hat \hbar} \in P^c(M_0)(2)\otimes \ik\fph =  \Hom_{D(X)}(M_0\otimes^!M_0,M_0)\otimes\ik\fph\oplus \Hom_{D(X^2)}(M_0\boxtimes M_0,\Delta^{(2)}_*M_0)\otimes \ik\fph~.
\]

Restricted to the central fibre, $m_0 : M_0\otimes^!M_0\rightarrow M_0$ is a commutative product and $b_0:M_0\boxtimes M_0\to \Delta^{(2)}_*M_0$ is a $\Lie^*$-bracket on $M_0$ compatible with $m_0$ in the sense that $(M_0,m_0,b_0)$ is a coisson algebra.
\end{rmk}

\subsubsection{}

We set $X=\IA^1$ and work in the weakly translation invariant setting. We can identify the weakly translation invariant $D_\hhb(X)$-modules with $\ik[\del]\Mod(\ik\fph\Mod)$. In this setting, following Section \ref{ssec:weakly trans inv chiral alg}, one recovers the notion of quantization of a vertex Poisson algebra as introduced in \cite{Arakawa2015:local}. Let us explore this in some detail.

Given a trivializable $V_\hhb\in\ik[\del]\Mod(\ik\fph\Mod)$, we can make sense of $P^{\ch,\IG_a}_\hhb(V_\hhb)$ by way of the Rees construction---as we did before.

We have an isomorphism of the central fibres,
\[
P^{\ch,\IG_a}_\hhb(V_0)  \cong P^{\IG_a,c}(V_0)~, 
\]
and a trivialization of $\Ss$-modules,
\[
P^{\ch,\IG_a}_\hbar(V_\hbar) \xrightarrow{\sim} P^{\ch,\IG_a}_0(V_0)\otimes \ik[\hbar]\cong P^{c,\IG_a}(V_0)\otimes \ik[\hbar]~.
\]

Concretely, $\varphi_\hhb\in \Alg_{\Lie_\hhb}(P^{\ch,\IG_a}_\hhb(V_\hhb))$ corresponds to an integral $\lambda$-bracket $I_{\hhb,\lambda}:V_{\hat \hbar}\otimes_{\ik\fph}V_{\hat \hbar} \to V_{\hat \hbar}[\lambda]$ with the property that
\[
I_{0,\lambda}:V_0\otimes V_0 \to V_0
\]
is a commutative multiplication on $V_0$ and 
\[
\lim_{\hbar\to 0}\frac{1}{\hbar}\tfrac{d I_{\hhb,\lambda}}{d\lambda}: V_0\otimes V_0\rightarrow V_0[\lambda]
\]
defines a $\lambda$-bracket on $V_0$ compatible with the commutative multiplication---making $V_0$ a vertex Poisson algebra. To summarize:

\begin{prop}
	The data of $\varphi_{\hat \hbar} \in \Alg_{\Lie_\hbar}(P^{\ch,\IG_a}_{\hat \hbar}(V_{\hat \hbar})$ is an almost commutative $\hbar$-adic vertex algebra structure on $V_\hbar$ that quantizes---in the sense of Definition \ref{defn:chiral quantization}---the central fibre: the vertex Poisson algebra $(V_0, \varphi_0)$.

	Similarly, $\varphi_{\hat \hbar}\in \Alg^\gd_{\Lie_\hbar}(P^{\ch,\IG_a}_{\hat \hbar}(V_{\hat \hbar}))$ is a graded quantization of the graded vertex Poisson algebra $(V_0, \varphi_0)$.
\end{prop}

\subsection{Quantization of coisson algebras}\label{quantchegsec}

Let $R\in D(X)$ be a unital-projective D-module and $\varphi\in \Alg_{\Lie}(P^c(R))$, so that $(R,\varphi)$ is a coisson algebra. 

In light of Remarks \ref{rmk:algebras in chiral Rees} and \ref{rmk:concrete algebraic sturcture in chiral Rees}, we make the following definition
\begin{defn}\label{defn:quant of coisson}
	Given a coisson algebra $\varphi\in \Alg_{\Lie}(P^c(R))$, a \emph{quantization}, $(R_{\hat \hbar},\varphi_{\hat \hbar})$, is a
	\begin{enumerate}[(i)]
	\item trivializable $R_{\hat \hbar}\in D_{\hat \hbar}(X)$ with an isomorphism of the central fibre $R_0= R_\hbar \otimes_{\ik\fph}\ik_0\xrightarrow{\sim}R$
	\item $\varphi_{\hat \hbar} \in \Alg_{\Lie_{\hat \hbar}}(P^\ch_{\hat \hbar}(R_{\hat \hbar}))$ 
\end{enumerate}
	such that the isomorphism $R_0\cong R$ extends to an isomorphism of coisson algebras $(R_0,\varphi_0)\cong (R,\varphi)$.

	If $(R,\varphi)$ is a graded coisson algebra, then we say that a quantization $(R_{\hat \hbar}, \varphi_{\hat \hbar})$ is \emph{graded} if $\varphi_{\hat \hbar}\in \Alg^\gd_{\Lie_{\hat \hbar}}(P^\ch_{\hat \hbar}(R_{\hat \hbar}))$ and $(R_0,\varphi_0)\cong (R,\varphi)$ as graded coisson algebras.
\end{defn}

\subsubsection{} 

Fix $R_{\hat \hbar}$ to be the trivial family, $R \widehat{\otimes}\ik\fph$. Consider a deformation of $\varphi$,
\[
\varphi_{\hat \hbar}\in \Def_{\varphi}^{\Lie_\hbar,P^\ch_\hbar(R_\hbar)}(\ik\fph)~.
\]
By construction, $R_{\hat \hbar}$ is a trivializable $D_{\hat \hbar}(X)$-module with central fibre \emph{equal} to $R$ and $\varphi_{\hat \hbar}\in \Alg_{\Lie_{\hat \hbar}}(P^\ch_{\hat \hbar}(R_{\hat \hbar}))$ is such that $\varphi_0 = \varphi$. Thus, $(R_{\hat \hbar}, \varphi_{\hat \hbar})$ is a quantization of  the coisson algebra $(R,\varphi)$. Indeed, the scheme of deformations parameterises all quantizations of the coisson algebra $(R,\varphi)$ whose underlying $D_{\hat \hbar}(X)$ module is our fixed $R_\hbar$.

Let $S$ be the Artinian local ring, $ \ik[\hbar]/\hbar^m$, or the complete local ring $\ik\fph$.
Therefore, we make the following definitions:
\begin{equation}
	\Quant_\varphi(S) =\Def_\varphi^{\Lie_\hbar, P_\hbar^\ch(R_\hbar)}(S) ~, \qquad \text{ and } \qquad  \Quant_{\varphi,\gd}(S) =\Def_{\varphi,\gd}^{\Lie_{\hbar}, P^\ch_\hbar(R_\hbar)}(S)~,
\end{equation}
of the schemes of quantizations and graded quantizations of $(\varphi_0,R_0)$ whose underlying (graded) $D_\hbar(X)$-module is the trivial family.

Similarly, we define the stack of quantizations and graded quantizations as
\begin{equation}\label{eq:chiral quant defn}
	\fQuant_\varphi(S)  =\fDef_\varphi^{\Lie_\hbar, P^\ch_\hbar(R_\hbar)}(S)  ~, \qquad \text{ and } \qquad  \fQuant_{\varphi,\gd}(S)  =\fDef_{\varphi,\gd}^{\Lie_{\hbar}, P^\ch_\hbar(R_\hbar)}(S)   .
\end{equation}

The following remark is mirror to Remark \ref{rmk:stack of finite quant is good}.
\begin{rmk}
The stacks $\fQuant_\varphi(S)$ and $\fQuant_{\varphi,\gd}(S)$ actually parameterise the isomorphism classes of all (graded) quantizations of $(R,\varphi)$. This is because the underlying $D_\hhb(X)$-module of every quantization is isomorphic to $R\widehat{\otimes}\ik\fph$ and so every quantization of $(R,\varphi)$ has a representative whose underlying $D_\hhb(X)$-module is trivial

One may think of our scheme of quantizations as a kind of partial gauge fixing of a larger, more glorious, scheme of quantizations, which in addition parameterizes the data of the central fibre isomorphism.
\end{rmk}

\subsubsection{}

We can now apply the results of Section \ref{chsec} to the chiral setting.
\begin{theo}\label{DefObsQuantChThm} Let $\varphi\in \Alg_{\Lie}(P^c(R))$ be a coisson algebra and $\varphi_{\hbar/\hbar^{m}}\in\Quant^\ch_\varphi(\K[\hbar]/\hbar^{m}) $ a quantization defined modulo $\hbar^m$. Then there exists an extension
	\[ \varphi_{ \hbar / \hbar^{m+1}} \  \in \Quant^\ch_\varphi(\K[\hbar]/\hbar^{m+1}) \times_{\Quant^\ch\varphi(\K[\hbar]/\hbar^{m}) } \{\varphi_{\hbar/\hbar^{m}} \}\]
	if and only if the cohomology class of the obstruction vanishes, that is,
	\[ 
	[ \Obs^{\Lie_\hbar, P_\hbar^\ch(R_\hbar)}_m(\varphi_{\hbar/\hbar^{m}})]= 0 \ \in H^{2}_{c}(\varphi) \ .
	\]
	Further, when the obstruction class vanishes, the space of such extensions up to isomorphism,
	\[  \fQuant^\ch_\varphi(\K[\hbar]/\hbar^{m+1}) \times_{\fQuant_\varphi^\ch(\K[\hbar]/\hbar^{m}) } \{\varphi_{\hbar/\hbar^{m}} \} \ , \]
	is a torsor for $H^{1}_{c}(\varphi)$.
\end{theo}
\begin{proof} This follows from applying Theorem \ref{DefObsThm} to the present setting; the operad $\Lie_\hbar$ is Koszul, the underlying $\Ss$-module of $\Lie_\hbar$ is evidently canonically trivialized, and the trivialization of the underlying $\Ss$-module of $P^\ch_\hbar(R_\hbar)$ is given by that of Proposition \ref{prop:joint split}.
\end{proof}

\begin{coro} \label{Coisquantexistthm}
Let $\varphi\in \Alg_{\Lie}(P^c(R))$ a coisson algebra such that $H^{2}_{c}(\varphi)=0$. Then $\Quant^\ch_\varphi(\ik\fph)$ is nonempty.
\end{coro}
\begin{proof}
This follows from applying Corollary \ref{forexcoro} to the present setting.
\end{proof}

\begin{coro} \label{Coisquantrigidthm}
Let $\varphi\in \Alg_{\Lie}(P^c(R))$ be a coisson algebra such that $H^{1}_{c}(\varphi)=0$. Then, if there exists $\vph\in\Quant^\ch_\varphi(\ik\fph)$, every quantization of $(R,\varphi)$ is isomorphic to $(R_\hhb,\vph)$.
\end{coro}
\begin{proof}
	This follows from applying Corollary \ref{foruncoro} to the present setting.
\end{proof}

\begin{theo}\label{DefObsgrQuantChThm} Let $\varphi\in \Alg^\gd_{\Lie}(P^c(R))$ a graded coisson algebra and $\varphi_{\hbar/\hbar^{m}}\in \Quant_{\varphi,\gd}^\ch(\K[\hbar]/\hbar^{m}) $ a graded quantization defined modulo $\hbar^m$. Then there exists an extension
	\[ \varphi_{ \hbar / \hbar^{m+1}} \  \in \Quant_{\varphi,\gd}^\ch(\K[\hbar]/\hbar^{m+1}) \times_{\Quant^\ch_{\varphi,\gd}(\K[\hbar]/\hbar^{m}) } \{\varphi_{\hbar/\hbar^{m}} \}\]
	if and only if the cohomology class of the obstruction vanishes, that is,
	\[ [ \Obs^{\Lie_{\hbar}, P_{\hbar}^{\ch}(R_\hbar)}_m(\varphi_{\hbar/\hbar^{m}})]= 0 \ \in H^{2,\langle -m \rangle}_{c}(\varphi) \ .\]
	Further, when the obstruction class vanishes, the space of such extensions up to isomorphism,
	\[  \fQuant_{\varphi,\gd}^\ch(\K[\hbar]/\hbar^{m+1}) \times_{\fQuant^\ch_{\varphi,\gd}(\K[\hbar]/\hbar^{m}) } \{\varphi_{\hbar/\hbar^{m}} \} \ , \]
	is a torsor for $H^{1,\langle -m \rangle}_{c}(\varphi)$, the $-m^{th}$ graded component of $H^1_{c}(\varphi)$.
\end{theo}
\begin{proof}
	This follows from applying Theorem \ref{DefObsgrThm} to the present setting, noting the required hypotheses hold as in the proof of Theorem \ref{DefObsQuantChThm}.
\end{proof}

\begin{coro}  Let $\varphi\in \Alg^\gd_{\Lie}(P^c(R))$ a graded coisson algebra such that $H^{2,\langle-m \rangle}_{c}(\varphi)=0$ for all $m>0$. Then $\Quant^\ch_{\varphi,\gd}(\ik\fph)$ is nonempty.
\end{coro}
\begin{proof}
This follows from applying Corollary \ref{forexgrcoro} to the present setting.
\end{proof}

\begin{coro}\label{cor:unique graded quantization}
Let $\varphi\in \Alg_{\Lie}(P^c(R))$ a graded coisson algebra such that \mbox{$H^{1,\langle -m \rangle}_{c}(\varphi)=0$} for all $m>0$. Then, if there exists $\varphi_{\hat \hbar}\in\Quant^\ch_{\varphi,\gd}(\ik\fph)$ every graded quantization of $(R,\varphi)$ is isomorphic to $(R_\hhb, \vph)$.
\end{coro}
\begin{proof}
This follows from applying Corollary \ref{forungrcoro} to the present setting.
\end{proof}

\begin{coro}\label{cor:torsor if cohomology is in single graded degree}
	Let $\varphi\in \Alg_{\Lie}(P^c(R))$ a graded coisson algebra such that \mbox{$H^{1,\langle -m \rangle}_{c}(\varphi)=\{0\}$} for all $m\neq m_0$, $m_0 \in \IN$. Then, if there exists a  graded quantization $\varphi_{\hat{\hbar}}\in \Quant^\ch_{\varphi,\gd}(\K\fph)$, the stack of  graded quantizations $\fQuant^\ch_{\varphi,\gd}(\K\fph)  $ is a torsor for $H^{1,\langle -m_0 \rangle}_{c}(\varphi)$.
\end{coro}
\begin{proof}
	This follows from applying Corollary \ref{grtorcoro} to the present setting.
\end{proof}

\section{Variational Poisson cohomology of jets of maps to affine symplectic varieties}\label{varsec}

\subsection{Overview of main result on variational Poisson cohomology}
The goal of this section is to prove the following theorem, which is one of the main results of the present work:

Let $Y$ be an affine symplectic variety, $\mc Y=\mc JY$ the corresponding affine D-scheme of jets of maps to $Y$ and $R=(R,m)=\mc O_{\mc JY}$ the corresponding $\Comm^!(X)$ algebra. The symplectic structure $\omega_Y$ on $Y$ induces a coisson structure $b_{\omega_Y}$ on $R$, and by mild abuse of notation we also let $R=(R,m,b_{\omega_Y})$ denote the resulting coisson algebra. Let $C^\bullet_c(R)=C^\bullet_c(R,m,b_{\omega_Y})$ denote the classical cochains on $R$ from Proposition \ref{chPoiscxprop}. The main result of this section is the following identification of the cohomology groups $H^\bullet_c(R)=H_{\Lie,P^c(R)}^{\bullet}(\varphi)$ for $\varphi:\Lie\to P^c(R)$ the corresponding map of operads.

\begin{theo}\label{mainchtheo} For $X=\bb A^1$, there is a natural isomorphism
	\[  H_{c}^{k}(\mc O_{\mc JY}) \cong  H^{k+2}_\dR(Y)  \]
	for $k\geq 1$.
\end{theo}

Conceptually, the proof is essentially identical to that of the analogous fact for affine symplectic varieties, which was the content of Theorem \ref{QuantdRthm}. We defer the analogous discussion of the implications for the deformation theory of coisson algebras to Section \ref{sympsec} below.

The formal proof is given in Section \ref{ssec:proof of mainchtheo}, with each of the main ingredients proved in the following sections: In Section \ref{inthomsec}, we introduce the requisite internal $\Hom$ objects for the categories $\DD(X)$ and $\DD(R)$, as well as for their spaces of multi-linear maps with respect to the various pseudo-tensor category structures used in the proof. In Section \ref{intglobalcompsec}, we show that the middle de Rham cohomology of the internal analogue of the variational cochains recovers the global variational cochains, which by Corollary \ref{GOequivchcoro} computes the cohomology of the classical cochains as on the right hand side of the isomorphism in Theorem \ref{mainchtheo}. In Section \ref{varsympsec}, we show that the internal variational cochains admit a natural isomorphism of complexes with the internal Chevalley--Eilenberg cochains on the $\Lie^*$ algebroid $\Omega^1_R$ determined by the coisson structure on $R$, and moreover that under our symplectic hypothesis on $Y$, this admits a further isomorphism of complexes to the truncated relative de Rham complex of $\mc Y \to X$. Finally, in Section \ref{reldRsec} we show that the corresponding middle de Rham cohomology of the truncated relative de Rham complex is isomorphic in the derived category to the truncated de Rham complex of $Y$.

\subsection{Proof of Theorem \ref{mainchtheo}}\label{ssec:proof of mainchtheo}

The proof amounts, essentially, to combining the main results of the sections which follow below.

Let $X = \IA^1$, $Y$ an affine symplectic variety and $\YY = \JJ Y$. Since $Y$ is smooth, $\YY=\JJ Y$ is very smooth and by Corollary \ref{GOequivchcoro}, the inclusion $C_{\PV}^\bullet(\OO_{\JJ Y})\into C_c^\bullet(\OO_{\JJ Y})$ is a quasi-isomorphism. We have a chain of isomorphisms
\[
C^\bullet_{\PV}(\OO_{\JJ Y}) \cong h\big(\CC_{\PV}^\bullet(\OO_{\JJ Y})\big) \cong h\big(\Omega^{\bullet\ge1}_{\mc Y/X}\big)[1]~,
\]
where the first isomorphism follows from \ref{prop:deRham of internal PV} and the second from Proposition \ref{prop:int PV is de Rham}.

Now, by Proposition \ref{jetsdRprop}, we have an isomorphism in cohomology $H^k(h(\Omega^{\bullet\ge1}_{\mc Y/X}))\cong H^{k+1}_{\dR}(Y) $ for $k\ge 2$. Combining, we have the desired result for $k\ge 1$
\[
H^k_c(\OO_{\JJ Y}) \cong H^{k+2}_\dR(Y)~.
\]
\hfill $\Box$

\subsection{Internal Hom objects for D-modules and quasicoherent sheaves on D-schemes}\label{inthomsec}

There are natural internal Hom objects for the derived category $\Dd(X)$:
\begin{defn}
	The naive internal Hom object $ \Homii_{\Dd(X)}(L,M) \in\Dd(\mc O_X)$ for the category $\Dd(X)$ is defined by
	\[  \Homii_{\Dd(X)}(L,M)(U) = \Hom_{\Dd(U)^!}(j^!L,j^!M)(U) \]
	on any open set $U\subset X$, where $j:U\to X$ the inclusion.
\end{defn}

Let $\Dd_\coh(X) \subset \Dd(X)$ denote the full subcategory of objects of $\Dd(X)$ with  $\mc D_X$-coherent cohomology objects, and let $L\in \Dd_\coh(X)$ and $M\in \Dd(X)$.

\begin{defn}  The (genuine) internal Hom object $\Homi_{\Dd(X)}(L,M) \in \Dd(X)$ is defined by 
	\[\Homi_{\Dd(X)}(L,M)  = \Homii_{\Dd^r(X)}(L^r, M^r \otimes_{\mc O_X} \mc D_X ) \ . \]
\end{defn}

This satisfies the following relation with the usual Hom space in $\DD(X)$:
\begin{prop}\label{inthommatchprop} There exists a natural quasi-isomorphism
	\[\pi_* \Homi_{\Dd(X)}(L,M) \xrightarrow{\cong}\Hom_{\Dd(X)}(L,M) \  .  \]
\end{prop}

\begin{defn} The duality functor $\DDD:\Dd_\coh(X) \to \Dd_\coh(X)$ is defined by
	\[ \DDD M = \Homi_{\Dd(X)}(M, \omega_X) = \Homii_{\Dd^r(X)}( M^r,\omega_X\otimes_{\mc O_X}\mc D_X) \ . \]
\end{defn}

This satisfies the following relation with the genuine internal Hom functor
\begin{prop}\label{inthomdescprop} Let $L\in \Dd_\coh(X)$ and $M\in \Dd(X)$. There is a natural isomorphism
	\[\Homi_{\Dd(X)}(L,M)  \cong (\DDD L) \otimes^{\bb L,!} M \ . \]
\end{prop}

\subsubsection{} There are also analogues of the above internal Hom objects in the abelian category $\DD(X)$:

\begin{defn}
	The naive internal Hom object $ \Homii_{\Dd(X)}(L,M) \in\Dd(\mc O_X)$ for the category $\Dd(X)$ is defined by
	\[  \Homii_{\Dd(X)}(L,M)(U) = \Hom_{\Dd(U)^!}(j^{!,\heartsuit}L,j^{!,\heartsuit}M)(U) \]
	on any open set $U\subset X$, where $j:U\to X$ the inclusion.
\end{defn}

Let $\DD_\coh(X) \subset \DD(X)$ denote the full subcategory on coherent D-modules, and let $L\in \DD_\coh(X)$ and $M\in \DD(X)$.

\begin{defn}  The (genuine) internal Hom object $\Homi_{\DD(X)}(L,M) \in \DD(X)$ is defined by 
	\[\Homi_{\DD(X)}(L,M)  = \Homii_{\DD^r(X)}(L^r, M^r \otimes_{\mc O_X} \mc D_X )  \ \in \DD^r(X) \ . \]
\end{defn}

We also introduce the following abelian variant of the duality functor:

\begin{defn} The duality functor $(\cdot)^{\circ}:\DD_\coh(X) \to \DD_\coh(X)$ is defined by
	\[ M^{\circ} = \Homi_{\DD(X)}(M, \Omega^{d_X}_X) = \Homii_{\DD^r(X)}( M^r,\Omega^{d_X}_X\otimes_{\mc O_X}\mc D_X) \ . \]
\end{defn}
Note that if $M\in \DD(X)$ is locally projective then so is $M^\circ$, and we have that $M^\circ \cong \DDD M[-d_X]$.

This satisfies the following analogous relation with the genuine internal Hom functor
\begin{prop}\label{inthomdescabprop} Let $L\in \DD_\coh(X)$ and $M\in \DD(X)$. There is a natural isomorphism
	\[\Homi_{\DD(X)}(L,M)  \cong L^{\circ} \otimes^! M \ . \]
\end{prop}

Although it will not in general be true that the plain internal Hom object correctly computes the space of Homs in the abelian category, we have the following:

\begin{prop}Let $L\in \DD_\coh(X)$ be locally projective and $M\in \DD(X)$. Then
	\[ \Homi_{\Dd(X)}(L,M) \cong \Homi_{\DD(X)}(L,M) \ ,\]
	and if in addition $X$ is affine and $\Homi_{\DD(X)}(L,M)$ is locally projective, we have
	\[ h(\Homi_{\DD(X)}(L,M)) \xrightarrow{\cong} \Hom_{\DD(X)}(L,M)  \ .\]
\end{prop}

\subsubsection{} Let $L_i\in \DD_\coh(X)$ and $M\in \DD(X)$. We also have the following natural variants internal to $\DD(X)$ of the spaces of multilinear operations for the various pseudo-tensor structures on $\DD(X)$:

\begin{defn} The internal operation spaces for $\DD(X)^!$ are defined by
\[\mc P^!(\{L_i\}, M) = \Homi_{\DD(X)}(\otimes^!_{i\in I} L_i , M)  \ . \]
\end{defn}

\begin{defn} The internal operation spaces for $\DD(X)^*$ are defined by
\[ 
\mc P^*(\{L_i\}, M) = \Delta^{I!}_X\Homi_{\DD(X^I)}(\boxtimes_{i\in I} L_i , \Delta^I_* M)  \]
\end{defn}

\begin{defn} The internal operation spaces for $\DD(X)^c$ are defined by
	\begin{equation}
		\mc P^c_I(\{L_i\}, M)  = \bigoplus_{(\pi_S,S)\in Q(I)}  \mc P_I^c(\{L_i\},M)_S
	\end{equation}
	where for each $\pi_S:I\to S$ we define
	\[\mc P_I^c(\{L_i\},M)_S= \mc P_S^*\big(\{ \otimes!_{i\in I_s} L_i\}_{s\in S},M\big)\otimes \big( \otimes_{s\in S} \Lie(|I_s|)\big)  \ . \]
\end{defn}

We also have the following concrete description of the internal operation spaces for the endomorphism operad of a single object $M\in \DD(X)^c$:
\[ 
\iEnd_{\DD(X)^c}(M)(n)= \mc P_n^c(\{ M\}_{i=1}^n,M) =\bigoplus_{p=0}^{n-1}\PP_n^c(\{M\}_{i=1}^n,M)^{\langle p \rangle}~,
\]
where the graded components are defined as in Proposition \ref{Oendcprop} by
\begin{equation}\label{Pclinteqn}
	\PP_n^c(\{M\}_{i=1}^n,M)^{\langle p \rangle}=   \left( \bigoplus_{n_1+...+n_{p+1}=n} \Ind_{\Ss_{n_1}\times ... \times \Ss_{n_{p+1}}}^{\Ss_n}\bigg(\mc P^*_{p+1}(\{M^{\otimes^! n_i}\}_{i=1}^{p+1}, M)\otimes \bigotimes_{i=1}^{p+1}\Lie(n_i)\bigg)\right)_{\Ss_{p+1}}   \ .
\end{equation}

\subsubsection{} There are natural internal Hom objects for the categories $\Dd(R)$ and $\DD(R)$, analogous to those for $\Dd(X)$ and $\DD(X)$ defined above:

\begin{defn}
	The naive internal Hom object $ \Homii_{\Dd(R)}(L,M) \in \Dd(R\Mod(\QC(X)))$ for the category $\DD(R)$ is defined by
	\[  \Homii_{\Dd(R)}(L,M)(U) = \Hom_{\Dd(j^!R\Mod(\DD(U)^!))}(j^!L,j^!M)(U) \]
	on any open set $U\subset X$, where $j:U\to X$ the inclusion.
\end{defn}

Let $\Dd_\coh(R) \subset \Dd(R)$ denote the full subcategory of complexes with $\mc D_X$-coherent, $R$-finitely generated cohomology modules, and let $L\in \Dd_\coh(R)$ and $M\in \DD(R)$.

\begin{defn}  The (genuine) internal Hom object $\Homi_{\Dd(R)}(L,M) \in \Dd(R)$ is defined by 
	\[\Homi_{\Dd(R)}(L,M)  = \Homii_{\Dd(R)}(L^r, M^r \otimes_{\mc O_X} \mc D_X ) \ . \]
\end{defn}
Note that in the case $R^{\ell,\heartsuit}=\mc O_X$ is the unit $\Comm^!$ algebra, these reduce to the usual definitions of the naive and genuine internal Hom objects in the category of D-modules recalled in Section \ref{DModsec}.

The internal Hom object satisfies the following relation with the usual Hom spaces in $\Dd(R)$:
\begin{prop}\label{Rinthommatchprop} There is a natural isomorphism
	\[ \pi_* \Homi_{\Dd(R)}(L,M) \xrightarrow{\cong} \Hom_{\Dd(R)}(L,M) \  .  \]
\end{prop}

\subsubsection{} In particular, we obtain a duality functor on $\Dd(R)_\coh$:

\begin{defn} The duality functor $\DDD_R:\Dd_\coh(R) \to \Dd_\coh(R)$ is defined by
	\[ \DDD_R M = \Homi_{\Dd(R)}(M, R) = \Homii_{\Dd^r(R)}( M^r, R^r\otimes_{\mc O_X} \mc D_X) \ . \]
\end{defn}

This satisfies the following relation with the genuine internal Hom functor
\begin{prop}\label{Rinthomdescderprop} Let $L\in \Dd_\coh(R)$ and $M\in \Dd(R)$. There is a natural isomorphism
	\[\Homi_{\Dd(R)}(L,M)  \cong (\DDD_R L) \otimes^{\bb L,!}_R M \ . \]
\end{prop}

We also have the following plain (underived) analogues of the internal Hom functors for $\DD(R)$:

\begin{defn}
	The naive internal Hom object $ \Homii_{\DD(R)}(L,M) \in R\Mod(\QC(X))$ for the category $\DD(R)$ is defined by
	\[  \Homii_{\DD(R)}(L,M)(U) = \Hom_{j^!R\Mod(\DD(U)^!)}(j^!L,j^!M)(U) \]
	on any open set $U\subset X$, where $j:U\to X$ the inclusion.
\end{defn}

Let $\DD_\coh(R) \subset \DD(R)$ denote the full subcategory of $\mc D_X$-coherent, $R$-finitely generated modules, and let $L\in \DD_\coh(R)$ and $M\in \DD(R)$.

\begin{defn}  The (genuine) internal Hom object $\Homi_{\DD(R)}(L,M) \in \DD(R)$ is defined by 
	\[\Homi_{\DD(R)}(L,M)  = \Homii_{\DD^r(R)}(L^r, M^r \otimes^! \mc D_X ) \ . \]
\end{defn}
Note that in the case $R=\Omega_X^{d_X}$ is the unit $\Comm^!$ algebra, these reduce to the usual definitions of the naive and genuine internal Hom objects in the category of D-modules recalled in Section \ref{DModsec}.

\subsubsection{} In particular, we obtain a duality functor on $\DD(R)_\coh$:

\begin{defn}\label{circRdefn} The duality functor $(\cdot)^{\circ_R}:\DD_\coh(R) \to \DD_\coh(R)$ is defined by
	\[ M^{\circ_R} = \Homi_{\DD(R)}(M, R) = \Homii_{\DD^r(R)}( M^r, R^r\otimes_{\mc O_X} \mc D_X) \ . \]
\end{defn}
Note that if $M\in \DD(R)$ is projective then so is $M^{\circ_R}$, and we have that $M^{\circ_R} \cong \DDD_R M[-d_X]$.

This satisfies the following relation with the genuine internal Hom functor
\begin{prop}\label{Rinthomdescprop} Let $L\in \DD_\coh(R)$ and $M\in \DD(R)$. There is a natural isomorphism
	\[\Homi_{\DD(R)}(L,M)  \cong L^{\circ_R} \otimes^!_R M \ . \]
\end{prop}

Again, it will not in general be true that the plain internal Hom object correctly computes the space of Homs in the abelian category, we have the following:

\begin{prop}\label{inthommatchabprop} Let $L\in \DD_\coh(R)$ be projective and $M\in \DD(R)$. Then
	\[ \Homi_{\Dd(R)}(L,M) \cong \Homi_{\DD(R)}(L,M) \ ,\]
	and if in addition $X$ is affine and $\Homi_{\DD(R)}(L,M)$ is locally projective in $\DD(X)$, we have
	\[ h(\Homi_{\DD(R)}(L,M)) \cong  \Hom_{\DD(R)}(L,M)  \ .\]
\end{prop}

\subsubsection{}Let $L_i \in \DD_\coh(R)$ and $M\in \DD(R)$. We also have the following natural variants internal to $\DD(R)$ of the spaces of multilinear operations for the various pseudo-tensor structures on $\DD(R)$:

\begin{defn} The internal operation spaces for $\DD(R)^!$ are defined by
	\[ \mc P^!_{I,R}(\{L_i\},M) = \Homi_{\DD(R)}( \otimes_{R,i\in I}^{!} L_i, M)  \  .  \]
\end{defn}

More explicitly, by Proposition \ref{Rinthomdescprop} we have
\[  \mc P^!_{I,R}(\{L_i\},M) \cong (\otimes_{R,i\in I}^{!} L_i )^\circ \otimes^{R,!} M \ . \]

\begin{defn} \label{Rstarinthomdefn}
The internal operation spaces for $\DD(R)^*$ are defined by
	\[ \mc P_{I,R}^*(\{L_i\},M)= \Delta_R^{I\bullet} \Delta^{I!}_X \Homi_{\DD(R^{\boxtimes I })}(\boxtimes_{i\in I} L_i, \Delta_*M) \ . \]
\end{defn}

\subsection{Comparison of global and internal variational cochains}\label{intglobalcompsec}

To begin, we introduce the internal analogues of some of the various deformation-obstruction complexes for algebras of the chiral kind:

\begin{defn}For $R=(R,m_\varphi)$ a $\Comm^!$ algebra, we define the internal Harrison and degree $p$ higher Andr\'e--Quillen cochains as the objects in $\Dd(X)$ defined by
\begin{align*}
	 \mc C^\bullet_{\Har}(R) & = \left( \bigoplus_{n=1}^\infty  \mc P_n^!(\{R\}_{i=1}^n , R)^{\Sh_n}[1-n] \ , \ d_\textup{Har} \right)  &  \in \Dd(X) \\
	 \mc C^{\bullet,\langle p \rangle}_\AQ(R) & = \Homi_{\Dd(R)^*}(\{\bL_R\}_{k=1}^{p+1};R)_{\Ss_{p+1}} &  \in \Dd(X) 
\end{align*}
where $d_\text{Har}:C^\bullet_{\Har}(M,m) \to C^\bullet_{\Har}(M,m_\varphi)[1]$ the Harrison differential determined by $m$, internal to the monoidal category $\DD(X)^!$ enriched over $\DD(X)$, and $\Homi_{\Dd(R)^*}$ denotes the derived analogue of the internal Hom object introduced in Definition \ref{Rstarinthomdefn}, defined for $L_i,M \in \DD(R)$ by
	\[ \Homi_{\Dd(R)^*}(\{L_i\},M)= \bb L \Delta_R^{I\bullet} \circ \Delta^{I!} \Homi_{\Dd(R^{\boxtimes I })}(\boxtimes_{i\in I} L_i, \Delta_*M) \ . \]

Similarly, for $L=(R,b)$ a $\Lie^*$ algebra, we define the internal Chevalley--Eilenberg cochains 
\[ \mc C^\bullet_{CE}(R,b) = \left( \bigoplus_{n=1}^\infty   \mc P_n^*(\{ R\}_{i=1}^n,R)_{\Ss_n}[1-n] \ , \ d_\text{CE} \right) \qquad  \in \Dd(X)   \  , 
\]
for $d_\text{CE}:\mc C^\bullet_{\CE}(R,b)\to C^\bullet_{\CE}(R,b)[1]$ the Chevalley--Eilenberg differential determined by $b$, internal to the pseudo-tensor category $\DD(X)^*$.

\end{defn}

We the following analogue of Corollary \ref{smthchcoro} also follows from Proposition \ref{verysmoothprop}:

\begin{coro}\label{smthchintcoro} Let $R$ be a very smooth $\Comm^!$ algebra. Then $\mc C^\bullet_\AQ(R)$ is acyclic in non-zero degree, and
	\[ \mc H^0_\AQ(R)\cong \Homi_{\DD(R)}(\Omega^1_R,R) = \Theta_R  \ , \]
	the tangent module $\Theta_R  \in \DD(R)$.
	
	More generally, for each $p\geq 0$, $\mc C^{\bullet,\langle p \rangle }_\AQ(R)$ is acyclic except in cohomological degree $p$, and
	\[ \mc H^{p,\langle p \rangle}_\AQ(R)\cong \Homi_{\DD(R)^*}(\{\Omega^1_R\}_{k=1}^{p+1},R)_{\Ss_{p+1}} \cong \wedge^{p+1}_R \Theta_R \ . \]
the module of degree ${p+1}$ polyvector fields $\wedge^{p+1}_R \Theta_R \in \DD(R)$.
\end{coro}

Throughout the remainder of this section, let $R=(R,m,b)$ be a coisson algebra. Following Proposition \ref{chPoiscxprop}, we make the following definition:

\begin{defn} The internal classical cochains on $R$ are defined by
	\[ \mc C^\bullet_c(R) = \left( \bigoplus_{n=1}^\infty \mc P_n^c(\{R\}_{i=1}^{n+1},R)_{\Ss_n}[1-n] \ , \ d_\CE  \right) \qquad \in \Dd(X)\ , \]
where $d_\CE:\mc C^\bullet_c(R) \to \mc C^\bullet_c(R)[1]$ is the Chevalley--Eilenberg differential determined by $\mu=(m,b)\in P_2^c(\{R,R\},R)_{\Ss_2}$
\end{defn}

In analogy with Proposition \ref{Coisbicxprop}, we have:

\begin{prop}\label{Coisbicxintprop} There is a natural bicomplex $\mc C^{\bullet,\bullet}_c(R)=\oplus_{p,q=0}^\infty \mc C^{p,q}_c(R)$ of objects in $\DD(X)$ defined by
	\[\mc C^{p,q}_c(R)=\mc C^{p,q}_c(R,m,b)= \mc P_{p+q+1}^c(\{ R\}_{i=1}^{p+q+1},R)_{\Ss_{p+q+1}}^{\langle p \rangle } \qquad d_h=d_{b}:\mc C^{p,q}_c \to \mc C^{p+1,q}_c \qquad d_v=d_{m}:\mc C^{p,q}_c \to \mc C^{p,q+1}_c \]
	such that the total complex is naturally isomorphic to the internal classical cochains, that is, \[\textup{Tot } \mc C^{\bullet,\bullet}_c(R) \xrightarrow{\cong} \mc C_c^\bullet(R) \ .\]
\end{prop}
\begin{proof}
	The proof is precisely the same as that of Proposition \ref{Coisbicxprop}, \emph{mutatis mutandis}.
\end{proof}

The first non-zero row and column define quotient complexes in $\Dd(X)$
\[ \mc C^\bullet_{c,m}(R)=(\bigoplus_{n=1}^\infty  \mc  P_n^c(\{ R\}_{i=1}^n,R)_{\Ss_n}^{\langle 0 \rangle}[1-n] \ , \ d_{m} ) \qquad \text{and}\qquad \mc C^\bullet_{c,b}(R)=(\bigoplus_{n=1}^\infty \mc  P_n^c(\{ R\}_{i=1}^n,R)_{\Ss_n}^{\langle n-1 \rangle}[1-n]  \ , \ d_{b} )  \ . \]
We have the following analogue of Proposition \ref{firstrcprop}:

\begin{prop}There are natural identifications of complexes
	\[\mc C^\bullet_{c,m}(R) \xrightarrow{\cong} \mc C^\bullet_{\textup{Har}}(R,m) \qquad \text{and}\qquad  \mc C^\bullet_{c,b}(R) \xrightarrow{\cong} \mc C^\bullet_{\textup{CE}}(R,b)  \ . \]
\end{prop}
\begin{proof} The proof is the same as that of Proposition \ref{firstrcprop}, \emph{mutatis mutandis}.
\end{proof}

More generally, for each $p\geq 0$ the $p^{th}$ row of the above bicomplex defines a complex
\[ \mc C^{\bullet,\langle p \rangle}_{c,m}(R)= \left(\bigoplus_{n=p+1}^\infty \mc  P_n^c(\{ R\}_{i=1}^n,R)_{\Ss_n}^{\langle p \rangle} [1-n] \ , \ d_{m} \ \right) \qquad \in \Dd(X) \]
and we have the following analogue of Proposition \ref{AQPoisequivprop}:

\begin{prop}\label{AQPoischequivintprop} For each $p\geq 1$, there are natural identifications of complexes
	\[ \mc C^{\bullet,\langle p \rangle}_{c,m}(R)[p] \xrightarrow{\cong} \mc C^{\bullet,\langle p \rangle}_\AQ(R) \ , \]
	for $R=(M,m)$. In particular, we have $\mc H^{p,\langle p \rangle}_{c,m}(R) \cong \Homi_{\DD(R)^*}(\{\Omega^1_R\}_{k=1}^{p+1},R)_{\Ss_{p+1}} $.
\end{prop}
\begin{proof}
Let $\mc P=\Comm$ and $\mc D=\DD(X)^!$ considered as a $\K$-linear symmetric monoidal category, enriched over the $\K$-linear category $\DD(X)$, so that the $\Comm^!$ algebra $R=(R,m)\in \Alg_{\mc P}(\mc D)$. By Theorem 12.4.3 of \cite{LV} internal to $\mc D$, the object $\bL_R \in R\Mod(\Dd(X)^!)=\Dd(R) $ admits a canonical representative
\[\bL_R \cong \left( R\otimes^! \mc P^\ash(R) \ , \ d_{\pi_k} \right) = \left( R \otimes^! \bigoplus_{n=1}^\infty \Lie(n)^*\otimes_{\K[\Ss_n]} R^{\otimes^! n}[n-1] \ , \ d^{\Har} \   \right) \ ,  \]
where $d^{\Har}$ denotes the differential on Harrison chains on $R$ internal to $\DD(X)$, as follows from Lemma 12.4.4 and Proposition 12.4.5 of \emph{loc. cit.}. In the case $p=1$, we use this to calculate
\begin{align*}
\mc C^\bullet_\AQ(R) & = \Homi_{\Dd(R)}(\bL_R,R) \\
& \cong \left(  \Homi_{\DD(R)}(R \otimes^! \bigoplus_{n=1}^\infty \Lie(n)^*\otimes_{\K[\Ss_n]} R^{\otimes^! n}[n-1] , R ) \ , \ d_{\Har} \right)  \\
& \cong \left( \bigoplus_{n=1}^\infty\Homi_{\DD(X)}( \Lie(n)^*\otimes_{\K[\Ss_n]} R^{\otimes^! n} , R )[1-n]   \ , \ d_{\Har} \right) \\
& \cong \left( \bigoplus_{n=1}^\infty\Homi_{\DD(X)}( R^{\otimes^! n} , R )^{\Sh_n}[1-n]   \ , \ d_{\Har} \right) \\
& = \mc C_\Har^\bullet(R)  & , 
\end{align*}
where $d_\Har=(d^\Har)^*$ denotes the differential on Harrison cochains on $R$ internal to $\DD(X)$, the first isomorphism follows from the fact our chosen representative is a complex of projective objects of $\DD(R)$, the second follows from the induction adjunction, and the third follows from the proof of Proposition 13.1.4 of \cite{LV}, recalled above as Proposition \ref{Harcxprop}.

More generally, for $p\geq 0$ we can analogously calculate the representative of $\bL_R^{\boxtimes p+1}\in \Dd(R^{\boxtimes I})$ as
\begin{align*}
	\bL_R^{\boxtimes p+1} & \cong \left( \  \left( R \otimes^! \bigoplus_{n=1}^\infty \Lie(n)^*\otimes_{\K[\Ss_n]} R^{\otimes^! n}[n-1] \right)^{\boxtimes p+1} \ , \ d^\Har_{\textup{Tot},\langle p \rangle } \ \right) \\ 
	& = \left( \bigoplus_{n=p+1}^\infty R^{\boxtimes p+1} \otimes^! \left(  \bigoplus_{n_1+...+n_{p+1}=n} \boxtimes_{i=1}^{p+1} R^{\otimes^! n_{i}} \right)\otimes_{\K[\prod_{i=1}^{p+1} \Ss_{n_i}]} \bigotimes_{i=1}^{p+1} \Lie(n_i)^*[n-p-1]  \ , \ d^\Har_{\textup{Tot},\langle p \rangle } \ \right) \\ 
\end{align*}
and thus analogously compute the degree $p$ higher Andr\'e--Quillen cochains as
\begin{align*}\hspace*{-3cm}
& \mc  C^{\bullet,\langle p \rangle}_\AQ(R)  = \Homi_{\Dd(R)^*}(\{\bL_R\}_{k=1}^{p+1};R)_{\Ss_{p+1}} \\
& = \bb L\Delta_R^{I\bullet}\Delta^{(p+1)!}\Homi_{\Dd(R^{\boxtimes p+1})}(\bL_R^{\boxtimes p+1}, \Delta^{(p+1)}_*R )_{\Ss_{p+1}} \\
& \hspace{-1.7cm} \cong \bigoplus_{n=p+1}^\infty    \Delta_R^{I\bullet}\Delta^{(p+1)!}_X\Homi_{\DD(R^{\boxtimes p+1})}\left( \bigoplus_{n_1+...+n_{p+1}=n} R^{\boxtimes p+1} \otimes^! \left( \boxtimes_{i=1}^{p+1} R^{\otimes^! n_{i}} \right)\otimes_{\K[\prod_{i=1}^{p+1} \Ss_{n_i}]} \bigotimes_{i=1}^{p+1} \Lie(n_i)^*[n-p-1] ,\Delta^{(p+1)}_* R \right)_{\Ss_{p+1}} \\
& \cong  \bigoplus_{n=p+1}^\infty   \left( \bigoplus_{n_1+...+n_{p+1}=n}\Delta_R^{I\bullet}\Delta^{(p+1)!}_X\Homi_{\DD(X^{\boxtimes p+1})}\left(  \boxtimes_{i=1}^{p+1} R^{\otimes^! n_{i}}  ,\Delta^{(p+1)}_* R \right)\otimes_{\K[\prod_{i=1}^{p+1} \Ss_{n_i}]} \bigotimes_{i=1}^{p+1} \Lie(n_i)[1-n+p] \right)_{\Ss_{p+1}} \\
& \cong  \bigoplus_{n=p+1}^\infty   \left(  \bigoplus_{n_1+...+n_{p+1}=n} \Ind_{\prod_{i=1}^{p+1} \Ss_{n_i}}^{\Ss_n}\left(\mc P^*_{p+1}\left(   \{R^{\otimes^! n_{i}}\}_{i=1}^{p+1}  , R \right)\otimes \bigotimes_{i=1}^{p+1} \Lie(n_i) \right)_{\Ss_n}   \right)_{\Ss_{p+1}} [1-n+p]\\
& \cong  \bigoplus_{n=p+1}^\infty   \left( \left( \bigoplus_{n_1+...+n_{p+1}=n} \Ind_{\prod_{i=1}^{p+1} \Ss_{n_i}}^{\Ss_n}\left(\mc P^*_{p+1}\left(   \{R^{\otimes^! n_{i}}\}_{i=1}^{p+1}  , R \right)\otimes \bigotimes_{i=1}^{p+1} \Lie(n_i) \right) \right)_{\Ss_{p+1}}   \right)_{\Ss_{n}}[1-n+p] \\
	& \cong  \bigoplus_{n=p+1}^\infty \PP_n^c(\{R\}_{i=1}^n,R)^{\langle p \rangle}_{\Ss_{n}}[1-n+p] \\
	& = \mc C^{\bullet,\langle p \rangle}_{c,m}(R)[p] & 
\end{align*}
where we have equipped the cohomologically graded vector spaces given above following the initial isomorphism with $d_\Har^{\textup{Tot},\langle p \rangle }=(d^\Har_{\textup{Tot},\langle p \rangle })^*$, which identifies with the differential $d_m$ on $\mc C^{\bullet,\langle p \rangle}_{c,m}(R)$ as claimed in the final equality, again by Lemma 12.4.4 of \cite{LV}. The final, non-trivial isomorphism follows from the expression in Equation \ref{Pclinteqn}.

\end{proof}

We now give the internal analogue of Definition \ref{gPCchdefn} and Corollary \ref{gPCchcoro}:

\begin{defn}
The (reduced) internal variational Poisson cochains on $R$ is the subcomplex of $\mc C_c^\bullet(R)$ given by
\[ \mc C^\bullet_\PV(R)= \left( \bigoplus_{p=0}^\infty \mc H^{p,\langle p \rangle}_{c,m}(R,m_\varphi)[-p] \ , \ d_{b}  \ \right)  \qquad \in \Dd(X) .\]
\end{defn}

\begin{coro} The internal variational Poisson cochains on $R$ admits an isomorphism of complexes
	\[ \mc C^\bullet_\PV(R)\cong \left( \bigoplus_{p=0}^\infty \mc P^*_{R,p+1}(\{\Omega^1_R \}_{k=1}^{p+1};R)_{\Ss_{p+1}}[-p] \ , \ \tilde d_{b}  \ \right)  \ ,\]
	where $\tilde d_{b}:\mc C_\PV^\bullet(R)\to \mc C_\PV^\bullet(R)[1]$ is determined uniquely by the fact that the following diagram commutes
\begin{equation}\label{tildedbeqn}
		\begin{tikzcd}
		\mc P^*_{R,p}(\{\Omega^1_R \}_{k=1}^{p};R)_{\Ss_{p}} \arrow[r,"\tilde d_b"] \arrow[d," (\cdot)\circ (d_R^{\dR})^{\boxtimes p}"] & \mc P^*_{R,p+1}(\{\Omega^1_R \}_{k=1}^{p+1};R)_{\Ss_{p+1}}  \arrow[d," (\cdot)\circ (d_R^{\dR})^{\boxtimes p+1}"]\\
		\mc P^*_p(\{R\}_{k=1}^p,R)_{\Ss_{p}} \arrow[r,"d_b"]&  	\mc P^*_{p+1}(\{R\}_{k=1}^{p+1},R)_{\Ss_{p+1}} 	\end{tikzcd}  \ .
\end{equation}
\end{coro}

Although we will not use it, note that by Corollary \ref{smthchintcoro}, we obtain the following:

\begin{coro}\label{GOequivchintcoro} Suppose the underlying $\Comm^!$ algebra $R=(M,m)$ is very smooth. Then the inclusion of the subcomplex of internal variational Poisson cochains is a quasi-isomorphism,
	\[\mc C^\bullet_\PV(R) \xrightarrow{\cong} \mc C_c^\bullet(R) \ .\]
\end{coro}

Moreover, we have the following comparison of the internal and global variational Poisson cochains on $R$, which is the first main result of this section:

\begin{prop}\label{prop:deRham of internal PV}
Suppose $X$ is affine and the underlying $\Comm^!$ algebra of $R$ is very smooth and locally free as an $\mc O_X$ module. There is a natural isomorphism of objects in the derived category
\[ h(\mc C^\bullet_\PV(R)) \cong C^\bullet_\PV(R) \ . \]
\end{prop}
\begin{proof}
	We begin by analyzing the underlying cohomologically graded vector space. Note that $\Omega^1_R$ is a projective object of $\DD(R)$ since $R$ is very smooth, and thus $(\Omega^1_R)^{\boxtimes p}\in \DD(R^{\boxtimes p})$ is projective, so that we have
\begin{equation}\label{ndhommatcheqn}
		 \Hom_{\DD(R^{\boxtimes p})}((\Omega^1_R)^{\boxtimes p},\Delta_*R) \cong \Hom_{\Dd(R^{\boxtimes p})}((\Omega^1_R)^{\boxtimes p},\Delta_*R) \ .  
\end{equation}
	Similarly, by Proposition \ref{inthommatchabprop} we have
\begin{equation}\label{ndhommatchinteqn}
			 \Homi_{\DD(R^{\boxtimes p})}((\Omega^1_R)^{\boxtimes p},\Delta_*R) \cong \Homi_{\Dd(R^{\boxtimes p})}((\Omega^1_R)^{\boxtimes p},\Delta_*R) \ .  
\end{equation}
	Thus, we can compute
	\begin{align*}
		\Hom_{\DD(R^{\boxtimes p})}((\Omega^1_R)^{\boxtimes p},\Delta_*R) & \cong	\Hom_{\Dd(R^{\boxtimes p})}((\Omega^1_R)^{\boxtimes p},\Delta_*R) \\
		& \cong (\pi^{\times p})_* \Homi_{\Dd(R^{\boxtimes p})}((\Omega^1_R)^{\boxtimes p},\Delta_*R) \\
		& \cong (\pi^{\times p})_* \Delta_* \Delta^!\Homi_{\Dd(R^{\boxtimes p})}((\Omega^1_R)^{\boxtimes p},\Delta_*R) \\
		& \cong\pi_* \Delta^!\Homi_{\Dd(R^{\boxtimes p})}((\Omega^1_R)^{\boxtimes p},\Delta_*R) \\
		& \cong \pi_*\Delta^!_X\Homi_{\DD(R^{\boxtimes p})}((\Omega^1_R)^{\boxtimes p},\Delta_*R) \\
		& \cong \pi_* \Delta_R^{\bullet}\Delta^!\Homi_{\DD(R^{\boxtimes p})}((\Omega^1_R)^{\boxtimes p},\Delta_*R) \\
		& = \pi_* \mc P^*_{R,p}(\{\Omega^1_R\}_{i=1}^p,R) \ , 
	\end{align*}
	where the first isomorphism follows from Equation \ref{ndhommatcheqn}, the second from Proposition \ref{Rinthommatchprop}, the third from Kashiwara's lemma noting $\Homi_{\Dd(R^{\boxtimes p})}((\Omega^1_R)^{\boxtimes p},\Delta_*R)$ is supported on the diagonal, the fourth by naturality of pushforward, the fifth by Equation \ref{inthommatchabprop}, and the sixth since the object $\Delta^!\Homi_{\DD(R^{\boxtimes p})}((\Omega^1_R)^{\boxtimes p},\Delta_*R)\in\DD(R^{\otimes^! p})$ is supported scheme theoretically on the image of the diagonal embedding $\Delta_R:\mc Y \to \mc Y^{\times p}$.
	 
	 Now, we can explicitly compute the internal Hom object $\mc P^*_{R,p}(\{\Omega^1_R\}_{i=1}^p,R)\in \DD(R)$ by
	 \begin{align*}
	 	\mc P^*_{R,p}(\{\Omega^1_R\}_{i=1}^p,R) & \cong \Delta_R^{(p)\bullet}\Delta^{(p)!}( \left((\Omega^1_R)^{\boxtimes p}\right)^{\circ_{R^{\boxtimes p}}} \otimes^!_{R^{\boxtimes p}} \Delta_* R) \\
	 	& \cong \Delta_R^{(p)\bullet}\Delta^!( (\Theta_R)^{\boxtimes p} \otimes^!_{R^{\boxtimes p}} \Delta_* R) \\
	 	& \cong \Theta_R^{\otimes^{!}_{R}p}\otimes^!_R R \\
	 	& \cong \Theta_R^{\otimes^{!}_{R}p} \ ,
	 \end{align*}
	 where the first isomorphism follows from Proposition \ref{Rinthomdescprop}, the second by the compatibility of duality with the exterior product, and the third noting $\Delta^!$ is symmetric monoidal and applying Kashiwara's lemma.
	 
Note that $\Theta_R^{\otimes^{!}_{R}p} $ is also projective as an object of $\DD(R)$, or equivalently as a module over $R[\mc D_X]$. Since $R$ is locally free as an $\mc O_X$ module, $R[\mc D_X]$ is locally free over $\mc D_X$ and thus locally projective as an object of $\DD(X)$. Thus, $R[\mc D_X]$ is locally (on X) a direct summand of a free $\mc D_X$ module, and similarly $\Theta_R^{\otimes^{!}_{R}p} $ is globally a direct summand of a free $R[\mc D_X]$ module, so that $\Theta_R^{\otimes^{!}_{R}p}$ is itself locally (on X) a direct summand of a free $\mc D_X$ module.
	 
	 In summary, we conclude that $\Theta_R^{\otimes^{!}_{R}p}$ is itself locally projective in $\DD(X)$, and thus by Proposition \ref{locprojaffprop} together with the preceding calculation, we have
	 \[\Hom_{\DD(R^{\boxtimes p})}((\Omega^1_R)^{\boxtimes p},R)  \cong  \pi_* \mc P^*_{R,p}(\{\Omega^1_R\}_{i=1}^p,R)  \cong h(\mc P^*_{R,p}(\{\Omega^1_R\}_{i=1}^p,R) ) \ , \]
	 and similarly for the modules of $\Ss_p$-coinvariants. This completes the identification of the underlying cohomologically graded vector spaces, and the identification of the differentials follows by definition.
\end{proof}

\subsection{Variational cochains on symplectic coisson algebras}\label{varsympsec}

We begin with an alternative interpretation of the complex of internal variational cochains.

Let $R$ be a commutative algebra object in $\D(X)^!$. Following section 1.4.18 of \cite{BD1}, we have:

\begin{prop}\label{Coisalgdprop} A coisson structure $P\in P_2^*(\{R,R\},R)$ is equivalent to a $\Lie^*$ algebroid structure on $\Omega^1_R$ such that the induced map \[\tilde P:=\tau\circ (d_{\mc Y}^\dR \boxtimes \id_R) \in P_2^*(\{R,R\},R)\] is skew-symmetric and satisfies $d_{\mc Y}^\dR \circ \tilde P = b\circ (d_{\mc Y}^\dR \boxtimes d_{\mc Y}^\dR) \in P_2^*(\{R,R\},R)$, where $d_{\mc Y}^\dR:R\to \Omega_R^1$ is the de Rham differential on $\mc Y=\Spec_{\mc D_X} R$.
\end{prop}
\begin{proof} The $\Lie^*$ bracket on $R$ induced by $P$, together with the adjoint action of $R$ on itself by derivations, makes $R\otimes R$ into a $\Lie^*$ algebroid on $\mc Y=\Spec R$. The natural projection map $R\otimes R \to \Omega^1_R$ defined by $r_1\otimes r_2 \mapsto r_1 \cdot  d_{\mc Y}^\dR r_2$ has kernel a $\Lie^*$ ideal so that the bracket descends to $\Omega_R^1$. The induced map $\tilde P$ is canonically identified with the underlying coisson structure $P$ since the action $\tau$ was defined in terms of $P$ acting on a choice of representative, and such a choice is already made by precomposition with $d$. This satisfies the property that $d_{\mc Y}^\dR:R \to \Omega^1_R$ is a map of $\Lie^*$ algebras, since we may compute the $\Lie^*$ bracket in the latter by choosing representatives of the form $1\otimes r$ to compute the bracket on $\Omega_R$; this is equivalent to the condition $d_{\mc Y}^\dR \circ \tilde P = b\circ (d_{\mc Y}^\dR \boxtimes d_{\mc Y}^\dR)$.
	
	Conversely, the $\Lie^*$ algebroid structure evidently determines a 2-ary $\tilde{P}$ according to the formula, which is skew-symmetric by hypothesis, satisfies the Jacobi identity since $b$ does, and acts by derivations on itself in the second argument since $\tau$ acts by derivations, and therefore in the first argument by skew-symmetry.
\end{proof}

In particular, note that we have shown:

\begin{coro}\label{ancmapcoro} Let $R$ be a coisson algebra and $\mc Y=\Spec R$. Then $d_{\mc Y}^\dR:R \to \Omega^1_R$ is a map of $\Lie^*$ algebras, and the adjoint $\Lie^*$ module structure on $R$ is the restriction along $d_{\mc Y}^\dR$ of the action $\tau$ of $\Omega^1_R$ on $R$ by derivations.
\end{coro}

Given a $\Lie^*$ algebroid $\mc L$ on $R$, we make the following definition:

\begin{defn} The (reduced) internal Chevalley--Eilenberg cochains $C^\bullet_{\CE, R}(\mc L)$ is defined by
\[ 
 C^\bullet_{\CE,R}(\mc L) = \left( \bigoplus_{n=1}^\infty   \mc P_{R,n}^*(\{ \mc L\}_{i=1}^n,R)_{\Ss_n}[1-n] \ , \ d_\text{CE,R} \right)  \ , 
\]
where $d_{\CE,R}:C^\bullet_{\CE,R}(\mc L) \to C^\bullet_{\CE,R}(\mc L) [1]$ is the Chevalley--Eilenberg differential determined by the $\Lie^*$ bracket on $\mc L$ and the action of $\mc L$ on $R$ by derivations.
\end{defn}

For example, the internal Chevalley--Eilenberg cochains on $\mc L=\Theta_R$ with its canonical $\Lie^*$ algebroid structure gives the relative de Rham complex viewed as an object of $\Dd(X)$, 
\[ \mc C_{\CE,R}^\bullet(\Theta_R) = ( \Omega^{\bullet\geq 1}_{\mc Y/X}  \ , \ d^\dR_{\mc Y} )[1]  \ .\]

\begin{prop}\label{PVCEprop} Let $R$ be a coisson algebra. There is a canonical isomorphism of complexes
	\[ \mc C_\PV^\bullet(R)  = \mc C_{\CE,R}^{\bullet}(\Omega^1_R) \ ,\]
where $\mc C_{\CE,R}^{\bullet}(\Omega^1_R)$ the internal Chevalley--Eilenberg cochains with respect to the $\Lie^*$ algebroid structure on $\Omega^1_R$ induced by the coisson structure on $R$.
\end{prop}
\begin{proof}
Recall from Corollary \ref{ancmapcoro} that the $\Lie^*$ algebroid structure on $\Omega^1_R$ determined by the coisson structure on $R$ is defined so that $d_R^\dR:R\to \Omega^1_R$ is a map of $\Lie^*$ algebras, and the restriction along $d_R^\dR$ of the action of $\Omega^1_R$ on $R$ by derivations agrees with the adjoint action of the underlying $\Lie^*$ algebra of $R$ on itself.

Thus, by the functoriality of (reduced) internal Chevalley--Eilenberg cochains on a $\Lie^*$ algebra with coefficients in a $\Lie^*$ module, we obtain a natural map $\mc C_\CE^\bullet(\Omega^1_R;R) \to \mc C_\CE^\bullet(R)$, defined by
\[   \mc P_{n}^*(\{\Omega^1_R\}_{i=1}^n, R) \xrightarrow{(\cdot)\circ (d_R^\dR)^{\boxtimes n}} \mc P_{n}^*(\{R\}_{i=1}^n, R) \]
on each cohomologically graded component, where $\mc C_\CE^\bullet(\Omega^1_R;R)$ denotes the (reduced) internal Chevalley--Eilenberg cochains on $\Omega^1_R$ viewed as a plain $\Lie^*$ algebra, with coefficients in the $\Lie^*$ module $R$.

Further, note that by construction we have a natural map $\mc C_{\CE,R}^\bullet(\Omega^1) \to \mc C_\CE^\bullet(\Omega^1_R;R)$ defined by
\[\mc P_{R,n}^*(\{\Omega^1_R\}_{i=1}^n, R) \xrightarrow{\ob} \mc P_{n}^*(\{\Omega^1_R\}_{i=1}^n, R) \]
on each cohomologically graded component, where $\ob:\DD(R)^*\to \DD(X)^*$ denotes the natural forgetful functor.
Moreover, the composition of the preceding maps identifies $\mc C_{\CE,R}^\bullet(\Omega^1) $ with a subcomplex of $\mc C^\bullet_\CE(R)$, so that the differential $d_{\CE,R}$ on the former is uniquely determined by commutativity of the following diagram
\[
\begin{tikzcd}
			\mc P^*_{R,p}(\{\Omega^1_R \}_{k=1}^{p};R)_{\Ss_{p}} \arrow[r,"d_{\CE,R}"] \arrow[d," (\cdot)\circ (d_R^{\dR})^{\boxtimes p}"] & \mc P^*_{R,p+1}(\{\Omega^1_R \}_{k=1}^{p+1};R)_{\Ss_{p+1}}  \arrow[d," (\cdot)\circ (d_R^{\dR})^{\boxtimes p+1}"]\\
			\mc P^*_p(\{R\}_{k=1}^p,R)_{\Ss_{p}} \arrow[r,"d_b"]&  	\mc P^*_{p+1}(\{R\}_{k=1}^{p+1},R)_{\Ss_{p+1}} 	\end{tikzcd}  \ .\]

This is precisely the same diagram of Equation \ref{tildedbeqn} which characterized the differential $\tilde d_b$ on $\mc C^\bullet_\PV$, and thus we obtain the desired result.
\end{proof}

\begin{defn} A coisson structure on $R$ is called symplectic if the induced map $\Omega^1_R\xrightarrow{\cong} \Theta_R$ of $R$ modules internal to $\D(X)^!$ is an isomorphism.
\end{defn}

\begin{eg}The natural coisson structure on $\mc O_{\mc J Y}$ for $Y$ a symplectic variety is symplectic.
\end{eg}

\begin{prop}\label{prop:int PV is de Rham}
Let $R$ be a symplectic coisson algebra. There is a canonical isomorphism of complexes
	\[ 
	\mathcal{C}_\PV^\bullet(R) \cong \mc C^{\bullet}_{\CE,R}(\Theta_R)= (\Omega^{\bullet\ge1}_{\YY/X},d_{\YY}^\dR)[1] \ .  
	\]
\end{prop}
\begin{proof} The anchor map is a map of $\Lie^*$ algebras, so by the symplectic hypothesis we have an isomorphism of $\Lie^*$ algebroids $\Omega^1_R\xrightarrow{\cong} \Theta_R$. Thus, the result follows from Proposition \ref{PVCEprop}.
	
\end{proof}

\subsection{Comparison of relative and absolute de Rham cohomology}\label{reldRsec}

Let $X$ be a smooth variety of dimension $d_X$ and recall from Section \ref{DModsec} that pushforward along the map $\pi:X\to \pt$ defines defines a functor $\pi_*:\Dd(X)\to \Dd(\pt)=\Dd(\K\Mod)$ which for $M\in \Dd^r(X)$ is computed by
\[ \pi_*M=\bb R\pi_\bullet( M^r\otimes^{\bb L}_{\mc D_X} \mc O_X) \cong \bb R\pi_\bullet( M^\ell \otimes_{\mc O_X} \Omega^\bullet_X )[2d_X] \]
where $ M^\ell \otimes_{\mc O_X} \Omega^\bullet_X\in \Dd(\Sh_\z(X))$ denotes the algebraic de Rham complex with coefficients in the complex $M^\ell$. For $M\in \DD(X)$, this gives
\[\pi_*M=\bb R\pi_\bullet( M^r\otimes^{\bb L}_{\mc D_X} \mc O_X) \cong \bb R\pi_\bullet( M^{\ell,\heartsuit} \otimes_{\mc O_X} \Omega^\bullet_X )[d_X] \]
where $ M^{\ell,\heartsuit} \otimes_{\mc O_X} \Omega^\bullet_X\in \Dd(\Sh_\z(X))$ denotes the algebraic de Rham complex with coefficients in the left D-module $M^{\ell,\heartsuit}$.

We define the de Rham functor sheaf functor $\dR_X:\Dd(X) \to \Dd(\Sh_\z(X))$ by
\[ \dR_X(M) = M^\ell \otimes_{\mc O_X} \Omega^\bullet_X [d_X] \ , \]
so that for $M\in \DD(X)$ and in particular $\mc O_X\in \DD^\ell(X)$ we have
\[ dR_X(M) \cong  M^{\ell,\heartsuit} \otimes_{\mc O_X} \Omega^\bullet_X \qquad \text{and in particular} \qquad \dR_X(\mc O_X) = \Omega^\bullet_X \ .\]
Similarly, we define the middle de Rham cohomology sheaf as the functor of abelian categories
\[ \underline{h}:\DD(X)\to \Sh_\z(X) \qquad \text{by}\qquad \underline{h}(M)=M^r\otimes_{\mc D_X} \mc O_X \cong M^r/(M^r\cdot \theta_X) \]
for $M=M^r\in \DD(X)\cong \DD^r(X)$, so that we have $h=\pi_\bullet \circ \underline{h}$.

We have the following analogue of Proposition \ref{locprojaffprop} in this setting:
\begin{coro} Let $M\in \DD(X)$ be a locally projective. Then
\[ \dR_X(M) \xrightarrow{\cong} \underline{h}(M)[-d_X] \ .\]
\end{coro}

\subsubsection{}Throughout the remainder of this section, let $X$ be a smooth curve, $R$ a very smooth $\Comm^!$ algebra on $X$, and $\mc Y=\Spec_{\mc D_X}R$ the corresponding affine D-scheme on $X$, a (pro-finite-type) variety affine over $X$ equipped with a flat Ehresmann connection relative to $X$.

Following 2.8.11 of \cite{BD1}, we let $\Omega^\bullet_{\mc Y/X}\in \Dd(X)$ denote the complex of D-modules
\[\Omega^\bullet_{\mc Y/X} = \left[ R \xrightarrow{d^\dR_R} \Omega^1_R[-1] \xrightarrow{d^\dR_R}\Omega^2_R[-2] \to\cdots \right]  \ . \]
The notation $\Omega_{\mc Y/X}^\bullet$ is chosen because the underlying complex of Zariski sheaves on $\mc Y$ is naturally identified with the usual relative algebraic de Rham complex of $\mc Y\to X$. Note that while each cohomologically graded summand of the complex is an object of $\DD(R)$, this does not define an object of $\Dd(R)$ since the maps are not $R$-linear, analogous to the fact that the usual de Rham complex $\Omega^\bullet_X\in \Dd(\Sh_\z(X))$ is not a complex of quasicoherent sheaves on $X$, despite that fact that each of its analogous summands are objects of $\QC(X)$.

Let $\Omega^\bullet_{\mc Y}\in \DD(\Sh_\z(X))$ denote the absolute de Rham complex of $\mc Y$, and note that the Ehresmann connection on $\mc Y \xrightarrow{\pi_{\mc Y}} X$ determines by definition a splitting
\[ \Omega^1_{\mc Y} \cong \Omega^1_{\mc Y/X} \oplus \pi_{\mc Y}^\bullet \Omega^1_X \qquad \text{and thus} \qquad \Omega^k_{\mc Y} \cong \Omega^k_{\mc Y/X} \oplus( \Omega^{k-1}_{\mc Y/X}\otimes_{\mc O_Y} \pi_{\mc Y}^\bullet \Omega^1_X  )\]
which together define an isomorphism of complexes of Zariski sheaves on $X$
\[  \dR_X( \Omega^\bullet_{\mc Y/X}) \xrightarrow{\cong} \Omega^\bullet_{\mc Y} \ .\]

Similarly, we let $\Omega^{\bullet\geq 1}_{\mc Y/X}\in \Dd(X)$ denote the complex of D-modules given by the stupid truncation at $1$, that is
\[\Omega^{\bullet\geq 1}_{\mc Y/X} = \sigma_1\Omega^\bullet_{\mc Y/X} = \left[  \Omega^1_R[-1] \xrightarrow{d^\dR_R}\Omega^{2}_R[-2] \to\cdots \right]  \ . \]
The above isomorphism restricts to an isomorphism of complexes
\[ \dR_X(\Omega^{\bullet \geq 1}_{\mc Y/X}) \xrightarrow{\cong} \left[  \Omega^1_{\mc Y/X}[-1] \xrightarrow{d^\dR_{\mc Y}} \Omega^{\bullet \geq 2}_{\mc Y} \right] \ . \]
In particular, note that we have the compatible inclusions of complexes
\begin{equation}\label{subcxeqn}
	 \dR_X(\Omega^{\bullet \geq 1}_{\mc Y/X}) \subseteq \dR_X( \Omega^\bullet_{\mc Y/X}) \xrightarrow{\cong} \Omega^\bullet_{\mc Y} 
\end{equation}   
and thus we obtain a canonical map of complexes
\[ \bb R \pi_\bullet \dR_X(\Omega^{\bullet \geq 1}_{\mc Y/X}) \to  \bb R \pi_\bullet  \Omega^\bullet_{\mc Y} =C^\bullet_\dR(\mc Y) \ .  \]

\begin{prop}\label{hdRcompprop} Suppose in addition that $R$ is locally free as an $\mc O_X$ module. Then
\[ \dR_X(\Omega^{\bullet \geq 1}_{\mc Y/X})  \cong \underline h(\Omega^{\bullet \geq 1}_{\mc Y/X})[-1] \]
If in addition $X$ is affine, we obtain canonical isomorphisms of objects in $\Dd(\K\Mod)$
\[ h(\Omega^{\bullet \geq 1}_{\mc Y/X})[-1]=\pi_\bullet \underline h(\Omega^{\bullet \geq 1}_{\mc Y/X})[-1]  \cong \bb R\pi_\bullet  \underline h(\Omega^{\bullet \geq 1}_{\mc Y/X})[-1]  \cong  \bb R\pi_\bullet \dR_X(\Omega^{\bullet \geq 1}_{\mc Y/X}) \ , \]
so that in particular we have a natural map in the derived category
\begin{equation}\label{compmapeqn}
	 h(\Omega^{\bullet \geq 1}_{\mc Y/X})[-1] \to \Omega^\bullet_{\mc Y} \ .
\end{equation}
\end{prop}

Now suppose $\mc Y=\mc J(Y)$ for $Y$ a smooth affine variety and $X=\bb A^1$. We have:

\begin{prop} \label{prop:dR of arc is dR}
There is a natural quasi-isomorphism
\[ C^\bullet_\dR(Y) \xrightarrow{\sim}  C^\bullet_\dR(\mc J(Y)) \ . \]
\end{prop}
\begin{proof} Let $\mc J(Y)_0=\mc J(Y)\times_{\bb A^1}\{0\}$ be the fiber of $\mc J(Y)$ at the origin, and recall we have a natural identification $\mc J(Y)_0=\Maps(\bb D_0,Y)$ where $\bb D_0 = (\bb A^1)^\wedge_{\{0\}}$ denotes the formal disk at the origin. We will construct algebraic homotopy equivalences between $\mc J(Y)$ and $\mc J(Y)_0$, and between $\mc J(Y)_0$ and $Y$, so that the claim follows from invariance of algebraic de Rham cohomology of varieties under homotopy equivalences.

Since $Y\times \bb A^1 \to \bb A^1$ is canonically trivialized as a scheme over $\bb A^1$, $\mc J(Y)$ admits a natural identification $\mc J(Y)= \mc J(Y)_0\times \bb A^1$, and we let $r_0:\mc J(Y) \to \mc J(Y)_0$ denote the induced projection map, which corresponds geometrically to restriction to the origin of an algebraic family of jets of maps to $Y$. We let $\iota_0:\mc J(Y)_0 \to \mc J(Y)$ denote the natural inclusion, and note that $r_0\circ \iota = \id_{\mc J(Y)_0}$. Further, define the algebraic homotopy
\[ h:\bb A^1\times \mc J(Y) \to \mc J(Y) \qquad \text{by}\qquad (t,\varphi_0,x) \mapsto (\varphi_0,tx) \ .  \]
Evidently we have $h|_{\{0\}\times \mc J(Y)}=\iota_0\circ r_0$ and $h|_{\{1\}\times \mc J(Y)}=\id_{ \mc J(Y)}$, so that $r_0$ and $\iota_0$ define an algebraic homotopy equivalence, as desired.

Let $\ev_0:\mc J(Y)_0\to Y$ denote the evaluation at the closed point $0\in \bb D_0$, so that $\ev_0(\varphi_0)=\varphi_0(0)$, where on the right hand side we have written $\varphi_0 \in \mc J(Y)_0$ in terms of the corresponding map $\varphi_0:\bb D_0 \to Y$. There is also a canonical constant jet extension $c:Y \to \mc J(Y)_0$ mapping a closed point $y \in Y$ to the constant jet $\varphi_y:\bb D_0 \to \{y\} \into Y$. Note that we have $\ev_0\circ c = \id_Y$. Further, define the algebraic homotopy
\[ h:\bb A^1\times \mc J(Y)_0 \to \mc J(Y)_0 \qquad \text{by}\qquad \varphi_0\mapsto \varphi_0\circ m_t\]
where $m_t:\bb D_0\to \bb D_0$ denotes the restrition of the multiplicative monoid action $m:\bb A^1\times \bb D_0\to \bb D_0$ to $\{t\}\times \bb D_0$. Evidently we have $h|_{\{0\}\times \mc J(Y)_0}=c\circ \ev_0$ and $h|_{\{1\}\times \mc J(Y)}=\id_{ \mc J(Y)_0}$, so that $\ev_0$ and $c$ define another algebraic homotopy equivalence, as desired.

\end{proof}

We now combine the preceding two propositions to establish the main result of this section:

\begin{prop}\label{jetsdRprop} Let $X=\bb A^1$ and $\mc Y=\mc J(Y)$ for $Y$ a smooth affine variety. Then the map $	 h(\Omega^{\bullet \geq 1}_{\mc Y/X})[-1] \to \Omega^\bullet_{\mc Y} $ of Equation \ref{compmapeqn} induces isomorphisms in cohomology
\[ H^k(h(\Omega^{\bullet \geq 1}_{\mc Y/X})[-1] )\xrightarrow{\cong} H^k_\dR(Y)   \qquad \text{for $k \geq 3$.} \]
\end{prop}

\begin{proof}
	
By Proposition \ref{hdRcompprop}, the map of Equation \ref{compmapeqn} is modelled by the inclusion of subcomplexes of Equation \ref{subcxeqn}. This inclusion determines the exact triangle of complexes
\[\dR_X(\Omega^{\bullet \geq 1}_{\mc Y/X}) \to \Omega^\bullet_{\mc Y} \to \dR_X(\mc O_{\mc Y}) \ , \]
and the cokernel $\dR_X(\mc O_{\mc Y})$ is concentrated in cohomological degree $0$ and $1$, so that its cohomology is vanishing in degree greater than or equal to $2$. Thus, the inclusion defines an isomorphism in cohomology
\[ H^k(\dR_X(\Omega^{\bullet \geq 1}_{\mc Y/X})) \xrightarrow{\cong} H^k_\dR(\mc Y) \]
for $k\geq 3$. By Proposition \ref{prop:dR of arc is dR}, we have $H^k_\dR(\mc Y) \cong H^k_\dR(Y)$ and thus we obtain the desired result.
\end{proof}

\section{Examples}\label{sec:examples}

In this section, we explain the application of the results of Subsection \ref{quantchegsec} to several families of examples. We begin with the case of chiral quantizations of affine symplectic varieties, applying Theorem \ref{mainchtheo} to establish the chiral analogue of Theorem \ref{QuantdRthm}. Next, we return to the various examples from Section \ref{ssec:chiral quantization intro} and describe their deformation theory, culminating in a proof of the rigidity of boundary Virasoro minimal models in Subsection \ref{ssec:minimal models}.

All the examples we consider are naturally vertex algebras, and thus in terms of D-modules we fix $X=\IA^1$ and work in the weakly translation invariant setting.

\subsection{Chiral quantizations of affine symplectic varieties}\label{sympsec} Throughout this section, let $Y$ be an affine, symplectic variety, $\mc J Y$ the affine D-scheme of jets to $Y$, and recall there is a natural Coisson algebra structure on $\mc O_{\mc J Y}$, as explained in the beginning of Section \ref{varsec}. Further, we denote by $\varphi\in \Alg_{\Lie}(P^{\IG_a,c}(\mc O_{\mc J Y}))$ the vertex Poisson algebra structure on $\mc O_{\mc J Y}$.

We begin by establishing the chiral analogue of Theorem \ref{QuantdRthm}:

\begin{theo}\label{symptheo}
	  Suppose $\varphi_{\hbar/\hbar^m}$ is a quantization of $\varphi$ defined modulo $\hbar^m$. If the obstruction class vanishes, then the space of extensions of $\varphi_{\hbar/\hbar^m}$, up to isomorphism,
	\[
	\fQuant^\ch(\ik[\hbar/\hbar^{m+1}])\times_{\fQuant_{\varphi}^\ch(\ik[\hbar/\hbar^m])}\{\varphi_{\hbar/\hbar^m}\}
	\]
	is a torsor for $H^3_{\dR}(Y)$, the third de Rham cohomology of $Y$.
\end{theo}
\begin{proof}
	This follows directly from Theorems \ref{DefObsQuantChThm} and \ref{mainchtheo}.
\end{proof}

Similarly, by Corollaries \ref{Coisquantexistthm} and\ref{Coisquantrigidthm}, together with Theorem \ref{mainchtheo}, we have:

\begin{coro} Suppose $H^4_\dR(Y)=\{0\}$. Then $\Quant^\ch_\varphi(\ik\fph)$ is nonempty.
\end{coro}

\begin{coro} Suppose $H^3_\dR(Y)=\{0\}$. Then if there exists $\vph\in\Quant^\ch_\varphi(\ik\fph)$, a formal chiral quantization of $Y$, every such quantization is isomorphic to $\vph$.
\end{coro}

Now suppose $Y$ is an affine symplectic variety with an action of $\IG_\hbar$ that scales the symplectic form with unit weight. Then $\OO(Y)$ is a graded Poisson algebra and $\OO(\JJ Y)$ is a graded vertex Poisson algebra. The grading on $\OO(Y)$ induces one on $\Omega_Y$, and thus on the de Rham cohomology, which we denote $H^\bullet_\dR(Y) = \bigoplus_{m\in \IZ}H^{\bullet,\langle m \rangle}(Y)$.

\begin{theo}
	Suppose $H^{3,\langle m \rangle}_\dR(Y) =0 $ unless $m=m_0\in \IN$. Then if there exists $\varphi_\hbar \in \fQuant^\ch_{\varphi,\gd}(\ik\fph)$, a graded chiral quantization of $Y$, we have
	\[
	\fQuant^\ch_{\varphi,\gd}(\ik\fph)  \cong H^3_\dR(Y)~.
	\]
\end{theo}

%
	%
%

\subsection{Chiral differential operators on smooth affine varieties}

Let $Y$ be a smooth affine variety with vanishing second Chern class, then from the discussion of \ref{sssec:cdos intro} we know that for each $\alpha\in \Omega^{3,cl}(Y)$ the homogenized cdos $\D^{\ch}_{\hbar,\alpha}(Y)$ provide a graded chiral quantization of $T^*Y$.

\begin{prop}[\cite{GMS2,GMS3}]
 Suppose $\alpha,\alpha'\in \Omega^{3,cl}(Y)$ with $\alpha-\alpha'=d\eta$ for some $\eta\in \Omega^{2}(Y)$, then
 \[
 \D^\ch_{\hbar,\alpha}(Y) \cong \D^\ch_{\hbar,\alpha'}(Y)
 \]
 Thus, the moduli space of cdos on $Y$, up to equivalence, is a torsor for $H^3_{\rm dR}(Y)$.
\end{prop}
\begin{proof}
By Proposition \ref{prop:hadic cdos are cdos}, ${\rm Fil}(\D^\ch_{\hbar,\alpha}(Y))\cong \D^\ch_\alpha(Y)$. From \cite{GMS2}, we know that $\D^\ch_\alpha(Y)\cong \D^\ch_{\alpha'}(Y)$ as ordinary vertex algebras. The result then follows since ${\rm Fil}$ reflects isomorphisms by Proposition \ref{prop:Rees for hadic VOA}.
\end{proof}

\begin{theo}
	Let $Y$ be a smooth affine variety with vanishing second Chern class. Then $T^*Y$ admits a graded chiral quantization and any graded chiral quantization of $T^*Y$ is isomorphic to a vertex algebra of cdos on $Y$.
\end{theo}
\begin{proof}
	Since the second Chern class of $Y$ vanishes, there exists at least one algebra of chiral differential operators on $Y$---which as we recalled in Section \ref{sssec:cdos intro}  provides a graded chiral quantization of $T^*Y$. 

	Now $H^3_\dR(T^*Y)\cong H^3_{\dR}(Y)$ sits in a single $\hbar$-degree, with respect to the $\IG_\hbar$ scaling the symplectic form. Thus, by Corollary \ref{cor:unique graded quantization}, the moduli space of chiral quantizations of $T^*Y$ is a torsor for $H^3_{\dR}(Y)$. Since the space of cdos is a torsor for the same group, we have the desired result.
\end{proof}

\begin{rmk}
	It is well known that cdos on $Y$ exist only if the second Chern class, $c_2\in H^4_{\rm dR}(Y)$ vanishes. We have shown that the obstruction to the existence of a chiral quantization lives in $H^4_{\rm dR}(T^*Y)\cong H^4_{\rm dR}(Y)$, the same cohomology group. While it is natural to ask how the two are related we leave this to future work.
\end{rmk}

\subsection{Affine current algebras}

Let $\gf$ be a simple Lie algebra, and let $\gf^*$ be its linear dual. Recall that for any $k\in \ik$, the homogenized universal affine vertex algebra $V^k_\hbar(\gf)$ is a graded chiral quantization of $\gf^*$.

\begin{prop}\label{prop:coisson cohomology of akm}
	The first coisson cohomology group of $\OO(\JJ\gf^*)$ is one-dimensional, \ie,
	\[
	H^1_{\rm c}(\OO(\JJ \gf^*)) = \ik~.
	\]
\end{prop}
\begin{proof}
	Since the underlying $\Comm^!$-algebra $\OO(\JJ\gf^*)$ is very smooth, by Corollary \ref{GOequivchcoro},
	\[
	H^\bullet_{c}(\OO(\JJ \gf^*))\cong H^\bullet_{\rm PV}(\OO(\JJ \gf^*))~.
	\]
From \cite{Bakalov:2020computation}, $H^1_{\rm PV}(\OO(\JJ\gf^*)))\cong H^2_{\Lie}(\gf,\ik)_\oplus H^3_{\Lie}(\gf,\ik)$. Since $\gf$ is simple, $H^2_{\Lie}(\gf,\ik)=0$ by the Whitehead Lemma and $H^3_{\Lie}(\gf,\ik)=\ik$.
\end{proof}

\begin{theo}
	Any graded chiral quantization of $\gf^*$ is isomorphic to $V^k_\hbar(\gf^*)$ for some $k\in \ik$. 
\end{theo}
\begin{proof}
	By Proposition \ref{prop:coisson cohomology of akm}, the coisson cohomology is in a single graded degree and so the result follows from Corollary \ref{cor:torsor if cohomology is in single graded degree}.
\end{proof}

This is similar to the result of \cite{Kovalchuk:2024rwd} on the moduli space of first order deformations of affine current algebras, though they use the cohomology theory of \cite{Huang:2010ud}.

\subsection{The Virasoro algebra}
Recall the vertex Poisson algebra $V = \JJ \ik[T]$ with $\lambda$-bracket
\[
[T_\lambda T] = 2\lambda T +\partial T~,
\]
introduced as the classical limit of the Virasoro algebra in \ref{ssec:dfn virasoro}.

\begin{prop}
	The vertex Poisson algebra $V$ admits a one parameter family of graded quantizations and any such quantization is isomorphic to the homogenized Virasoro algebra, $\Vir^c_{\hbar}$, for some $c\in \ik$.
\end{prop}
\begin{proof}
 As a $\Comm^!$-algebra, $V $ is very smooth and so by Corollary \ref{GOequivchcoro},
 \[
 H^\bullet_{\rm c}(V) \cong H^\bullet_{\rm PV}(V)~.
 \]
 From \cite{Bakalov:2020computation}, 
 \[
 H^1_{\rm PV}(V) = \ik
 \]
 and so by Corollary \ref{cor:torsor if cohomology is in single graded degree}, the moduli space of quantisations is isomorphic to ${\bb A^1}$.
\end{proof}

\subsection{Boundary minimal models}\label{ssec:minimal models}

Recall from \ref{ssec:dfn virasoro} that when $c = c_{2,2n+1}$ the homogenized Virasoro algebra develops a vertex ideal $\II_{\hbar,n}$ defining the boundary minimal model
\[
\MM_{\hbar,n} = \Vir^{c_{2,2n+1}}_\hbar/ \II_{\hbar, n}~.
\]
The classical limit of $\MM_{\hbar,n}$ can be identified with the arc algebra $M_n = \JJ( \ik[T]/T^n)$, of which it is a graded quantization.

In the rest of this section we shall show that the boundary Virasoro minimal models, $\MM_n$, are rigid, \ie, they admit no deformations. 

Unlike the examples we considered previously, $\JJ(\ik[T]/T^n)$ is not smooth, let alone very smooth. As a result, its Harrison cohomology is non-vanishing and, moreover, since $\JJ(\ik[T]/T^n)$ is not symplectic, we cannot relate its variational cohomology to the de Rham cohomology. 

\subsubsection{}
Our main computational tool will be the bicomplex from Proposition \ref{Coisbicxprop}, $ C_{c}^{\bullet,\bullet}(M_n) $, whose total cohomology is isomorphic to the classical cohomology. The total cohomology can be computed by the first quadrant spectral sequence whose first page is
\begin{equation}
E_1^{p,q} = H_{\AQ}^q(C^{p,q}_c(M_n))~.
\end{equation}
for $p,q\ge0$, where we, implicitly, make use of the quasi-isomorphism between the Harrison and Andr\'e--Quillen cohomologies.

Note that we have the identifications
\[
E_1^{0,q} = H^q_{\AQ}(M_n)~, \qquad \text{ and }\qquad E_2^{p,0}=C^\bullet_{\rm PV}(M_n)~.
\]

The spectral sequence converges to the total cohomology and since it is first quadrant, the first classical cohomology group can be computed by the entries of the third page.
\begin{equation}\label{eq:H1 coisson is fixed by E3}
H^1_{c}(M_n) \cong E_3^{1,0}\oplus E_3^{0,1} = H_{\PV}^{1}(M_n)\oplus E_3^{0,1}~,
\end{equation}
where we make use of the identification $E_3^{1,0}= E_2^{1,0}=H_{\PV}^1 (M_n)$.

We will also make extensive use of Andr\'e--Quillen and variational cohomology with coefficients not necessarily in the adjoint module. The reader can find a comprehensive, concrete definition of these in \cite{BDHKV21} but, for the sake of completeness, let us recall some basic definitions.

For $A\in\Comm^!(\ik[\del]\Mod)$ is a commutative $\ik$-algebra with derivation and $B\in A\Mod(\ik[\del]\Mod)$, a module over $A$, we write
\[
C_\AQ^\bullet(A,B) = \Hom_{A\Mod(\ik[\partial]\Mod)}(\mathbb{L}^\bullet_A,B)~,
\]
where $\mathbb{L}_A^\bullet$ is the cotangent complex of $A$. Moreover, 
\[
H^0_{\AQ}(A,B) \cong \Der_\del(A,B)~.
\]
by construction.

Similarly for $A$ a vertex Poisson algebra and $B$ a vertex Poisson module over $A$, we write
\[
C_\PV^n(A,B) = \Hom_{\otimes_{i=1}^n A[\del_i] \Mod}((\Omega_A^1)^{\otimes n},B\otimes_{\ik[\del]}\ik[\lambda_1,\dots,\lambda_n])~,
\]
with differential $d_{PV}$ given by Schouten bracket against the Poisson bracket.

More concretely, a cochain $f_{\lambda_1,\dots,\lambda_n}\in C_\PV^n(A,B)$ is a multilinear $f_{\lambda_1,\dots,\lambda_n}\in \Hom_\ik(A^n,B[\lambda_1,\dots,\lambda_n])$ such that for any $a_1,\dots,a_n,c\in A$
\begin{enumerate}[(i)]
	\item $f(a_1,\dots,\partial a_i,\dots a_n) = - \lambda_if(a_1,\dots,a_i,\dots,a_n)~$
	\item $f(a_{\sigma^{-1}(1)},\dots a_{\sigma^{-1}(n)}) = \textup{sgn}(\sigma)f(a_1,\dots,a_n)$ for any $\sigma \in \Ss_n$.
	\item $f(a_1,\dots, c a_i,\dots, a_n) = cf_{\lambda_1,\dots,\lambda_i+\del,\dots,\lambda_n}(a_1,\dots,a_i,\dots,a_n)_{\leftarrow} + a_i f_{\lambda_1,\dots,\lambda_i+\del,\dots,\lambda_n}(a_1,\dots,c,\dots,a_n)_{\leftarrow} $, where the subscript $\leftarrow$ indicates that $\lambda_i+\del$ acts towards the left.
\end{enumerate}
with differential 
\begin{align*}
(d_\PV f)_{\lambda_1,\dots,\lambda_{n+1}}(a_1,\dots,a_{n+1}) & = (-1)^n\sum_{i}(-1)^i \bigg[v_i\ _{\lambda_i} f_{\lambda_1,\dots, \hat{\lambda_i},\dots,\lambda_{n+1}}(a_1,\dots,\hat{a_i},\dots,a_n)\bigg]\\
& + (-1)^{n+1} \sum_{i<j}f_{\lambda_i+\lambda_j,\lambda_1, \dots,\hat{\lambda_i},\dots,\hat{\lambda_j},\dots,\lambda_{n+1}}([a_i\ _{\lambda_i}a_j],a_1,\dots,\hat{a_i},\dots,\hat{a_j},\dots,a_{n+1})~,
\end{align*}
where the hat denotes omission.

Breaking with our former conventions, we shall now explicitly denote the module of coefficients in the following.

\subsubsection{The Andr\'e--Quillen cohomology}

We have an inclusion
\[
\iota: C_{\AQ}^\bullet(M_n,M_n) \rightarrow C^\bullet_{\AQ}(V,M_n)~,
\]
with image equal to the subcomplex of cochains in $C_{\AQ}^\bullet(V,M_n)$ that vanish if any argument is in $I_n$. The cokernel complex, $\coker(\iota)$, admits a nice description in low degrees. Namely,
\[
\coker^0(\iota) \cong \Hom_{\ik[\partial]}(I_n,M_n)~,
\]
The quotient map $C^0_{\AQ}(V,M_n)=\Hom_\partial(V,M_n) \rightarrow \Hom_{\partial}(I_n,M_n)$ is given by restricting the domain. For $f\in \Hom_{\partial}(I_n,M_n)$, the induced differential on the cokernel is 
\[
df(a,b) = af(b) + f(a)b - f(ab)~,
\]
for $a,b\in V$. For $f$ to be closed, it must vanish on $I_n$, and so being closed is equivalent to
\[
f(ab) \overset{!}{=} af(b)
\]
for $a\in V$ and $b\in I_n$, \ie, $f$ is $V$-linear and we can identify
\begin{equation}\label{eq:cocycle in Harrison cokernel}
H^0(\coker(\iota)) = \Hom_V(I_n, M_n)~.
\end{equation}

The short exact sequence
\[
0 \rightarrow C_{\AQ}^\bullet(M_n,M_n)\rightarrow C_{\AQ}^\bullet(V, M_n)\rightarrow  \coker^\bullet(\iota)\rightarrow 0
\]
induces the usual long exact sequence in cohomology:
\[
0 \rightarrow  \Der_\partial(M_n,M_n) \rightarrow  \Der_\partial(V,M_n) \rightarrow  H^0(\coker(\iota)) \rightarrow  H^1_{\AQ}(M_n,M_n)\rightarrow 0 \rightarrow  \dots~,
\]
where we used the fact that $H^1_{\AQ}(V,M_n)=0$ since $V$ is very smooth.

Thus, 
\begin{equation}\label{eq:identifying Harrison of quotient}
H^1_{\AQ}(M_n,M_n) \cong H^0(\coker(\iota)) / \Der_\partial(V,M_n)~,
\end{equation}
\ie, the cocycles in $\coker(\iota)$ that do not arise from the restriction of a derivation on $V$. Concretely, for every such $f\in H^0(\coker(\iota))$, $d_{\AQ}f$ is a cocycle in $C^1_{\AQ}(M_n,M_n)$ and $[d_{\AQ} f]\in H^1_{\AQ}(M_n,M_n)$.

\begin{lemma}\label{lem:E2 Harrison vanishes for min model}
The $(0,1)$-entry of the second page of the spectral sequence vanishes, \ie,
\[
E_2^{0,1} = H^0_{\rm PV}(H^1_{\AQ}(M_n,M_n)) =0 ~,
\]
and so $E_3^{0,1}=0$.
\end{lemma}
\begin{proof}
In this presentation, the kernel of the $E_1$-differential $d_1^{0,1}:H^1_{\AQ}(M_n,M_n) \rightarrow H^{1,\langle 1 \rangle}_{\rm AQ}(M_n,M_n)$ admits a nice description. A class $[d_{\AQ} f]\in H^1_{\AQ}(M_n,M_n)$, with $f\in H^0(\coker(\iota))$, is closed for this differential if and only if there exists some $f'\in C^{0,1}_{\rm c}(M_n,M_n)$ such that
\[
d_{\rm PV} d_{\AQ} f = d_{\AQ} f' ~.
\] 
We will show that $f'$ exists only if $[d_{\AQ} f]= H^1_{\AQ}(M_n,M_n)$, \ie, we will show that the kernel of $d_1^{0,1}$ is trivial.

A naive ansatz for $f'$ is just $d_{\rm PV} f$ since $d_{\AQ} d_{\rm PV} f = -d_{\rm PV}d_{\AQ} f$. However, $d_{\rm PV} f$ generically fails to be a well-defined cochain for $M_n$ since it might fail to vanish on the ideal. Thus, the existence of $f'$ is equivalent to the existence of a $\delta f \in C^{0,\langle 1\rangle}_{\rm AQ}(V,M_n)$ such that  
\[
f'=d_{\rm PV}f -\delta f \in C_{\rm AQ}^{0,\langle 1\rangle}(M_n,M_n)~.
\]
So we want to show that $\delta f$ can only exist if $[d_{\AQ} f]=0\in H^1_{\AQ}(M_n,M_n)$. Since,
\[
d_{\AQ} ( d_{\rm PV} f - f') =0~,
\]
we can restrict $\delta f\in C^1_{\rm  PV}(V,M_n)$. Since $f'$ must vanish on $I_n$, in particular it should vanish when one of the arguments is $T^{n+k}$ for $k\ge0$,\ie
\[
f'(T,T^{n+k})=(d_{\rm PV}f-\delta f)(T,T^{n+k}) \overset{!}{=}0~.
\]
But since $\delta f$ satisfies the right Leibniz rule, $\delta f(T,T^{n+k}) =(n+k)T^{n+k-1}\delta f(T,T)$ and we see that
\begin{equation}\label{eq:order constraint on f}
d_{\rm PV}f(T,T^{n+k}) \overset{!}{=}
\begin{cases}
	nT^{n-1} \delta f(T,T) \text{ for } k=0~,\\
	0 \text{ for } k>0~.
\end{cases}
\end{equation}
Expanding for $k>0$,
\[
d_{\rm PV}f (T,T^{n+k}) = [T_\lambda f(T^{n+k})] - f([T_\lambda T^{n+k}]) =  [T_ \lambda f(T^{n+k})] - 2(n+k)\lambda f(T^{n+k}) - \partial f(T^{n+k})~.
\]
For $k=1$, $d_\PV f(T,T^{n+1})\overset{!}{=}0$ implies that
\[
 [T_{\lambda} f(T^{n+1})] \overset{!}{=} 2(n+1)\lambda f(T^{n+1}) + \partial f(T^{n+1}) ~.
\]
So, $f(T^{n+1})\in M_n$ must have conformal weight $2(n+1)$ and its $\lambda$-bracket with $T$ has no terms of order $\lambda^2$ or higher---but this implies that $f(T^n+1)$ has no derivatives. The only candidate for $f(T^{n+1})$ is, therefore, a scalar multiple  of $T^{n+1}$, which vanishes in $M_n$. Thus, we conclude that $f(T^{n+1})=0$ and by \eqref{eq:cocycle in Harrison cokernel}, 
\[
f(T^{n+k})=0 ~, \text{ for }k\ge1~.
\]
Moreover, $f(T^{n+1})=T f(T^n)=0$ and so $f(T^n)$ must be of the form
\[
 f(T^n) = c\, T^{n-2} \partial T + T^{n-1}g
\]
for some $c\in \ik$ and $g\in M_n$ since these are the only complementary zero divisors of $T$ in $M_n$. To order $\lambda^2$,
\[
d_{\rm PV}f(T,T^n) = (-c\, T^{n-2}\partial T )\lambda + (\Delta_g -1)T^{n-1}g \lambda + O(\lambda^2)~.
\]
The consistency conditions of \eqref{eq:order constraint on f} state that $d_{\rm PV}f(T,T^n)$ must be in the algebraic ideal generated by $T^{n-1}$, which forces $c=0$. But, $f(T^{n-1}) = T^{n-1} g$ is actually the restriction of the derivation 
\[
F = \frac1n \sum_m  \partial^m g \frac{\partial}{\partial \partial^m T}\in \Der_{\partial}(V,M_n)~,
\]
to $I_n$. Therefore, under the identification of \eqref{eq:identifying Harrison of quotient}, $[d_{\AQ} f]=0\in H^1_{\AQ}(M_n,M_n)$ and so the kernel of the $E_1$-differential is trivial, as desired.
\end{proof}

\subsubsection{The variational cohomology}

To compute the variational cohomology of $M_n$, we make use of the same strategy: once again, we have an inclusion of complexes 
\[
C^\bullet_{\rm PV}(M_n,M_n) \xrightarrow{\iota} C^\bullet_{\rm PV}(V,M_n)~,
\]
with image equal to the subcomplex of cochains in $C^\bullet_{\rm PV}(V,M_n)$ vanishing on $I_n$. We write $\coker(\iota)$ for the cokernel complex of $\iota$. 
In degree zero, we have the identification
\[
\coker^0(\iota) \cong \Der_{\partial}(V,M_n)|_{I_n}
\]
with the quotient map $C^0_{\rm PV}(V,M_n)\rightarrow \Der_{\partial}(V,M_n)|_{I_n}$ given by restriction of vector fields. The cohomology $H^0(\coker(\iota))$ consists of $f\in\Der_{\partial}(V,M_n)|_{I_n}$ such that 
\begin{equation}\label{eq:cocycles of variational cokernel}
[a_\lambda f(b)] = f([a_\lambda b])~,
\end{equation}
for $a\in V$ and $b\in I_n$, \ie, the cocycles intertwine the $\Lie^*$-action of $V$.


\begin{lemma}\label{lem: Ham vect fields of ideal}
	The zeroth cohomology of the cokernel vanishes, \ie,
	\[
	 H^0(\coker(\iota)) =0
	\]
\end{lemma}
\begin{proof}
 Since $\coker^0(\iota)\cong \Der(V,M_n)|_{I_n}$, we can represent our putative cocycle as $F|_{I_n}$ for some derivation $F\in \Der_{\partial}(V,M_n)$. Then by \eqref{eq:cocycles of variational cokernel} this cochain is closed if and only if it intertwines the $\Lie^*$-action of $V$, and so
 \[
 [T_\lambda F(T^n)] \overset{!}{=} F[T_\lambda T^n]= 2n \lambda F(T^n) + \partial F(T^n) ~.
 \]
 This in turn implies that $F(T^n)$ has conformal weight $2n$ and has no derivatives, which implies that $F(T^n)$ is a scalar multiple of $T^n$, which vanishes in the quotient. Therefore, $F|_{I_n} =0\in \coker^0(\iota) $ as desired.
\end{proof}

\begin{lemma}[\cite{Bakalov:2020computation}]\label{lem:Virasoro variational is 1d}
	The first variational cohomology $H^1_{\rm PV}(V,M_n)$ is one dimensional and spanned by the class of $C\in C^1_{\rm var}(V,M_n)$, with
	\[
	C_\lambda(T,T) = \lambda^3~.
	\]
\end{lemma}
\begin{proof}
	The proof is essentially the same as that of Theorem 4.17 in \cite{Bakalov:2020computation}, which computes $H^1_{\PV}(V,V)=\ik$. The change of coefficients to $M_n$ does not affect the proof strategy overmuch.
\end{proof}

\begin{prop}\label{prop:var of min model vanishes}
	The first variational cohomology of $M_n$ vanishes, \ie,
	\[
	H^1_{\rm PV}(M_n,M_n) =0~.
	\]
\end{prop}
\begin{proof}
	The short exact sequence $0 \rightarrow C^\bullet_{\rm PV}(M_n,M_n) \xrightarrow{\iota} C^\bullet_{\rm PV}(V,M_n) \rightarrow \coker^\bullet(\iota) \rightarrow 0$ induces the usual long exact sequence in cohomology, 
	\begin{equation}
		\dots \rightarrow H^0(\coker(\iota)) \rightarrow H^1_{\rm PV}(M_n,M_n) \rightarrow H^1_{\rm PV}(V,M_n) \rightarrow H^1(\coker(\iota)) \rightarrow \dots~.
	\end{equation}

 First, we show that $H^1_{\rm PV}(V,M_n)\rightarrow H^1(\coker(\iota))$ is injective. By Lemma \ref{lem:Virasoro variational is 1d}, $H^1_{\rm PV}(V,M_n)$ is spanned by the class $[C]$ and since 
 \[
	C_\lambda(T,T^n) = nT^{n-1} \lambda^3 \neq0~,
 \]
 its image in $\coker^1(\iota)$ is nonzero. Because $C^\bullet_{\rm PV}(V,M_n)\rightarrow \coker^\bullet(\iota)$ is surjective the image of $[C]$ in $H^1(\coker(\iota))$ is nonzero and $H^1_{\rm PV}(V,M_n)\rightarrow H^1(\coker(\iota))$ is injective.

 Therefore, by exactness, the map $H^1_{\rm PV}(M_n,M_n)\rightarrow H^1_{\rm PV}(V,M_n)$ is identically zero. And so, again by exactness,
 \[
 H^1_{\rm PV}(M_n,M_n) = \im\big( H^0(\coker(\iota)) \rightarrow H^1_{\rm PV}(M_n,M_n)\big)~,
 \]
 But, by Lemma \ref{lem: Ham vect fields of ideal}, $H^0(\coker(\iota))=0$ and thus $H^1_{\rm PV}(M_n,M_n)=0$, as desired.
\end{proof}

\subsubsection{}
Combining the results of the preceding sections, we have:

\begin{prop}\label{prop:coisson cohomology of minimal model}
 For any $n\in \IN$, 
 \[
 H^1_{\rm c}(M_n,M_n) =0 ~,
 \]
 and, therefore
 \[
 H^1_{\rm ch}(\MM_n,\MM_n) =0~.
 \]
\end{prop}
\begin{proof}
	Recall from \eqref{eq:H1 coisson is fixed by E3}, that $H^1_c(M_n,M_n) = E_3^{1,0}\oplus E_3^{0,1}$ for the spectral sequence arising from the bicomplex structure of Proposition \ref{Coisbicxprop}. But by Lemma \ref{lem:E2 Harrison vanishes for min model} $E_2^{0,1}=0$ and so $E_3^{0,1}=0$.

	Furthermore, $E_2^{1,0} = H_{\rm PV}^1(M_n,M_n)$, which vanishes by Proposition \ref{prop:var of min model vanishes}. Therefore, $E_3^{1,0}=0$ and we conclude that $H^1_{c}(M_n,M_n)=0$.

	Note that the vertex algebra $\MM_n$ has a split filtration induced by assigning the generator, $T$, with weight one under $\IG_\hbar$. Moreover, the associated graded of $\MM_n$, with respect to this filtration, can be identified with $M_n$, as vertex Poisson algebras. This is the content of Lemma \ref{lem:filtered version of Virasoro}.

	Therefore, by Proposition \ref{prop:bound on cohomology}, the dimension of $H^1_\ch(\MM_n,\MM_n)$ is bounded by the dimension of $H^1_c(M_n,M_n)$, giving the desired result.

\end{proof}

This gives us two immediate corollaries.

\begin{theo}
	The Virasoro minimal model $\MM_{\hbar,n}$ is the unique (up to isomorphism), graded quantization of the vertex Poisson algebra $M_n$.
\end{theo}
\begin{proof}
	Follows from Corollary \ref{cor:unique graded quantization} noting that the classical cohomology vanishes by Proposition \ref{prop:coisson cohomology of minimal model}.
\end{proof}

\begin{theo}\label{thm:rigidity of minimal model}
	The Virasoro minimal model $\MM_n$ is rigid and admits no deformations.
\end{theo}
\begin{proof}
	Follows from the vanishing of $H^1_\ch(\MM_n,\MM_n)$ in Proposition \ref{prop:coisson cohomology of minimal model}.
\end{proof}

%

\bibliographystyle{amsalpha}
\bibliography{refsdef}

\end{document}